\documentclass[11pt,reqno]{amsart}
\usepackage{amsmath, amssymb}
\usepackage{hyperref}
\newcommand{\df}{\dfrac}
\newcommand{\tf}{\tfrac}

\renewcommand{\Re}{\operatorname{Re}}

\newcommand{\s}{{\sigma}}

\newcommand*\pFq[5]{{}_{#1}F_{#2}\left(\genfrac{}{}{0pt}{}{#3}{#4};{#5}\right)}

 \renewcommand{\a}{\alpha}
\renewcommand{\b}{\beta}

\renewcommand{\d}{{\delta}}
\newcommand{\g}{\gamma}
\newcommand{\G}{\Gamma}
\renewcommand{\l}{\lambda}

\newcommand{\z}{\zeta}

\renewcommand{\(}{\left(}
\renewcommand{\)}{\right)}
\let\dotlessi=\i
\newcommand{\iy}{\infty}
\allowdisplaybreaks[4]
\renewcommand{\pmod}[1]{\,(\textup{mod}\,#1)}
\numberwithin{equation}{section}
 \theoremstyle{plain}
\newtheorem{theorem}{Theorem}[section]
\newtheorem{lemma}[theorem]{Lemma}
\newtheorem{corollary}[theorem]{Corollary}

\newtheorem{definition}[theorem]{Definition}

\newtheorem{entry}[theorem]{Entry}

\newcommand{\twopartdef}[4]
{
        \left\{
                \begin{array}{ll}
                        #1 & \mbox{if } #2 \\
                        #3 & \mbox{if } #4
                \end{array}
        \right.
}

\begin{document}

\title[New Pathways and Connections in Number Theory and Analysis]{New Pathways and Connections in Number Theory and Analysis 
Motivated by Two Incorrect Claims of Ramanujan}
\author{Bruce C.~Berndt}
\thanks{The first author's research was partially supported by NSA grant H98230-11-1-0200 and a Simons Foundation
Collaboration Grant.}
\address{Department of Mathematics, University of Illinois, 1409 West Green
Street, Urbana, IL 61801, USA} \email{berndt@illinois.edu}

\author{Atul Dixit}
\thanks{The second author is funded in part by the grant NSF-DMS 1112656 of Professor
Victor H.~Moll of Tulane University and sincerely thanks him for the support.}
\address{Department of Mathematics, Tulane University, New Orleans, LA 70118, USA}
\email{adixit@tulane.edu}
\curraddr{Department of Mathematics, Indian Institute of Technology Gandhinagar, Ahmedabad 382424, India}
\email{adixit@iitgn.ac.in}

\author{Arindam Roy}
\address{Department of Mathematics, University of Illinois, 1409 West Green
Street, Urbana, IL 61801, USA} \email{roy22@illinois.edu}
\curraddr{Department of Mathematics, Rice University, 6100 Main St, Houston, 77005, USA}
\email{arindam.roy@rice.edu}

 \author{Alexandru Zaharescu}
\address{Department of Mathematics, University of Illinois, 1409 West Green
Street, Urbana, IL 61801, USA, and Institute of Mathematics of the Romanian
Academy, P.O.~Box 1-764, Bucharest RO-70700, Romania}
\email{zaharesc@illinois.edu}


\dedicatory{
\begin{center}
In Memory of W.~Keith Moore, Professor of Mathematics at Albion College
\end{center}
Dedicated to Pratibha Kulkarni who first showed me how beautiful Mathematics is
\begin{center}
In Memory of Mohanlal S. Roy, Professor at Ramakrishna Mission
Vidyamandira
\end{center}
\begin{center}
In Memory of Professor Nicolae Popescu
\end{center}
}

















\begin{abstract} The focus of this paper commences with an
  examination of  three (not obviously related) pages in Ramanujan's lost notebook, pages 336, 335,
  and 332, in decreasing order of attention.
    On page 336,
   Ramanujan proposes two identities, but the formulas are wrong  --
  each is vitiated by divergent series. We concentrate on only one of the two
incorrect ``identities,'' which
   may have been devised to attack the extended divisor
  problem.   We prove here a corrected version of Ramanujan's
  claim, which contains the convergent series appearing in it.   The
  convergent series in Ramanujan's faulty claim is similar to one used by
  G.~F.~Vorono\"{\dotlessi}, G.~H.~Hardy, and others in their study of the
  classical Dirichlet divisor problem.   This now brings us to page 335, which comprises two
  formulas featuring doubly infinite series of Bessel functions, the first being conjoined with the
  classical circle problem initiated by Gauss, and the second being associated
  with the Dirichlet divisor problem.  The first and fourth authors, along
  with Sun Kim, have written several papers providing proofs of these two
  difficult formulas in different interpretations.  In this monograph, we
  return to these two formulas and examine them in more general settings.

    The aforementioned convergent series in Ramanujan's ``identity'' is also similar to one that
  appears in a curious identity found in Chapter 15 in Ramanujan's second
  notebook,  written in a more elegant, equivalent
  formulation on page 332 in the lost notebook.  This formula may be regarded
  as a formula for $\zeta(\tf12)$, and in 1925, S. Wigert obtained a
  generalization giving a formula for $\zeta(\tf1k)$ for any even integer
  $k\geq2$.  We extend the work of Ramanujan
  and Wigert in this paper.

The Vorono\"{\dotlessi} summation formula appears prominently in our study.  In particular, we generalize work of J.~R.~Wilton and derive an analogue involving the sum of divisors function $\sigma_s(n)$.

The modified Bessel
functions $K_{s}(x)$  arise in
several contexts, as do Lommel functions.   We establish here new series and integral identities
involving modified Bessel functions and modified Lommel functions.  Among other results, we establish a modular transformation for an infinite series involving $\sigma_{s}(n)$ and modified Lommel functions.  We also discuss certain
obscure related work of N.~S.~Koshliakov.  We
define and discuss two new related classes of integral transforms, which we
call Koshliakov transforms, because he first found elegant
special cases of each.

\end{abstract}
\maketitle
\tableofcontents


\section{Introduction}\label{sect1}
The Dirichlet divisor problem is one of the most notoriously difficult unsolved problems in analytic number theory.  Let $d(n)$ denote the number of divisors of $n$. Define the error term $\Delta(x)$, for $x>0$, by
\begin{equation}\label{ddp}
\sideset{}{'}\sum_{n\leq x}d(n)=x\log x+(2\gamma-1)x+\frac{1}{4}+\Delta(x),
\end{equation}
where $\gamma$ denotes Euler's constant.  Here, and in the sequel, a prime $'$ on the summation
sign in ${\sum_{n\leq x}}^{\prime}a(n)$ indicates that only $\frac{1}{2}a(x)$ is counted when $x$ is an
integer. The Dirichlet divisor problem asks for the correct order of
magnitude of $\Delta(x)$ as $x\to\infty$. At this writing, the
best estimate $\Delta(x) =O(x^{131/416+\epsilon})$, for each
$\epsilon>0$, as $x\to\iy$,  is due to M.~N.~Huxley \cite{huxley2}
($\tf{131}{416}=0.3149\dots$).  On the other hand, G.~H.~Hardy \cite{hardiv}
proved that $\Delta(x)\neq O(x^{1/4})$, as $x\to\infty$, with the best result
in this direction currently due to K.~Soundararajan \cite{sound}. It is conjectured that
$\Delta(x)=O\left(x^{1/4+\epsilon}\right)$, for each $\epsilon>0$, as $x\to\iy$.

The conditionally convergent series
\begin{equation}\label{impser}
\frac{x^{1/4}}{\pi\sqrt{2}}\sum_{n=1}^{\infty}\frac{d(n)}{n^{\frac{3}{4}}}\cos\left(4\pi\sqrt{nx}-\frac{\pi}{4}\right)
\end{equation}
arose in  G.~F.~Vorono\"{\dotlessi}'s \cite[p.~218]{voronoi} work on the
Dirichlet divisor problem, and its importance was further emphasized by Hardy
\cite[equation (6.32)]{hardiv}.  Moreover,
J.~L.~Hafner \cite{hafner} and   Soundararajan \cite[equation (1.8)]{sound}
in their improvements of Hardy's $\Omega$-theorem on the Dirichlet divisor
problem also crucially employed \eqref{impser}.


Let $\sigma_s(n)=\sum_{d|n}d^s$, and let $\zeta(s)$ denote the Riemann zeta function.  For $0<s<1$, define $\Delta_{-s}(x)$ \footnote{We use $\Delta_{-s}(x)$ instead of $\Delta_{s}(x)$, as is customarily used, so as to be consistent with the results in this paper, most of which require Re $s>0$.} by
\begin{equation}\label{gddp}
\sideset{}{'}\sum_{n\leq x}\sigma_{-s}(n)=\zeta(1+s)x+\frac{\zeta(1-s)}{1-s}x^{1-s}-\frac{1}{2}\zeta(s)+\Delta_{-s}(x).
\end{equation}
The problem of determining the correct order of magnitude of the error term $\Delta_{-s}(x)$, as $x\to\infty$,  is known as the extended divisor problem \cite{lau}.  As $x\to\iy$, it is conjectured that for each $\epsilon>0$, $\Delta_{-s}(x)= O(x^{1/4-s/2+\epsilon})$ for $0<s\leq\frac{1}{2}$
and $\Delta_{-s}(x)= O(x^{\epsilon})$ for $\frac{1}{2}\leq s<1$.

 In analogy with  \eqref{impser}, the series
\begin{equation}\label{cw}
\sum_{n=1}^{\infty}\frac{\s_{k}(n)}{n^{\frac{5}{4}+\frac{k}{2}}}\sin\left(4\pi\sqrt{nx}-\frac{\pi}{4}\right),
\end{equation}
for $|k|<\frac{3}{2}$, arises in work \cite[p.~282]{segal}, \cite{kanrao} related to a conjecture of S.~Chowla and H.~Walum \cite{chowwal}, \cite[pp.~1058--1063]{chowlacp}, which is another extension of the Dirichlet divisor problem. It is conjectured that if $a, r\in\mathbb{Z}, a\geq 0, r\geq 1$, and if $B_{r}(x)$ denotes the $r$-th Bernoulli polynomial, then for every $\epsilon>0$, as $x\to\infty$,
\begin{equation}\label{extddp}
\sum_{n\leq\sqrt{x}}n^{a}B_{r}\left(\left\{\frac{x}{n}\right\}\right)=O\left(x^{a/2+1/4+\epsilon}\right),
\end{equation}
where $\{x\}$ denotes the fractional part of $x$.
The conjectured correct order of magnitude in the Dirichlet divisor problem is equivalent to (\ref{extddp}) with $a=0, r=1$.

Our last example is as famous as the Dirichlet divisor problem with which we opened this paper. Let $r_2(n)$ denote the number of representations of $n$ as a sum of two
squares.  The equally celebrated circle problem asks for the precise order of magnitude
of the error term $P(x)$, as $x\to\infty$, where
\begin{equation*}
\sideset{}{'}\sum_{n\leq x}r_2(n)=\pi x +P(x).
\end{equation*}

During the five years that Ramanujan visited Hardy at Cambridge, there is
considerable evidence, from Hardy in his papers and from Ramanujan in his
lost notebook \cite{lnb}, that the two frequently discussed both the circle
and divisor problems.
For details of Ramanujan's contributions to these
problems, see either the first author's book with G.~E.~Andrews \cite[Chapter
2]{ab4} or the survey paper by the first author, S.~Kim, and the last author
\cite{besselsurvey}.

 It is possible that Ramanujan also thought of the extended divisor problem, for on page $336$ in his lost notebook \cite{lnb}, we find the following claim.

\textit{Let $\sigma_s(n)=\sum_{d|n}d^s$, and let $\zeta(s)$ denote the Riemann zeta function. Then
{\allowdisplaybreaks\begin{align}\label{qc}
&\Gamma\left(s+\frac{1}{2}\right)
\bigg\{\frac{\zeta(1-s)}{(s-\frac{1}{2})x^{s-\frac{1}{2}}}+\frac{\zeta(-s)\tan\frac{1}{2}\pi s}{2x^{s+\frac{1}{2}}}\nonumber\\
&\quad+\sum_{n=1}^{\infty}\frac{\sigma_s(n)}{2i}\left((x-in)^{-s-\frac{1}{2}}-(x+in)^{-s-\frac{1}{2}}\right)\bigg\}
\nonumber\\
&=(2\pi)^s\bigg\{\frac{\zeta(1-s)}{2\sqrt{\pi x}}-2\pi\sqrt{\pi x}\zeta(-s)\tan\tfrac{1}{2}\pi s\nonumber\\
&\quad+\sqrt{\pi}\sum_{n=1}^{\infty}\frac{\sigma_s(n)}{\sqrt{n}}e^{-2\pi\sqrt{2nx}}
\sin\left(\frac{\pi}{4}+2\pi\sqrt{2nx}\right)\bigg\}.
\end{align}}}%

In view of the identities for \eqref{impser} and \eqref{cw}, it is possible that Ramanujan developed the series on the right-hand side of (\ref{qc}) to study a generalized divisor problem.  Unfortunately, \eqref{qc} is incorrect, since the series on the left-hand side, which can be written as
$$\displaystyle\sum_{n=1}^{\infty}\frac{\sigma_s(n)\sin\left(\left(s+\tfrac{1}{2}\right)
\tan^{-1}\left(\frac{n}{x}\right)\right)}{(x^2+n^2)^{\frac{s}{2}+\frac{1}{4}}},$$
 diverges for all real values of $s$ since $\sigma_{s}(n)\geq n^s$. See
 \cite{bclz} for further discussion. In this paper, we obtain a corrected
 version of
 Ramanujan's claim, where we start with the series on the right-hand side,
 since we know that it converges.

  Before stating our version, we need to define a general hypergeometric function.
  Define the rising or shifted factorial  $(a)_n$ by
  \begin{equation}\label{rising}
  (a)_n=a(a+1)(a+2)\cdots(a+n-1), \qquad n\geq1, \qquad (a)_0=1.
  \end{equation}
  Let $p$ and $q$ be non-negative integers, with $q\leq p+1$.
  Then, the generalized hypergeometric function $_qF_p$ is defined by
   \begin{equation}\label{hyper}
   _qF_p(a_1,a_2,\dots,a_q;b_1,b_2,\dots,b_p;z)
   :=\sum_{n=0}^{\iy}\df{(a_1)_n(a_2)_n\cdots(a_q)_n}{(b_1)_n(b_2)_n\cdots(b_p)_n}\df{z^n}{n!},
   \end{equation}
   where $|z|<1$, if $q=p+1$, and $|z|<\iy$, if $q<p+1$.

We emphasize further notation.  Throughout the paper, $s=\sigma+it$, with $\sigma$ and $t$ both real.  We also set $R_a(f)=R_a$ to denote the residue of a meromorphic function $f(z)$ at a pole $z=a$.

\begin{theorem}\label{p3361}
Let $_3F_2$ be defined by \eqref{hyper}.  Fix $s$ such that $\sigma>0$. Let $x\in\mathbb{R}^{+}$. Let $a$ be the number defined by
\begin{equation}\label{a}
a=
\begin{cases}
0, \qquad \text{if}\hspace{1.5mm} s\hspace{1.5mm}\text{is an odd integer},\\
1, \qquad \text{otherwise}.
\end{cases}
\end{equation}
Then,
{\allowdisplaybreaks\begin{align}\label{tf}
&\sum_{n=1}^{\infty}\frac{\sigma_s(n)}{\sqrt{n}}e^{-2\pi\sqrt{2nx}}
\sin\left(\frac{\pi}{4}+2\pi\sqrt{2nx}\right)\notag\\
&=4\pi\left(\frac{\zeta(1-s)}{8\pi^2\sqrt{x}}+\frac{1}{4\sqrt{2}\pi}\zeta
\left(\frac{1}{2}\right)\zeta\left(\frac{1}{2}-s\right)-\frac{2^{-s-3}}{\pi^{s+\frac{3}{2}}}
\frac{\G(s+1/2)\cot\(\frac{\pi s}{2}\)\zeta(-s)}{x^{s+\frac{1}{2}}}\right)\nonumber\\
&\quad+\frac{\sqrt{x}}{\pi^{s}}\bigg\{\sum_{n<x}\frac{\s_{s}(n)}{n^{s+1}}
\left[-\frac{\sqrt{n}\G\(\frac{1}{4}+\frac{s}{2}\)}{\sqrt{2x}\G\(\frac{1}{4}-\frac{s}{2}\)}\right.\nonumber\\
&\quad\left.-\frac{a\G\(s+\frac{1}{2}\)\cot\(\frac{\pi s}{2}\)}{2^{s+1}\sqrt{\pi}}\left(\frac{n}{x}\right)^{s+1}
\left\{\(1+\frac{in}{x}\)^{-\(s+\frac{1}{2}\)}+\(1-\frac{in}{x}\)^{-\(s+\frac{1}{2}\)}\right\}\right.\nonumber\\
&\quad+\left.\frac{n2^{-s}}{x\sin\left(\frac{\pi
        s}{2}\right)\G(1-s)}\pFq{3}{2}{\frac{1}{4}, \frac{3}{4},
    1}{\frac{1-s}{2}, 1-\frac{s}{2}}{-\frac{n^2}{x^2}}\right]\nonumber\\
&\quad+\sum_{n\geq
  x}\frac{\s_{s}(n)}{n^{s+1}}\left[-\frac{n\G(s)\cos\(\frac{\pi s}{2}\)}{
    2^{s-1}\pi x}\left\{\pFq{3}{2}{\frac{s}{2},\frac{1+s}{2}, 1}{\frac{1}{4},
      \frac{3}{4}}{-\frac{x^2}{n^2}}-1\right\}\right.\nonumber\\
&\quad-\frac{i\sqrt{n}\G\(s+\frac{1}{2}\)}{2^{s+1}\sqrt{\pi
    x}}\left\{\sin\(\frac{\pi}{4}+\frac{\pi
    s}{2}\)\left(\left(1+\frac{ix}{n}\right)^{-(s+\frac{1}{2})}
-\left(1-\frac{ix}{n}\right)^{-(s+\frac{1}{2})}\right)\right.
\nonumber\\
&\quad+\left.\left.i\cos\(\frac{\pi}{4}+\frac{\pi
      s}{2}\)\left(\left(1+\frac{ix}{n}\right)^{-(s+\frac{1}{2})}
+\left(1-\frac{ix}{n}\right)^{-(s+\frac{1}{2})}-2\right)\right\}\right]\bigg\},
\end{align}}%
where, if $x$ is an integer, we additionally require that $\sigma<\frac{1}{2}$.
\end{theorem}

The following lemma, which is interesting in its own right, is the main
ingredient of our proof. We use the notation $\int_{(c)}$ to designate
$\int_{c-i\infty}^{c+i\infty}$.
\begin{lemma}\label{lemma1}
Fix $s$ such that $\sigma>0$. Fix $x\in\mathbb{R}^{+}$. Let $-1<\l<0$ and let
$a$ be defined by \textup{(\ref{a})}. Define $I(s,x)$ by
\begin{align}\label{Is}
I(s,x)&:=\frac{1}{2\pi i}\int_{(\l)}\G(z-1)\G\left(1-\frac{z}{2}\right)
\G\left(1-\frac{z}{2}+s\right)\notag\\&\qquad\times
\sin^2\left(\frac{\pi z}{4}\right)
\sin\left(\frac{\pi z}{4}-\frac{\pi s}{2}\right)
\left(4x\right)^{-\frac{1}{2}z}\,dz.
\end{align}
Then,\\
\textup{(i)} for $x>1$,
{\allowdisplaybreaks\begin{align}\label{ig1}
I(s,x)&=-\frac{\pi}{2^{2-s}}\left[\frac{\G\(\frac{1}{4}
+\frac{s}{2}\)}{\sqrt{2x}\G\(\frac{1}{4}-\frac{s}{2}\)}+\frac{ax^{-s-1}\cot\(\frac{\pi
      s}{2}\)}{2^{s+1}\sqrt{\pi}}\G\(s+\frac{1}{2}\)\left\{
\(1+\frac{i}{x}\)^{-\(s+\frac{1}{2}\)}\right.\right.\notag
\\&\quad\left.+\(1-\frac{i}{x}\)^{-\(s+\frac{1}{2}\)}\right\}
-\left.\frac{1}{x 2^{s}\sin\left(\frac{\pi
s}{2}\right)\G(1-s)}\pFq{3}{2}{\frac{1}{4}, \frac{3}{4}, 1}{\frac{1-s}{2};
1-\frac{s}{2}}{-\frac{1}{x^2}}\right];
\end{align}}

\textup{(ii)} for $x\leq 1$,
{\allowdisplaybreaks\begin{align}\label{ig2}
I(s,x)&=-\frac{\pi}{2^{2-s}}\biggl[\frac{\G(s)\cos\(\frac{\pi s}{2}\)}{ 2^{s-1}\pi x}\left\{\pFq{3}{2}{\frac{s}{2},\frac{1+s}{2}, 1}{\frac{1}{4}, \frac{3}{4}}{-x^2}-1\right\}\nonumber\\
&\quad+\frac{i\G\(s+\frac{1}{2}\)}{2^{s+1}\sqrt{\pi x}}\left\{\sin\(\frac{\pi}{4}+\frac{\pi s}{2}\)\left((1+ix)^{-(s+\frac{1}{2})}-(1-ix)^{-(s+\frac{1}{2})}\right)\right.\nonumber\\
&\quad+\left.i\cos\(\frac{\pi}{4}+\frac{\pi s}{2}\)\left((1+ix)^{-(s+\frac{1}{2})}+(1-ix)^{-(s+\frac{1}{2})}-2\right)\right\}\biggr],
\end{align}}%
where, if $x=1$, we additionally require that  $\sigma<\frac{1}{2}$.
\end{lemma}
 We note in passing that each ${}_3F_{2}$ in Theorem
\ref{p3361}, as well as in Lemma \ref{lemma1}, can be written, using the
duplication formula for the Gamma function (see \eqref{dup} below), as a sum
of two ${}_2F_{1}$'s.

If we replace the `+' sign in the argument of the sine function in the series
on the left-hand side of (\ref{tf})  by a `$-$' sign, then we obtain the
following theorem.

\begin{theorem}\label{p3361a}
Fix $s$ such that $\sigma>0$. Let $x\in\mathbb{R}^{+}$. Then,
{\allowdisplaybreaks\begin{align}\label{tf-}
&\sum_{n=1}^{\infty}\frac{\sigma_s(n)}{\sqrt{n}}e^{-2\pi\sqrt{2nx}}
\sin\left(\frac{\pi}{4}-2\pi\sqrt{2nx}\right)\nonumber\\
&=4\pi\left(\frac{\sqrt{x}}{2}\zeta(-s)
+\frac{\zeta\left(\frac{1}{2}\right)}{4\pi\sqrt{2}}\zeta
\left(\frac{1}{2}-s\right)+\frac{\Gamma\left(s+\frac{1}{2}\right)
\zeta(-s)}{2^{s+3}\pi^{s+\frac{3}{2}}x^{s+\frac{1}{2}}}\right)\nonumber\\
&\quad+\frac{\Gamma\left(s+\frac{1}{2}\right)}{2^{s}\pi^{s+\frac{1}{2}}}
\bigg\{\sum_{n<x}\frac{\sigma_{s}(n)}{n^{s+\frac{1}{2}}}\left[-\sin\left(\frac{\pi}{4}-\frac{\pi
      s}{2}\right)+\frac{n^{s+\frac{1}{2}}}{2x^{s+\frac{1}{2}}}
\right.\notag\\
&\hspace{1.7in}\left.\times
\left(\left(1+\frac{in}{x}\right)^{-\left(s+\frac{1}{2}\right)}
+\left(1-\frac{in}{x}\right)^{-\left(s+\frac{1}{2}\right)}\right)\right]\nonumber\\
&\quad+\dfrac12\sum_{n\geq x}\frac{\sigma_{s}(n)}{n^{s+\frac{1}{2}}}\bigg[\cos\left(\frac{\pi}{4}+\frac{\pi s}{2}\right)\left(\left(1+\frac{ix}{n}\right)^{-\left(s+\frac{1}{2}\right)}
+\left(1-\frac{ix}{n}\right)^{-\left(s+\frac{1}{2}\right)}-2\right)\nonumber\\
&\quad+i\sin\left(\frac{\pi}{4}+\frac{\pi s}{2}\right)\left(\left(1+\frac{ix}{n}\right)^{-\left(s+\frac{1}{2}\right)}
-\left(1-\frac{ix}{n}\right)^{-\left(s+\frac{1}{2}\right)}\right)\bigg]\bigg\}.
\end{align}}
\end{theorem}
On page 332 in his lost notebook \cite{lnb}, Ramanujan gives an elegant reformulation of a formula for $\zeta\left(\tfrac{1}{2}\right)$ that appears  in Chapter $15$ of his second notebook \cite{secnb}, \cite[p.~314, Entry 8]{II}.

\textit{Let $\alpha$ and $\beta$ be two positive numbers such that $\alpha\beta=4\pi^3$. If $\phi(n)$, $n\geq 1$, and $\psi(n), n\geq 1$, are defined by
\begin{equation*}
\sum_{j=1}^{\infty}\frac{x^{j^2}}{1-x^{j^2}}=\sum_{n=1}^{\infty}\phi(n)x^n
\end{equation*}
and
\begin{equation}\label{d2}
\sum_{j=1}^{\infty}\frac{jx^{j^2}}{1-x^{j^2}}=\sum_{n=1}^{\infty}\psi(n)x^n,
\end{equation}
respectively, then\footnote{Ramanujan inadvertently omitted the term $\tf14$ on the right-hand side of (\ref{p332}).}
\begin{align}\label{p332}
\sum_{n=1}^{\infty}\phi(n)e^{-n\alpha}&=\frac{\pi^2}{6\alpha}
+\frac{1}{4}+\frac{\sqrt{\beta}}{\pi\sqrt{2}}\left(\frac{1}{2\sqrt{2}}\zeta\left(\frac{1}{2}\right)
+\sum_{n=1}^{\infty}\frac{\psi(n)}{\sqrt{n}}e^{-\sqrt{n\beta}}
\sin\left(\frac{\pi}{4}-\sqrt{n\beta}\right)\right).
\end{align}}%
Recall that \cite[p.~340, Theorem 3.10]{Hardy}
$$\sum_{n=1}^{\infty}d(n)x^n=\sum_{n=1}^{\infty}\df{x^n}{1-x^n}.$$
A similar elementary argument shows that
$$ \sum_{n=1}^{\infty}\sigma(n)x^n=\sum_{n=1}^{\infty}\df{nx^n}{1-x^n}.$$
Hence, we see that $\phi(n)$ and $\psi(n)$ are analogues of $d(n)$ and $\sigma(n)$,
respectively. The identity \eqref{p332} was rediscovered by S.~Wigert
\cite[p.~9]{wig}, who actually gave a general formula for
$\zeta\left(\frac{1}{k}\right)$ for each positive even integer $k$. See
\cite[pp.~191--193]{ab4} for more details about \eqref{p332}.

We have included \eqref{p332} here to demonstrate the similarity in the
structure
of the series on its right-hand side with the series on the left-hand sides
of (\ref{tf}) and (\ref{tf-}). One might therefore ask if other arithmetic
functions, analogous to $\sigma_{s}(n)$ in (\ref{tf}) and $\psi(n)$ in
(\ref{p332}), produce interesting series identities like those in (\ref{tf}) and
(\ref{p332}).

The special case $s=\tf12$ of Theorem \ref{p3361} (see (\ref{corhalf})) is very interesting, since the two sums, one over $n<x$ and the other over $n\geq x$, coalesce into a single infinite sum.  If $K_{s}(x)$ denotes the modified Bessel function or the Macdonald function \cite[p.~78]{watsonbessel} of order $s$, and if we use the identities
 \cite[p.~80, equation (13)]{watsonbessel}
\begin{equation}\label{K1}
K_{1/2}(z)=\sqrt{\frac{\pi}{2 z}}e^{-z}
\end{equation}
and  \cite[p.~79, equation (8)]{watsonbessel}
\begin{equation}\label{K2}
K_{-s}(z)=K_{s}(z),
\end{equation} we see that this special case of
the series on the left-hand side of (\ref{tf}) can be realized as a special case of the series
\begin{equation}\label{minus}
2\sum_{n=1}^{\infty}\sigma_{-s}(n)n^{\frac{1}{2}s}\left(e^{\pi is/4}K_{s}\left(4\pi e^{\pi i/4}\sqrt{nx}\right)-e^{-\pi is/4}K_{s}\left(4\pi e^{-\pi i/4}\sqrt{nx}\right)\right)
\end{equation}
when $s=-\tf12$.
If we replace the minus sign by a plus sign between the Bessel functions in the summands of \eqref{minus}, then the resulting series
is a generalization of the series
\begin{equation}\label{minusspl}
\varphi(x):=2\sum_{n=1}^{\infty}d(n)\left(K_{0}\left(4\pi e^{i\pi/4}\sqrt{nx}\right)+K_{0}\left(4\pi e^{-i\pi/4}\sqrt{nx}\right)\right),
\end{equation}
extensively studied by N.~S.~Koshliakov (also spelled N.~S.~Koshlyakov) \cite{koshliakov, koshlond, koshli, kosh1937}.
 See also \cite{dixitmoll} for properties of this series and some integral transformations involving it.  The authors of this paper feel that Koshliakov's work has not earned the respect that it deserves in the mathematical community.  Some of his best work was achieved under extreme hardship, as these excerpts from a paper written for the centenary of his birth clearly demonstrate \cite{kosh}.

\begin{quote} The repressions of the thirties which affected scholars in Leningrad continued even after the outbreak of the Second World War. In the winter of 1942 at the height of the blockade of Leningrad, Koshlyakov along with a group \dots was arrested on fabricated \dots dossiers and condemned to 10 years correctional hard labour.  After the verdict he was exiled to one of the camps in the Urals. \dots On the grounds of complete exhaustion and complicated pellagra, Koshlyakov was classified in the camp as an invalid and was not sent to do any of the usual jobs. \dots very serious shortage of paper.  He was forced to carry out calculations on a piece of plywood, periodically scraping off what he had written with a piece of glass. Nevertheless, between 1943 and 1944 Koshlyakov wrote two long memoirs \dots \end{quote}

A natural question arises -- what may have motivated Ramanujan to consider the series
\begin{equation}\label{ramser}
\sum_{n=1}^{\infty}\frac{\sigma_s(n)}{\sqrt{n}}e^{-2\pi\sqrt{2nx}}
\sin\left(\frac{\pi}{4}+2\pi\sqrt{2nx}\right)?
\end{equation}
We provide a plausible answer to this question in Section \ref{sect5}, demonstrating that \eqref{ramser} is related to a generalization of the famous Vorono\"{\dotlessi} summation formula and also to the generalization of Koshliakov's series \eqref{minusspl} discussed above and its analogue.

This paper is organized as follows. The preliminary results are given in
Section \ref{sect2}. The proof of Theorem \ref{p3361} appears in Section
\ref{sect3}. We do not give a proof of Theorem \ref{p3361a} since it is
similar to that of Theorem \ref{p3361}. Lemma \ref{lemma1}, which is crucial
in the proof of Theorem \ref{p3361}, is derived in Section \ref{sect4}.  Special cases
of Theorems
\ref{p3361} and \ref{p3361a} are examined in Section \ref{bessie}, and
connections  with modified Bessel functions are made. In
Section \ref{sect5}, we relate \eqref{ramser} and Theorem \ref{p3361} to
Vorono\"{\dotlessi}'s formula for ${\sum_{n\leq x}}^{\prime}d(n)$ and work of
Hardy, Koshliakov, and A.~Oppenheim.  In the following section, we examine an analogue of the
Vorono\"{\dotlessi} summation formula with $d(n)$ replaced by $\sigma_s(n)$.  The work of Ramanujan \cite{lnb} and Wigert \cite{wig}, evinced in \eqref{p332}, is
extended in Section \ref{sect9}.  On page 335 in his lost notebook
\cite{lnb}, Ramanujan stated two beautiful identities connected,
respectively, with the circle and divisor problems.  We extend these
identities in Sections \ref{sect10}--\ref{sect13}.  The linear combination of Bessel functions appearing in our representation for $\sum_{n\leq x}\sigma_{-s}(n)$ was remarkably shown by Koshliakov \cite{kosh1938} to be the kernel of an integral transform for which the modified Bessel function $K_{\nu}(x)$ is self-reciprocal.  We study these transforms in Section \ref{ktmt}.

\section{Preliminary Results}\label{sect2}

We recall below the functional equation, the reflection formula (along with a
variant), and Legendre's duplication formula for the Gamma function
$\G(s)$. To that end,
{\allowdisplaybreaks\begin{align}
 \G(s+1)&=s\G(s),\label{feg}\\
 \G(s)\G(1-s)&=\frac{\pi}{\sin(\pi s)},\label{ref}\\
 \G\left(\frac{1}{2}+s\right)\G\left(\frac{1}{2}-s\right)&=\frac{\pi}{\cos(\pi s)},\label{ref2}\\
 \G(s)\G\left(s+\frac{1}{2}\right)&=\frac{\sqrt{\pi}}{2^{2s-1}}\G(2s).\label{dup}
\end{align}}%
 Throughout the paper, we shall need Stirling's formula for the Gamma function in a vertical strip \cite[p.~224]{cop}.  Thus, for $\sigma_1\leq\sigma\leq\sigma_2$, as $|t|\to\infty$,
 \begin{equation}\label{strivert}
 |\Gamma(s)|=\sqrt{2\pi}|t|^{\sigma-1/2}e^{-\pi|t|/2}\left(1+O\left(\df{1}{|t|}\right)\right).
 \end{equation}

The functional equation of the Riemann zeta function $\z(s)$ in its
asymmetric form is given by \cite[p.~24]{titch}
\begin{equation}\label{fe}
 \zeta(1-s)=2^{1-s}\pi^{-s}\cos\left(\tfrac{1}{2}\pi s\right)\G(s)\zeta(s),
\end{equation}
whereas its symmetric form takes the shape
\begin{equation}\label{fesym}
\pi^{-s/2} \Gamma \left( \tfrac{1}{2}s \right) \zeta(s) =
\pi^{-(1-s)/2} \Gamma \left( \tfrac{1}{2}(1-s) \right) \zeta(1-s).
\end{equation}
Since $\zeta(s)$ has a simple pole at $s=1$ with residue 1, i.e.,
\begin{equation}\label{zlim}
\lim_{s\to 1}(s-1)\zeta(s)=1,
\end{equation}
from \eqref{fe} and \eqref{zlim}, we find the value \cite[p.~19]{titch}
\begin{equation*}
\zeta(0)=-\tf12.
\end{equation*}
The Riemann $\xi$-function $\xi(s)$ is defined by
\begin{equation}\label{xis}
\xi(s) := \tfrac{1}{2}s(s-1)\pi^{-s/2} \Gamma \left( \tfrac{1}{2}s \right)\zeta(s),
\end{equation}
where $\Gamma(s)$ and $\zeta(s)$ are the Gamma and the Riemann zeta functions respectively. The Riemann $\Xi$-function is defined by
\begin{equation}\label{Xit}
\Xi(t) := \xi \left( \tfrac{1}{2} + i t \right).
\end{equation}
For $0<c=$ Re $w<\sigma$ \cite[p.~908, formula \textbf{8.380.3}; p.~909, formula \textbf{8.384.1}]{grn},
\begin{equation}\label{betamel}
\frac{1}{2\pi i}\int_{(c)}\frac{\Gamma(w)\Gamma(s-w)}{\Gamma(s)}x^{-w}\, dw=\frac{1}{(1+x)^{s}}.
\end{equation}
We note Parseval's identity \cite[pp.~82--83]{kp}
\begin{equation}\label{pf}
\int_{0}^{\infty}f(x)g(x)\, dx=\frac{1}{2\pi i}\int_{c-i\infty}^{c+i\infty}\mathfrak{F}(1-w)\mathfrak{G}(w)\, dw,
\end{equation}
where $\mathfrak{F}$ and $\mathfrak{G}$ are Mellin transforms of $f$ and $g$, and which is valid for Re $w=c$ lying in the common strip of analyticity of $\mathfrak{F}(1-w)$ and $\mathfrak{G}(w)$. A variant of the above identity \cite[p.~83, equation (3.1.13)]{kp} is
\begin{equation}\label{melconv}
\frac{1}{2\pi i}\int_{(k)}\mathfrak{F}(w)\mathfrak{G}(w)t^{-w}\, dw=\int_{0}^{\infty}f(x)g\left(\frac{t}{x}\right)\frac{dx}{x}.
\end{equation}

We close this section by recalling facts about Bessel functions. The ordinary Bessel function $J_{\nu}(z)$ of order $\nu$ is defined by \cite[p.~40]{watsonbessel}
\begin{align}\label{sumbesselj}
	J_{\nu}(z)=\sum_{m=0}^{\infty}\frac{(-1)^m(z/2)^{2m+\nu}}{m!\Gamma(m+1+\nu)}, \quad |z|<\infty.
	\end{align}
As customary, $Y_{\nu}(z)$ denotes the Bessel function of order $\nu$ of the second kind.  Its relation to $J_{\nu}(z)$ is given in the identity  \cite[p.~64]{watsonbessel}
\begin{align}
Y_{\nu}(z)=\frac{J_{\nu}(z)\cos(\pi \nu)-J_{-\nu}(z)}{\sin{\pi \nu}}\label{yj}.
\end{align}
 and, as above, $K_{\nu}(z)$ denotes the modified Bessel function of order $\nu$.
The asymptotic formulas of the Bessel functions $J_{\nu}(z), Y_{\nu}(z)$, and $K_{\nu}(z)$, as $|z|\to\infty$, are given  by 
\cite[p.~199 and p.~202]{watsonbessel}
\begin{align}
J_{\nu}(z)&\sim \left(\frac{2}{\pi z}\right)^{\tf12}\bigg(\cos w\sum_{n=0}^{\infty}\frac{(-1)^n(\nu, 2n)}{(2z)^{2n}} -\sin w\sum_{n=0}^{\infty}\frac{(-1)^n(\nu, 2n+1)}{(2z)^{2n+1}}\bigg),\label{asymbess}\\
Y_{\nu}(z)&\sim \left(\frac{2}{\pi z}\right)^{\tf12}\bigg(\sin w\sum_{n=0}^{\infty}\frac{(-1)^n(\nu, 2n)}{(2z)^{2n}}+\cos w\sum_{n=0}^{\infty}\frac{(-1)^n(\nu, 2n+1)}{(2z)^{2n+1}}\bigg),\label{asymbess1}\\
K_{\nu}(z)&\sim \left(\frac{\pi}{2 z}\right)^{\tf12}e^{-z}\sum_{n=0}^{\infty}\frac{(\nu, n)}{(2z)^{n}},\label{asymbess2}
\end{align}
for $|\arg z|<\pi$. Here $w=z-\tfrac{1}{2}\pi \nu-\tfrac{1}{4}\pi$ and
\begin{align*}
(\nu,n)=\frac{\Gamma(\nu+n+1/2)}{\Gamma(n+1)\Gamma(\nu-n+1/2)}.
\end{align*}
\section{Proof of Theorem \ref{p3361}}\label{sect3}
Let
\begin{equation}\label{ssx}
S(s,x):=\sum_{n=1}^{\iy} \frac{\s_s(n)}{\sqrt{n}}e^{-2\pi \sqrt{2nx}}\sin\left(\frac{\pi}{4}+2\pi \sqrt{2nx}\right).
\end{equation}
From \cite[p.~45, equations (5.19), (5.20)]{ob}, we have
\begin{align}
\frac{1}{2\pi i}\int_{(c)}\frac{\G (z)}{(a^2+b^2)^{z/2}}\sin \left(z\tan^{-1}\left(\frac{a}{b}\right)\right)x^{-z}\,dz & = e^{-bx}\sin (ax),\label{ob1}\\
\frac{1}{2\pi i}\int_{(c)}\frac{\G (z)}{(a^2+b^2)^{z/2}}\cos \left(z\tan^{-1}\left(\frac{a}{b}\right)\right)x^{-z}\,dz & = e^{-bx}\cos (ax), \label{ob2}
\end{align}
where $a,b>0$, and Re $z>0$ for (\ref{ob1}) and Re $z>-1$ for (\ref{ob2}). Let $a=b=2\pi\sqrt{2n}$, replace $x$ by $\sqrt{x}$, add (\ref{ob1}) and (\ref{ob2}), and then simplify, so that for $c=$ Re $z>0$,
\begin{equation}\label{uob}
\frac{1}{2\pi i}\int_{(c)}\frac{\G(z)}{(16\pi^2n)^{z/2}}\sin\left(\frac{\pi (z+1)}{4}\right)x^{-z/2}\, dz=e^{-2\pi\sqrt{2nx}}\sin\left(\frac{\pi}{4}+2\pi\sqrt{2nx}\right).
\end{equation}
Now replace $z$ by $z-1$ in (\ref{uob}), so that for $c=$ Re $z>1$,
\begin{equation}\label{uob1}
\frac{1}{2\pi i}\int_{(c)}\frac{\G(z-1)}{(4\pi)^{z-1}n^{z/2}}\sin\left(\frac{\pi z}{4}\right)x^{(1-z)/2}\, dz=\frac{e^{-2\pi\sqrt{2nx}}}{\sqrt{n}}\sin\left(\frac{\pi}{4}+2\pi\sqrt{2nx}\right).
\end{equation}
Now substitute (\ref{uob1}) in (\ref{ssx}) and interchange the order of summation and integration to obtain
\begin{align}\label{su1}
S(s,x)
&=\frac{2}{i}\int_{(c)}\left(\sum_{n=1}^{\iy}\frac{\s_{s}(n)}{n^{z/2}}\right)\frac{\G(z-1)}{(4\pi)^z}\sin\left(\frac{\pi z}{4}\right)x^{(1-z)/2}\, dz.
\end{align}
It is well-known \cite[p.~8, equation (1.3.1)]{titch} that for Re $\nu>1$ and Re $\nu>1+$ Re $\mu$,
\begin{align}\label{sz}
\zeta(\nu)\zeta(\nu-\mu)=\sum_{n=1}^{\infty}\frac{\sigma_{\mu}(n)}{n^{\nu}}.
\end{align}
Invoking (\ref{sz}) in (\ref{su1}), we see that
\begin{equation}\label{su2}
S(s,x)=\frac{2}{i}\int_{(c)}\Omega(z, s, x)\, dz,
\end{equation}
where $c>2\sigma+2$ (since $\sigma>0$) and
\begin{equation}\label{omzsx}
\Omega(z, s, x):=\zeta\left(\frac{z}{2}\right)\zeta\left(\frac{z}{2}-s\right)\frac{\G(z-1)}{(4\pi)^z}\sin\left(\frac{\pi z}{4}\right)x^{(1-z)/2}.
\end{equation}
We want to shift the line of integration from Re $z=c$ to Re $z=\l$, where
$-1<\l<0$. Note that the integrand in (\ref{su2}) has poles at $z=1$, $2$, and
$2s+2$. Consider the positively oriented rectangular contour formed by $[c-iT,
c+iT], [c+iT,\l+iT]$,  $[\l+iT,\l-iT]$, and $[\l-iT, c-iT]$, where $T$ is any
positive real number. By Cauchy's residue theorem,
\begin{align}\label{res}
\frac{1}{2\pi
  i}&\left\{\int_{c-iT}^{c+iT}+\int_{c+iT}^{\l+iT}
+\int_{\l+iT}^{\l-iT}+\int_{\l-iT}^{c-iT}\right\}\Omega(z,
s, x)\, dz\notag\\&=R_{1}(\Omega)+R_2(\Omega)+R_{2s+2}(\Omega),
\end{align}
where we recall that $R_a(f)$ denotes the residue of a function $f$ at the
pole $z=a$. The
residues are now calculated.
First,
\begin{align}
R_{2s+2}(\Omega)&=\lim_{z\to 2s+2}
(z-2s-2)\zeta\left(\frac{z}{2}-s\right)\zeta\left(\frac{z}{2}\right)\frac{\G(z-1)}{(4\pi)^z}\sin\left(\frac{\pi
    z}{4}\right)x^{(1-z)/2}\nonumber\\
&=2\zeta(s+1)\frac{\G(2s+1)}{(4\pi)^{2s+2}}\sin\left(\frac{\pi (2s+2)}{4}\right)x^{-s-\frac{1}{2}}\nonumber\\
&=-\frac{2^{-s-3}}{\pi^{s+\frac{3}{2}}}\frac{\G(s+\tf12)\cot\(\frac{1}{2}\pi s\)\zeta(-s)}{x^{s+\frac{1}{2}}},\label{res1}
\end{align}
where in the first step we used (\ref{zlim}), and in the last step we employed
(\ref{dup}) and (\ref{fe}) with $s$ replaced by $s+1$. Second and third,
{\allowdisplaybreaks\begin{align}
R_1(\Omega) &=\lim_{z\to 1} (z-1)\frac{\G(z-1)}{(4\pi)^z}\zeta\left(\frac{z}{2}\right)\zeta\left(\frac{z}{2}-s\right)\sin\left(\frac{\pi z}{4}\right)x^{(1-z)/2}\nonumber\\
&=\frac{1}{4\sqrt{2}\pi}\zeta\left(\frac{1}{2}\right)\zeta\left(\frac{1}{2}-s\right),\label{res2s2}\\
R_2(\Omega) &=\lim_{z\to 2} (z-2)\zeta\left(\frac{z}{2}\right)\zeta\left(\frac{z}{2}-s\right)\frac{\G(z-1)}{(4\pi)^z}\sin\left(\frac{\pi z}{4}\right)x^{(1-z)/2}\nonumber\\
&=\frac{\zeta(1-s)}{8\pi^2\sqrt{x}},\label{reshalf}
\end{align}}%
where, in (\ref{res2s2}) we utilized (\ref{feg}), and in (\ref{reshalf}) we used
(\ref{zlim}). Next, we show that as $T\to\infty$, the integrals along the
horizontal segments $[c+iT,\l+iT]$ and $[\l-iT,c-iT]$ tend to zero. To that
end, note that if $s=\sigma+it$, for $\sigma\geq -\delta$
\cite[p.~95, equation (5.1.1)]{titch},
\begin{equation}\label{zetab}
\zeta(s)=O(t^{\tf{3}{2}+\delta}).
\end{equation}
 Also, as  $|t|\to\infty$,
\begin{equation}\label{sinest}
\left|\sin\left(\frac{\pi s}{4}\right)\right|=\left|\frac{e^{\frac{1}{4}i\pi s}-e^{-\frac{1}{4}i\pi s}}{2i}\right|=O\left(e^{\frac{1}{4}\pi|t|}\right).
\end{equation}
Thus from (\ref{zetab}), (\ref{strivert}), and (\ref{sinest}), we see that the
integrals along the horizontal segments tend to zero as $T\to \iy$. Along
with (\ref{res}), this implies that
\begin{align}\label{cel}
&\int_{(c)}\Omega(z, s, x)\, dz=\int_{(\l)}\Omega(z, s, x)\, dz\\
&\quad+2\pi i\left(\frac{\zeta(1-s)}{8\pi^2\sqrt{x}}+\frac{1}{4\sqrt{2}\pi}\zeta\left(\frac{1}{2}\right)
\zeta\left(\frac{1}{2}-s\right)-\frac{2^{-s-3}}{\pi^{s+\frac{3}{2}}}\frac{\G(s+1/2)\cot\(\frac{\pi s}{2}\)\zeta(-s)}{x^{s+\frac{1}{2}}}\right).\notag
\end{align}

We now evaluate the integral along the vertical line $\Re z=\l$. Using (\ref{fe}) twice, we have
{\allowdisplaybreaks\begin{align}\label{ixn}
\int_{(\l)}\Omega(z, s, x)\, dz&=\int_{(\l)}2^{z-s}\pi^{z-s-2}\zeta\left(1-\frac{z}{2}\right)
\zeta\left(1-\frac{z}{2}+s\right)\G\left(1-\frac{z}{2}\right)\nonumber\\
&\quad\times\G\left(1-\frac{z}{2}+s\right)\frac{\G(z-1)}{(4\pi)^z}\sin^2\left(\frac{\pi z}{4}\right)\sin\left(\frac{\pi z}{4}-\frac{\pi s}{2}\right)x^{(1-z)/2}\, dz\nonumber\\
&=\frac{\sqrt{x}}{2^s\pi^{s+2}}\sum_{n=1}^{\iy}\frac{\s_{s}(n)}{n^{s+1}}
\int_{(\l)}\G(z-1)\G\left(1-\frac{z}{2}\right)\G\left(1-\frac{z}{2}+s\right)\nonumber\\
&\quad\quad\quad\quad\quad\quad\quad\quad\quad\quad\times\sin^2\left(\frac{\pi z}{4}\right)\sin\left(\frac{\pi z}{4}-\frac{\pi s}{2}\right)\left(\frac{4x}{n}\right)^{-z/2}\, dz\nonumber\\
&=\frac{i\sqrt{x}}{2^{s-1}\pi^{s+1}}\sum_{n=1}^{\iy}\frac{\s_{s}(n)}{n^{s+1}}I\left(s,\frac{x}{n}\right),
\end{align}}%
where in the penultimate step we used (\ref{sz}), since $\l<0$, and used the
notation for $I(s,x)$ in the lemma. From (\ref{su2}), (\ref{cel}), and
(\ref{ixn}), we deduce that
\begin{align*}
&S(s,x)=\frac{\sqrt{x}}{2^{s-2}\pi^{s+1}}\sum_{n=1}^{\iy}\frac{\s_{s}(n)}{n^{s+1}}I\left(s,\frac{x}{n}\right)\\
&+4\pi\left(\frac{\zeta(1-s)}{8\pi^2\sqrt{x}}+\frac{1}{4\sqrt{2}\pi}\zeta\left(\frac{1}{2}\right)
\zeta\left(\frac{1}{2}-s\right)-\frac{2^{-s-3}}{\pi^{s+\frac{3}{2}}}\frac{\G(s+1/2)\cot\(\frac12\pi s\)\zeta(-s)}{x^{s+\frac{1}{2}}}\right).
\notag\end{align*}
The final result follows by substituting the expressions for
$I\left(s,\frac{x}{n}\right)$ from the lemma, accordingly as $n<x$ or $n\geq x$. This completes the proof.

\section{Proof of Lemma \ref{lemma1}}\label{sect4}

Multiplying and dividing the integrand in (\ref{Is}) by $\Gamma\left(\frac12(3-z)\right)$ and then applying (\ref{dup}) and (\ref{ref}), we see that
\begin{equation}\label{Is1}
I(s,x)=-\frac{{\pi}^{\frac{3}{2}}}{4\pi i}\int_{(\l)}\frac{\sin^2\left(\frac14\pi z\right)\sin\left(\frac14\pi z-\frac12\pi s\right)}{\sin\pi z}\frac{\G\left(1-\frac12z+s\right)}{\G\left(1-\frac12z+\frac{1}{2}\right)}x^{-\frac12z}\,dz.
\end{equation}
We now apply (\ref{ref}), (\ref{ref2}), and (\ref{dup}) repeatedly to simplify the integrand in (\ref{Is1}). This gives
\begin{equation}\label{is}
 I(s,x)=\frac{1}{2\pi i}\frac{-\pi}{2^{2-s}}\int_{(\l)}F(z, s, x)\, dz,
\end{equation}
where
\begin{equation}\label{fz}
F(z, s, x):=\frac{\tan\(\frac14\pi z\)}{2^{z/2}(1-z)}\frac{\G\left(\frac{1}{2}-\frac14z
+\frac12s\right)\G\left(\frac12(1+z)\right)}{\G\left(\frac14z-\frac12s\right)}x^{-z/2}.
\end{equation}
The poles of $F(z,s,x)$ are at $z=1$,
 at $z=2(2k+1+s), k\in\mathbb{N}\cup\{0\}$,
at $z=2(2m+1), m\in\mathbb{Z}$, and
 at $z=-(2j+1)$, $j\in
\mathbb{N}\cup\{0\}$.\\

{\bf Case (i):} When $x>1$, we would like to move the  vertical line of
integration to $+\infty$. To that end, let $X>\l$ be such that
the line $(X-i\infty,X+i\infty)$ does not pass through any poles of
$F(z)$. Consider the positively oriented rectangular contour formed by $[\l-iT,
X-iT], [X-iT,X+iT]$, $[X+iT,\l+iT]$, and $[\l+iT, \l-iT]$, where $T$ is any
positive real number. Then by Cauchy's residue theorem,
\begin{align*}
&\frac{1}{2\pi i}\left\{\int_{\l-iT}^{X-iT}+\int_{X-iT}^{X+iT}+\int_{X+iT}^{\l+iT}+\int_{\l+iT}^{\l-iT}\right\}F(z, s, x)\, dz\nonumber\\
&=R_1(F)+\sum_{0\leq k< \frac{1}{2}\(\frac12X-1-\text{Re}\hspace{0.5mm} s\)}R_{2(2k+1+s)}(F)+\sum_{0\leq m<\frac{1}{2}\(\frac12X-1\)}R_{2(2m+1)}(F).
\end{align*}
We now calculate the residues.  First,
{\allowdisplaybreaks\begin{align}
R_1(F)&=\lim_{z\to 1}(z-1)\frac{\tan\(\frac14\pi z\)}{2^{\frac{z}{2}}(1-z)}\frac{\G\left(\frac{1}{2}-\frac14z+\frac12s\right)
\G\left(\frac12(1+z)\right)}{\G\left(\frac14z-\frac12s\right)}x^{-\frac12z}\nonumber\\
&=-\frac{1}{\sqrt{2x}}\frac{\G\(\frac{1}{4}+\frac12s\)}{\G\(\frac{1}{4}-\frac12s\)}.\label{lr1}
\end{align}
Second,
\begin{align}
&R_{2(2k+1+s)}(F)\notag\\
&=\lim_{z\to2(2k+1+s)}\{z-2(2k+1+s)\}\frac{\tan\(\frac14\pi z\)}{2^{\frac{z}{2}}(1-z)}\frac{\G\left(\frac{1}{2}-\frac14z+\frac12s\right)
\G\left(\frac12(1+z)\right)}{\G\left(\frac14z-\frac12s\right)}x^{-\frac{1}{2}z}\nonumber\\
&=\frac{4(-1)^{k+1}\cot\(\frac12\pi s\)}{k!2^{2k+2+s}}\frac{\G\(\frac{1}{2}+2k+s\)}{\G\(\frac12(2k+1)\)}x^{-(2k+1+s)}\nonumber\\
&=\frac{(-1)^{k+1}\cot\(\frac12\pi s\)}{(2k)!2^{s}\sqrt{\pi}}\G\(s+\frac{1}{2}\)\(s+\frac{1}{2}\)_{2k}x^{-(2k+1+s)},\label{lr2s}
\end{align}}%
where in the second calculation, we used the fact
$\lim_{z\to-n}(z+n)\G(z)=(-1)^{n}/n!,$ followed by (\ref{feg}) and
(\ref{dup}). Here $(y)_n$ denotes the rising factorial defined in
\eqref{rising}.
Note that we do not have a pole at $2(2k+1+s)$ when $s$ is an odd
integer. Also,
{\allowdisplaybreaks\begin{align}\label{xpp}
&R_{2(2m+1)}(F)\notag\\
&=\lim_{z\to 2(2m+1)}\{z-2(2m+1)\}\frac{\tan\(\frac14\pi z\)}{2^{z/2}(1-z)}\frac{\G\left(\frac{1}{2}-\frac14{z}+\frac12{s}\right)
\G\left(\frac12(1+z)\right)}{\G\left(\frac14{z}-\frac12{s}\right)}x^{-z/2}\nonumber\\
&=\frac{1}{\pi 2^{2m}}\frac{\G\(\frac12{s}-m\)\G\(2m+\frac{1}{2}\)}{\G\(m-\frac12{s}+\frac{1}{2}\)}x^{-(2m+1)}\nonumber\\
&=\frac{(-1)^{m}}{ 2^{s}\sin\left(\frac12{\pi s}\right)\G(1-s)}\frac{\(\frac{1}{2}\)_{2m}}{\(1-s\)_{2m}}x^{-(2m+1)},
\end{align}}%
where we used (\ref{ref}) and (\ref{dup}). As in the proof of Theorem \ref{p3361}, using Stirling's formula (\ref{strivert}), we see  that the
integrals along the horizontal segments tend to zero as $T\to\infty$. Thus,
\begin{align}\label{lfx}
\frac{1}{2\pi i}&\int_{(X)}F(z, s, x)\,dz=\frac{1}{2\pi i}\int_{(\l)}F(z, s,
x)\,dz\\
\quad&+R_1(F)+a\sum_{0\leq k\leq \frac{1}{2}\(\frac{1}{2}X-1-\text{Re}\hspace{0.5mm} s\)}R_{2(2k+1+s)}(F)+\sum_{0\leq m<\frac{1}{2}\(\frac{1}{2}X-1\)}R_{2(2m+1)}(F),\notag
\end{align}
where $a$ is defined in \eqref{a}. From (\ref{fz}), we see that
\begin{align}\label{fzp}
F(z+4, s, x)=-\frac{F(z, s, x)(z-1)\(\frac{1}{2}(z+1)\)\(\frac{1}{2}(z+3)\)}{4x^2(z+3)\(\frac{1}{4}z-\frac{1}{2}(s-1)\)\(\frac{1}{4}z-\frac{1}{2}s\)},
\end{align}
so that
\begin{align}\label{fzp1}
\left|F(z+4, s, x)\right|=\frac{\left|F(z, s, x)\right|}{x^2}\(1+O_s\(\frac{1}{|z|}\)\).
\end{align}
Applying (\ref{fzp}) and (\ref{fzp1}) repeatedly, we find that
\begin{equation*}
\left|F(z+4\ell, s, x)\right|=\frac{\left|F(z, s, x)\right|}{x^{2\ell}}\(1+O_s\(\frac{1}{|z|}\)\)^\ell,
\end{equation*}
for any positive integer $\ell$ and $\Re z>0$. Therefore,
{\allowdisplaybreaks\begin{align}\label{fzp2}
\left|\int_{(X+4\ell)}F(z, s, x)\,dz\right|\notag&=\left|\int_{(X)}\frac{F(z, s,
    x)}{x^{2\ell}}\(1+O_s\(\frac{1}{|z|}\)\)^\ell\,dz\right| \nonumber\\
&=\frac{1}{|x|^{2\ell}}\(1+O_s\(\frac{1}{|X|}\)\)^{\ell}\left|\int_{(X)}F(z,
  s, x)\,dz\right|.
\end{align}}%
Since $x>1$, we can choose $X$ large enough so that
\begin{equation*}
|x|>\sqrt{1+O_s\(\frac{1}{|X|}\)}.
\end{equation*}
With this choice of $X$ and the fact that $\left|\int_{(X)}F(z, s,
  x)\,dz\right|$ is finite, if we let $\ell\to \iy$, then, from (\ref{fzp2}),
we find that
\begin{equation}\label{infz}
\lim_{\ell\to\infty}\int_{X+4\ell-i\iy}^{X+4\ell+i\iy}F(z, s, x)\,dz=0.
\end{equation}
Hence, if we shift the vertical line $(X)$ through the sequence of vertical lines $\{(X+4\ell)\}_{\ell=1}^{\iy}$, then, from (\ref{lfx}) and (\ref{infz}), we arrive at
\begin{equation}\label{lfx1}
\frac{1}{2\pi i}\int_{(\l)}F(z, s, x)\,dz=-R_1(F)-a\sum_{k=0}^{\iy}R_{2(2k+1+s)}(F)-\sum_{m=0}^{\iy}R_{2(2m+1)}(F).
\end{equation}
Since $x>1$, from (\ref{lr2s}) and the binomial theorem, we deduce that
\begin{align}\label{resks}
a\sum_{k=0}^{\iy}&R_{2(2k+1+s)}(F)=-a\frac{x^{-s-1}\cot\(\tfrac{1}{2}\pi s\)}{2^{s}\sqrt{\pi}}\G\(s+\frac{1}{2}\)\sum_{k=0}^{\iy}\frac{\(s+\frac{1}{2}\)_{2k}}{(2k)!}
\left(\frac{i}{x}\right)^{2k}\\
&=-a\frac{x^{-s-1}\cot\(\tfrac{1}{2}\pi s\)}{2^{s+1}\sqrt{\pi}}\G\(s+\frac{1}{2}\)
\left\{\(1+\frac{i}{x}\)^{-\(s+\frac{1}{2}\)}+\(1-\frac{i}{x}\)^{-\(s+\frac{1}{2}\)}\right\}.\notag
\end{align}
From (\ref{xpp}),
\begin{align}\label{resk}
\sum_{m=0}^{\iy}R_{2(2m+1)}(F)&=\frac{1}{x 2^{s}\sin\left(\tfrac{1}{2}\pi s\right)\G(1-s)}\sum_{m=0}^{\iy}\frac{\(\frac{1}{2}\)_{2m}}{\(1-s\)_{2m}}\left(\frac{i}{x}\right)^{2m}\nonumber\\
&=\frac{1}{x 2^{s}\sin\left(\tfrac{1}{2}\pi s\right)\G(1-s)}\pFq{3}{2}{\frac{1}{4}, \frac{3}{4}, 1}{\frac{1}{2}(1-s), 1-\frac{1}{2}s}{-\frac{1}{x^2}}.
\end{align}
Therefore from (\ref{lr1}), (\ref{lfx1}), (\ref{resks}), and (\ref{resk}) we
deduce that
\begin{align*}
&\frac{1}{2\pi i}\int_{(\l)}F(z, s, x)\,dz\\
&=a\frac{x^{-s-1}\cot\(\tfrac{1}{2}\pi s\)}{2^{s+1}\sqrt{\pi}x}\G\(s+\frac{1}{2}\)
\left\{\(1+\frac{i}{x}\)^{-\(s+\frac{1}{2}\)}+\(1-\frac{i}{x}\)^{-\(s+\frac{1}{2}\)}\right\}\\
&\quad-\frac{1}{x 2^{s}\sin\left(\frac{1}{2}\pi s\right)\G(1-s)}\pFq{3}{2}{\frac{1}{4}, \frac{3}{4}, 1}{\frac{1}{2}(1-s), 1-\frac{1}{2}s}{-\frac{1}{x^2}}
+\frac{1}{\sqrt{2x}}\frac{\G\(\frac{1}{4}+\frac{1}{2}s\)}{\G\(\frac{1}{4}
-\frac{1}{2}s\)}.
\end{align*}
Using (\ref{is}), we complete the proof of (\ref{ig1}).\\


{\bf Case (ii):} Now consider $x\leq 1$. We would like to shift the line of
integration all the way to $-\infty$.  Let $X<\l$ be such that the line
$[X-i\infty,X+i\infty]$ again does not pass through any pole of
$F(z)$. Consider a positively oriented rectangular contour formed by $[\l-iT,
\l+iT], [\l+iT,X+iT]$, $[X+iT,X-iT]$, and $[X-iT, \l-iT]$, where $T$ is any
positive real number. Again, by Cauchy's residue theorem,
\begin{align*}
&\frac{1}{2\pi
  i}\left[\int_{\l-iT}^{\l+iT}+\int_{\l+iT}^{X+iT}+\int_{X+iT}^{X-iT}+\int_{X-iT}^{\l-iT}\right]F(z,s, x)\, dz\nonumber\\
&=\sum_{0\leq k< \frac{1}{2}\(-\frac{1}{2}X-1\)}R_{-2(2k+1)}(F)+\sum_{0\leq j<
  \frac{1}{2}\(-X-1\)}R_{-(2j+1)}(F).
\end{align*}
The residues in this case are calculated below. First,
{\allowdisplaybreaks\begin{align}
&R_{-2(2k+1)}(F)\notag\\
&=\lim_{z\to -2(2k+1)}\left\{(z+2(2k+1))\tan\left(\frac{\pi
      z}{4}\right)\right\}\frac{1}{2^{z/2}}(1-z)
\notag\\&\quad\times\frac{\G\left(\frac{1}{2}-\frac{1}{4}z+\frac{1}{2}s\right)
\G\left(\frac{1}{2}(1+z)\right)}{\G\left(\frac{1}{4}z-\frac{1}{2}s\right)}x^{-\frac{1}{2}z}\nonumber\\
&=\frac{(-1)^{k+1}}{\sqrt{\pi}2^{s}\sin\left(\frac{1}{2}\pi s\right)}\frac{\G\(\frac{1}{2}-2(k+1)\)}{\G\(1-2(k+1)-s\)}x^{(2k+1)}\nonumber\\
&=\frac{(-1)^{k+1}}{\sqrt{\pi}2^{s}\sin\left(\frac{1}{2}\pi s\right)}\frac{\G\(\frac{1}{2}+2(k+1)\)\G\(\frac{1}{2}-2(k+1)\)}{\G\(2(k+1)+s\)\G\(1-2(k+1)-s\)}
x^{(2k+1)}\frac{\G\(2(k+1)+s\)}{\G\(\frac{1}{2}+2(k+1)\)}\nonumber\\
&=\frac{(-1)^{k+1}\cos\(\frac{1}{2}\pi s\)\G(s)}{ 2^{s-1}\pi x}\frac{\(s\)_{2(k+1)}}{\(\frac{1}{2}\)_{2(k+1)}}x^{2(k+1)},\label{lro3}
\end{align}}
where in the last step we used (\ref{ref}) and (\ref{ref2}). Second,
\begin{align}
&R_{-(2j+1)}(F)\notag\\
&=\lim_{z\to -(2j+1)}(z+(2j+1))\frac{\tan\left(\frac{1}{4}\pi z\right)}{2^{z/2}(1-z)}\frac{\G\left(\frac{1}{2}-\frac{1}{4}z+\frac{1}{2}s\right)
\G\left(\frac{1}{2}(1+z)\right)}{\G\left(\frac{1}{4}z-\frac{1}{2}s\right)}x^{-z/2}\nonumber\\
&=-\frac{2^{j+\frac{1}{2}}}{(j+1)!}\frac{\G\left(\frac{5}{4}+\frac{1}{2}j
+\frac{1}{2}s\right)}{\G\left(-\frac{1}{4}-\frac{1}{2}j-\frac{1}{2}s\right)}x^{j+\frac{1}{2}}\nonumber\\
&=\frac{1}{\sqrt{\pi}2^s(j+1)!}\G\(s+\frac{3}{2}\)\(s+\frac{3}{2}\)_j
\sin\(\pi\(\frac{j}{2}+\frac{1}{4}+\frac{s}{2}\)\)x^{j+\frac{1}{2}},\label{lro4}
\end{align}
where we multiplied the numerator and denominator by
$\G\left(\frac{3}{4}+\frac{1}{2}j+\frac{1}{2}s\right)$ in the last step and then used
(\ref{ref}) and (\ref{dup}).  Thus, by \eqref{lro3} and \eqref{lro4},
\begin{multline}\label{lfx2}
\frac{1}{2\pi i}\int_{(\l)}F(z, s, x)\,dz=\frac{1}{2\pi i}\int_{(X)}F(z, s,
x)\,dz\\+\sum_{0\leq k\leq
  \frac{1}{2}\(-\frac{1}{2}X-1\)}R_{-2(2k+1)}(F)+\sum_{0\leq k\leq
  \frac{1}{2}\(-X-1\)}R_{-(2k+1)}(F).
\end{multline}
From (\ref{fzp}),
\begin{equation*}
\left|F(z-4, s, x)\right|=|x|^2\(1+O_s\(\frac{1}{|z|}\)\)\left|F(z, s,
  x)\right|,
\end{equation*}
and hence
\begin{equation}\label{fzp-1}
\left|F(z-4\ell, s,
  x)\right|=|x|^{2\ell}\(1+O_s\(\frac{1}{|z|}\)\)^{\ell}\left|F(z, s,
  x)\right|,
\end{equation}
for any positive integer $\ell$ and $\Re z<0$. Therefore, from (\ref{fzp-1}),
\begin{align*}
\left|\int_{(X-4k)}F(z, s, x)\,dz\right|&=\left|\int_{(X)}F(z, s, x)x^{2\ell}\(1+O_s\(\frac{1}{|z|}\)\)^{\ell}\,dz\right|\nonumber\\
&=|x|^{2\ell}\(1+O_s\(\frac{1}{|X|}\)\)^{\ell}\left|\int_{(X)}F(z, s,
  x)\,dz\right|.
\end{align*}
Since $x<1$, we can find an $X<\l$, with $|X|$ sufficiently large, so that
\begin{equation}\label{xz2}
x^2\(1+O_s\(\frac{1}{|X|}\)\)<1.
\end{equation}
With the given choice of $X$ and the fact that $\left|\int_{(X)}F(z, s, x)\,dz\right|$ is finite, upon letting $\ell\to\infty$ and using (\ref{xz2}), we find that
\begin{equation}\label{infz2}
\lim_{\ell\to\infty}\int_{X-4\ell-i\iy}^{X-4\ell+i\iy}F(z, s, x)\,dz=0.
\end{equation}
Thus if we shift the line of integration $(X)$ to $-\infty$ through the sequence of vertical lines $\{(X-4k)\}_{k=1}^{\iy}$, from (\ref{lfx2}) and (\ref{infz2}), we arrive at
 \begin{equation}\label{lfx3}
\frac{1}{2\pi i}\int_{(\l)}F(z, s, x)\,dz=\sum_{k=0}^{\iy}R_{-2(2k+1)}(F)+\sum_{j=0}^{\iy}R_{-(2j+1)}(F).
\end{equation}
Since $x\leq 1$, using (\ref{lro3}), we find that
\begin{align}\label{reskn}
\sum_{k=0}^{\iy}R_{-2(2k+1)}(F)&=\frac{\G(s)\cos\(\tfrac{1}{2}\pi s\)}{ 2^{s-1}\pi x}\sum_{k=0}^{\iy}\frac{\(s\)_{2(k+1)}}{\(1/2\)_{2(k+1)}}(ix)^{2(k+1)}\nonumber\\
&=\frac{\G(s)\cos\(\tfrac{1}{2}\pi s\)}{ 2^{s-1}\pi x}\left\{\pFq{3}{2}{\frac{s}{2},\frac{1+s}{2}, 1}{\frac{1}{4}, \frac{3}{4}}{-x^2}-1\right\},
\end{align}
where for $x=1$, we additionally require that $\sigma<\frac{1}{2}$ in order to
ensure the conditional convergence of
the ${}_3F_{2}$  \cite[p.~62]{aar}.

From (\ref{lro4}),
{\allowdisplaybreaks\begin{align}\label{resn}
\sum_{j=0}^{\iy}R_{-(2j+1)}(F)&=\frac{\G\(s+\frac{3}{2}\)}{2^s\sqrt{\pi}}\sum_{j=0}^{\iy}
\sin\(\pi\(\frac{j}{2}+\frac{1}{4}+\frac{s}{2}\)\)\frac{\(s+\frac{3}{2}\)_j}{(j+1)!}x^{(j+\frac{1}{2})}\\
&=\frac{\G\(s+\frac{3}{2}\)}{2^s\sqrt{\pi}}\left\{\sqrt{x}\sin\(\pi\(\frac{1}{4}+\frac{s}{2}\)\)
\sum_{j=0}^{\iy}\frac{\(s+\frac{3}{2}\)_{2j}}{(2j+1)!}(ix)^{2j}\right.\nonumber\\
&\hspace{2.2cm}+\left.x^{3/2}\cos\(\pi\(\frac{1}{4}+\frac{s}{2}\)\)\sum_{j=0}^{\iy}
\frac{\(s+\frac{3}{2}\)_{2j+1}}{(2j+2)!}(ix)^{2j}\right\}\nonumber\\
&=\frac{i\G\(s+\frac{1}{2}\)}{2^{s+1}\sqrt{\pi x}}\left[\sin\(\frac{\pi}{4}+\frac{\pi s}{2}\)\left\{(1+ix)^{-(s+\frac{1}{2})}-(1-ix)^{-(s+\frac{1}{2})}\right\}\right.\nonumber\\
&\hspace{1cm}+\left.i\cos\(\frac{\pi}{4}+\frac{\pi s}{2}\)\left\{(1+ix)^{-(s+\frac{1}{2})}+(1-ix)^{-(s+\frac{1}{2})}-2\right\}\right],\nonumber
\end{align}}%
where in the last step we used the identities
\begin{align*}
\sum_{j=0}^{\infty}\frac{(a)_{2j}x^{2j}}{(2j+1)!}&=\frac{(1+x)^{1-a}-(1-x)^{1-a}}{2x(1-a)},\\
\sum_{j=0}^{\infty}\frac{(a)_{2j+1}x^{2j+1}}{(2j+2)!}
&=\frac{-\left((1+x)^{1-a}+(1-x)^{1-a}-2\right)}{2x(1-a)},
\end{align*}
valid for $|x|<1$. Combining (\ref{lfx3}), (\ref{reskn}), and (\ref{resn}),
we deduce that
{\allowdisplaybreaks\begin{align*}
\frac{1}{2\pi i}&\int_{(\l)}F(z, s, x)\,dz=\frac{\cos\(\frac{\pi s}{2}\)\G(s)}{ 2^{s-1}\pi x}\left\{\pFq{3}{2}{\frac{s}{2},\frac{1+s}{2}, 1}{\frac{1}{4}, \frac{3}{4}}{-x^2}-1\right\}\nonumber\\
&\quad+\frac{i\G\(s+\frac{1}{2}\)}{2^{s+1}\sqrt{\pi x}}\left[\sin\(\frac{\pi}{4}+\frac{\pi s}{2}\)\left\{(1+ix)^{-(s+\frac{1}{2})}-(1-ix)^{-(s+\frac{1}{2})}\right\}\right.\nonumber\\
&\hspace{1cm}+\left.i\cos\(\frac{\pi}{4}+\frac{\pi s}{2}\)\left\{(1+ix)^{-(s+\frac{1}{2})}+(1-ix)^{-(s+\frac{1}{2})}-2\right\}\right].
\end{align*}}%
Using (\ref{is}), we see that this proves (\ref{ig2}). This completes the
proof of Lemma \ref{lemma1}.

 If $x$ is an integer in Theorem \ref{p3361}, then the term corresponding to
 it on the right-hand side of (\ref{tf}) can be included either in the first
 (finite) sum or in the second (infinite) sum. This follows from the fact
 that the integral $I(s, x)$ in the lemma above is continuous at
 $x=1$. Though elementary, we warn readers that it is fairly tedious to
 verify this by showing that the right-hand sides of (\ref{ig1}) and
 (\ref{ig2}) are equal when $x=1$, and requires the following transformation
 between ${}_3F_{2}$ hypergeometric functions, which is actually the special
 case $q=2$ of a general connection formula between ${}_pF_{q}$'s
 \cite[p.~410, formula \textbf{16.8.8}]{nist}.
\begin{theorem}
For $a_1-a_2, a_1-a_3, a_2-a_3\notin\mathbb{Z}$, and $z\notin (0, 1)$,
\begin{align}\label{3f2trans}
&{}_3F_{2}(a_1,a_2,a_3;b_1,b_2;z)
=\frac{\G(b_1)\G(b_2)}{\G(a_1)\G(a_2)\G(a_3)}
\bigg(\frac{\G(a_1)\G(a_2-a_1)\G(a_3-a_1)}{\G(b_1-a_1)\G(b_2-a_1)}(-z)^{-a_1}\notag\\
&\quad\times{}_3F_{2}\left(a_1, a_1-b_1+1, a_1-b_2+1; a_1-a_2+1, a_1-a_3+1; \frac{1}{z}\right)\nonumber\\
&\quad+\frac{\G(a_2)\G(a_1-a_2)\G(a_3-a_2)}{\G(b_1-a_2)\G(b_2-a_2)}(-z)^{-a_2}\notag\\
&\quad\times{}_3F_{2}\left(a_2, a_2-b_1+1, a_2-b_2+1; -a_1+a_2+1, a_2-a_3+1; \frac{1}{z}\right)\nonumber\\
&\quad+\frac{\G(a_3)\G(a_1-a_3)\G(a_2-a_3)}{\G(b_1-a_3)\G(b_2-a_3)}(-z)^{-a_3}\notag\\
&\quad\times{}_3F_{2}\left(a_3, a_3-b_1+1, a_3-b_2+1; -a_1+a_3+1, -a_2+a_3+1; \frac{1}{z}\right)\bigg).
\end{align}
\end{theorem}
\section{Coalescence}\label{bessie}

In the proofs of Theorems \ref{p3361} and \ref{p3361a} using contour
integration, the convergence of the series of residues of the corresponding
functions necessitates the consideration of two sums -- one over $n<x$ and
the other over $n\geq x$. However, for some special values of $s$, namely
$s=2m+\frac{1}{2}$, where $m$ is a non-negative integer, the two sums over
$n<x$ and $n\geq x$ coalesce into a single infinite sum. This section
contains corollaries of these theorems when $s$ takes these special values.

\begin{theorem}\label{coal1}
Let $x\notin\mathbb{Z}$. Then, for any non-negative integer $m$,
\begin{align}\label{coal1id}
&\sum_{n=1}^{\infty}\frac{\sigma_{2m+\frac{1}{2}}(n)}{\sqrt{n}}e^{-2\pi\sqrt{2nx}}
\sin\left(\frac{\pi}{4}+2\pi\sqrt{2nx}\right)\nonumber\\
&=\frac{\zeta\left(\frac{1}{2}-2m\right)}{2\pi\sqrt{x}}-\frac{(2m)!\zeta\left(-\frac{1}{2}-2m\right)}{\sqrt{2}(2\pi x)^{2m+1}}+\frac{1}{\sqrt{2}}\zeta\left(\frac{1}{2}\right)\zeta(-2m)\nonumber\\
&\quad+\frac{\sqrt{x}}{\pi^{2m+\frac{1}{2}}}\sum_{n=1}^{\infty}\frac{\sigma_{2m+\frac{1}{2}}(n)}{n^{2m
+\frac{3}{2}}}\bigg[-\frac{(2m)!}{\sqrt{\pi}}\left(\frac{n}{2x}\right)^{2m+\frac{3}{2}}\nonumber\\
&\quad\times\left\{\left(1+\frac{in}{x}\right)^{-(2m+1)}+\left(1-\frac{in}{x}\right)^{-(2m+1)}\right\}\nonumber\\
&\quad+\frac{(-1)^mn}{2^{2m}\pi x}\Gamma\left(2m+\frac{1}{2}\right)\pFq{3}{2}{\frac{1}{4}, \frac{3}{4}, 1}{\frac{1}{4}-m, \frac{3}{4}-m}{-\frac{n^2}{x^2}}\bigg].
\end{align}
\end{theorem}

\begin{proof}
Let $s=2m+\frac{1}{2}$, $m\geq0$, in Theorem \ref{p3361}. To examine the summands in the sum over $n<x$, observe first that $1/\Gamma\left(\frac{1}{4}-\frac{1}{2}s\right)=0$. Since $a=1$, the second expression in the summands is given by
{\allowdisplaybreaks\begin{align}\label{goofy}
&-\frac{\sigma_{2m+\frac{1}{2}}(n)}{n^{2m+\frac{3}{2}}}\frac{a\G\(s+\frac{1}{2}\)\cot\(\frac{\pi s}{2}\)}{2^{s+1}\sqrt{\pi}}\left(\frac{n}{x}\right)^{s+1}
\left\{\(1+\frac{in}{x}\)^{-\(s+\frac{1}{2}\)}+\(1-\frac{in}{x}\)^{-\(s+\frac{1}{2}\)}\right\}\nonumber\\
&=-\frac{\sigma_{2m+\frac{1}{2}}(n)}{n^{2m+\frac{3}{2}}}\frac{(2m)!}
{\sqrt{\pi}}\left(\frac{n}{2x}\right)^{2m+\frac{3}{2}}
\frac{\left(1-\frac{in}{x}\right)^{2m+1}+\left(1+\frac{in}{x}\right)^{2m+1}}
{\left(1+n^2/x^2\right)^{2m+1}}\nonumber\\
&=-\frac{\sigma_{2m+\frac{1}{2}}(n)}{n^{2m+\frac{3}{2}}}\frac{(2m)!}
{\sqrt{\pi}}\frac{n^{2m+\frac{3}{2}}x^{2m+\frac{1}{2}}}{2^{2m+\frac{1}{2}}(x^2+n^2)^{2m+1}}
\sum_{k=0}^{m}(-1)^k\binom{2m+1}{2k}\left(\frac{n}{x}\right)^{2k}.
\end{align}}%
The third expressions in the summands become
\begin{align}\label{3ex}
&\frac{\sigma_{2m+\frac{1}{2}}(n)}{n^{2m+\frac{3}{2}}}\frac{n2^{-s}}{x\sin\left(\frac{1}{2}\pi s\right)\G(1-s)}\pFq{3}{2}{\frac{1}{4}, \frac{3}{4}, 1}{\frac{1-s}{2}, 1-\frac{s}{2}}{-\frac{n^2}{x^2}}\nonumber\\
&=\frac{\sigma_{2m+\frac{1}{2}}(n)}{n^{2m+\frac{3}{2}}}\frac{(-1)^mn2^{-2m}}
{x\Gamma\left(\frac{1}{2}-2m\right)}\pFq{3}{2}{\frac{1}{4}, \frac{3}{4}, 1}{\frac{1}{4}-m, \frac{3}{4}-m}{-\frac{n^2}{x^2}}.
\end{align}
Hence, by \eqref{goofy} and \eqref{3ex}, the summands over $n<x$ are given by
\begin{align}\label{nlxsum}
\frac{\sigma_{2m+\frac{1}{2}}(n)}{n^{2m+\frac{3}{2}}}\bigg\{&-\frac{(2m)!}{\sqrt{\pi}}
\frac{n^{2m+\frac{3}{2}}x^{2m+\frac{1}{2}}}{2^{2m+\frac{1}{2}}(x^2+n^2)^{2m+1}}
\sum_{k=0}^{m}(-1)^k\binom{2m+1}{2k}
\left(\frac{n}{x}\right)^{2k}\nonumber\\
&+\frac{(-1)^mn2^{-2m}}{x\Gamma\left(\frac{1}{2}-2m\right)}\pFq{3}{2}{\frac{1}{4}, \frac{3}{4}, 1}{\frac{1}{4}-m, \frac{3}{4}-m}{-\frac{n^2}{x^2}}\bigg\}.
\end{align}
For the summands over $n>x$, observe that the third expression is equal to zero, since $\cos\left(\frac{1}{4}\pi+\frac{1}{2}\pi\left(2m+\frac{1}{2}\right)\right)=0$. The first expression becomes
\begin{align}\label{2nd1}
&-\frac{\sigma_{2m+\frac{1}{2}}(n)}{n^{2m+\frac{3}{2}}}\frac{n\G(s)\cos\(\frac{1}{2}\pi s\)}{ 2^{s-1}\pi x}\left\{\pFq{3}{2}{\frac{s}{2},\frac{1+s}{2}, 1}{\frac{1}{4}, \frac{3}{4}}{-\frac{x^2}{n^2}}-1\right\}\nonumber\\
&=\frac{\sigma_{2m+\frac{1}{2}}(n)}{n^{2m+\frac{3}{2}}}\frac{(-1)^{m+1}n2^{-2m}}
{x\G\left(\frac{1}{2}-2m\right)}\left\{\pFq{3}{2}{\frac{1}{4}+m,\frac{3}{4}+m, 1}{\frac{1}{4}, \frac{3}{4}}{-\frac{x^2}{n^2}}-1\right\},
\end{align}
where we used \eqref{ref2} with $s=2m$. The second expressions of the summands become
\begin{equation}\label{7.6}
\frac{\sigma_{2m+\frac{1}{2}}(n)}{n^{2m+\frac{3}{2}}}\frac{i(-1)^{m+1}\sqrt{n}(2m)!}{2^{2m+\frac{3}{2}}\sqrt{\pi x}}\frac{\left(1-ix/n\right)^{2m+1}-\left(1+ix/n\right)^{2m+1}}{\left(1+x^2/n^2\right)^{2m+1}}.
\end{equation}
Note that
\begin{equation*}
\left(1-\frac{ix}{n}\right)^{2m+1}-\left(1+\frac{ix}{n}\right)^{2m+1}=
-\sum_{k=0}^{2m+1}\binom{2m+1}{k}\left(\frac{ix}{n}\right)^{k}
\left(1+(-1)^{2m+1-k}\right).
\end{equation*}
These summands are non-zero only when $k$ is odd, and so if we let
$2j=2m+1-k$, we see that
\begin{equation*}
\left(1-\frac{ix}{n}\right)^{2m+1}-\left(1+\frac{ix}{n}\right)^{2m+1}=2i(-1)^{m+1}\left(\frac{x}{n}\right)^{2m+1}
\sum_{j=0}^{m}(-1)^{j}\binom{2m+1}{2j}\left(\frac{n}{x}\right)^{2j}.
\end{equation*}
Thus, after simplification, the second expressions \eqref{7.6} equal
\begin{equation}\label{2nd2}
-\frac{\sigma_{2m+\frac{1}{2}}(n)}{n^{2m+\frac{3}{2}}}\frac{(2m)!}{\sqrt{\pi}}
\frac{n^{2m+\frac{3}{2}}x^{2m+\frac{1}{2}}}{2^{2m+\frac{1}{2}}(x^2+n^2)^{2m+1}}
\sum_{j=0}^{m}(-1)^j\binom{2m+1}{2j}\left(\frac{n}{x}\right)^{2j}.
\end{equation}
Thus, by \eqref{2nd1} and \eqref{2nd2}, the summands over $n>x$ equal
\begin{align}\label{ngxsum}
&\frac{\sigma_{2m+\frac{1}{2}}(n)}{n^{2m+\frac{3}{2}}}\bigg[-\frac{(2m)!}{\sqrt{\pi}}
\frac{n^{2m+\frac{3}{2}}x^{2m+\frac{1}{2}}}{2^{2m+\frac{1}{2}}(x^2+n^2)^{2m+1}}
\sum_{j=0}^{m}(-1)^j\binom{2m+1}{2j}\left(\frac{n}{x}\right)^{2j}\nonumber\\
&\quad+\frac{(-1)^{m+1}n2^{-2m}}{x\G\left(\frac{1}{2}-2m\right)}\left\{\pFq{3}{2}{\frac{1}{4}+m,\frac{3}{4}+m, 1}{\frac{1}{4}, \frac{3}{4}}{-\frac{x^2}{n^2}}-1\right\}\bigg].
\end{align}
From \eqref{nlxsum} and \eqref{ngxsum}, it is clear that we want to prove that
\begin{equation}\label{hyident}
\pFq{3}{2}{\frac{1}{4}, \frac{3}{4}, 1}{\frac{1}{4}-m, \frac{3}{4}-m}{-\frac{n^2}{x^2}}+\pFq{3}{2}{\frac{1}{4}+m, \frac{3}{4}+m, 1}{\frac{1}{4}, \frac{3}{4}}{-\frac{x^2}{n^2}}=1,
\end{equation}
for $x>0$ and $n\in\mathbb{N}$.
To that end, use \eqref{3f2trans} with $a_1=\frac{1}{4}, a_2=\frac{3}{4}, a_3=1, b_1=\frac{1}{4}-m, b_2=\frac{3}{4}-m$, and $z=-n^2/x^2$. This gives, for all $x, n>0$,
\begin{equation}\label{hysimp1}
\pFq{3}{2}{\frac{1}{4}, \frac{3}{4}, 1}{\frac{1}{4}-m, \frac{3}{4}-m}{-\frac{n^2}{x^2}}=\frac{(4m+3)(4m+1)x^2}{3n^2}\pFq{3}{2}{\frac{7}{4}+m, \frac{5}{4}+m, 1}{\frac{7}{4}, \frac{5}{4}}{-\frac{x^2}{n^2}}.
\end{equation}
Now for $n>x$, we can use the series representation \eqref{hyper} for  ${}_3F_{2}$ on the right-hand side to obtain
\begin{align}\label{hysimp2}
&\pFq{3}{2}{\frac{7}{4}+m, \frac{5}{4}+m, 1}{\frac{7}{4}, \frac{5}{4}}{-\frac{x^2}{n^2}}=1+\sum_{k=1}^{\infty}
\frac{\left(\frac{7}{4}+m\right)_{k}\left(\frac{5}{4}+m\right)_{k}(1)_k}{\left(\frac{7}{4}\right)_{k}
\left(\frac{5}{4}\right)_{k}k!}\left(-\frac{x^2}{n^2}\right)^k\nonumber\\
&=1-\frac{3n^2}{(4m+3)(4m+1)x^2}\sum_{k=1}^{\infty}
\frac{\left(\frac{3}{4}+m\right)_{k+1}
\left(\frac{1}{4}+m\right)_{k+1}(1)_{k+1}}{\left(\frac{3}{4}\right)_{k+1}
\left(\frac{1}{4}\right)_{k+1}(k+1)!}\left(-\frac{x^2}{n^2}\right)^{k+1}
\nonumber\\
&=\frac{-3n^2}{(4m+3)(4m+1)x^2}\left\{\pFq{3}{2}{\frac{1}{4}+m,
    \frac{3}{4}+m, 1}{\frac{1}{4}, \frac{3}{4}}{-\frac{x^2}{n^2}}-1\right\}.
\end{align}
Combining \eqref{hysimp1} and \eqref{hysimp2}, we obtain \eqref{hyident} for
$n>x$.

Now set $a_1=\frac{1}{4}+m$, $a_2=\frac{3}{4}+m$, $a_3=1$, $b_1=\frac{1}{4}$,
$b_2=\frac{3}{4}$, and $z=-x^2/n^2$ in \eqref{3f2trans} and use, for $n<x$,
the series representation for the ${}_3F_{2}$ on the right-hand side of the
resulting identity to arrive at \eqref{hyident} for $n<x$. This shows that
\eqref{hyident} holds for all $x>0$ and $n\in\mathbb{N}$.

 Hence, the summands in the sums over $n<x$ and $n>x$ in Theorem \ref{p3361}
 are the same when $s=2m+\frac{1}{2}$. Now slightly rewrite \eqref{nlxsum}
 to finish the proof of Theorem \ref{coal1}.
\end{proof}

Similarly, when $s=2m+\frac{1}{2}$ in Theorem \ref{p3361a}, we obtain the
following.
\begin{theorem}\label{coal2}
For any non-negative integer $m$,
{\allowdisplaybreaks\begin{align}\label{coal2id}
&\sum_{n=1}^{\infty}\frac{\sigma_{2m+\frac{1}{2}}(n)}{\sqrt{n}}e^{-2\pi\sqrt{2nx}}
\sin\left(\frac{\pi}{4}-2\pi\sqrt{2nx}\right)\nonumber\\
&=\left(2\pi\sqrt{x}+\frac{(2m)!}{\sqrt{2}(2\pi x)^{2m+1}}\right)\zeta\left(-\frac{1}{2}-2m\right)+\frac{1}{\sqrt{2}}\zeta\left(\frac{1}{2}\right)\zeta(-2m)\nonumber\\
&\quad+\frac{\sqrt{\pi}(2m)!}{(2\pi)^{2m+\frac{3}{2}}}
\sum_{n=1}^{\infty}\sigma_{2m+\frac{1}{2}}(n)\left\{(x-in)^{-(2m+1)}+(x+in)^{-(2m+1)}\right\}.
\end{align}}
\end{theorem}

Notice the resemblance of the series on the right-hand side of
\eqref{coal2id} with the divergent series in Ramanujan's incorrect ``identity''
\eqref{qc}. Since the series on the right side above has a $+$ sign between
the two binomial expressions in the summands, the order of $n$ in the summand
is at least $-\tf32+\epsilon$, for each $\epsilon>0$,  unlike $-\tf12+\epsilon$
in Ramanujan's series, because of which the latter is divergent.

When $m\geq 1$, we can omit the term
$\frac{1}{\sqrt{2}}\zeta\left(\frac{1}{2}\right)\zeta(-2m)$ from both
(\ref{coal1id}) and (\ref{coal2id}) since $\zeta(-2m)=0$.

 In Theorem \ref{coal1}, we assume $x\notin\mathbb{Z}$, whereas there is no
 such restriction in Theorem \ref{coal2}, because  Theorems 1.1 and
 \ref{coal1} involve ${}_3F_{2}$'s that are conditionally convergent, with
 the restriction  $\sigma<\tf12$ when $x$ is an integer. Thus, the condition
 $\sigma\geq \tf12$ implies that $x\notin\mathbb{Z}$, which is the case when
 $s=2m+\frac{1}{2}$ for $m\geq 0$. However,  ${}_3F_{2}$'s do not appear in
 Theorem 1.3, and so the restriction on $x$ (other than the requirement $x>0$) is not needed.

Adding (\ref{coal1id}) and (\ref{coal2id}) and simplifying gives the next theorem.
\begin{theorem}
For $x\notin\mathbb{Z}$,
\begin{align}\label{coal12plus}
&\sum_{n=1}^{\infty}\frac{\sigma_{2m+\frac{1}{2}}(n)}{\sqrt{n}}e^{-2\pi\sqrt{2nx}}\cos\left(2\pi\sqrt{2nx}\right)\nonumber\\
&=\frac{1}{2\pi\sqrt{2x}}\zeta\left(\frac{1}{2}-2m\right)+\pi\sqrt{2x}\zeta\left(\frac{-1}{2}-2m\right)
+\zeta\left(\frac{1}{2}\right)\zeta(-2m)\nonumber\\
&\quad+\frac{(-1)^m}{\pi\sqrt{x}(2\pi)^{2m+\frac{1}{2}}}\Gamma\left(2m+\frac{1}{2}\right)
\sum_{n=1}^{\infty}\frac{\sigma_{2m+\frac{1}{2}}(n)}{n^{2m+\frac{1}{2}}}\pFq{3}{2}{\frac{1}{4}, \frac{3}{4}, 1}{\frac{1}{4}-m, \frac{3}{4}-m}{-\frac{n^2}{x^2}}.
\end{align}
\end{theorem}

Subtracting (\ref{coal1id}) from (\ref{coal2id}) and simplifying leads to the next result.
\begin{theorem}
For $x\notin\mathbb{Z}$,
\begin{align}\label{coal12minus}
&\sum_{n=1}^{\infty}\frac{\sigma_{2m+\frac{1}{2}}(n)}{\sqrt{n}}e^{-2\pi\sqrt{2nx}}\sin\left(2\pi\sqrt{2nx}\right)\nonumber\\
&=\frac{\zeta\left(\frac{1}{2}-2m\right)}{2\pi\sqrt{2x}}-\sqrt{2}\left(\pi\sqrt{x}+\frac{(2m)!}{\sqrt{2}(2\pi x)^{2m+1}}\right)\zeta\left(\frac{-1}{2}-2m\right)\nonumber\\
&\quad+\frac{\sqrt{x}}{\sqrt{2}\pi^{2m+\frac{1}{2}}}
\sum_{n=1}^{\infty}\frac{\sigma_{2m+\frac{1}{2}}(n)}{n^{2m+\frac{3}{2}}}
\bigg[-2\frac{(2m)!}{\sqrt{\pi}}\left(\frac{n}{2x}\right)^{2m+\frac{3}{2}}\notag\\
&\quad\times
\left\{\left(1+\frac{in}{x}\right)^{-(2m+1)}+\left(1-\frac{in}{x}\right)^{-(2m+1)}\right\}\nonumber\\
&\quad+\frac{(-1)^mn}{2^{2m}\pi x}\Gamma\left(2m+\frac{1}{2}\right)\pFq{3}{2}{\frac{1}{4}, \frac{3}{4}, 1}{\frac{1}{4}-m, \frac{3}{4}-m}{-\frac{n^2}{x^2}}\bigg].
\end{align}
\end{theorem}

 In Theorem \ref{coal1}, as well as in \eqref{coal12plus} and
 \eqref{coal12minus}, we should be careful while interpreting the
 ${}_3F_{2}$-function. For example, if $n<x$, then it can be expanded as a
 series. Otherwise, for $n>x$, the ${}_3F_{2}$-function represents the
 analytic continuation of the series. Of course, when $n>x$, one can replace
 the ${}_3F_{2}$-function by
\begin{equation*}
-\left\{\pFq{3}{2}{\frac{1}{4}+m, \frac{3}{4}+m, 1}{\frac{1}{4},
    \frac{3}{4}}{-\frac{x^2}{n^2}}-1\right\},
\end{equation*}
as can be seen from \eqref{hyident}, and then use the series expansion of
this other ${}_3F_{2}$-function.

\subsection{The Case $m=0$}

When $m=0$ in Theorem \ref{coal1}, we obtain the following corollary.

\begin{corollary}\label{fir}
Let $x\notin\mathbb{Z}$ and $x>0$. Then,
\begin{align}\label{corhalf}
&\sum_{n=1}^{\infty}\frac{\sigma_{1/2}(n)}{\sqrt{n}}e^{-2\pi\sqrt{2nx}}
\sin\left(\frac{\pi}{4}+2\pi\sqrt{2nx}\right)\\
&=\frac{1}{2}\left\{\left(\frac{1}{\pi\sqrt{x}}-\frac{1}{\sqrt{2}}\right)\zeta\left(\frac{1}{2}\right)-\frac{1}{\pi x\sqrt{2}}\zeta\left(-\frac{1}{2}\right)\right\}+\frac{x}{\pi\sqrt{2}}
\sum_{n=1}^{\infty}\frac{\sigma_{1/2}(n)}{\sqrt{n}}\frac{(\sqrt{2x}-\sqrt{n})}{x^2+n^2}. \nonumber
\end{align}
\end{corollary}

\begin{proof}
The corollary follows readily from Theorem \ref{coal1}. We only need to
observe that when $n<x$,
\begin{equation*}
\pFq{3}{2}{\frac{1}{4}, \frac{3}{4}, 1}{\frac{1}{4}, \frac{3}{4}}{-\frac{n^2}{x^2}}=\frac{x^2}{x^2+n^2},
\end{equation*}
and when $n>x$,
\begin{align*}
\pFq{3}{2}{\frac{1}{4}, \frac{3}{4}, 1}{\frac{1}{4}, \frac{3}{4}}{-\frac{n^2}{x^2}}&=-\left\{\pFq{3}{2}{\frac{1}{4}+m, \frac{3}{4}+m, 1}{\frac{1}{4}, \frac{3}{4}}{-\frac{x^2}{n^2}}-1\right\}\\
&=-\left(\frac{1}{1+x^2/n^2}-1\right)
=\frac{x^2}{x^2+n^2}
\end{align*}
to complete our proof.
\end{proof}

Similarly, when $m=0$ in Theorem \ref{coal2}, we derive the following corollary.

\begin{corollary} For $x>0$,
\begin{align}\label{3aux}
&\sum_{n=1}^{\infty}\frac{\sigma_{1/2}(n)}{\sqrt{n}}e^{-2\pi\sqrt{2nx}}
\sin\left(\frac{\pi}{4}-2\pi\sqrt{2nx}\right)\nonumber\\
&=\left(2\pi\sqrt{x}+\frac{1}{2\sqrt{2}\pi x}\right)\zeta\left(-\frac{1}{2}\right)-\frac{1}{2\sqrt{2}}\zeta\left(\frac{1}{2}\right)
+\frac{x}{\pi\sqrt{2}}\sum_{n=1}^{\infty}\frac{\sigma_{1/2}(n)}{x^2+n^2}.
\end{align}
\end{corollary}

We now show that the two previous corollaries can also be obtained by evaluating special cases of the infinite series
\begin{equation}
2\sum_{n=1}^{\infty}\sigma_{-s}(n)n^{\frac{1}{2}s}\left(e^{\pi is/4}K_{s}\left(4\pi e^{\pi i/4}\sqrt{nx}\right)\mp e^{-\pi is/4}K_{s}\left(4\pi e^{-\pi i/4}\sqrt{nx}\right)\right).
\end{equation}

\medskip

\begin{proof}[Second Proof of Corollary \ref{fir}.]
Use the remarks following \eqref{ab} and then replace $x$ by $xe^{\pi i/2}$ and by $xe^{-\pi i/2}$ in \eqref{cohenres}, and
then subtract the resulting two identities to obtain, in particular for $x>0$,
\begin{align}\label{two}
&2\sum_{n=1}^{\infty}\sigma_{-s}(n)n^{\frac{s}{2}}\left(e^{\pi is/4}K_{s}\left(4\pi e^{\pi i/4}\sqrt{nx}\right)-e^{-\pi is/4}K_{s}\left(4\pi e^{-\pi i/4}\sqrt{nx}\right)\right)\nonumber\\
&=-\frac{ix^{s/2-1}}{2\pi}\cot\left(\frac{\pi s}{2}\right)\zeta(s)-\frac{i(2\pi)^{-s-1}}{\pi x^{1+s/2}}\Gamma(s+1)\zeta(s+1)-\frac{ix^{s/2}}{2}\tan\left(\frac{\pi s}{2}\right)\zeta(s+1)\nonumber\\
&\quad+\frac{i\pi x}{6}\frac{\zeta(2-s)}{\sin\left(\frac{1}{2}\pi
    s\right)}-\frac{ix^{3-s/2}}{\pi\sin\left(\frac{1}{2}\pi s\right)}\sum_{n=1}^{\infty}\frac{\sigma_{-s}(n)}{x^2+n^2}\left(n^{s-2}+x^{s-2}\cos\left(\frac{\pi s}{2}\right)\right).
\end{align}
Now let $s=-\tf12$ in \eqref{two}. Using \eqref{K1} and \eqref{K2},
we see that the left-hand side simplifies to
\begin{align}
&\frac{1}{\sqrt{2}x^{1/4}}\sum_{n=1}^{\infty}\frac{\sigma_{1/2}(n)}{n^{1/4}}\left(e^{-\pi i/4-4\pi e^{\pi i/4}\sqrt{nx}}-e^{\pi i/4-4\pi e^{-\pi i/4}\sqrt{nx}}\right)\notag\\
&=-\frac{i\sqrt{2}}{x^{1/4}}\sum_{n=1}^{\infty}
\frac{\sigma_{1/2}(n)}{\sqrt{n}}e^{-2\pi\sqrt{2nx}}
\sin\left(\frac{\pi}{4}+2\pi\sqrt{2nx}\right).\label{ltwo3}
\end{align}
The right-hand side of (\ref{two}) becomes
{\allowdisplaybreaks\begin{align}\label{rtwo1}
&\frac{i}{2\pi x^{5/4}}\zeta\left(-\frac{1}{2}\right)-\frac{i}{\sqrt{2}\pi x^{3/4}}\zeta\left(\frac{1}{2}\right)+\frac{i}{2x^{1/4}}\zeta\left(\frac{1}{2}\right)\nonumber\\
&\quad-\frac{i\pi x^{5/4}}{3\sqrt{2}}\zeta\left(\frac{5}{2}\right)+\frac{ix^{3/4}}{\pi}\sum_{n=1}^{\infty}
\frac{\sigma_{1/2}(n)}{x^2+n^2}+\frac{i\sqrt{2}x^{13/4}}{\pi}\sum_{n=1}^{\infty}\frac{\sigma_{1/2}(n)}{n^{5/2}(x^2+n^2)}.
\end{align}}%
Thus, from (\ref{ltwo3}) and (\ref{rtwo1}), we deduce that
\begin{align}\label{corhalf2}
&\sum_{n=1}^{\infty}\frac{\sigma_{1/2}(n)}{\sqrt{n}}e^{-2\pi\sqrt{2nx}}
\sin\left(\frac{\pi}{4}+2\pi\sqrt{2nx}\right)\nonumber\\
&=\frac{1}{2}\left\{\left(\frac{1}{\pi\sqrt{x}}-\frac{1}{\sqrt{2}}\right)\zeta\left(\frac{1}{2}\right)-\frac{1}{\pi
    x\sqrt{2}}\zeta\left(\frac{-1}{2}\right)\right\}+\frac{\pi
  x^{3/2}}{6}\zeta\left(\frac{5}{2}\right)\notag\\
&\quad-\frac{x}{\pi\sqrt{2}}\sum_{n=1}^{\infty}\frac{\sigma_{1/2}(n)}{x^2+n^2}
-\frac{x^{7/2}}{\pi}\sum_{n=1}^{\infty}\frac{\sigma_{1/2}(n)}{n^{5/2}(x^2+n^2)}.
\end{align}
From (\ref{corhalf}) and (\ref{corhalf2}), it is clear that we want to prove
that
\begin{equation}\label{zeta}
\frac{\pi
  x^{3/2}}{6}\zeta\left(\frac{5}{2}\right)-\frac{x^{7/2}}{\pi}\sum_{n=1}^{\infty}
\frac{\sigma_{1/2}(n)}{n^{5/2}(x^2+n^2)}=\frac{x^{3/2}}{\pi}
\sum_{n=1}^{\infty}\frac{\sigma_{1/2}(n)}{\sqrt{n}(x^2+n^2)}.
\end{equation}
To that end, observe that
\begin{align*}
\frac{x^{7/2}}{\pi}\sum_{n=1}^{\infty}\frac{\sigma_{1/2}(n)}{n^{5/2}(x^2+n^2)}
+\frac{x^{3/2}}{\pi}
\sum_{n=1}^{\infty}\frac{\sigma_{1/2}(n)}{x^2+n^2}
=\frac{x^{3/2}}{\pi}\sum_{n=1}^{\infty}\frac{\sigma_{1/2}(n)}{n^{5/2}}.
\end{align*}
Finally, from (\ref{sz}) and the fact that $\zeta(2)=\pi^2/6$, we find that
\begin{equation}
\sum_{n=1}^{\infty}\frac{\sigma_{1/2}(n)}{n^{5/2}}=\frac{\pi^2}{6}\zeta\left(\frac{5}{2}\right).
\end{equation}
This proves (\ref{zeta}) and hence completes an alternative proof of (\ref{corhalf}).
\end{proof}

Similarly, if we let $s=-\tf12$ in (\ref{one}), then we obtain \eqref{3aux} upon simplification. Adding (\ref{corhalf}) and (\ref{3aux}), we obtain the following result.


\begin{theorem}\label{sec}
Let $x\notin\mathbb{Z}$. Then,
\begin{align*}
&\sum_{n=1}^{\infty}\frac{\sigma_{1/2}(n)}{\sqrt{n}}e^{-2\pi\sqrt{2nx}}
\cos\left(2\pi\sqrt{2nx}\right)\nonumber\\
&=\left(\frac{1}{2\pi\sqrt{2x}}-\frac{1}{2}\right)\zeta\left(\frac{1}{2}\right)
+\pi\sqrt{2x}\zeta\left(-\frac{1}{2}\right)+\frac{x^{3/2}}{\pi\sqrt{2}}
\sum_{n=1}^{\infty}\frac{\sigma_{1/2}(n)}{\sqrt{n}(x^2+n^2)}.
\end{align*}
\end{theorem}

Subtracting (\ref{corhalf}) from (\ref{3aux}) gives the next result.

\begin{theorem}\label{thi}
Let $x\notin\mathbb{Z}$. Then,
\begin{align*}
&\sum_{n=1}^{\infty}\frac{\sigma_{1/2}(n)}{\sqrt{n}}e^{-2\pi\sqrt{2nx}}
\sin\left(2\pi\sqrt{2nx}\right)\\
&=\frac{1}{2\pi\sqrt{2x}}\zeta\left(\frac{1}{2}\right)-\left(\frac{1}{2\pi x}+\pi\sqrt{2x}\right)\zeta\left(-\frac{1}{2}\right)+\frac{x}{\pi\sqrt{2}}
\sum_{n=1}^{\infty}\frac{\sigma_{1/2}(n)}{\sqrt{n}}\frac{(\sqrt{x}-\sqrt{2n})}{x^2+n^2}. \nonumber
\end{align*}
\end{theorem}

\section{Connection with the Vorono\"{\dotlessi} Summation Formula}\label{sect5}

A celebrated formula of Vorono\"{\dotlessi}  \cite{voronoi} for $\sum_{n\leq x}d(n)$ is given by
\begin{align}\label{vsf}
\sideset{}{'}\sum_{n\leq x}d(n)&=x(\log x+(2\gamma-1))+\frac{1}{4}\nonumber\\
&\quad+\sqrt{x}\sum_{n=1}^{\infty}\frac{d(n)}{\sqrt{n}}\left(-Y_{1}(4\pi\sqrt{nx})-\frac{2}{\pi}K_{1}(4\pi\sqrt{nx})\right),
\end{align}
where $Y_{\nu}(x)$ denotes the Bessel function of order $\nu$ of the second kind, and $K_{\nu}(x)$ denotes the modified Bessel function of order $\nu$. Thus, the error term $\Delta(x)$ in the Dirichlet divisor problem \eqref{ddp} admits the infinite series representation
\begin{equation*}
\Delta(x)=\sqrt{x}\sum_{n=1}^{\infty}\frac{d(n)}{\sqrt{n}}\left(-Y_{1}(4\pi\sqrt{nx})-\frac{2}{\pi}K_{1}(4\pi\sqrt{nx})\right).
\end{equation*}

In \cite{voronoi}, Vorono\"{\dotlessi} also gave a more general form of \eqref{vsf}, namely,
\begin{align}\label{kvsf}
\sum_{\alpha<n<\beta}d(n)f(n)&=\int_{\alpha}^{\beta}(2\gamma+\log t)f(t)\, dt\nonumber\\
&\quad+2\pi\sum_{n=1}^{\infty}d(n)\int_{\alpha}^{\beta}f(t)\left(\dfrac{2}{\pi}K_{0}(4\pi\sqrt{nt})-Y_{0}(4\pi\sqrt{nt})\right)\, dt,
\end{align}
where $f(t)$ is a function of bounded variation in $(\alpha,\beta)$ and $0<\alpha<\beta$. A.~L.~Dixon and W.~L.~Ferrar \cite{dixfer1} gave a proof of \eqref{kvsf} under the more restrictive condition that $f$ has a bounded second differential coefficient in $(\a,\b)$. J.~R.~Wilton \cite{wilton} proved \eqref{kvsf} under less restrictive conditions. In his proof, he assumed $f(t)$ has compact support on $[\a,\b]$ and $V_{\a}^{\b-\epsilon}f(t)\to V_{\a}^{\b-0}f(t)$ as $\epsilon$ tends to $0$. Here $V_{\a}^{\b}f(t)$ denotes the total variation of $f(t)$ over $(\a,\b)$. In 1929,   Koshliakov \cite{koshliakov} gave a very short proof of \eqref{kvsf} for $0<\alpha<\beta$, $\alpha, \beta\notin\mathbb{Z}$, for $f$ analytic inside a closed contour strictly containing the interval $[\alpha, \beta]$.
%
Koshliakov's proof in \cite{koshliakov} is based on the series $\varphi(x)$, defined in \eqref{minusspl}, and its representation
\begin{equation*}
\varphi(x)=-\gamma-\frac{1}{2}\log x-\frac{1}{4\pi x}+\frac{x}{\pi}\sum_{n=1}^{\infty}\frac{d(n)}{x^2+n^2}.
\end{equation*}
 See also \cite{V, bcbconf} for Vorono\"{\dotlessi}-type summation formulas for a large class of arithmetical functions generated by Dirichlet series satisfying a functional equation involving the Gamma function. For Vorono\"{\dotlessi}-type summation formulas involving an exponential factor, see \cite{jutila}. The Vorono\"{\dotlessi} summation formula has been found to be useful in physics too; for example, S.~Egger and F.~Steiner \cite{es1, es2} showed that it plays the role of an exact trace formula for a Schr\"{o}dinger operator on a certain non-compact quantum graph. They also gave a short proof of the Vorono\"{\dotlessi} summation formula in \cite{es3}.

The extension of \eqref{kvsf} for $\a=0$ is somewhat more difficult, since one needs to impose a further condition on $f(t)$. When $f''(t)$ is bounded in $(\delta, \a)$ and $t^{3/4}f''(t)$ is integrable over $(0, \delta)$ for $0<\delta<\a$, Dixon and Ferrar \cite{dixfer1} proved that
 \begin{align}\label{kvsf1}
\sum_{0<n<\b}d(n)f(n)&=\frac{f(0+)}{4}+\int_{0}^{\b}(2\gamma+\log t)f(t)\, dt\nonumber\\
&\quad+2\pi\sum_{n=1}^{\infty}d(n)\int_{0}^{\b}f(t)\left(\dfrac{2}{\pi}K_{0}(4\pi\sqrt{nt})-Y_{0}(4\pi\sqrt{nt})\right)\, dt.
\end{align}
 Wilton \cite{wilton} obtained \eqref{kvsf1} under the assumption that $\log x V_{0+}^{x}f(t)$ tends to $0$ as $x\to 0+$. D.~A.~Hejhal \cite{hejhal} gave a proof of \eqref{kvsf1} for $\b\to \infty$ under the assumption that $f$ is twice continuously differentiable and possesses compact support. For other proofs of the Vorono\"{\dotlessi} summation formula, the reader is referred to papers by  T.~Meurman \cite{meurmanast} and A.~Ivi\'{c} \cite{ivic}.

Consider the following Vorono\"{\dotlessi} summation formula in an extended
form due to A.~Oppenheim \cite{oppenheim}, and in the version given by
A.~Laurin\u{c}ikas \cite{al}.  For $x>0,
x\notin\mathbb{Z}$, and $-\tf12<\s <\tf12$,
\begin{align}\label{lauform}
\sum_{n<x}&\sigma_{-s}(n)=\zeta(1+s)x+\frac{\zeta(1-s)}{1-s}x^{1-s}-\frac{1}{2}\zeta(s)
+\frac{x}{2\sin\left(\frac{1}{2}\pi s\right)}\sum_{n=1}^{\infty}\sigma_{s}(n)\\&\times\left(\sqrt{nx}\right)^{-1-s}
\left(J_{s-1}(4\pi\sqrt{nx})+J_{1-s}(4\pi\sqrt{nx})-\frac{2}{\pi}\sin(\pi s)K_{1-s}(4\pi\sqrt{nx})\right),\nonumber
\end{align}
so that, by \eqref{gddp}, $\Delta_{-s}(x)$ is represented by the expression
involving the series on the right-hand side of \eqref{lauform}. (Note that
Laurin\u{c}ikas proved \eqref{lauform} for $0<s<\tf12$. However, one can extend it to
$-\tf12<\sigma<\tf12$.) Wilton
\cite{wiltonextended} proved the same result in a more general setting by
considering the `integrated function', that is, the Riesz sum
\begin{align*}
\frac{1}{\Gamma(\lambda+1)}\sideset{}{'}\sum_{n\leq x}\sigma_{-s}(n)(x-n)^{\lambda}.
\end{align*}
Laurin\u{c}ikas \cite{al} gave a different proof of \eqref{lauform} many
years later.

We will now explain the connection of Ramanujan's series
\begin{equation*}
\sum_{n=1}^{\infty}\frac{\sigma_s(n)}{\sqrt{n}}e^{-2\pi\sqrt{2nx}}
\sin\left(\frac{\pi}{4}+2\pi\sqrt{2nx}\right)
\end{equation*}
and its companion with the extended form of the Vorono\"{\dotlessi} summation formula.

As mentioned by Hardy \cite{hardiv}, \cite[pp.~268--292]{hardycpII}, if we use the  asymptotic formulas \eqref{asymbess1} and \eqref{asymbess2} for $Y_{1}(4\pi\sqrt{nx})$ and $K_{1}(4\pi\sqrt{nx})$, respectively, in \eqref{vsf}, we find that
\begin{equation}\label{impser2}
\Delta(x)=\frac{x^{1/4}}{\pi\sqrt{2}}\sum_{n=1}^{\infty}\frac{d(n)}{n^{3/4}}\cos\left(4\pi\sqrt{nx}-\frac{\pi}{4}\right)+R(x),
\end{equation}
where $R(x)$ is a series absolutely and uniformly convergent for all positive
values of $x$. The first series on the left side of \eqref{impser2} is
convergent for all real values of $x$, and uniformly convergent throughout
any compact interval not containing an integer. At each integer $x$, it has
a finite discontinuity.

If we replace the Bessel functions in \eqref{lauform} by
their asymptotic expansions, namely \eqref{asymbess} and \eqref{asymbess2}, similar
to what Hardy did, then the most important part of the error term
$\Delta_{-s}(x)$ is given by
\begin{equation*}
\frac{x^{\frac{1}{4}-\frac{1}{2}s}\cot\left(\frac{1}{2}\pi s\right)}{\pi\sqrt{2}}\sum_{n=1}^{\infty}\frac{\sigma_s(n)}{n^{\frac{s}{2}+\frac{3}{4}}}
\cos\left(4\pi\sqrt{nx}-\frac{\pi}{4}\right).
\end{equation*}
This series, though similar to the one in \eqref{impser2} or in \eqref{impser}, is different from Ramanujan's series \eqref{ramser} in that the exponential factor, namely $e^{-2\pi\sqrt{2nx}}$, is not present.

A generalization of \eqref{minusspl}, namely,
\begin{equation}\label{plus}
\varphi(x, s):=2\sum_{n=1}^{\infty}\sigma_{-s}(n)n^{\frac{1}{2}s}\left(e^{\pi is/4}K_{s}\left(4\pi e^{\pi i/4}\sqrt{nx}\right)+e^{-\pi is/4}K_{s}\left(4\pi e^{-\pi i/4}\sqrt{nx}\right)\right),
\end{equation}
was studied in \cite{dixitmoll}. Note that $\varphi(x, 0)=\varphi(x)$, and that $\varphi(x)$ was used by Koshliakov \cite{koshliakov} in his short proof of \eqref{kvsf}.

Replacing the Bessel functions in \eqref{plus} by their asymptotic expansions from \eqref{asymbess2}, we find that the main terms are given by
\begin{align}\label{2ser}
&\frac{\sqrt{2}}{x^{1/4}}\cos\left(\frac{\pi}{4}\left(s+\frac{1}{2}\right)\right)
\sum_{n=1}^{\infty}\frac{\sigma_{s}(n)}{n^{s/2+1/4}}e^{-2\pi\sqrt{2nx}}
\sin\left(\frac{\pi}{4}-2\pi\sqrt{2nx}\right)\nonumber\\
&+\frac{\sqrt{2}}{x^{1/4}}\sin\left(\frac{\pi}{4}\left(s+\frac{1}{2}\right)\right)\sum_{n=1}^{\infty}
\frac{\sigma_{s}(n)}{n^{s/2+1/4}}e^{-2\pi\sqrt{2nx}}
\sin\left(\frac{\pi}{4}+2\pi\sqrt{2nx}\right).
\end{align}
 In our extensive study, the forms of the series in \eqref{2ser} are the closest that we could find that resemble the series in Ramanujan's original claim \eqref{qc}, or in our Theorem \ref{p3361}, or the companion series
 \begin{equation*}
 \sum_{n=1}^{\infty}\frac{\sigma_s(n)}{\sqrt{n}}e^{-2\pi\sqrt{2nx}}
\sin\left(\frac{\pi}{4}- 2\pi\sqrt{2nx}\right).
\end{equation*}
Note that the only place where they differ is in the power of $n$. Similar remarks can be made about \eqref{minus} and \eqref{2ser}.

Series similar to these arise in the mean square estimates of
$\int_{1}^{x}\Delta_{-s}(t)^2\, dt$ by Meurman \cite[equations (3.7),
(3.8)]{meurman}.  (An excellent survey on recent progress on divisor problems and mean square theorems has been written by K.--M.~Tsang \cite{tsang}.)  Similar series have also arisen in the work of H.~Cram\'{e}r
\cite{cramer}, and in the recent work of S.~Bettin and J.~B.~Conrey
\cite[p.~220--223]{betcon}. Thus it seems that the  two series in \eqref{2ser} are more
closely connected to the generalized Dirichlet divisor problem than are
Ramanujan's series and its companion. We have found identities, similar to
those in Theorems \ref{p3361} and \ref{p3361a}, for each of the series in
\eqref{2ser}. However, we refrain ourselves from stating them as they are
similar to the ones already proved.

\textbf{Remark.} It is interesting to note here that at the bottom of page $368$ in \cite{lnb},
one finds the following note in Hardy's handwriting: ``Idea. You can replace the Bessel functions of the Vorono\"{\dotlessi}
identity by circular functions, at the price of complicating the `sum'. Interesting idea, but probably of no value for the study of the divisor
problem.'' In view of the applications of such series mentioned in the above paragraph, it appears that Hardy's judgement was incorrect.

The series in \eqref{plus} can be used to derive an extended form of the
Vorono\"{\dotlessi} summation formula \eqref{kvsf} in the form contained in
the following theorem. This proof generalizes the technique enunciated by Koshliakov in \cite{koshliakov}.

\begin{theorem}\label{varlauform}
Let $0<\alpha<\beta$ and $\alpha, \beta\notin\mathbb{Z}$. Let $f$ denote a
function analytic inside a closed contour strictly containing  $[\alpha,
\beta]$. Assume that $-\tf12< \s <\tf12$. Then,
\begin{align}\label{varlauform1}
\sum_{\alpha<j<\beta}\sigma_{-s}(j)f(j)&=\int_{\alpha}^{\beta}(\zeta(1+s)+t^{-s}\zeta(1-s))f(t)\, dt\nonumber\\
&\quad+2\pi\sum_{n=1}^{\infty}\sigma_{-s}(n)n^{\frac{1}{2}s}\int_{\alpha}^{\beta}t^{-\frac{1}{2}s}f(t)
\bigg\{\left(\frac{2}{\pi}K_{s}(4\pi\sqrt{nt})-Y_{s}(4\pi\sqrt{nt})\right)\nonumber\\
&\quad\quad\quad\quad\quad\quad\times\cos\left(\frac{\pi s}{2}\right)-J_{s}(4\pi\sqrt{n t})\sin\left(\frac{\pi s}{2}\right)\bigg\}\, dt.
\end{align}
\end{theorem}
We wish to extend \eqref{varlauform1} to allow $\a=0$ so as to obtain \eqref{lauform} as a special case of Theorem \ref{varlauform}. To do this, we need to impose some additional
restrictions on $f$. As an intermediate result, we state the following theorem which generalizes Theorem 3 in \cite{wilton}.

\begin{theorem}\label{voronoisumgen}
Let $0<\a<\tf12$, $-\tf12<\s <\tf12$, and $0<\theta<\min\left(1,\frac{1+2\s}{1-2\s}\right)$. Let $N\in\mathbb{N}$ such that $N^{\theta}\a>1$. If $f$ is twice differentiable as a function of $t$, and is of bounded variation in $(0,\a)$,  then as $N\to\infty$,
\begin{align*}
&f(0+)\frac{\z(-s)}{2}-\int_{0}^{\a}f(t)(\z(1-s)+t^s\z(1+s))\, dt\nonumber\\
&\quad+2\pi\sideset{}{'}\sum_{n=1}^{N}\frac{\s_s(n)}{n^{s/2}}\int_{0}^{\a}f(t)t^{\frac{1}{2}s}
\bigg\{J_{s}(4\pi\sqrt{n t})\sin\left(\frac{\pi s}{2}\right)\nonumber\\
&\hspace{5cm}+\left(Y_{s}(4\pi\sqrt{n t})-\frac{2}{\pi}K_{s}(4\pi\sqrt{n t})\right)\cos\left(\frac{\pi s}{2}\right)\bigg\}\, dt\nonumber\\
&\ll\twopartdef{ (2\g+\log N)(V_{0}^{N^{-\theta}}f(t)+N^{(\theta-1)/4}(|f(\a)|+V_{0}^{\a}f(t))),}{s=0,\\
\vspace{-5pt}}{V_{0}^{N^{-\theta}}f(t)+(N^{(1-\theta)(2\s-1)/4}+N^{(\theta(1-2\s)-(2\s+1))/4})\\
\hspace{5cm}\times(|f(\a)|+V_{0}^{\a}f(t)),}{s\neq 0.}
\end{align*}
Additionally, if we assume the limits
\begin{align}\label{limbdvar}
\lim_{x\to 0+}V_{0}^{x}f(t)=0,\text{ if }s\neq 0\quad\text{ and }\quad\lim_{x\to 0+}\log x\, V_{0}^{x}f(t)=0,\text{ if }s=0,
\end{align}
then
\begin{align}\label{varlauform3}
&f(0+)\frac{\z(-s)}{2}-\int_{0}^{\a}f(t)(\z(1-s)+t^s\z(1+s))\, dt\nonumber\\
&\quad+2\pi\sum_{n=1}^{\infty}\frac{\s_s(n)}{n^{s/2}}\int_{0}^{\a}f(t)t^{\frac{s}{2}}
\bigg\{J_{s}(4\pi\sqrt{n t})\sin\left(\frac{\pi s}{2}\right)\nonumber\\
&\hspace{3cm}+\left(Y_{s}(4\pi\sqrt{n t})-\frac{2}{\pi}K_{s}(4\pi\sqrt{n t})\right)\cos\left(\frac{\pi s}{2}\right)\bigg\}\, dt=0.
\end{align}
\end{theorem}
Clearly, for $0<\a<\tf12$, we have
\begin{align}\label{rhszero}
\sideset{}{'}\sum_{0<j\leq \a}\sigma_{-s}(j)f(j)=0.
\end{align}
Also, if we substitute for $ Y_{s}(4\pi\sqrt{n t})$ via \eqref{yj} and  employ \eqref{K2}, we find  that the kernel in \eqref{varlauform3}, namely,
\begin{equation*}
J_{s}(4\pi\sqrt{n t})\sin\left(\frac{\pi s}{2}\right)+\left(Y_{s}(4\pi\sqrt{n t})-\frac{2}{\pi}K_{s}(4\pi\sqrt{n t})\right)\cos\left(\frac{\pi s}{2}\right)
\end{equation*}
is invariant under the replacement of $s$ by $-s$.  Therefore replacing $s$ by $-s$ in \eqref{varlauform3}, then
replacing zero on the right-hand side of \eqref{varlauform3} by
$-\sum_{0<j\leq \a}\sigma_{-s}(j)f(j)$ using \eqref{rhszero},
and then finally subtracting the resulting equation so obtained from
\eqref{varlauform1}, we arrive at the following result.

\begin{theorem}\label{vorsumgen}
Let $0<\alpha<\tf12, \a<\beta$ and $\beta\notin\mathbb{Z}$. Let $f$ denote a
function analytic inside a closed contour strictly containing  $[\alpha,
\beta]$, and of bounded variation in $0<t<\a$. Furthermore, if $f$ satisfies the limit conditions in
\eqref{limbdvar}, and  $-\tf12<\s <\tf12$, then
\begin{align*}
\sum_{0<j<\beta}\sigma_{-s}(j)f(j)&=-f(0+)\frac{\z(s)}{2}+\int_{0}^{\beta}(\zeta(1+s)+t^{-s}\zeta(1-s))f(t)\, dt\nonumber\\
&\quad+2\pi\sum_{n=1}^{\infty}\sigma_{-s}(n)n^{\frac{1}{2}s}\int_{0}^{\beta}t^{-\frac{1}{2}s}f(t)
\bigg\{\left(\frac{2}{\pi}K_{s}(4\pi\sqrt{nt})-Y_{s}(4\pi\sqrt{nt})\right)\nonumber\\
&\quad\quad\quad\quad\quad\quad\times\cos\left(\frac{\pi s}{2}\right)-J_{s}(4\pi\sqrt{n t})\sin\left(\frac{\pi s}{2}\right)\bigg\}\, dt.
\end{align*}
\end{theorem}


\subsection{Oppenheim's Formula \eqref{lauform} as a Special Case of Theorem \ref{vorsumgen}} Letting $\lambda=-s+1$, $\mu=s$, and $x=4\pi\sqrt{nt}$ in \cite[p.~37, equation (1.8.1.1)]{prud}, \cite[p.~42, equation (1.9.1.1)]{prud} \footnote{This formula, as is stated, contains many misprints. The correct version should read
\begin{align*}
\int_{x_1}^{x_2}y^{\lambda}Y_{\nu}(y)\, dy &={-1 \brace 1} \frac{\cos(\nu\pi)\G(-\nu) x^{\lambda+\nu+1}}{2^{\nu}\pi (\lambda+\nu+1)}{}_1F_{2}\left(\frac{\lambda+\nu+1}{2};1+\nu,\frac{\lambda+\nu+3}{2};-\frac{x^2}{4}\right)\nonumber\\
&\quad+{-1 \brace 1}\frac{2^{\nu}\G(\nu)x^{\lambda-\nu+1}}{\pi (\lambda-\nu+1)}{}_1F_{2}\left(\frac{\lambda-\nu+1}{2};1-\nu,\frac{\lambda-\nu+3}{2};-\frac{x^2}{4}\right)\nonumber\\
&\quad-{0 \brace 1}\frac{2^{\lambda}}{\pi}\cos\left(\frac{(\lambda-\nu+1)\pi}{2}\right)\G\left(\frac{\lambda+\nu+1}{2}\right)
\G\left(\frac{\lambda-\nu+1}{2}\right).\nonumber\\
&\quad\quad\quad\quad\quad\quad\quad\quad\quad\quad\quad\quad\quad\left[{{x_1=0, x_2=x;\hspace{2mm} \text{Re}(\lambda)>|\text{Re}(\nu)|-1} \brace {x_1=x, x_2=\infty;\hspace{2mm} \text{Re}(\lambda)<\frac{1}{2}}}\right].
\end{align*}}
and \cite[p.~47, equation (1.12.1.2)]{prud}, and then simplifying, we see that
{\allowdisplaybreaks\begin{align}\label{diffint}
&\int t^{-\frac{1}{2}s}\bigg\{\left(\frac{2}{\pi}K_{s}(4\pi\sqrt{nt})-Y_{s}(4\pi\sqrt{nt})\right)\cos\left(\frac{\pi s}{2}\right)-J_{s}(4\pi\sqrt{n t})\sin\left(\frac{\pi s}{2}\right)\bigg\}\, dt\nonumber\\
&=\frac{t^{(1-s)/2}}{4\pi\sqrt{n}\sin\left(\frac{1}{2}\pi s\right)}\left(J_{s-1}(4\pi\sqrt{nt})+J_{1-s}(4\pi\sqrt{nt})-\frac{2}{\pi}\sin(\pi s)K_{1-s}(4\pi\sqrt{nt})\right).
\end{align}}%
%
Let $f(t)\equiv 1$ and  $\beta=x\notin\mathbb{Z}$ in Theorem \ref{vorsumgen}. Then,
{allowdisplaybreaks\begin{align}\label{varlauformspl}
\sum_{j<x}\sigma_{-s}(j)&=-\frac{1}{2}\zeta(s)+\int_{0}^{x}(\zeta(1+s)+t^{-s}\zeta(1-s))\, dt\notag\\
&\quad+2\pi\sum_{n=1}^{\infty}
\sigma_{-s}(n)n^{\frac{1}{2}s}\int_{0}^{x}t^{-\frac{1}{2}s}
\bigg\{\left(\frac{2}{\pi}K_{s}(4\pi\sqrt{nt})
-Y_{s}(4\pi\sqrt{nt})\right)\nonumber\\
&\quad\times\cos\left(\frac{\pi s}{2}\right)-J_{s}(4\pi\sqrt{n t})\sin\left(\frac{\pi s}{2}\right)\bigg\}\, dt.
\end{align}}%
Note that
\begin{equation}\label{easyint}
\int(\zeta(1+s)+t^{-s}\zeta(1-s))\, dt=t\zeta(1+s)+\frac{t^{1-s}}{1-s}\zeta(1-s).
\end{equation}
Since $-\frac{1}{2}<\sigma<\frac{1}{2}$ and the right-hand sides of \eqref{diffint} and \eqref{easyint} vanish as $t$ tends to $0$, from \eqref{diffint}, \eqref{varlauformspl}, and \eqref{easyint}, we obtain \eqref{lauform}.\\

\textbf{Remark.} The analysis above also shows that for $\a>0, \a\notin\mathbb{Z}$,
{\allowdisplaybreaks\begin{align}\label{varlauformspl1}
	&\sum_{\alpha<j<x}\sigma_{-s}(j)=x\zeta(1+s)+\frac{x^{1-s}}{1-s}\zeta(1-s)-\a\zeta(1+s)-\frac{\a^{1-s}}{1-s}\zeta(1-s)\notag\\
	&+\frac{1}{2\sin\left(\frac{1}{2}\pi s\right)}\sum_{n=1}^{\infty}\frac{\sigma_{s}(n)}{n^{(s+1)/2}}
	\bigg\{x^{(1-s)/2}\biggl(J_{s-1}(4\pi\sqrt{nx})+J_{1-s}(4\pi\sqrt{nx})\notag\\
	&\hspace{1.6in}-\frac{2}{\pi}\sin(\pi s)K_{1-s}(4\pi\sqrt{nx})\biggr)\nonumber\\
	&-
	\alpha^{(1-s)/2}\left(J_{s-1}(4\pi\sqrt{n\a})+J_{1-s}(4\pi\sqrt{n\a})-\frac{2}{\pi}\sin(\pi s)K_{1-s}(4\pi\sqrt{n\a})\right)\bigg\}.
	\end{align}}%
From \eqref{lauform} and \eqref{varlauformspl1}, we conclude that, for $-\tf12< \s <\tf12$,
\begin{align*}
&\lim_{\a\to 0+}\frac{\a^{(1-s)/2}}{\sin\left(\frac{1}{2}\pi s\right)}\sum_{n=1}^{\infty}\frac{\sigma_{s}(n)}{n^{(s+1)/2}}
\biggl(J_{s-1}(4\pi\sqrt{n\a})+J_{1-s}(4\pi\sqrt{n\a})\notag\\
&\hspace{1.6in}-\frac{2}{\pi}\sin(\pi
s)K_{1-s}(4\pi\sqrt{n\a})\biggr)=\zeta(s),
\end{align*}
which is likely to be difficult to prove directly.




\section{Proof of Theorem \ref{varlauform}}\label{sect6}
We begin with a result due to H.~Cohen \cite[Theorem 3.4]{cohen}.

\begin{theorem}\label{cohen}
Let $x>0$ and $s\notin\mathbb{Z}$, where $\sigma\geq 0$ \footnote{As
  mentioned in \cite{cohen}, the condition $\sigma\geq 0$ is not restrictive
  since, because of \eqref{K2}, the left side of the identity in this theorem
  is invariant under the replacement of $s$ by $-s$.}. Then, for any integer $k$
such that $k\geq\lfloor\left(\sigma+1\right)/2\rfloor$,
{\allowdisplaybreaks\begin{align}\label{cohenres}
&8\pi x^{s/2}\sum_{n=1}^{\infty}\sigma_{-s}(n)n^{s/2}K_{s}(4\pi\sqrt{n x})=
A(s, x)\zeta(s)+B(s, x)\zeta(s+1)\\
&\quad+\frac{2}{\sin\left(\pi s/2\right)}\left(\sum_{1\leq j\leq k}\zeta(2j)\zeta(2j-s)x^{2j-1}+x^{2k+1}\sum_{n=1}^{\infty}\sigma_{-s}(n)\frac{n^{s-2k}-x^{s-2k}}{n^2-x^2}\right),\notag
\end{align}}%
where
\begin{align}\label{ab}
A(s, x)&=\frac{x^{s-1}}{\sin\left(\pi s/2\right)}-(2\pi)^{1-s}\Gamma(s),\nonumber\\
B(s, x)&=\frac{2}{x}(2\pi)^{-s-1}\Gamma(s+1)-\frac{\pi x^{s}}{\cos\left(\pi s/2\right)}.
\end{align}
\end{theorem}

By analytic continuation, the identity in Theorem \ref{cohen} is valid not only for $x>0$ but for $-\pi<\arg x<\pi$. Take $k=1$ in \eqref{cohenres}. The condition $\lfloor\left(\sigma+1\right)/2\rfloor\leq 1$ implies that $0\leq\sigma <3$. We consider $0\leq\s<\tf12$. Note that Koshliakov \cite{koshliakov}  proved the case $s=0$, and the theorem follows for the remaining values of $\s$, i.e., for $-\tf12<\s<0$, by the invariance noted in the previous footnote.

Replace $x$ by $iz$ in \eqref{cohenres} for $-\pi<\arg z<\frac{1}{2}\pi$, and then by $-iz$ for $-\frac{1}{2}\pi<\arg z<\pi$. Now add the resulting two identities and simplify, so that for $-\tfrac{1}{2}\pi<\arg z<\tfrac{1}{2}\pi$,
\begin{align}\label{one}
\Lambda(z,s)=\Phi(z,s),
\end{align}
where
\begin{equation}\label{Lambda}
\Lambda(z,s):=z^{-s/2}\varphi(z, s),
\end{equation}
with $\varphi(x, s)$ defined in \eqref{plus}, and
\begin{equation}\label{Phi}
\Phi(z,s):=-(2\pi z)^{-s}\Gamma(s)\zeta(s)+\frac{\zeta(s)}{2\pi z}-\frac{1}{2}\zeta(1+s)+\frac{z}{\pi}\sum_{n=1}^{\infty}\frac{\sigma_{-s}(n)}{z^2+n^2}.
\end{equation}
As a function of $z$, $\Phi(z, s)$ is analytic in the entire complex plane except on the negative real axis and at $z= in, n\in\mathbb{Z}$. Hence,
$\Phi(iz, s)$ is analytic in the entire complex plane except on the positive imaginary axis and at $z\in\mathbb{Z}$. Similarly, $\Phi(-iz, s)$ is analytic in the entire complex plane except on the negative imaginary axis and at $z= n\in\mathbb{Z}$. This implies that $\Phi(iz, s)+\Phi(-iz, s)$ is analytic in both the left and right half-planes, except possibly when $z$ is an integer. However, it is easy to see that
\begin{equation*}
\lim_{z\to\pm n}(z\mp n)\Phi(iz,s)=\frac{1}{2\pi i}\sigma_{-s}(n)\hspace{2mm}\text{and}\hspace{2mm} \lim_{z\to\pm n}(z\mp n)\Phi(-iz,s)=-\frac{1}{2\pi i}\sigma_{-s}(n),
\end{equation*}
so that
\begin{equation*}
\lim_{z\to\pm n}(z\mp n)\left(\Phi(iz, s)+\Phi(-iz, s)\right)=0.
\end{equation*}
In particular, this implies that $\Phi(iz, s)+\Phi(-iz, s)$ is analytic in
the entire right half-plane.

Now observe that for $z$ inside an interval $(u,v)$ on the positive real line
not containing any integer, we have, using the definition \eqref{Phi},
\begin{align}\label{mainreln}
\Phi(iz, s)+\Phi(-iz, s)=-2(2\pi
z)^{-s}\Gamma(s)\zeta(s)\cos\left(\tfrac{1}{2}
\pi s\right)-\zeta(1+s).
\end{align}
Since both $\Phi(iz, s)+\Phi(-iz, s)$ and $-2(2\pi
z)^{-s}\Gamma(s)\zeta(s)\cos\left(\tfrac{1}{2}\pi s\right)-\zeta(1+s)$ are
analytic in the right half-plane as functions of $z$, by analytic
continuation, the identity \eqref{mainreln} holds for any $z$ in the right
half-plane. Finally, using  the functional equation \eqref{fe} for $\zeta(s)$, we can simplify \eqref{mainreln} to deduce that, for
$-\tfrac{1}{2}\pi<\arg z<\tfrac{1}{2}\pi$,
\begin{align}\label{mainreln1}
\Phi(iz, s)+\Phi(-iz, s)=-z^{-s}\zeta(1-s)-\zeta(1+s).
\end{align}
Next, let $f$ be an analytic function of $z$ within a closed contour
intersecting the real axis in $\alpha$ and $\beta$, where $0<\alpha<\beta$,
$m-1<\alpha<m$, $n<\b<n+1$, and $m, n\in\mathbb{Z}$. Let $\gamma_1$ and
$\gamma_2$  denote the portions of the contour in the upper and
lower half-planes, respectively, so that the notations $\alpha\gamma_1\beta$
and $\alpha\gamma_2\beta$, for example, denote paths from $\alpha$ to $\beta$
in the upper and lower half-planes, respectively. By the residue theorem,
\begin{equation*}
\frac{1}{2\pi i}\int_{\alpha\gamma_2\beta\gamma_1\alpha}f(z)\Phi(iz,s)\,
dz=\sum_{\alpha<j<\beta}R_j(f(z)\Phi(iz,s)).
\end{equation*}
 Since $f(z)\Phi(iz,s)$ has a simple pole at each integer $j$,
 $\alpha<j<\beta$, with residue $\frac{1}{2\pi i}\sigma_{s}(j)f(j)$, we find
 that
\begin{align*}
&\sum_{\alpha<j<\beta}\sigma_{-s}(j)f(j)=\int_{\alpha\gamma_2\beta}f(z)\Phi(iz,s)\, dz-\int_{\alpha\gamma_1\beta}f(z)\Phi(iz,s)\, dz\nonumber\\
&=\int_{\alpha\gamma_2\beta}f(z)\Phi(iz,s)\, dz-\int_{\alpha\gamma_1\beta}f(z)\left(-\Phi(-iz,s)-z^{-s}\zeta(1-s)-\zeta(1+s)\right)\, dz\nonumber\\
&=\int_{\alpha\gamma_2\beta}f(z)\Phi(iz,s)\, dz+\int_{\alpha\gamma_1\beta}f(z)\Phi(-iz,s)\, dz
+\int_{\alpha\gamma_1\beta}f(z)\left(z^{-s}\zeta(1-s)+\zeta(1+s)\right)\, dz,
\end{align*}
where in the penultimate step, we used \eqref{mainreln1}. Using the residue
theorem again, we readily see that
\begin{equation*}
\int_{\alpha\gamma_1\beta}f(z)\left(z^{-s}\zeta(1-s)+\zeta(1+s)\right)\,
dz=\int_{\alpha}^{\beta}f(t)\left(\zeta(1+s)+t^{-s}\zeta(1-s)\right)\, dt.
\end{equation*}
Since $\Lambda(z, s)=\Phi(z, s)$ for $-\tfrac{1}{2}\pi<\arg
z<\tfrac{1}{2}\pi$, it is easy to see that $\Lambda(iz, s)=\Phi(iz, s)$, for
$-\pi<\arg z<0$, and $\Lambda(-iz, s)=\Phi(-iz, s)$, for $0<\arg
z<\pi$. Thus,
\begin{align}\label{genvor1}
\sum_{\alpha<j<\beta}\sigma_{-s}(j)f(j)&=\int_{\alpha\gamma_2\beta}f(z)\Lambda(iz,s)\,
dz+\int_{\alpha\gamma_1\beta}f(z)\Lambda(-iz,s)\, dz\nonumber\\
&\quad+\int_{\alpha}^{\beta}f(t)\left(\zeta(1+s)+t^{-s}\zeta(1-s)\right)\, dt.
\end{align}
Using the asymptotic expansion \eqref{asymbess2}, we see that the series
\begin{align*}
\Lambda(iz,s)=2(iz)^{-\frac{s}{2}}\sum_{n=1}^{\infty}\sigma_{-s}(n)n^{\frac{1}{2}s}\left(e^{i\pi
      s/4}K_{s}\left(4\pi
    e^{i\pi/4}\sqrt{inz}\right)\right.
\left.+e^{-i\pi s/4}K_{s}\left(4\pi
    e^{-i\pi/4}\sqrt{inz}\right)\right)
\end{align*}
is uniformly convergent in compact subintervals of $-\pi<\arg z<0$, and the
series
\begin{align*}
\Lambda(-iz,s)=2(-iz)^{-\frac{1}{2}s}\sum_{n=1}^{\infty}\sigma_{-s}(n)n^{\frac{1}{2}s}&\left(e^{i\pi
      s/4}K_{s}\left(4\pi e^{i\pi/4}\sqrt{-inz}\right)
\right.\notag\\&\left.+e^{-i\pi s/4}K_{s}\left(4\pi
    e^{-i\pi/4}\sqrt{-inz}\right)\right)
\end{align*}
is uniformly convergent in compact subsets of $0<\arg z<\pi$. Thus,
interchanging the order of summation and integration in \eqref{genvor1}, we
deduce that
{\allowdisplaybreaks\begin{align*}
\sum_{\alpha<j<\beta}\sigma_{-s}(j)f(j)
&=2\sum_{n=1}^{\infty}\sigma_{-s}(n)n^{\frac{1}{2}s}\int_{\alpha\gamma_2\beta}f(z)(iz)^{-\frac{1}{2}s}
\left(e^{i\pi s/4}K_{s}\left(4\pi e^{i\pi/4}\sqrt{inz}\right)
\right.\notag\\
&\left.\hspace{1.9in}+e^{-i\pi s/4}K_{s}\left(4\pi e^{-i\pi/4}\sqrt{inz}\right)\right)\, dz\nonumber\\
&\quad+2\sum_{n=1}^{\infty}\sigma_{-s}(n)n^{\frac{1}{2}s}
\int_{\alpha\gamma_1\beta}f(z)(-iz)^{-\frac{1}{2}s}\left(e^{i\pi
      s/4}K_{s}\left(4\pi
    e^{i\pi/4}\sqrt{-inz}\right)\right.\notag\\
&\left.\hspace{1.9in}+e^{-i\pi s/4}K_{s}\left(4\pi
    e^{-i\pi/4}\sqrt{-inz}\right)\right)\, dz\nonumber\\
&\quad+\int_{\alpha}^{\beta}f(t)\left(\zeta(1+s)+t^{-s}\zeta(1-s)\right)\, dt.
\end{align*}}%
Employing the residue theorem again, this time for each of the integrals
inside the two sums,  and simplifying, we find that
\begin{align}\label{genvor3}
&\sum_{\alpha<j<\beta}\sigma_{-s}(j)f(j)
=2\sum_{n=1}^{\infty}\sigma_{-s}(n)n^{\frac{1}{2}s}\notag\\&\quad\times\int_{\alpha}^{\beta}
t^{-\frac{1}{2}s}f(t)\bigg(K_{s}\left(4\pi i\sqrt{nt}\right)+K_{s}\left(-4\pi i\sqrt{nt}\right)+2\cos\left(\frac{\pi s}{2}\right)K_{s}\left(4\pi\sqrt{nt}\right)\bigg)\, dt\nonumber\\
&\quad+\int_{\alpha}^{\beta}f(t)\left(\zeta(1+s)+t^{-s}\zeta(1-s)\right)\, dt.
\end{align}
Note that for $-\pi<\arg z\leq \tfrac{1}{2}\pi$, the modified Bessel function
$K_{\nu}(z)$ is related to the Hankel function $H_{\nu}^{(1)}(z)$ by
\cite[p.~911, formula \textbf{8.407.1}]{grn}
\begin{equation}\label{besshank}
K_{\nu}(z)=\frac{\pi i}{2}e^{\nu\pi i/2}H_{\nu}^{(1)}(iz),
\end{equation}
where the Hankel function is defined by \cite[p.~911, formula \textbf{8.405.1}]{grn}
\begin{align}\label{hankel}
H_{\nu}^{(1)}(z):=J_{\nu}(z)+iY_{\nu}(z).
\end{align}
Employing the relations \eqref{besshank} and \eqref{hankel}, we have, for $x>0$,
\begin{align}\label{besshank1}
K_{s}(ix)+K_{s}(-ix)&=\frac{\pi i}{2}e^{i\pi s/2}\left(H_{s}^{(1)}(-x)+H_{s}^{(1)}(x)\right)\nonumber\\
&=\frac{\pi i}{2}e^{i\pi s/2}\left\{\left(J_{s}(x)+J_{s}(-x)\right)+i\left(Y_{s}(x)+Y_{s}(-x)\right)\right\}.
\end{align}
For $m\in\mathbb{Z}$ \cite[p.~927, formulas \textbf{8.476.1, 8.476.2}]{grn}
\begin{align}
J_{\nu}(e^{m\pi i}z)&=e^{m\nu\pi i}J_{\nu}(z),\label{be}\\
Y_{\nu}(e^{m\pi i}z)&=e^{-m\nu\pi i}Y_{\nu}(z)+2i\sin\left(m\nu\pi\right)\cot\left(\nu\pi\right)J_{\nu}(z).\label{bessreln}
\end{align}
Using the relations \eqref{be} and \eqref{bessreln} with $m=1$, we can simplify \eqref{besshank1} and put it in the form
{\allowdisplaybreaks\begin{align}\label{besshank2}
&K_{s}(ix)+K_{s}(-ix)\notag\\
&=\frac{\pi i}{2}e^{i\pi s/2}\left\{\left(J_{s}(x)+e^{i\pi s}J_{s}(x)\right)+i\left(Y_{s}(x)+e^{-i\pi s}Y_{s}(x)+2i\cos\left(\pi s\right)J_{s}(x)\right)\right\}\nonumber\\
&=\frac{\pi i}{2}e^{i\pi s/2}\left\{\left(1-e^{-i\pi s}\right)J_{s}(x)+i\left(1+e^{-i\pi s}\right)Y_{s}(x)\right\}\nonumber\\
&=-\pi\left(J_{s}(x)\sin\left(\frac{\pi s}{2}\right)+Y_{s}(x)\cos\left(\frac{\pi s}{2}\right)\right).
\end{align}}%
Now replace $x$ by $4\pi\sqrt{nt}$ in \eqref{besshank2} and substitute in
\eqref{genvor3} to obtain \eqref{varlauform1} after simplification. This
completes the proof.

\section{Proof of Theorem \ref{voronoisumgen}}
For any integer $\lambda$, define
\begin{align}\label{capg}
G_{\lambda+s}(z):=-J_{\lambda+s}(z)\sin\left(\frac{\pi s}{2}\right)-\left(Y_{\lambda+s}(z)-(-1)^{\l}\frac{2}{\pi}K_{\lambda+s}(z)\right)\cos\left(\frac{\pi s}{2}\right)
\end{align}
and
\begin{align}\label{capf}
	F_{\lambda+s}(z):=-J_{\lambda+s}(z)\sin\left(\frac{\pi s}{2}\right)-\left(Y_{\lambda+s}(z)+(-1)^{\l}\frac{2}{\pi}K_{\lambda+s}(z)\right)\cos\left(\frac{\pi s}{2}\right).
\end{align}
\textbf{Remark.} Throughout this section, we keep $s$ fixed such that $-\tf12<\s<\tf12$. So
while interpreting $F_{s+\l}(z)$ or $G_{s+\l}(z)$, care should be taken to not conceive them as functions obtained
after replacing $s$ by $s+\l$ in $F_s(z)$ or $G_s(z)$, but instead as those where $s$ remains fixed and only $\l$ varies.

From \cite[pp.~66, 79]{watsonbessel} we have
\begin{align}
\frac{d}{dz}\left\{z^{\nu}J_{\nu}(z)\right\}&=z^{\nu}J_{\nu-1}(z),\label{jd}\\
\frac{d}{dz}\left\{z^{\nu}K_{\nu}(z)\right\}&=-z^{\nu}K_{\nu-1}(z),\label{kd}\\
\frac{d}{dz}\left\{z^{\nu}Y_{\nu}(z)\right\}&=z^{\nu}Y_{\nu-1}(z).\label{yd}
\end{align}
Using \eqref{jd}, \eqref{kd}, and \eqref{yd} we deduce that
\begin{align}\label{gtransform}
\frac{d}{dt}\bigg\{\left(\frac{t}{u}\right)^{(s+\lambda)/2}G_{s+\lambda}(4\pi \sqrt{tu})\bigg\}=2\pi\left(\frac{t}{u}\right)^{(s+\lambda-1)/2}G_{s+\lambda-1}(4\pi \sqrt{tu}),
\end{align}
for $u>0$.
Similarly,
\begin{align}\label{ftransform}
\frac{d}{dt}\bigg\{\left(\frac{t}{u}\right)^{(s+\lambda)/2}F_{s+\lambda}(4\pi \sqrt{tu})\bigg\}=2\pi\left(\frac{t}{u}\right)^{(s+\lambda-1)/2}F_{s+\lambda-1}(4\pi \sqrt{tu}),
\end{align}
for $u>0$.

From \eqref{gddp} and \eqref{lauform}, recall the definition
\begin{align*}
\Delta_{-s}(x)&=\frac{x}{2\sin\left(\frac{1}{2}\pi s\right)}\sum_{n=1}^{\infty}\sigma_{s}(n)\left(\sqrt{nx}\right)^{-1-s}\nonumber\\
&\quad\times\left(J_{s-1}(4\pi\sqrt{nx})+J_{1-s}(4\pi\sqrt{nx})-\frac{2}{\pi}\sin(\pi s)K_{1-s}(4\pi\sqrt{nx})\right),
\end{align*}
for $-\tf12<\s <\tf12$ and $x>0$. If we replace $s$ by $-s$ in the equation above and use \eqref{yj}, we find by a straightforward computation that
\begin{align}\label{delsxat}
\Delta_s(x)=\sum_{n=1}^{\infty}\left(\frac{x}{n}\right)^{(s+1)/2}\s_s(n)G_{s+1}(4\pi \sqrt{nx}),
\end{align}
for $-\tf12< \s <\tf12$ and $x>0$. Fix $x>0$. By the asymptotic expansions of Bessel functions \eqref{asymbess}, \eqref{asymbess1}, and \eqref{asymbess2}, there exists a sufficiently large integer $N_0$ such that
\begin{align}\label{Gasym}
G_{\nu}(4\pi\sqrt{nx})\ll_{\nu}\frac{1}{(nx)^{1/4}}\quad\text{and}\quad F_{\nu}(4\pi\sqrt{nx})\ll_{\nu}\frac{1}{(nx)^{1/4}},
\end{align}
for all $n>N_0$. Hence, for $-\tf12<\s <\tf12$ and $x>0$,
\begin{align*}
\sum_{n>N_0}\left(\frac{x}{n}\right)^{\lambda+\frac{1}{2}s}\s_s(n)G_{s+2\lambda}(4\pi \sqrt{nx})\ll x^{\lambda+(2\s-1)/4}\sum_{n>N_0}\frac{\s_{\s}(n)}{n^{\lambda+(1+2\s)/4}}\ll x^{\lambda+(2\s-1)/4},
\end{align*}
provided that $2\lambda>|\s|+\tfrac{3}{2}$.
Therefore, for $\lambda\geq 1$, $-\tf12<\s <\tf12$, and $x>0$, the
series
\begin{align*}
\sum_{n=1}^{\infty}\left(\frac{x}{n}\right)^{\lambda+\frac{1}{2}s}\s_s(n)G_{s+2\lambda}(4\pi \sqrt{nx})
\end{align*}
is absolutely convergent. Similarly, for $\lambda\geq 1$, $-\tf12<\s <\tf12$, and $x>0$, the
series
\begin{align*}
\sum_{n=1}^{\infty}\left(\frac{x}{n}\right)^{\lambda+\frac{s}{2}}\s_s(n)F_{s+2\lambda}(4\pi \sqrt{nx})
\end{align*}
is absolutely convergent.
Denote
\begin{align}\label{defD}
D_s(x):=\sideset{}{'}\sum_{n\leq x}\s_s(n)
\end{align}		
and
\begin{align}\label{defph}
\Phi_s(x):=x\z(1-s)+\frac{x^{1+s}}{1+s}\z(1+s)-\frac{1}{2}\z(-s).
\end{align}
Therefore, from \eqref{lauform}, we write
\begin{align}\label{Dphidel}
D_s(x)=\Phi_s(x)+\Delta_s(x)
\end{align}
for $-\tf12<\s <\tf12$.

The following lemmas are key ingredients in the proof of Theorem \ref{voronoisumgen}. They are special cases of
two results in \cite{wiltonextended}. We note, however, that the definitions of $G$ and $F$ in \cite{wiltonextended} are different from those in \eqref{capg} and \eqref{capf} that we use.

\begin{lemma}\label{lemtail}
If $x>0$, $N>0$, and $-\tf12<\s <\tf12$, then
{\allowdisplaybreaks\begin{align}\label{talisdelx}
\Delta_s(x)&=\sideset{}{'}\sum_{n=1}^{N}\left(\frac{x}{n}\right)^{(s+1)/2}\s_s(n)G_{s+1}(4\pi \sqrt{nx})-\left(\frac{x}{N}\right)^{(s+1)/2}G_{s+1}(4\pi \sqrt{Nx})\Delta_s(N)\nonumber\\
&\quad+\frac{N^s\zeta(1+s)+\zeta(1-s)}{2\pi}\left(\frac{x}{N}\right)^{s/2}F_{s}(4\pi \sqrt{Nx})\nonumber\\
&\quad+\frac{s\zeta(1+s)}{2\pi}\int_{N}^{\infty}\left(\frac{x}{t}\right)^{s/2}F_{s}(4\pi \sqrt{xt})t^{s-1}\,dt\\
&\quad+2\pi\sum_{n=1}^{\infty}\s_s(n)\int_{N}^{\infty}\left(\frac{x}{t}\right)^{(s+2)/2}F_{s+2}(4\pi \sqrt{xt})\left(\frac{t}{n}\right)^{(s+1)/2}G_{s+1}(4\pi \sqrt{nt})\,dt.\notag
\end{align}}
\end{lemma}	
\begin{proof}
Take $\lambda=0$, $\kappa=1$, and $\theta=1$ in Theorem 2 of \cite[p.~404]{wiltonextended}, and make use of the notations (1.21) and (3.13) given in it.
\end{proof}	

We wish to invert the order of summation and integration in the last expression
on the right-hand side of \eqref{talisdelx}. In order to justify that, we need the following lemma.

\begin{lemma}\label{sumprodint}
If $N>A$, $Nx>A$, $-\tf12<\s <\tf12$, and
\begin{align*}
I_s(x,n;N):=2\pi \int_{N}^{\infty}\left(\frac{x}{t}\right)^{(s+2)/2}F_{s+2}(4\pi \sqrt{xt})\left(\frac{t}{n}\right)^{(s+1)/2}G_{s+1}(4\pi \sqrt{nt})\,dt,
\end{align*}
 then
\begin{align*}
\sum_{n=1}^{\infty}\s_s(n)I_s(x,n;N)=C_s(x,N)+O\left(\frac{x^{1+\epsilon}}{\sqrt{N}}\right),
\end{align*}
for every $\epsilon>0$, where
{\allowdisplaybreaks\begin{align*}
C_s(x,N)&=0, \quad\text{if}\quad x<\tf12\quad\text{or}\quad x\in\mathbb{N},\\
C_s(x,N)&=\frac{1}{\pi}\left(\frac{x}{y}\right)^{(2s+5)/4}\s_s({y})
\int_{4\pi\sqrt{N}|\sqrt{y}-\sqrt{x}|}^{\infty}\frac{\sin(t\operatorname{sgn}(y-x))}{t}\,dt,\notag\\
&\qquad\text{if}\quad x\neq y=\lfloor x+\tf12\rfloor \geq 1.
\end{align*}}
\end{lemma}
\begin{proof}
This is the special case $\lambda=0, \kappa=1$ of Lemma 6 of \cite[p.~412]{wiltonextended}.
\end{proof}
\begin{proof}[Proof of Theorem \ref{voronoisumgen}.]

By Lemma \ref{sumprodint}, we see that the last expression on the
right-hand side of \eqref{talisdelx} tends to $0$ as $N\to \infty$. Hence, by
interchanging the summation and integration in this expression, we deduce that
\begin{align}\label{taildelxs}
\Delta_s(x)&=\sideset{}{'}\sum_{n=1}^{N}\left(\frac{x}{n}\right)^{(s+1)/2}\s_s(n)G_{s+1}(4\pi
\sqrt{nx})\notag\\
&\quad+\frac{N^s\zeta(1+s)+\zeta(1-s)}{2\pi}\left(\frac{x}{N}\right)^{s/2}F_{s}(4\pi \sqrt{Nx})\nonumber\\
&\quad-\left(\frac{x}{N}\right)^{(s+1)/2}G_{s+1}(4\pi \sqrt{Nx})\Delta_s(N)
\notag\\&\quad+\frac{s\zeta(1+s)}{2\pi}\int_{N}^{\infty}\left(\frac{x}{t}\right)^{s/2}F_{s}(4\pi \sqrt{xt})t^{s-1}\,dt\nonumber\\
&\quad+2\pi\int_{N}^{\infty}\left(\frac{x}{t}\right)^{(s+2)/2}F_{s+2}(4\pi \sqrt{xt})\Delta_s(t)\,dt.
\end{align}
Let $a\geq 0$ and $b\geq 0$. From \eqref{defD},
\begin{align}\label{lsint}
\sum_{a\leq n\leq b}f(n)\s_s(n)=\int_{a}^{b}f(t)\,dD_s(t),
\end{align}
where we write the sum as a Lebesgue-Stieltjes integral.

For $a=0$ and $b=\a<\tf12$, the left-hand side of \eqref{lsint} equals
$0$. Therefore, from \eqref{gtransform}, \eqref{ftransform}, \eqref{defph}, \eqref{Dphidel}, \eqref{taildelxs}, \eqref{lsint}, and \eqref{delsxat},
\begin{align}\label{a0bh}
-\int_{0}^{\a}&f(t)(\z(1-s)+t^s\z(1+s))\, dt=\int_{0}^{\a}f(t)\, d\Delta_s(t)
\nonumber\\
&=2\pi\sideset{}{'}\sum_{n=1}^{N}\frac{\s_s(n)}{n^{s/2}}\int_{0}^{\a}t^{s/2}G_{s}(4\pi \sqrt{nt})f(t)\, dt\nonumber\\
&\quad+\frac{N^s\zeta(1+s)+\zeta(1-s)}{N^{(s-1)/2}}\int_{0}^{\a}t^{(s-1)/2}F_{s-1}(4\pi \sqrt{Nt})f(t)\, dt\nonumber\\
&\quad-\frac{2\pi}{N^{s/2}}\Delta_s(N)\int_{0}^{\a}t^{s/2}G_{s}(4\pi \sqrt{Nt})f(t)\, dt\nonumber\\
&\quad+\frac{s\zeta(1+s)}{2\pi}\int_{0}^{\a}f(t)\frac{d}{dt}\left(\int_{N}^{\infty}\left(\frac{t}{u}\right)^{s/2}F_{s}(4\pi \sqrt{tu})u^{s-1}\,du\right)\, dt\nonumber\\
&\quad+2\pi\int_{0}^{\a}f(t)\frac{d}{dt}\left(\int_{N}^{\infty}\left(\frac{t}{u}\right)^{(s+2)/2}F_{s+2}(4\pi \sqrt{tu})\Delta_s(u)\,du\right)\, dt.
\end{align}
Using \eqref{ftransform} twice, we see that
\begin{align}\label{tNG}
\frac{d}{dt}\left(\int_{N}^{\infty}\left(\frac{t}{u}\right)^{s/2}F_{s}(4\pi \sqrt{tu})u^{s-1}\,du\right)&=2\pi\int_{N}^{\infty}(tu)^{(s-1)/2}F_{s-1}(4\pi \sqrt{tu})\,du\nonumber\\
&=\left.t^{s/2-1}u^{s/2}F_{s}(4\pi \sqrt{tu})\right\rvert_{N}^{\infty}\nonumber\\
&=-t^{s/2-1}N^{s/2}F_{s}(4\pi \sqrt{tN}),
\end{align}
where in the last step we use \eqref{asymbess}--\eqref{asymbess2}, and the fact that $\sigma<\tf12$. The interchange of differentiation and integration  above is
justified from \eqref{Gasym}.
Denote
\begin{align}\label{ItN}
I_s(t,N):=2\pi\int_{N}^{\infty}\left(\frac{t}{u}\right)^{(s+2)/2}F_{s+2}(4\pi \sqrt{tu})\Delta_s(u)\,du.
\end{align}
 Performing an integration by parts on the last expression on the right-hand side of \eqref{a0bh}
 and using \eqref{tNG} and \eqref{ItN}, we find that
\begin{align}\label{fininequa}
-\int_{0}^{\a}&f(t)(\z(1-s)+t^s\z(1+s))\, dt-2\pi\sideset{}{'}\sum_{n=1}^{N}\frac{\s_s(n)}{n^{s/2}}\int_{0}^{\a}t^{s/2}G_{s}(4\pi \sqrt{nt})f(t)\, dt\nonumber\\
&=\frac{N^s\zeta(1+s)+\zeta(1-s)}{N^{(s-1)/2}}\int_{0}^{\a}t^{(s-1)/2}F_{s-1}(4\pi \sqrt{Nt})f(t)\, dt\nonumber\\
&\quad-\frac{2\pi}{N^{s/2}}\Delta_s(N)\int_{0}^{\a}t^{s/2}G_{s}(4\pi \sqrt{Nt})f(t)\, dt\nonumber\\
&\quad-\frac{s\zeta(1+s)N^{s/2}}{2\pi}\int_{0}^{\a}f(t)t^{s/2-1}F_{s}(4\pi \sqrt{Nt})\, dt\nonumber\\
&\quad+f(\a)I_s(\a,N)-\int_{0}^{\a}I_s(t,N)f'(t)\,dt,
\end{align}
where in the last step we made use of the fact that for $-\tf12<\s<\tf12$,
\begin{equation*}
\lim_{t\to 0}t^{(s+2)/2}F_{s+2}(4\pi \sqrt{tu})=0.
\end{equation*}
Here again the limit can be moved inside the integral because of \eqref{Gasym}.

Since $\a<\tf12$, by Lemma \ref{sumprodint},  $I_s(t,N)\ll N^{-1/2}$, for all $0<t\leq\a$. Also by hypothesis, $f$ is differentiable, so
\begin{align*}
V_0^{\a}f(t)=\int_0^{\a}|f'(t)|\, dt,
\end{align*}
where $V_0^{\a}f(t)$ is the total variation of $f$ on the interval $(0,\a)$.
Therefore the last two terms on the right-hand side of \eqref{fininequa} are of the form
\begin{align}\label{totvar}
O(N^{-1/2}(|f(\a)|+V_0^{\a}f(t))).
\end{align}
Recall the bound $\Delta_s(N)\ll  N^{\frac12(1+\s)}$ \cite[Lemma 7]{wiltonextended}. From \eqref{gtransform} and \eqref{Gasym},
\begin{align}\label{Delbd}
&\frac{2\pi}{N^{s/2}}\Delta_s(N)\int_{0}^{\a}t^{s/2}G_{s}(4\pi \sqrt{Nt})\, dt\nonumber\\
&=\Delta_s(N)\left(\frac{\a}{N}\right)^{(s+1)/2}G_{s+1}(4\pi \sqrt{N\a})\nonumber\\
&\ll \a^{\frac{\s}{2}}\left(\frac{\a}{N}\right)^{\frac{1}{4}}.
\end{align}
Here we  also made use of the fact that
\begin{equation*}
\lim_{t\to 0}t^{(s+1)/2}G_{s+1}(4\pi \sqrt{Nt})=0.
\end{equation*}
Again, from \eqref{gtransform} and \eqref{Gasym},
\begin{align}\label{zebd}
&\frac{N^s\zeta(1+s)+\zeta(1-s)}{N^{(s-1)/2}}\int_{0}^{\a}t^{(s-1)/2}F_{s-1}(4\pi \sqrt{Nt})\, dt\nonumber\\&=\frac{N^s\zeta(1+s)+\zeta(1-s)}{2\pi N^{s/2}}\a^{s/2}F_{s}(4\pi \sqrt{N\a})\nonumber\\
&\ll \twopartdef{(2\g+\log N)(\a N)^{-1/4},}{s=0,\\\vspace{-5pt}}{(\a N)^{(2\s-1)/4}+\a^{(2\s-1)/4}N^{(-2\s-1)/4},}{s\neq 0,}
\end{align}
since $\lim_{t\to 0}t^{s/2}F_{s}(4\pi \sqrt{Nt})=0$. Finally,
\begin{align*}
&\frac{s\zeta(1+s)N^{s/2}}{2\pi}\int_{0}^{\a}t^{s/2-1}F_{s}(4\pi \sqrt{Nt})\,
dt\notag\\
&=\frac{s\zeta(1+s)N^{s/2}}{2\pi}\left(\int_{0}^{\infty}-\int_{\a}^{\infty}\right)t^{s/2-1}F_{s}(4\pi \sqrt{Nt})\, dt\nonumber\\
&=:I_1-I_2.
\end{align*}
Using the functional equation of $\z(s)$, namely \eqref{fe}, and the formula \cite[p.~409, equation 4.65]{wiltonextended}, we find that
\begin{align*}
I_1&=\frac{s\zeta(1+s)N^{s/2}}{2\pi}\int_{0}^{\infty}t^{s/2-1}F_{s}(4\pi
\sqrt{Nt})\, dt\notag\\
&=-(2\pi)^{-s-1}\sin(\pi s/2) \G(s+1)\zeta(1+s)=\frac{\z(-s)}{2}.
\end{align*}
Using \eqref{Gasym}, we deduce that
\begin{align}\label{I2}
I_2=\frac{s\zeta(1+s)N^{s/2}}{2\pi}\int_{\a}^{\infty}t^{s/2-1}F_{s}(4\pi \sqrt{Nt})\, dt\ll (\a N)^{(2\s-1)/4},
\end{align}
since $-\tf12<\s<\tf12$. Using \eqref{totvar}--\eqref{I2} in \eqref{fininequa}, we find that
\begin{align}\label{fininequa01}
f&(0+)\frac{\z(-s)}{2}-\int_{0}^{\a}f(t)(\z(1-s)+t^s\z(1+s))\, dt\\
&\quad-2\pi\sideset{}{'}\sum_{n=1}^{N}\frac{\s_s(n)}{n^{s/2}}\int_{0}^{\a}t^{s/2}G_{s}(4\pi \sqrt{nt})f(t)\, dt\nonumber\\
&=\frac{N^s\zeta(1+s)+\zeta(1-s)}{N^{(s-1)/2}}\int_{0}^{\a}t^{(s-1)/2}F_{s-1}(4\pi \sqrt{Nt})(f(t)-f(0+))\, dt\nonumber\\
&\quad-\frac{2\pi}{N^{s/2}}\Delta_s(N)\int_{0}^{\a}t^{s/2}G_{s}(4\pi \sqrt{Nt})(f(t)-f(0+))\, dt\nonumber\\
&\quad-\frac{s\zeta(1+s)}{2\pi}\int_{0}^{\a}(f(t)-f(0+))t^{s/2-1}N^{s/2}F_{s}(4\pi \sqrt{Nt})\, dt\nonumber\\
&\quad+O((\a N)^{(2\s-1)/4}+\a^{(2\s-1)/4}N^{(-2\s-1)/4})+O((2\g+\log N)(\a N)^{-1/4}).\nonumber
\end{align}
By the second mean value theorem for integrals in the form given in \cite[p.~31]{wilton},
\begin{align}\label{secondmean}
\left\lvert\int_{a}^{b}f(t)\phi(t)\, dt-f(b)\int_{a}^{b}\phi(t)\, dt\right\rvert\leq V_{a}^{b}f(t)\max_{a\leq c< d\leq b}\left\lvert\int_{c}^{d}\phi(t)\, dt\right\rvert,
\end{align}
where $\phi$ is integrable on $[a,b]$.

Recall that $N^{\theta}\a>1$ for some $0<\theta<\min\left(1,\frac{1+2\s}{1-2\s}\right)$. Dividing the interval $(0,\a)$ into two sub-intervals $(0,N^{-\theta})$ and $(N^{-\theta},\a)$, applying \eqref{secondmean}, and using an argument like that in \eqref{zebd}, we see that
\begin{align}\label{assymp01}
&\frac{N^s\zeta(1+s)+\zeta(1-s)}{N^{(s-1)/2}}\int_{N^{-\theta}}^{\a}t^{(s-1)/2}F_{s-1}(4\pi \sqrt{Nt})(f(t)-f(0+))\, dt\nonumber\\
&\ll\twopartdef{ (2\g+\log N)N^{(\theta-1)/4}V_{N^{-\theta}}^{\a}f(t),}{s=0,\\\vspace{-5pt}}{(N^{(1-\theta)(2\s-1)/4}
+N^{(\theta(1-2\s)-(2\s+1))/4})V_{N^{-\theta}}^{\a}f(t),}{s\neq 0,}
\end{align}
and
\begin{align*}
&\frac{N^s\zeta(1+s)+\zeta(1-s)}{N^{(s-1)/2}}\int_{0}^{N^{-\theta}}t^{(s-1)/2}F_{s-1}(4\pi \sqrt{Nt})(f(t)-f(0+))\, dt\nonumber\\
&\ll\twopartdef{ (2\g+\log N)V_{0}^{N^{-\theta}}f(t),}{s=0,\\\vspace{-5pt}}{V_{0}^{N^{-\theta}}f(t),}{s\neq 0.}
\end{align*}
By \eqref{secondmean} and arguments similar to those in \eqref{Delbd} and \eqref{I2},
\begin{align*}
\frac{2\pi}{N^{s/2}}\Delta_s(N)\int_{N^{-\theta}}^{\a}t^{s/2}G_{s}(4\pi \sqrt{Nt})(f(t)-f(0+))\, dt\ll  \a^{\frac{\s}{2}}\left(\frac{\a}{N}\right)^{\frac{1}{4}}V_{N^{-\theta}}^{\a}f(t),
\end{align*}
\begin{align*}
\frac{2\pi}{N^{s/2}}\Delta_s(N)\int_{0}^{N^{-\theta}}t^{s/2}G_{s}(4\pi \sqrt{Nt})(f(t)-f(0+))\, dt\ll N^{\frac{-2\theta\s-1-\theta}{4}} V_{0}^{N^{-\theta}}f(t),
\end{align*}
\begin{align*}
\frac{s\zeta(1+s)}{2\pi}\int_{N^{-\theta}}^{\a}(f(t)-f(0+))t^{s/2-1}N^{s/2}F_{s}(4\pi \sqrt{Nt})\, dt\ll N^{\frac{(1-\theta)(2\s-1)}{4}}V_{N^{-\theta}}^{\a}f(t),
\end{align*}%
and
\begin{align}\label{assymp06}
\frac{s\zeta(1+s)}{2\pi}\int_{0}^{N^{-\theta}}(f(t)-f(0+))t^{s/2-1}N^{s/2}F_{s}(4\pi \sqrt{Nt})\, dt\ll V_{0}^{N^{-\theta}}f(t).
\end{align}
Combining \eqref{assymp01}--\eqref{assymp06} together with \eqref{fininequa01}, we obtain
\begin{align*}
&f(0+)\frac{\z(-s)}{2}-\int_{0}^{\a}f(t)(\z(1-s)+t^s\z(1+s))\, dt\\&\hspace{.7in}-2\pi\sideset{}{'}\sum_{n=1}^{N}\frac{\s_s(n)}{n^{s/2}}\int_{0}^{\a}t^{s/2}G_{s}(4\pi \sqrt{nt})f(t)\, dt\nonumber\\
&\ll\twopartdef{ (2\g+\log N)(V_{0}^{N^{-\theta}}f(t)+N^{(\theta-1)/4}V_{N^{-\theta}}^{\a}f(t)),}{s=0,\\
\vspace{-5pt}}{V_{0}^{N^{-\theta}}f(t)+(N^{(1-\theta)(2\s-1)/4}+N^{(\theta(1-2\s)-(2\s+1))/4)})V_{N^{-\theta}}^{\a}f(t),}{s\neq 0,}\nonumber\\
&\ll\twopartdef{ (2\g+\log N)(V_{0}^{N^{-\theta}}f(t)+N^{(\theta-1)/4}(|f(\a)|+V_{0}^{\a}f(t))),} {s=0,\\ \vspace{-5pt}}
{V_{0}^{N^{-\theta}}f(t)
+(N^{(1-\theta)(2\s-1)/4}+N^{(\theta(1-2\s)-(2\s+1))/4})\\\hspace{6cm}\times(|f(\a)|+V_{0}^{\a}f(t)),} {s\neq 0.}\nonumber
\end{align*}
Furthermore, if $\log x\,V_{0}^{x}f(t)\to 0$ as $x\to 0+$ when $s=0$, and if $V_{0}^{x}f(t)\to 0$ as $x\to 0+$ when $s\neq 0$, then the assumption $0<\theta<\min\left(1,\frac{1+2\s}{1-2\s}\right)$ implies that
\begin{align*}
f(0+)\frac{\z(-s)}{2}-&\int_{0}^{\a}f(t)(\z(1-s)+t^s\z(1+s))\, dt\notag\\&\quad-2\pi\sum_{n=1}^{\infty}\frac{\s_s(n)}{n^{s/2}}\int_{0}^{\a}t^{s/2}G_{s}(4\pi \sqrt{nt})f(t)\, dt=0.
\end{align*}
This completes the proof of Theorem \ref{voronoisumgen}.
\end{proof}

%
%
%
\section{An Interpretation of Ramanujan's Divergent Series}
\label{sect8}
As mentioned in the introduction, the series on the left-hand side of
\eqref{qc} is divergent for all real values of $s$, since $\sigma_s(n)\geq
n^s$. However, as we show below, there is a valid interpretation of this
series using the theory of analytic continuation.

Throughout this section, we assume $x>0$, $\sigma>0$, and Re $w>1$. Define a
function $F(s, x, w)$ by
\begin{equation}\label{fsxw}
F(s, x,
w):=\sum_{n=1}^{\infty}\frac{\sigma_s(n)}{n^{w-\frac{1}{2}}}\left((x-in)^{-s-\frac{1}{2}}-(x+in)^{-s-\frac{1}{2}}\right).
\end{equation}
Ramanujan's divergent series corresponds to letting $w=\frac{1}{2}$ in
\eqref{fsxw}.
Note that
\begin{equation*}
(x-in)^{-s-\frac{1}{2}}-(x+in)^{-s-\frac{1}{2}}
=\frac{2i\sin\left(\left(s+\tfrac{1}{2}\right)\tan^{-1}\left(n/x\right)\right)}{(x^2+n^2)^{\frac{s}{2}+\frac{1}{4}}}.
\end{equation*}
Since for $\sigma>-\frac{3}{2}$ and $n>0$ \cite[p.~524, formula \textbf{3.944}, no.~5]{grn}
\begin{equation}\label{3.944}
\int_{0}^{\infty}e^{-xt}t^{s-\frac{1}{2}}\sin (nt)\, dt=\Gamma\left(s+\frac{1}{2}\right)\frac{\sin\left(\left(s+\tfrac{1}{2}\right)
\tan^{-1}\left(n/x\right)\right)}{(x^2+n^2)^{\frac{s}{2}+\frac{1}{4}}},
\end{equation}
we deduce from \eqref{fsxw}--\eqref{3.944} that
\begin{align*}
F(s, x, w)=\frac{2i}{\Gamma\left(s+\frac{1}{2}\right)}
\sum_{n=1}^{\infty}\frac{\sigma_s(n)}{n^{w-\frac{1}{2}}}\int_{0}^{\infty}e^{-xt}t^{s-\frac{1}{2}}\sin nt\, dt.\nonumber\\
\end{align*}
From \cite[p.~42, formula (5.1)]{ob}, for $-1<c=$ Re $z<1$,
\begin{equation*}
\sin(nt)=\frac{1}{2\pi i}\int_{c-i\infty}^{c+i\infty}\G(z)\sin\left(\frac{\pi z}{2}\right)(nt)^{-z}\, dz.
\end{equation*}
Hence,
\begin{align}\label{fsxw2}
&F(s, x, w)=\frac{1}{\pi\Gamma\left(s+\frac{1}{2}\right)}\int_{0}^{\infty}e^{-xt}t^{s-\frac{1}{2}}\sum_{n=1}^{\infty}
\frac{\sigma_s(n)}{n^{w-\frac{1}{2}}}\int_{c-i\infty}^{c+i\infty}\G(z)\sin\left(\frac{\pi z}{2}\right)(nt)^{-z}\, dz\, dt\notag\\
&=\frac{1}{\pi\Gamma\left(s+\frac{1}{2}\right)}\int_{0}^{\infty}e^{-xt}t^{s-\frac{1}{2}}\int_{c-i\infty}^{c+i\infty}
t^{-z}\G(z)\sin\left(\frac{\pi z}{2}\right)\left(\sum_{n=1}^{\infty}\frac{\sigma_s(n)}{n^{w+z-\frac{1}{2}}}\right)\, dz\, dt,
\end{align}
where the interchange of the order of summation and integration in both
instances is justified by absolute convergence. Now if Re $z>\frac{3}{2}-$ Re
$w$ and Re $z>\frac{3}{2}-$ Re $w+\sigma$, from \eqref{sz}, we see that
\begin{equation*}
\sum_{n=1}^{\infty}\frac{\sigma_s(n)}{n^{w+z-\frac{1}{2}}}
=\zeta\left(w+z-\frac{1}{2}\right)\zeta\left(w+z-s-\frac{1}{2}\right).
\end{equation*}
Substituting this in \eqref{fsxw2}, we find that
\begin{align}\label{fsxw3}
F(s, x, w)&=\frac{1}{\pi\Gamma\left(s+\frac{1}{2}\right)}\int_{0}^{\infty}e^{-xt}t^{s-\frac{1}{2}}
\int_{c-i\infty}^{c+i\infty}t^{-z}\G(z)\sin\left(\frac{\pi z}{2}\right)\nonumber\\
&\quad\quad\quad\quad\quad\quad\quad\quad\quad\times\zeta
\left(w+z-\frac{1}{2}\right)\zeta\left(w+z-s-\frac{1}{2}\right)\, dz\, dt\nonumber\\
&=\frac{1}{\pi\Gamma\left(s+\frac{1}{2}\right)}\int_{c-i\infty}^{c+i\infty}\G(z)\sin\left(\frac{\pi z}{2}\right)\zeta\left(w+z-\frac{1}{2}\right)\zeta\left(w+z-s-\frac{1}{2}\right)\nonumber\\
&\quad\quad\quad\quad\quad\quad\quad\times\int_{0}^{\infty}e^{-xt}t^{s-z-\frac{1}{2}}\, dt\, dz,\nonumber\\
\end{align}
with the interchange of the order of integration again being easily justifiable.
For Re $z<\sigma+\frac{1}{2}$, we have
\begin{equation*}
\int_{0}^{\infty}e^{-xt}t^{s-z-\frac{1}{2}}\, dt=\frac{\G\left(s-z+\frac{1}{2}\right)}{x^{s-z+\frac{1}{2}}}.
\end{equation*}
Substituting this in \eqref{fsxw3}, we obtain the integral representation
\begin{align}\label{fsxw4}
F(s, x, w)&=\frac{x^{-s-\frac{1}{2}}}{\pi\Gamma\left(s+\frac{1}{2}\right)}\int_{c-i\infty}^{c+i\infty}\G(z)\sin\left(\frac{\pi z}{2}\right)\zeta\left(w+z-\frac{1}{2}\right)\nonumber\\
&\quad\quad\quad\quad\quad\quad\quad\quad\times\zeta\left(w+z-s-\frac{1}{2}\right)\G\left(s-z+\frac{1}{2}\right)x^{z}\, dz.
\end{align}
Note that if we shift the line of integration Re $z=c$ to Re $z=d$ such that
$d=\frac{3}{2}+ \sigma-\eta$ with $\eta>0$, we encounter a simple pole of the
integrand due to $\G\left(s-z+\frac{1}{2}\right)$. Employing the residue
theorem and noting that, from \eqref{strivert} and \eqref{sinest}, the
integrals over the horizontal segments tend to zero as the height of the
rectangular contour tends to $\infty$, we have
\begin{align}\label{fsxw5}
F(s, x, w)&=\frac{x^{-s-\frac{1}{2}}}{\pi\Gamma\left(s+\frac{1}{2}\right)}\int_{d-i\infty}^{d+i\infty}\G(z)\sin\left(\frac{\pi z}{2}\right)\zeta\left(w+z-\frac{1}{2}\right)\nonumber\\
&\quad\quad\quad\quad\quad\quad\quad\quad\times\zeta\left(w+z-s-\frac{1}{2}\right)\G\left(s-z+\frac{1}{2}\right)x^{z}\, dz\nonumber\\
&\quad-\frac{2ix^{-s-\frac{1}{2}}}{\Gamma\left(s+\frac{1}{2}\right)}\G\left(s+\frac{1}{2}\right)
\sin\left(\frac{\pi}{2}\left(s+\frac{1}{2}\right)\right)\zeta(w+s)\zeta(w)x^{s+\frac{1}{2}}.
\end{align}
Note that the residue in equation \eqref{fsxw5} is analytic in $w$ except for
simple poles at $1$ and $1-s$. Consider the integrand in \eqref{fsxw5}. The
zeta functions $\zeta\left(w+z-\frac{1}{2}\right)$ and
$\zeta\left(w+z-s-\frac{1}{2}\right)$ have simple poles at $w=\frac{3}{2}-z$
and $w=\frac{3}{2}+s-z$, respectively. However, since Re
$z=\frac{3}{2}+\sigma-\eta$ and $\sigma>0$, the integrand is analytic as a
function of $w$ as long as Re $w>\eta$. By a well-known theorem \cite[p.~30, Theorem
2.3]{temme}, the integral is also analytic in $w$ for Re
$w>\eta$. Thus, the right-hand side of \eqref{fsxw5} is analytic in $w$,
which allows us to analytically continue $F(s, x, w)$ as a function of $w$ to
the region Re $w>\eta$, and hence to Re $w>0$, since $\eta$ is any arbitrary
positive number.

As remarked in the beginning of this section, letting $w=\frac{1}{2}$ in
\eqref{fsxw} yields Ramanujan's divergent series. However, the analytic
continuation of $F(s, x, w)$ to Re $w>0$ allows us to substitute
$w=\frac{1}{2}$ in \eqref{fsxw5} and thereby give a valid interpretation of
Ramanujan's divergent series. The only exception to this is when
$s=\frac{1}{2}$, since then $w=\frac{1}{2}=1-s$ is a pole of the right-hand
side of \eqref{fsxw5}, as discussed above.

 If we further shift the line of integration in \eqref{fsxw4} from Re
 $z=\frac{3}{2}+ \sigma-\eta$ to Re $z=\frac{5}{2}+ \sigma-\eta$, and
 likewise to $+\infty$, we obtain a meromorphic continuation of $F(s, x, w)$,
 as a function of $w$, to the whole complex plane.


\section{Generalization of the Ramanujan--Wigert Identity}
\label{sect9}

The identity \eqref{p332} found by Ramanujan, and then extended by Wigert,
has the following one variable generalization in the same spirit as Theorem
\ref{p3361}.

\begin{theorem}\label{rwgthm}
Let $\psi_{s}(n):=\sum_{j^2|n}j^s$. Let $\a>0$ and  $\b>0$ be such that $\a\b=4\pi^3$. Recall that $_1F_1$ is defined in \eqref{hyper}.  Then, for $\sigma>0$,
\begin{align}\label{rwg}
&\sum_{n=1}^{\infty}\frac{\psi_{s}(n)}{\sqrt{n}}e^{-\sqrt{n\beta}}\sin\left(\frac{\pi}{4}-\sqrt{n\b}\right)\nonumber\\
&=\frac{\sqrt{\b}}{\sqrt{2}}\zeta(-s)+\G(s)\cos\left(\frac{\pi}{4}+\frac{\pi s}{4}\right)\zeta\left(\frac{1+s}{2}\right)(2\b)^{-\frac{1}{2}s}+\frac{1}{\sqrt{2}}\zeta\left(\frac{1}{2}\right)\zeta(1-s)
\nonumber\\
&\quad+2^{-s-\frac{1}{2}}\pi^{-s-2}\sqrt{\b}\sum_{n=1}^{\infty}\frac{\psi_{1-s}(n)}{n}
\left\{\frac{-\pi^3\sqrt{n}}{\sqrt{\b}\Gamma(1-s)\sin\left(\frac{\pi s}{2}\right)}\right.\notag\\&\quad\left.+2^{2s}\pi^{\frac{3}{2}s+2}\left(\frac{n}{\b}\right)^{(s+1)/2}
\Gamma\left(\frac{s}{2}\right){}_1F_{1}\left(\frac{s}{2};\frac{1}{2};-n\a\right)\right\}.
\end{align}
\end{theorem}

Note that $\psi_s(n)$ is an extension of Ramanujan's definition of $\psi(n)$ that was defined earlier in \eqref{d2}.
The proof of Theorem \ref{rwgthm} is similar to that of Theorem \ref{p3361}, and so we will be brief.

\begin{proof}
From the definition of $\psi_{s}(n)$, we note that
\begin{equation*}
\sum_{n=1}^{\infty}\frac{\psi_{s}(n)}{n^{z/2}}=\zeta\left(\frac{z}{2}\right)\zeta(z-s)
\end{equation*}
for Re $z>2$ and Re $z>1+\sigma$. Let
\begin{equation*}
W(s,\beta):=\sum_{n=1}^{\infty}\frac{\psi_{s}(n)}{\sqrt{n}}e^{-\sqrt{n\beta}}\sin\left(\frac{\pi}{4}-\sqrt{n\b}\right).
\end{equation*}
Proceeding as we did in \eqref{ob1}--\eqref{omzsx}, we obtain for $c=$ Re $z>$ max$\{2, 1+\sigma\}$,
\begin{equation*}
W(s, \beta)=\frac{1}{2\pi i}\int_{(c)}\G(z-1)\cos\left(\frac{\pi z}{4}\right)\zeta\left(\frac{z}{2}\right)\zeta(z-s)(2\beta)^{(1-z)/2}\, dz.
\end{equation*}
Shifting the line of integration from Re $z=c$ to Re $z=\lambda,
-1<\lambda<0$, applying the residue theorem, and considering the
contributions of the poles at $z=0, 1$, and $1+s$, we find that
\begin{align}\label{ini1}
W(s,\beta)&=\frac{\sqrt{\beta}}{\sqrt{2}}\zeta(-s)+\G(s)\cos\left(\frac{\pi (1+s)}{4}\right)\zeta\left(\frac{1+s}{2}\right)(2\beta)^{-\frac{1}{2}s}+\frac{1}{\sqrt{2}}\zeta\left(\frac{1}{2}\right)\zeta(1-s)
\nonumber\\
&\quad+\frac{1}{2\pi i}\int_{(\lambda)}G(z, s, \beta)\, dz,
\end{align}
where
\begin{align}\label{ini2}
G(z, s, \beta)&=\G(z-1)\cos\left(\frac{\pi z}{4}\right)\zeta\left(\frac{z}{2}\right)\zeta(z-s)(2\beta)^{(1-z)/2}\nonumber\\
&=2^{z-s-\frac{1}{2}}\pi^{\frac{3}{2}z-s-2}\G(z-1)\G\left(1-\frac{z}{2}\right)\G(1-z+s)\sin\left(\frac{\pi z}{2}\right)\nonumber\\
&\quad\times\sin\left(\frac{\pi z}{2}-\frac{\pi s}{2}\right)\zeta\left(1-\frac{z}{2}\right)\zeta(1-z+s),
\end{align}
and where we used the functional equation of $\zeta(s)$ given in
\eqref{fe}. Since Re $z<0$ and $\sigma>0$,
\begin{equation}\label{ini3}
\zeta\left(1-\frac{z}{2}\right)\zeta(1-z+s)=\sum_{n=1}^{\infty}\frac{\psi_{1-s}(n)}{n^{1-z/2}}.
\end{equation}
Thus, \eqref{ini1}, \eqref{ini2}, and \eqref{ini3} imply that
\begin{align}\label{wsbbeff}
W(s,\beta)&=\frac{\sqrt{\beta}}{\sqrt{2}}\zeta(-s)+\G(s)\cos\left(\frac{\pi
    (1+s)}{4}\right)\zeta\left(\frac{1+s}{2}\right)(2\beta)^{-\frac{1}{2}s}
\notag\\&\quad+\frac{1}{\sqrt{2}}\zeta\left(\frac{1}{2}\right)\zeta(1-s)
+2^{-s-\frac{1}{2}}\pi^{-s-2}\sqrt{\beta}\sum_{n=1}^{\infty}\frac{\psi_{1-s}(n)}{n}H\left(s, \frac{\beta}{n}\right),
\end{align}
where
\begin{align*}
&H\left(s, \frac{\beta}{n}\right)=\frac{1}{2\pi
  i}\int_{(\lambda)}\G(z-1)\G\left(1-\frac{z}{2}\right)
\G(1-z+s)\\&\hspace{1in}\times\sin\left(\frac{\pi  z}{2}\right)
\sin\left(\frac{\pi z}{2}-\frac{\pi
    s}{2}\right)\left(\frac{\beta}{4\pi^3n}\right)^{-z/2}\, dz\nonumber\\
&=\frac{1}{2\pi i}\int_{(\lambda)}\frac{-2^{2z-3}\pi^{(3z+5)/2}
\left(\beta/n\right)^{-z/2}}{\G\left(\frac{3-z}{2}\right)\G(z-s)}\cos\left(\frac{\pi z}{2}\right)\cos\left(\frac{\pi (z-s)}{2}\right)\, dz,\notag
\end{align*}
by routine simplification using \eqref{feg}--\eqref{dup}.

To evaluate $H(s, \beta/n)$, we now move the line of integration
to $+\infty$ and apply the residue theorem. In this process, we encounter the
poles of the integrand at $z=1$ and at $z=2k+1+s,
k\in\mathbb{N}\cup\{0\}$. The residues at these poles are
\begin{equation*}
R_{1}(H)=\frac{\pi^{3}\sqrt{n}}{\beta\G(1-s)\sin\left(\frac{\pi s}{2}\right)}
\end{equation*}
and
\begin{align*}
R_{2k+1+s}(H)=\frac{(-1)^{k+1}2^{4k+2s}\pi^{3k+4+3s/2}
(\beta/n)^{-(k+(s+1)/2)}
\G\left(\frac{1}{2}s+k\right)}{\pi^2\G(2k+1)}.
\end{align*}
With the help of \eqref{strivert}, it is easy to show that the integrals
along the horizontal segments tend to zero as the height of the rectangular
contour tends to $\infty$. Also, arguing similarly as in
\eqref{fzp}--\eqref{infz}, we see that the integral along the shifted
vertical line in the limit also equals zero. Thus,
\begin{align}\label{hfin}
H\left(s, \frac{\beta}{n}\right)&=-R_{1}(H)-\sum_{k=0}^{\infty}R_{2k+1+s}\\
&=\frac{-\pi^{3}\sqrt{n}}{\beta\G(1-s)\sin\left(\frac{\pi s}{2}\right)}+2^{2s}\pi^{\frac{3}{2}s+2}\left(\frac{n}{\beta}\right)^{\frac{s+1}{2}}
\sum_{k=0}^{\infty}\frac{\G\left(\frac{s}{2}+k\right)}{(2k)!}\left(-\frac{16\pi^3n}{\beta}\right)^{k}\nonumber\\
&=\frac{-\pi^{3}\sqrt{n}}{\beta\G(1-s)\sin\left(\frac{\pi s}{2}\right)}+2^{2s}\pi^{\frac{3}{2}s+2}\left(\frac{n}{\beta}\right)^{\frac{s+1}{2}}
\G\left(\frac{s}{2}\right){}_1F_{1}\left(\frac{s}{2};\frac{1}{2};-\frac{4\pi^3n}{\beta}\right),\notag
\end{align}
where ${}_1F_{1}(a;c;z)$ is Kummer's confluent hypergeometric function.

 A result similar to the one in Theorem \ref{rwgthm} could be obtained when
 we replace the $-$ sign in the sine function by a $+$ sign.

The result now follows from \eqref{wsbbeff} and \eqref{hfin}, and from the
fact that $\a\b=4\pi^3$.
\end{proof}

\subsection{Special Cases of Theorem \ref{rwgthm}}

When $s=2m+1$, $m\in\mathbb{N}\cup\{0\}$, in Theorem \ref{rwgthm}, we obtain
the following result.
\begin{theorem}\label{sodd} For each non-negative integer $m$,
\begin{align}\label{m0}
&\sum_{n=1}^{\infty}\frac{\psi_{2m+1}(n)}{\sqrt{n}}e^{-\sqrt{n\beta}}\sin\left(\frac{\pi}{4}-\sqrt{n\b}\right)\nonumber\\
&=\frac{\sqrt{\b}}{\sqrt{2}}\zeta(-1-2m)-(2m)!\sin\left(\frac{\pi m}{2}\right)\zeta(1+m)(2\beta)^{-\left(m+\frac{1}{2}\right)}+\frac{1}{\sqrt{2}}\zeta\left(\frac{1}{2}\right)\zeta(-2m)\nonumber\\
&\quad+2^{m}\left(\frac{2\pi}{\beta}\right)^{m+\frac{1}{2}}\G\left(m+\frac{1}{2}\right)
\sum_{n=1}^{\infty}\psi_{-2m}(n)n^m{}_1F_{1}\left(m+\frac{1}{2};\frac{1}{2};-n\a\right).
\end{align}
\end{theorem}
When $m=0$ in \eqref{m0}, or equivalently when $s=1$ in Theorem \ref{rwgthm}, we obtain \eqref{p332}.
When $s=2m, m\in\mathbb{N}\cup\{0\}$, in Theorem \ref{rwgthm}, we obtain the
following companion result.

\begin{theorem}\label{seven} For each non-negative integer $m$,
\begin{align*}
&\sum_{n=1}^{\infty}\frac{\psi_{2m}(n)}{\sqrt{n}}e^{-\sqrt{n\beta}}\sin\left(\frac{\pi}{4}-\sqrt{n\b}\right)\nonumber\\
&=\frac{\sqrt{\b}}{\sqrt{2}}\zeta(-2m)-\frac{(2m-1)!}{\sqrt{2}}\sin\left(\frac{\pi m}{2}\right)\zeta\left(\frac{1+2m}{2}\right)(2\beta)^{-m}+\frac{1}{\sqrt{2}}\zeta\left(\frac{1}{2}\right)\zeta(1-2m)\nonumber\\
&\quad+2^{\frac{1}{2}-2m}\pi^{-2m}\sum_{n=1}^{\infty}\frac{\psi_{1-2m}(n)}{\sqrt{n}}
\left\{(-1)^{m+1}(2m)!\right.\notag\\&\quad\left.+2^{4m-1}\pi^{3m}\left(\frac{n}{\beta}\right)^m(m-1)!{}_1F_{1}
\left(m;\frac{1}{2};-n\a\right)\right\}.
\end{align*}
\end{theorem}

\subsection{A Common Generalization of $\sigma_s(n)$ and $\psi_s(n)$}
Note that if we define $\Omega_{s}(m,n):=\sum_{j^m|n}j^s$ for $m\in\mathbb{Z}, m\geq1$, then $\Omega_{s}(1, n)=\sigma_{s}(n)$ and $\Omega_{s}(2, n)=\psi_{s}(n)$. Thus, the series
\begin{equation}\label{seriesgen}
\sum_{n=1}^{\infty}\frac{\Omega_{s}(m, n)}{\sqrt{n}}e^{-\sqrt{n\beta}}\sin\left(\frac{\pi}{4}-\sqrt{n\b}\right)
\end{equation}
corresponds to the one in Theorem \ref{p3361a} when $m=1$, and to the one in Theorem \ref{rwgthm} when $m=2$. It might be interesting to find a representation of \eqref{seriesgen} analogous to those in these theorems for general $m$.

Attempting to evaluate \eqref{seriesgen} by first writing the sum as a line integral, then shifting the line of integration
to infinity and using the functional equation \eqref{fe} of $\zeta(s)$, we
encounter the sum of residues at the poles  $2(2k+1+s)/m$ to be a constant (independent of $k$ but depending on $m, s, n,$ and $\beta$) times
\begin{equation*}
\sum_{k=0}^{\infty}\frac{\Gamma\left(\frac{2k+1+s}{m}-\frac{1}{2}\right)}{(2k)!}\left(-\frac{2^{(2m+4)/m}
\pi^{(2m+2)/m}n^{2/m}}{\beta^{2/m}}\right)^k.
\end{equation*}
When $m=2$, this series essentially reduces to the ${}_1F_{1}$ present in (\ref{rwg}). If we let $\sqrt{\beta}=2\pi\sqrt{2x}$ and set $m=1$, we can write the series above as a ${}_3F_{2}$. However, we need to consider two cases, according as $n<x$ or as $n\geq x$, which should then give Theorem \ref{p3361a}. But for a general positive integer $m$, it is doubtful that the series above is summable in closed form.
\section{Ramanujan's Entries on Page 335 and Generalizations}
\label{sect10}

We begin this section by stating the two entries on page 335 in Ramanujan's
lost notebook \cite{lnb}.
 Define
\begin{equation}\label{ffn}
F(x)=\begin{cases}
	\lfloor x\rfloor, \quad &\text{if $x$ is not an integer},\\
x-\tf12, \quad&\text{if $x$ is an integer}.
\end{cases}
\end{equation}

\begin{entry}\label{1entry1} If $0<\theta<1$ and $F(x)$ is defined by \eqref{ffn}, then
\begin{multline}\label{entry1}
	\sum_{n=1}^{\infty}F\left(\frac{x}{n}\right)\sin(2\pi n\theta)=\pi x \left(\frac{1}{2}-\theta\right)-\frac{1}{4}\cot(\pi\theta)\\
	+\frac{1}{2}\sqrt{x}\sum_{m=1}^{\infty}\sum_{n=0}^{\infty}\left\{\frac{J_1(4\pi \sqrt{m(n+\theta)x})}{\sqrt{m(n+\theta)}}-\frac{J_1(4\pi \sqrt{m(n+1-\theta)x})}{\sqrt{m(n+1-\theta)}}\right\},
\end{multline}
where $J_{\nu}(x)$ denotes the ordinary Bessel function of order $\nu$.
\end{entry}

\begin{entry}\label{2entry2} If $0<\theta<1$ and $F(x)$ is defined by \eqref{ffn}, then
\begin{multline}\label{entry2}
	\sum_{n=1}^{\infty}F\left(\frac{x}{n}\right)\cos(2\pi n\theta)=\frac{1}{4}-x\log(2\sin(\pi\theta))\\
	+\frac{1}{2}\sqrt{x}\sum_{m=1}^{\infty}\sum_{n=0}^{\infty}\left\{\frac{I_1(4\pi \sqrt{m(n+\theta)x})}{\sqrt{m(n+\theta)}}+\frac{I_1(4\pi \sqrt{m(n+1-\theta)x})}{\sqrt{m(n+1-\theta)}}\right\},
\end{multline}
where
\begin{equation}\label{defofI}
	I_{\nu}(z):=-Y_{\nu}(z)+\frac{2}{\pi}\cos(\pi \nu)K_{\nu}(z),
\end{equation}
where $Y_{\nu}(x)$ denotes the Bessel function of the second kind of order $\nu$, and $K_{\nu}(x)$ denotes the modified Bessel function of order $\nu$.
\end{entry}

Entries \ref{1entry1} and \ref{2entry2} were established by Berndt, S.~Kim,
and Zaharescu under different conditions on the summation variables $m,n$ in
\cite{besselbbskaz, cdp1, bessel1}.  An expository account of their work
along with a survey of the circle and divisor problems can be found in
\cite{besselsurvey}.  See also the book \cite[Chapter 2]{ab4} by
Andrews and the first author.

It is easy to see from \eqref{ffn} that the left-hand sides of \eqref{entry1}
and \eqref{entry2} are finite. When $x\to 0{+}$, Entries \eqref{entry1} and \eqref{entry2} give the following interesting
limit evaluations:
\begin{align*}
\lim_{x\to 0{+}}\sqrt{x}\sum_{m=1}^{\infty}\sum_{n=0}^{\infty}\left\{\frac{J_1(4\pi \sqrt{m(n+\theta)x})}{\sqrt{m(n+\theta)}}-\frac{J_1(4\pi \sqrt{m(n+1-\theta)x})}{\sqrt{m(n+1-\theta)}}\right\}=\frac{1}{2}\cot(\pi\theta),
\end{align*}
and
\begin{align*}
\lim_{x\to 0{+}}\sqrt{x}\sum_{m=1}^{\infty}\sum_{n=0}^{\infty}\left\{\frac{I_1(4\pi \sqrt{m(n+\theta)x})}{\sqrt{m(n+\theta)}}+\frac{I_1(4\pi \sqrt{m(n+1-\theta)x})}{\sqrt{m(n+1-\theta)}}\right\}=-\frac{1}{2}.
\end{align*}
Direct proofs of these limit evaluations appear to be difficult.

As shown in \cite[equation (2.8)]{besselsurvey}, when $\theta=\tf14$, Entry
\ref{1entry1} is equivalent to the following famous identity due to Ramanujan
and Hardy
\cite{hardiv}, provided that the double sum in \eqref{entry1} is interpreted as
$\lim_{N\to\infty}\sum_{m,n\leq N}$, rather than as an iterated double sum
(see \cite[p.~26]{cdp1}):
\begin{equation*}
{\sum_{0<n\leq x}}^{\prime}r_2(n)=\pi x-1+\sum_{n=1}^{\infty}r_2(n)\left(\frac{x}{n}\right)^{1/2}J_{1}(2\pi\sqrt{nx}).
\end{equation*}
Note that the Bessel functions appearing in \eqref{entry2} are the
same as those appearing in \eqref{vsf}.  Indeed when $\theta=\tf12$,  Entry \ref{2entry2} is
connected with Vorono\"{\dotlessi}'s identity for $\sum_{n\leq x}d(n)$ as will be shown below. First, following the elementary
formula
\begin{equation*}
\sum_{n\leq x}d(n)=\sum_{n\leq x}\sum_{d|n}1=\sum_{dj\leq x}1
=\sum_{d\leq x}\left[\df{x}{d}\right],
\end{equation*}
we see that the left-hand side of \eqref{entry2}, for $\theta=\frac{1}{2}$, can be simplified as
\begin{equation*}
\sum_{n=1}^{\infty}F\left(\frac{x}{n}\right)\cos(\pi n)=\sideset{}{'}\sum_{n\leq x}\sum_{d|n}\cos(\pi d).
\end{equation*}
Second, let
\begin{equation*}
\ell=
\begin{cases}
0, \qquad \text{if}\hspace{1.5mm} n\hspace{1.5mm}\text{is odd},\\
1, \qquad \text{if}\hspace{1.5mm} n\hspace{1.5mm}\text{is even}.
\end{cases}
\end{equation*}
Note that
\begin{align*}
\sum_{d|n}\cos(\pi d)&=\#\hspace{1mm}\text{even divisors of}\hspace{1mm}n-\#\hspace{1mm}\text{odd divisors of}\hspace{1mm} n\nonumber\\
&=d\left(\frac{n}{2}\right)-\left\{d(n)-\ell d\left(\frac{n}{2}\right)\right\}\nonumber\\
&=(1+\ell)d\left(\frac{n}{2}\right)-d(n).
\end{align*}
Hence,
\begin{align*}
\sum_{n=1}^{\infty}F\left(\frac{x}{n}\right)\cos(\pi n)
&=-\sideset{}{'}\sum_{\substack{n\leq x\\n\text{ odd}}}d(n)+\sideset{}{'}\sum_{\substack{n\leq x\\n\text{ even}}}\left\{2d\left(\frac{n}{2}\right)-d(n)\right\}\nonumber\\
&=2\sideset{}{'}\sum_{n\leq\frac{1}{2}x}d(n)-\sideset{}{'}\sum_{n\leq x}d(n).
\end{align*}
Applying the Vorono\"{\dotlessi} summation formula \eqref{vsf} to each of
the sums above and simplifying, we find that
\begin{align*}
\sum_{n=1}^{\infty}F\left(\frac{x}{n}\right)&\cos(\pi n)
=-x\log 2+\frac{1}{4}-\sqrt{2x}\sum_{n=1}^{\infty}\frac{d(n)}{\sqrt{n}}\left(Y_{1}(2\pi\sqrt{2nx})+\frac{2}{\pi}K_{1}(2\pi\sqrt{2nx})\right)
\nonumber\\
&\quad+\sqrt{x}\sum_{n=1}^{\infty}\frac{d(n/2)}{\sqrt{n/2}}\left(Y_{1}(2\pi\sqrt{2nx})
+\frac{2}{\pi}K_{1}(2\pi\sqrt{2nx})\right)\nonumber\\
&=-x\log 2+\frac{1}{4}-\sqrt{2x}\sum_{k=1}^{\infty}\frac{1}{\sqrt{k}}\bigg(\sum_{\substack{d|k\\d\text{ odd}}}1\bigg)
\left(Y_{1}(2\pi\sqrt{2kx})+\frac{2}{\pi}K_{1}(2\pi\sqrt{2kx})\right).\notag
\end{align*}
Letting $k=m(2n+1)$ and interpreting the double sum as
$\lim_{N\to\infty}\sum_{m,n\leq N}$, we deduce that
\begin{align}
&\sum_{n=1}^{\infty}F\left(\frac{x}{n}\right)\cos(\pi n)=-x\log 2+\frac{1}{4}+\sqrt{x}\sum_{m=1}^{\infty}\sum_{n=0}^{\infty}\frac{I_{1}(4\pi\sqrt{m(n+\frac{1}{2})x})}{\sqrt{m(n+\frac{1}{2})}},
\label{goodsy}
\end{align}
where $I_1(z)$ is defined by \eqref{defofI}.  Then
\eqref{goodsy} is exactly Entry \ref{2entry2} with $\theta=\frac{1}{2}$.

It should be mentioned here that Dixon and Ferrar \cite{dixfer2}
established, for $a,b>0$, the identity
\begin{equation}\label{dixonferrar}
a^{\mu/2}\sum_{n=0}^{\infty}\frac{r_2(n)}{(n+b)^{\mu/2}}K_{\mu}(2\pi\sqrt{a(n+b)})
=b^{(1-\mu)/2}
\sum_{n=0}^{\infty}\frac{r_2(n)}{(n+a)^{(1-\mu)/2}}K_{1-\mu}(2\pi\sqrt{b(n+a)}).
\end{equation}
Generalizations have been given by Berndt \cite[p.~343, Theorem 9.1]{III} and
F.~Oberhettinger and K.~Soni \cite[p.~24]{os}.
Using Jacobi's identity
\begin{equation*}
r_2(n)=4\sum_{\substack{d|n\\ {d\hspace{0.5mm}\text{odd}}}}(-1)^{(d-1)/2},
\end{equation*}
we can recast \eqref{dixonferrar} as an identity between double
series
\begin{align*}
&a^{\mu/2}\sum_{n=0}^{\infty}\sum_{m=0}^{\infty}\left\{\frac{K_{\mu}\left(4\pi\sqrt{a\left(\left(n+\tfrac{1}{4}\right)m
+\tfrac{b}{4}\right)}\right)}{((4n+1)m+b)^{\mu/2}}
-\frac{K_{\mu}\left(4\pi\sqrt{a\left(\left(n+\tfrac{3}{4}\right)m+\tfrac{b}{4}\right)}\right)}
{((4n+3)m+b)^{\mu/2}}\right\}\nonumber\\
&=b^{(1-\mu)/2}\sum_{n=0}^{\infty}\sum_{m=0}^{\infty}\left\{\frac{K_{1-\mu}
\left(4\pi\sqrt{b\left(\left(n+\tfrac{1}{4}\right)m
+\tfrac{a}{4}\right)}\right)}{((4n+1)m+a)^{(1-\mu)/2}}\right.
\left.-\frac{K_{1-\mu}\left(4\pi\sqrt{b\left(\left(n+\tfrac{3}{4}\right)m
+\tfrac{a}{4}\right)}\right)}{((4n+3)m+a)^{(1-\mu)/2}}\right\}.
\end{align*}

In this section, we establish one-variable generalizations of Entries
\ref{1entry1} and \ref{2entry2}, where the double sums here are also
interpreted as
 $\lim_{N\to\infty}\sum_{m,n\leq N}$, instead of as iterated double sums.  It is
 an open problem to determine if the series can be replaced by
 iterated double series.
As in Entries \ref{1entry1} and \ref{2entry2}, the series on the left-hand sides of
Theorems \ref{bdrz01} and \ref{bdrz02} are finite.

\begin{theorem}\label{bdrz01} Let $\zeta(s,a)$ denote the Hurwitz zeta function.
Let $0<\theta<1$. Then, for $|\s|<\frac{1}{2}$,
\begin{align}\label{ebdrz01}
&\sum_{n=1}^{\infty}F\left(\frac{x}{n}\right)\frac{\sin\left(2\pi n \theta\right)}{n^{s}}
=-x\frac{\sin(\pi s/2)\Gamma(-s)}{(2\pi)^{-s}}(\zeta(-s,\theta)-\zeta(-s,1-\theta))\\\nonumber
&\quad-\frac{\cos(\pi s/2)\Gamma(1-s)}{2(2\pi)^{1-s}}(\zeta(1-s,\theta)-\zeta(1-s,1-\theta))
+\frac{x}{2}\sin\left(\frac{\pi s}{2}\right)\nonumber\\
&\quad\times\sum_{m=1}^{\infty}\sum_{n=0}^{\infty}\left\{\frac{M_{1-s}\left(4\pi \sqrt{mx\left(n+\theta\right)}\right)}{(mx)^{(1+s)/2}(n+\theta)^{(1-s)/2}}-\frac{M_{1-s}\left(4\pi \sqrt{mx\left(n+1-\theta\right)}\right)}{(mx)^{(1+s)/2}(n+1-\theta)^{(1-s)/2}}\right\},\notag
\end{align}
where
\begin{equation}\label{M}
	M_{\nu}(x)=\frac{2}{\pi}K_{\nu}(x)+\frac{1}{\sin(\pi
          \nu)}(J_{\nu}(x)-J_{-\nu}(x))=\frac{2}{\pi}K_{\nu}(x)+Y_{\nu}(x)+J_{\nu}(x)\tan\left(\frac{\pi \nu}{2}\right).
\end{equation}
\end{theorem}

We show that Entry \ref{1entry1} is identical with Theorem \ref{bdrz01} when $s=0$. First observe that \cite[p.~264, Theorem 12.13]{apostol}
\begin{equation}\label{hzetazero}
\zeta(0,\theta)=\frac{1}{2}-\theta
\end{equation}
 and
\begin{equation*}
\lim_{s\to 0}(\zeta(1-s,\theta)-\zeta(1-s,1-\theta))=\psi(1-\theta)-\psi(\theta)=\pi \cot(\pi\theta),
\end{equation*}
where $\psi(z)=\Gamma'(z)/\Gamma(z)$ denotes the digamma function. Since, by \eqref{sumbesselj}, $J_{-1}(x)=-J_1(x)$,
\begin{equation}\label{y1j1}
\lim_{s\to 0}\sin(\pi s/2)M_{1-s}(x)=J_1(x).
\end{equation}
Now taking the limit as $s\to 0$ on both sides of \eqref{ebdrz01} and using \eqref{hzetazero}--\eqref{y1j1}, we obtain Entry \ref{1entry1}.

\begin{theorem}\label{bdrz02}
Let $0<\theta<1$. Then, for $|\s|<\frac{1}{2}$,
\begin{align}\label{ebdrz02}
&\sum_{n=1}^{\infty}F\left(\frac{x}{n}\right)\frac{\cos\left(2\pi n \theta\right)}{n^{s}}
=x\frac{\cos(\pi s/2)\Gamma(-s)}{(2\pi)^{-s}}(\zeta(-s,\theta)+\zeta(-s,1-\theta))\\\nonumber
&\quad-\frac{\sin(\tf12\pi s)\Gamma(1-s)}{2(2\pi)^{1-s}}(\zeta(1-s,\theta)+\zeta(1-s,1-\theta))
-\frac{x}{2}\cos\left(\frac{\pi s}{2}\right)
\notag\\&\quad\times\sum_{m=1}^{\infty}\sum_{n=0}^{\infty}\left\{\frac{H_{1-s}\left(4\pi \sqrt{mx\left(n+\theta\right)}\right)}{(mx)^{\frac{1+s}{2}}(n+\theta)^{\frac{1-s}{2}}}+\frac{H_{1-s}\left(4\pi \sqrt{mx\left(n+1-\theta\right)}\right)}{(mx)^{\frac{1+s}{2}}(n+1-\theta)^{\frac{1-s}{2}}}\right\},\notag
\end{align}
where
\begin{equation}\label{H}
	H_{\nu}(x)=\frac{2}{\pi}K_{\nu}(x)-\frac{1}{\sin(\pi \nu)}(J_{\nu}(x)+J_{-\nu}(x))=\frac{2}{\pi}K_{\nu}(x)+Y_{\nu}(x)-J_{\nu}(x)\cot\left(\frac{\pi \nu}{2}\right).
\end{equation}
\end{theorem}

We demonstrate that Entry \ref{2entry2} can be obtained from Theorem \ref{bdrz02} as the particular case $s=0$.  First,
\begin{align}\label{s0}
\lim_{s\to 0}\Gamma(-s)(\zeta(-s,\theta)+\zeta(-s,1-\theta))&=\lim_{s\to 0}(-s)\Gamma(-s)\frac{(\zeta(-s,\theta)+\zeta(-s,1-\theta))}{-s}\nonumber\\\nonumber
&=\zeta'(0,\theta)+\zeta'(0,1-\theta)\\
&=-\log(2\sin(\pi\theta)),
\end{align}
where we used the fact that
$\zeta'(0,\theta)=\log(\Gamma(\theta))-\frac{1}{2}\log(2\pi)$
\cite{bcbhurwitz}. Second, since $s=1$ is a simple pole of $\zeta(s,\theta)$
with residue $1$, then
\begin{multline*}
\lim_{s\to 0}\sin(\pi s/2)(\zeta(1-s,\theta)+\zeta(1-s,1-\theta))\\=\lim_{s\to 0}\frac{\sin(\pi s/2)}{s}s(\zeta(1-s,\theta)+\zeta(1-s,1-\theta))
=-\pi.
\end{multline*}
Third, by \eqref{yj},
\begin{equation}\label{y1}
\lim_{s\to 0}\frac{1}{2\sin(\pi s/2)}(J_{1-s}(x)+J_{s-1}(x))
=-Y_1(x).
\end{equation}
Taking the limit as $s\to0$ in \eqref{ebdrz02} while using \eqref{s0}--\eqref{y1}, we obtain Entry \ref{2entry2}.

\section{Further Preliminary Results}\label{sect11}

 Let us define the generalized twisted divisor sum by
\begin{equation}\label{twistdivsum}
	\sigma_s(\chi,n):=\sum_{d\mid n}\chi(d)d^s,
\end{equation}
which, for Re $z>$ max$\{1,1+\sigma\}$, has the generating function
\begin{equation*}
	\zeta(z)L(z-s,\chi)=\sum_{n=1}^{\infty}\frac{\sigma_s(\chi,n)}{n^z}.
\end{equation*}

The following lemma from the papers of Vorono\"{\dotlessi} \cite{voronoi} and  Oppenheim \cite{oppenheim} is instrumental in proving our main theorems.

\begin{lemma}\label{ldivsum}
If $x>0$, $x\notin\mathbb{Z}$, and $-\tf12<\s<\tf12$, then
\begin{equation}\label{ohio}
\sideset{}{'}\sum_{n\leq x}\sigma_{-s}(n)=-\cos(\tf12 \pi s)\sum_{n=1}^{\infty}\sigma_{-s}(n)\left(\frac{x}{n}\right)^{(1-s)/2}H_{1-s}\left(4\pi\sqrt{nx}\right)
+xZ(s,x)-\frac{1}{2}\zeta(s),
\end{equation}
where $H_{\nu}(x)$ is defined in \eqref{H}, and where
\begin{equation}\label{Z}
Z(s,x)=\begin{cases}\zeta(1+s)+\dfrac{\zeta(1-s)}{1-s}x^{-s}, \quad &\text{if }s\neq 0,\\
\log x+2\gamma-1, \quad&\text{if }s=0,
\end{cases}
\end{equation}
is analytic for all $s$.
\end{lemma}

We show that \eqref{ohio} reduces to Vorono\"{\dotlessi}'s formula \eqref{vsf} when $s=0.$ From the definition \eqref{H} of $H_{\nu}$  and \eqref{y1}, we find that
\begin{equation*}
	H_{1}(4\pi\sqrt{nx})=Y_{1}(4\pi\sqrt{nx})+\frac{2}{\pi}K_{1}(4\pi\sqrt{nx})=-I_{1}(4\pi\sqrt{nx}).
\end{equation*}
We now show that
\begin{equation}\label{illini}
\lim_{s\to 0}Z(s,x)=\log x+2\gamma-1=Z(0,x).
\end{equation}
Recall that the Laurent series expansion of $\zeta(s)$ near the pole $s=1$ is given by
\begin{equation*}
\zeta(s)=\frac{1}{s-1}+\gamma+\sum_{n=1}^{\infty}\frac{(-1)^n\gamma_n(s-1)^n}{n!},
\end{equation*}
where $\gamma_n$, $n\geq1$,  are the Stieltjes constants defined by \cite{bcbrocky}
\begin{equation}\label{stiel}
\gamma_n=\lim_{N\to\infty}\left(\sum_{k=1}^N\df{\log^nk}{k}-\df{\log^{n+1}N}{n+1}\right).
\end{equation}
Thus, by \eqref{Z}, for $s>0$,
\begin{equation*}
Z(s,x)=\frac{s-1+x^{-s}}{s(s-1)}+\gamma\left(1-\frac{x^{-s}}{s-1}\right)+\sum_{n=1}^{\infty}\frac{(-1)^n\gamma_ns^n}{n!}
+\frac{x^{-s}}{1-s}\sum_{n=1}^{\infty}\frac{\gamma_ns^n}{n!}.
\end{equation*}
Hence,
\begin{equation*}
\lim_{s\to 0}Z(s,x)=\lim_{s\to 0}\frac{s-1+x^{-s}}{s(s-1)}+2\gamma=-\lim_{s\to 0}(1-\log x x^{-s})+2\gamma=\log x+2\gamma-1,
\end{equation*}
which proves \eqref{illini}.

\begin{lemma}\label{oddcharsin}
Let $F(x)$ be defined by \eqref{ffn}. For each character $\chi$ modulo $q$, where $q$ is  prime, define the Gauss sum
\begin{equation}\label{tau}
\tau(\chi)=\sum_{n\pmod q}\chi(n)e^{2\pi in/q}.
\end{equation}
If $0<a<q$ and $(a,q)=1$, then, for any complex number $s$,
\begin{equation*}
\sum_{n=1}^{\infty}F\left(\frac{x}{n}\right)\sin\left(\frac{2\pi n a}{q}\right)n^{s}=-iq^s\sum_{\substack{d\mid q \\ d>1}}\frac{1}{d^s\phi(d)}\sum_{\substack{\chi \operatorname{mod} d\\ \chi \operatorname {odd}}}\chi(a)\tau(\bar{\chi})\sideset{}{'}\sum_{1\leq n\leq dx/q}\sigma_s(\chi,n),
\end{equation*}
where $\phi(n)$ denotes Euler's $\phi$-function.
\end{lemma}

\begin{proof}
    First, we see that
\begin{equation}\label{gendivsum}
	\sideset{}{'}\sum_{n\leq x}\sigma_{s}(n)=\sideset{}{'}\sum_{n\leq x}\sum_{d\mid n}d^{s}=\sum_{d\leq x}d^{s}\sideset{}{'}\sum_{m=1}^{\left\lfloor x/d\right\rfloor}1=\sum_{n=1}^{\infty}F\left(\frac{x}{n}\right)n^{s}.
\end{equation}
Similarly, for any Dirichlet character $\chi$ modulo $q$,
\begin{equation}\label{genteistdivsum}
	\sideset{}{'}\sum_{n\leq x}\sigma_{s}(\chi,n)=\sum_{n=1}^{\infty}F\left(\frac{x}{n}\right)\chi(n)n^{s},
\end{equation}
where $\sigma_{s}(\chi,n)$ is defined in \eqref{twistdivsum}. We have
\begin{align*}
\quad\sum_{n=1}^{\infty}F\left(\frac{x}{n}\right)\sin\left(\frac{2\pi n a}{q}\right)n^{s}&=\sum_{n=1}^{\infty}\sum_{d\mid q}\sum_{(n,q)=q/d}F\left(\frac{x}{n}\right)\sin\left(\frac{2\pi n a}{q}\right)n^{s}\\\nonumber
&=\sum_{d\mid q}\sum_{\substack{m=1\\(m,d)=1}}^{\infty}F\left(\frac{dx}{qm}\right)\sin\left(\frac{2\pi m a}{d}\right)\left(\frac{qm}{d}\right)^{s}\\\nonumber
&=\sum_{\substack{d\mid q\\d>1}}\sum_{\substack{m=1\\(m,d)=1}}^{\infty}F\left(\frac{dx}{qm}\right)\sin\left(\frac{2\pi m a}{d}\right)\left(\frac{qm}{d}\right)^{s}.
\end{align*}
Now using the fact \cite[p.~72, Lemma 2.5]{besselcrelle}
\begin{equation}\label{echarcter}
\sin\left(\frac{2\pi ma}{d}\right)=\frac{1}{i\phi(d)}\sum_{\substack{\chi\operatorname{mod} d\\\chi\text{ odd}} }\chi(a)\tau(\bar{\chi})\chi(m),
\end{equation}
we find that
\begin{align*}
&\sum_{n=1}^{\infty}F\left(\frac{x}{n}\right)\sin\left(\frac{2\pi n a}{q}\right)n^{s}\\&=\sum_{\substack{d\mid q\\d>1}}\frac{1}{i\phi(d)}\sum_{\substack{m=1\\(m,d)=1}}^{\infty}F\left(\frac{dx}{qm}\right)\left(\frac{qm}{d}\right)^{s}
\sum_{\substack{\chi\operatorname{mod} d \\\chi\operatorname{odd}}}\tau(\bar{\chi})\chi(m)\chi(a)\\\nonumber
&=-iq^s\sum_{\substack{d\mid q\\d>1}}\frac{1}{d^s\phi(d)}\sum_{\substack{\chi\operatorname{mod} d \\\chi\operatorname{odd}}}\tau(\bar{\chi})\chi(a)\sideset{}{'}\sum_{n\leq dx/q}\sigma_s(\chi,n),
\end{align*}
as can be seen from \eqref{genteistdivsum}. This completes the proof of Lemma \ref{oddcharsin}.
\end{proof}

\begin{lemma}\label{evencahrcos}
If $0<a<q$ and $(a,q)=1$, then, for any complex number $s$,
\begin{multline*}
\sum_{n=1}^{\infty}F\left(\frac{x}{n}\right)\cos\left(\frac{2\pi n a}{q}\right)n^{s}\\=q^s\sideset{}{'}\sum_{1\leq n\leq x/q}\sigma_s(n)+q^s\sum_{\substack{d\mid q \\ d>1}}\frac{1}{d^s\phi(d)}\sum_{\substack{\chi \operatorname{mod} d\\ \chi \operatorname {even}}}\chi(a)\tau(\bar{\chi})\sideset{}{'}\sum_{1\leq n\leq dx/q}\sigma_s(\chi,n).
\end{multline*}
\end{lemma}

\begin{proof}
We have
\begin{align*}
\quad\sum_{n=1}^{\infty}F\left(\frac{x}{n}\right)&\cos\left(\frac{2\pi n a}{q}\right)n^{s}=\sum_{n=1}^{\infty}\sum_{d\mid q}\sum_{(n,q)=q/d}F\left(\frac{x}{n}\right)\cos\left(\frac{2\pi n a}{q}\right)n^{s}\\\nonumber
&=\sum_{d\mid q}\sum_{\substack{m=1\\(m,d)=1}}^{\infty}F\left(\frac{dx}{qm}\right)\cos\left(\frac{2\pi m a}{d}\right)\left(\frac{qm}{d}\right)^{s}\\\nonumber
&=\sum_{m=1}^{\infty}F\left(\frac{x}{qm}\right)(qm)^{s}+\sum_{\substack{d\mid q\\d>1}}\sum_{\substack{m=1\\(m,d)=1}}^{\infty}F\left(\frac{dx}{qm}\right)\cos\left(\frac{2\pi n a}{d}\right)\left(\frac{qm}{d}\right)^{s}.
\end{align*}
Invoking \eqref{gendivsum} and \eqref{echarcter} above, we find that
\begin{align*}
\quad\sum_{n=1}^{\infty}F\left(\frac{x}{n}\right)&\cos\left(\frac{2\pi n a}{q}\right)n^{s}
=q^s\sideset{}{'}\sum_{n\leq x/q}\sigma_s(n)\\\nonumber
&\quad+q^s\sum_{\substack{d\mid q\\d>1}}\frac{1}{d^s\phi(d)}\sum_{\substack{m=1\\(m,d)=1}}^{\infty}F\left(\frac{dx}{qm}\right)m^{s}
\sum_{\substack{\chi\operatorname{mod} d \\\chi\operatorname{even}}}\tau(\bar{\chi})\chi(a)\chi(m)\\
&=q^s\sideset{}{'}\sum_{n\leq x/q}\sigma_s(n)+q^s\sum_{\substack{d\mid q\\d>1}}\frac{1}{d^s\phi(d)}\sum_{\substack{\chi\operatorname{mod} d \\\chi\operatorname{even}}}\tau(\bar{\chi})\chi(a)\sideset{}{'}\sum_{n\leq dx/q}\sigma_s(\chi,n).\notag
\end{align*}
Thus, we have finished the proof of Lemma \ref{evencahrcos}.
\end{proof}
  We need a lemma from \cite[p.~5, Lemma 1]{chandrasekharannarasimhan}.

\begin{lemma}\label{titchma}
Let $\sigma_a$ denote the abscissa of absolute convergence for  $$\phi(s):=\sum_{n=1}^{\infty}a_n\lambda_n^{-s}.$$
Then for $k\geq 0$, $\s>0$, and $\s>\s_a$,
\begin{align*}
\frac{1}{\Gamma(k+1)}\sideset{}{'}\sum_{\lambda_n\leq x}a_n(x-\lambda_n)^k=\frac{1}{2\pi i}\int_{(\s)}\frac{\Gamma(s)\phi(s)x^{s+k}}{\Gamma(s+k+1)}\, ds,
\end{align*}
where the prime $\prime$ on the summation sign indicates that if $k=0$ and
$x=\lambda_m$ for some positive integer $m$, then we count only $\tf12a_m$.
\end{lemma}

We recall the following version of the Phragm\'en-Lindel\"of theorem \cite[p. 109]{littlewood}.

\begin{lemma}\label{phragmenlindelof}
Let $f$ be holomorphic in a strip $S$ given by $a<\s<b$, $|t|>\eta>0$, and continuous on the boundary. If for some constant $\theta<1$,
\begin{align*}
f(s)\ll \exp(e^{\theta \pi |s|/(b-a)}),
\end{align*}
uniformly in $ S $, $ f(a+it)=o(1) $, and $ f(b+it)=o(1) $ as $|t|\to
\infty$, then $ f(\s+it)=o(1) $ uniformly in $ S $ as $ |t|\to \infty $.
\end{lemma}

We also need two lemmas, proven by K.~Chandrasekharan and R.~Narasimhan
\cite[Corollaries 1 and 2, p.~11]{chandrasekharannarasimhan} (see also
\cite[Lemmas 12 and 13]{bII}), which are based on results of A.~Zygmund
\cite{zygmund} for equi-convergent series. We recall that two series
\begin{align*}
\sum_{j=-\infty}^{\infty}a_j(x)\quad\text{and}\quad\sum_{j=-\infty}^{\infty}b_j(x)
\end{align*}
are  uniformly equi-convergent on an interval if
\begin{align*}
\sum_{j=-n}^{n}[a_j(x)-b_j(x)]
\end{align*}
converges uniformly  on that interval as $n\to\infty$  \cite[Definition
5]{bII}.

\begin{lemma}\label{equiconser1}
Let $a_n$ be a positive strictly increasing sequence of numbers tending to
$\infty$, and suppose that $a_n=a_{-n}$. Suppose that $J$ is a closed
interval contained in an interval $I$ of length $2\pi$. Assume  that
\begin{align*}
\sum_{n=-\infty}^{\infty}|c_n|<\infty.
\end{align*}
Then, if $g$ is a function with period $2\pi$ which equals
\begin{align*}
\sum_{n=-\infty}^{\infty}c_ne^{ia_n x}
\end{align*}
on $I$, the Fourier series of $g$ converges uniformly on $J$.
\end{lemma}
\begin{lemma}\label{equiconser2}
With the same notation as Lemma \ref{equiconser1}, assume that
\begin{align*}
\operatorname{sup}_{0\leq h\leq 1}\left|\sum_{k<a_n<k+h}c_n\right|=o(1),
\end{align*}
as $k\to\infty$, and
\begin{align*}
\sum_{n=-\infty}^{\infty}\frac{|c_n|}{a_n}<\infty.
\end{align*}
Let $A(x)$ be a $\mathcal{C}^{\infty}$ function with compact support on $I$,
which is equal to $1$ on $J$. Furthermore, let $B(x)$ be a
$\mathcal{C}^{\infty}$ function. Then, the series
\begin{align*}
B(x)\sum_{n=-\infty}^{\infty}c_ne^{ia_n x}
\end{align*}
is uniformly equi-convergent on $J$ with the differentiated series of the Fourier series of a function with period $ 2\pi $, which equals
\begin{align*}
A(x)\sum_{n=-\infty}^{\infty}c(n)W_n(x)
\end{align*}
on $I$, where $W_n(x)$ is an antiderivative of $B(x)e^{ia_n x}$.
\end{lemma}

Let the Fourier series of any function $ f $ defined, say, in the interval $
(-\pi, \pi) $, be
\begin{align*}
S[f]:=\sum_{n=-\infty}^{\infty}c_ne^{inx}.
\end{align*}
The following result of Zygmund \cite[Theorem 6.6,
p.~53]{zygmundtrig} expresses the Riemann-Lebesgue localization
principle.

\begin{lemma}\label{localization}
If two functions $f_1$ and $f_2$ are equal in an interval $I$, then $S[f_1]$ and $S[f_2]$ are uniformly equi-convergent in any interval $I'$ interior to $I$.
\end{lemma}

For each integer $\lambda$ define
\begin{align}\label{capgt}
\tilde{G}_{\lambda+s}(z):=J_{\lambda+s}(z)\cos\left(\frac{\pi s}{2}\right)-\left(Y_{\lambda+s}(z)-(-1)^{\l}\frac{2}{\pi}K_{\lambda+s}(z)\right)\sin\left(\frac{\pi s}{2}\right).
\end{align}
By \eqref{jd}, \eqref{yd}, and \eqref{kd},
\begin{align}\label{dergt}
\frac{d}{dx}\left(\frac{x}{u}\right)^{(1+k-s)/2}\s_s(n)\tilde{G}_{1+k-s}(4\pi
\sqrt{xu})=2\pi
\left(\frac{x}{u}\right)^{(k-s)/2}\s_s(n)\tilde{G}_{k-s}(4\pi
\sqrt{xu}).
\end{align}
Let us consider the Dirichlet series $\sum_{n=1}^{\infty}a_n\mu_n^{-s}$ with abscissa of absolute convergence $\s_a$ and
$ 0<\mu_1<\mu_2<\cdots<\mu_n\to\infty.$
For $y>0$ and $\nu=\lambda+s$, define
\begin{align*}
\tilde{F}_{\nu}(y):=\sum_{n=1}^{\infty}a_n\left(\frac{qy^2}{\mu_n}\right)^{\nu/2}
\tilde{G}_{\nu}\left(4\pi
  y\sqrt{\frac{\mu_n}{q}}\right)
\end{align*}
and
\begin{align*}
F_{\nu}(y):=\sum_{n=1}^{\infty}a_n\left(\frac{qy^2}{\mu_n}\right)^{\nu/2}G_{\nu}
\left(4\pi y\sqrt{\frac{\mu_n}{q}}\right),
\end{align*}
where $G_{\lambda+s}(z)$ is defined in \eqref{capg}. Suppose that
\begin{align}\label{absconv}
\sum_{n=1}^{\infty}\frac{|a_n|}{\mu_n^{\frac{1}{2}\nu+\frac{3}{4}}}<\infty
\end{align}
and
\begin{align}\label{conconv}
\operatorname{sup}_{0\leq h\leq 1}\left\lvert\sum_{m^2<\mu_n\leq(m+h)^2}
\frac{a_n}{\mu_n^{\frac{1}{2}\nu+\frac{1}{4}}}\right\rvert=o(1),
\end{align}
as $m\to\infty$.

The following lemma is similar to Theorem II in
\cite{chandrasekharannarasimhan} and Lemma 14 in \cite{bII}.

\begin{lemma}\label{ftnuequi} The function
$2y\tilde{F}_{\nu}(y)$ is uniformly equi-convergent on any interval $J$ of
length less than $1$ with the differentiated series of the Fourier series of
a function with period $1$, which on $I$ equals $A(y)\tilde{F}_{\nu+1}(y)$,
where $I$ is of length $1$ and contains $J$. Moreover, $\tilde{F}_{\nu}(y)$
is a continuous function.
\end{lemma}

\begin{proof}
We examine the function
\begin{align}\label{funf}
&f(y):=2q^{\nu/2}y^{1+\nu}\sum_{n=1}^{\infty}\left(\frac{a_n}{\mu_n}\right)^{\nu/2}\bigg\{\tilde{G}_{\nu}\left(4\pi y\sqrt{\frac{\mu_n}{q}}\right)\\
&-\frac{q^{1/4}}{\pi\mu_n^{1/4}(2y)^{1/2}}\left(\cos\left(4\pi y\sqrt{\frac{\mu_n}{q}}-\dfrac{\pi \nu}{2}-\dfrac{\pi}{4}\right)d_0+\sin\left(4\pi y\sqrt{\frac{\mu_n}{q}}-\dfrac{\pi \nu}{2}-\dfrac{\pi}{4}\right)d_0'\right)\nonumber\\
&-\frac{q^{3/4}}{2\pi^2\mu_n^{3/4}y^{3/2}}\left(\sin\left(4\pi y\sqrt{\frac{\mu_n}{q}}-\dfrac{\pi \nu}{2}-\dfrac{\pi}{4}\right)d_1+\cos\left(4\pi y\sqrt{\frac{\mu_n}{q}}-\dfrac{\pi \nu}{2}-\dfrac{\pi}{4}\right)d_1'\right)\bigg\},
\notag
\end{align}
where $d_0, d_0', d_1$, and $d_1'$ are constants.
Since $y>0$, then by the definition \eqref{capgt}, \eqref{asymbess},
\eqref{asymbess1}, \eqref{asymbess2}, and \eqref{absconv}, the function
$f(y)$ in \eqref{funf} is a continuously differentiable function. Let $g$ be
a function with period $1$ which equals $f$ on $I$. Since $f$ is continuously
differentiable,
the Fourier series of $g$ is uniformly convergent on $J$. By the hypothesis
\eqref{absconv}, \eqref{conconv}, and Lemma \ref{equiconser2}, the series
\begin{align*}
&2q^{\nu/2}y^{1+\nu}\sum_{n=1}^{\infty}\left(\frac{a_n}{\mu_n}\right)^{\nu/2}\frac{q^{1/4}}{\pi\mu_n^{1/4}(2y)^{1/2}}\nonumber\\
&\hspace{1cm}\times\left(\cos\left(4\pi y\sqrt{\frac{\mu_n}{q}}-\dfrac{\pi \nu}{2}-\dfrac{\pi}{4}\right)d_0+\sin\left(4\pi y\sqrt{\frac{\mu_n}{q}}-\dfrac{\pi \nu}{2}-\dfrac{\pi}{4}\right)d_0'\right)
\end{align*}
is uniformly equi-convergent on $J$ with the derived series of the Fourier
series of a function that is of period $1$ and equals on $I$,
\begin{align}\label{derivese}
&A(y)\sum_{n=1}^{\infty}\left(\frac{a_n}{\mu_n}\right)^{\nu/2}\int_{\a}^{y}2q^{\nu/2}t^{1+\nu}\frac{q^{1/4}}{\pi\mu_n^{1/4}(2t)^{1/2}}
\nonumber\\
&\quad\times\left(\cos\left(4\pi t\sqrt{\frac{\mu_n}{q}}-\dfrac{\pi
      \nu}{2}-\dfrac{\pi}{4}\right)d_0+\sin\left(4\pi
    t\sqrt{\frac{\mu_n}{q}}-\dfrac{\pi
      \nu}{2}-\dfrac{\pi}{4}\right)d_0'\right)\, dt,
\end{align}
for some $\a>0$.
Using Lemma \ref{equiconser1}, we can prove a result similar to that of
\eqref{derivese} for the series
\begin{align*}
&2q^{\nu/2}y^{1+\nu}\sum_{n=1}^{\infty}\left(\frac{a_n}{\mu_n}\right)^{\nu/2}\frac{q^{3/4}}{2\pi^2\mu_n^{3/4}(y)^{3/2}}\nonumber\\
&\hspace{1cm}\times\left(\cos\left(4\pi y\sqrt{\frac{\mu_n}{q}}-\dfrac{\pi \nu}{2}-\dfrac{\pi}{4}\right)d_0+\sin\left(4\pi y\sqrt{\frac{\mu_n}{q}}-\dfrac{\pi \nu}{2}-\dfrac{\pi}{4}\right)d_0'\right).
\end{align*}
Hence, the series
\begin{align*}
2y\sum_{n=1}^{\infty}a_n\left(\frac{qy^2}{\mu_n}\right)^{\nu/2}\tilde{G}_{\nu}\left(4\pi y\sqrt{\frac{\mu_n}{q}}\right)
\end{align*}
is uniformly equi-convergent on $J$ with the derived series of the Fourier series of a function that is of period $1$ and equals on $I$,
\begin{align*}
A(y)&\sum_{n=1}^{\infty}a_n\int_{0}^{y}2t\left(\frac{qt^2}{\mu_n}\right)^{\nu/2}\tilde{G}_{\nu}\left(4\pi t\sqrt{\frac{\mu_n}{q}}\right)\,dt\\
&=\frac{A(y)}{2\pi}\sum_{n=1}^{\infty}a_n\left(\frac{qy^2}{\mu_n}\right)^{(\nu+1)/2}\tilde{G}_{\nu+1}\left(4\pi y\sqrt{\frac{\mu_n}{q}}\right).
\end{align*}
In the last step we use \eqref{dergt}. This completes the proof of the lemma.

\end{proof}
The following lemma is proved by the same kind of argument.

\begin{lemma}\label{fnuequi}
The function $2yF_{\nu}(y)$ is uniformly equi-convergent on any interval $J$ of length less than $1$ with the differentiated series of the Fourier series of a function with period $1$, which on $I$ equals $A(y)F_{\nu+1}(y)$, where $I$ is of length $1$ and contains $J$. Moreover, $F_{\nu}(y)$ is a continuous function.
\end{lemma}

\section{Proof of Theorem \ref{bdrz01}}\label{sect13}
We  prove the theorem under the assumption that the double series on the right-hand sides of \eqref{ebdrz01} and \eqref{ebdrz02} are summed symmetrically, i.e., the product  $mn$ of the indices of summation tends to $\infty$. Under this assumption, we prove that the double series in \eqref{ebdrz01} and \eqref{ebdrz02} are uniformly convergent with respect to $\theta$ on any compact subinterval of $(0,1)$. By continuity, it is sufficient to prove the theorem for all primes $q$ and all fractions $\theta=a/q$, where $0<a<q$.
 Therefore for these values of $\theta$, Theorem \ref{bdrz01} is equivalent to the following theorem.
\begin{theorem}\label{bdrz05}
Recall that $M_{\nu}$ is defined in \eqref{M}. Let $q$ be a prime and $0<a<q$. Let
\begin{align}\label{L}
&L_s(a,q,x)=-\frac{x}{2}\sin\left(\frac{\pi s}{2}\right)\\&\quad\times\sum_{m=1}^{\infty}\sum_{n=0}^{\infty}\left\{\frac{M_{1-s}\left(4\pi \sqrt{mx\left(n+a/q\right)}\right)}{(mx)^{(1+s)/2}(n+a/q)^{(1-s)/2}}-\frac{M_{1-s}\left(4\pi \sqrt{mx\left(n+1-a/q\right)}\right)}{(mx)^{(1+s)/2}(n+1-a/q)^{(1-s)/2}}\right\},\notag
\end{align}
where $M_{s}(z)$ is defined in \eqref{M}.
Then, for $ |\s|< \frac{1}{2}$,
\begin{align*}
L_s(a,q,x)+\sum_{n=1}^{\infty}F\left(\frac{x}{n}\right)&\frac{\sin\left(2\pi n a/q\right)}{n^s}
=-x\frac{\sin(\pi s/2)\Gamma(-s)}{(2\pi)^{-s}}\left(\zeta(-s,\frac{a}{q})-\zeta\left(-s,1-\frac{a}{q}\right)\right)\notag\\
&\quad-\frac{\cos(\pi s/2)\Gamma(1-s)}{2(2\pi)^{1-s}}\left(\zeta\left(1-s,\frac{a}{q}\right)-\zeta\left(1-s,1-\frac{a}{q}\right)\right),
\end{align*}
where $\zeta(s,a)$ denotes the Hurwitz zeta function.
\end{theorem}

First we need the following theorem.

\begin{theorem}\label{bdrz006}
If $\chi$ is a non-principal odd primitive character modulo $q$, $x>0$,
$|\s|<1/2$, and $k$ is a non-negative integer, then
\begin{align}
&\frac{1}{\Gamma(k+1)}\sideset{}{'}\sum_{n\leq x}\sigma_{-s}(\chi,n)(x-n)^k\nonumber\\
&=\frac{x^{k+1}L(1+s,\chi)}{\Gamma(k+2)}-\frac{x^kL(s,\chi)}{2\Gamma(k+1)}
+2\sum_{n=1}^{\left\lfloor\frac{k+1}{2}\right\rfloor}\frac{(-1)^{n-1}x^{k-2n+1}}{\Gamma(k-2n+2)}\frac{\z(2n)}{(2\pi)^{2n}}
L(1-2n+s,\chi)\nonumber\\
&\quad+\frac{i}{\tau(\bar{\chi})(2\pi)^k}\sum_{n=1}^{\infty}\sigma_{-s}(\bar{\chi},n)\left(\frac{qx}{n}\right)^{(1-s+k)/2}
\tilde{G}_{1-s+k}
\left(4\pi\sqrt{\frac{nx}{q}}\right),\notag
\end{align}
where $\tilde{G}_{\lambda-s}(z)$ is defined in \eqref{capgt}. The series on the right-hand side
converges uniformly on any interval for $x > 0$, where the left-hand side is continuous.
The convergence is bounded on any interval $0 < x_1\leq x\leq x_2<\infty$ when $k=0$.
\end{theorem}

\begin{proof}
From \eqref{twistdivsum} and Lemma \ref{titchma}, for a fixed $x>0$, we see that
\begin{align}\label{perron1}
\frac{1}{\Gamma(k+1)}\sideset{}{'}\sum_{n\leq x}\sigma_{-s}(\chi,n)(x-n)^k=\frac{1}{2\pi i}\int_{(c)}\zeta(w)L(w+s,\chi)\frac{\Gamma(w)x^{w+k}}{\Gamma(w+k+1)}dw,
\end{align}
where $\max\{1,1-\s, \s\}<c$ and $k\geq 0$. Consider the positively oriented rectangular contour $R$ with vertices $[c\pm iT, 1-c\pm iT]$.
 Observe that the integrand on the right-hand side of \eqref{perron1} has poles at $w=1$ and $w=0$ inside the contour $R$. By the residue  theorem,
\begin{align}\label{cresapp1}
&\frac{1}{2\pi i}\int_{R}\zeta(w)L(w+s,\chi)\frac{\Gamma(w)x^{w+k}}{\Gamma(w+k+1)}\,dw\\\nonumber
&\quad=R_1\left(\zeta(w)L(w+s,\chi)\frac{\Gamma(w)x^{w+k}}{\Gamma(w+k+1)}\right)
+R_0\left(\zeta(w)L(w+s,\chi)\frac{\Gamma(w)x^{w+k}}{\Gamma(w+k+1)}\right),
\end{align}
where we recall that $R_a(f(w))$ denotes the residue of the function $f(w)$ at the pole $w=a$. Straightforward computations show that
\begin{align}\label{r0zl1}
R_{0}\left(\zeta(w)L(w+s,\chi)\frac{\Gamma(w)x^{w+k}}{\Gamma(w+k+1)}\right)&=\frac{\z(0)L(s,\chi)x^{1+k}}{\Gamma(k+1)}
\end{align}
and
\begin{align}\label{r1zl1}
R_1\left(\zeta(w)L(w+s,\chi)\frac{\Gamma(w)x^{w+k}}{\Gamma(w+k+1)}\right)&=\frac{x^{k+1}L(1+s,\chi)}{\Gamma(k+2)}.
\end{align}
We show that the contribution from the integrals along the horizontal sides $(\s\pm iT, 1-c\leq\s\leq c)$ on the left-hand side of \eqref{cresapp1} tends to zero as $|t|\to\infty$. We prove this fact by showing that
\begin{align*}
\zeta(w)L(w+s,\chi)\frac{\Gamma(w)x^{w+k}}{\Gamma(w+k+1)}=o(1),
\end{align*}
as $|\textnormal{Im }w|\to \infty$, uniformly for $1-c\leq\textnormal{ Re }w\leq c$. The functional equation for $L(s,\chi)$ for an odd primitive Dirichlet character $\chi$  is given by \cite[p.~69]{davenport}
\begin{equation}\label{evenfe1}
\left(\df{\pi}{q}\right)^{-(1+s)/2}\Gamma\left(\frac{1+s}{2}\right)L(s,\chi)
=\df{i\tau(\chi)}{\sqrt{q}}\left(\df{\pi}{q}\right)^{-(2-s)/2}\Gamma\left(\frac{2-s}{2}\right)L(1-s,\overline{\chi}),
\end{equation}
where $\tau(\chi)$ is the Gauss sum defined in \eqref{tau}. Combining the functional equation \eqref{fe} of $\zeta(w)$  and the functional equation \eqref{evenfe1} of $L(w+s,\chi)$ for odd primitive $\chi$, we deduce the functional equation
\begin{equation}\label{fezetal1}
	\zeta(w)L(w+s,\chi)=\frac{i\pi^{2w+s-1}}{\tau(\bar{\chi})q^{w+s-1}}\eta(w,s)\zeta(1-w)L(1-w-s,\bar{\chi}),
\end{equation}
where
\begin{equation*}
	\eta(w,s)=\frac{\Gamma\left(\frac12(1-w)\right)\Gamma\left(\frac12(2-w-s)\right)}{\Gamma\left(\frac12{w}\right)
\Gamma\left(\frac12(1+w+s)\right)}.
\end{equation*}
Since $c>\max\{1,1-\s, \s\}$,
\begin{align*}
\zeta(c+it)L(c+it+s,\chi)=O(1),
\end{align*}
as $|t|\to\infty$.
Using \eqref{strivert}, we see that
\begin{align}\label{grster1}
\frac{\Gamma(w)}{\Gamma(w+k+1)}=O(|\textnormal{Im }w|^{-1-k}),
\end{align}
uniformly in $1-c\leq \textnormal{ Re }w\leq c$, as $|\textnormal{Im }w|\to\infty$. Therefore, for $w=c+it$,
\begin{align}\label{rhsest1}
\zeta(w)L(w,\chi)\frac{\Gamma(w)x^{w+k}}{\Gamma(w+k+1)}=o(1),
\end{align}
as $|t|\to \infty$.
Again, using Stirling's formula \eqref{strivert} for the Gamma function and the relation \eqref{fezetal1}, we find that, for $w=1-c+it$,
\begin{align}\label{lhsest1}
&\zeta(w)L(w+s,\chi)\frac{\Gamma(w)x^{w+k}}{\Gamma(w+k+1)},\nonumber\\
&=\frac{i\pi^{2w+s-1}}{\tau(\bar{\chi})q^{w+s-1}}\eta(w,s)\zeta(1-w)L(1-w-s,\bar{\chi})
\frac{\Gamma(w)x^{w+k}}{\Gamma(w+k+1)}\nonumber\\
&=O_{q,s}(t^{2c-\s-k-2})\nonumber\\
&=o(1),
\end{align}
as $|t|\to \infty$, provided that $k>2c-\s-2$. From \eqref{grster1} and \cite[pp.~79, 82, equations (2),(15)]{davenport},
\begin{align}\label{midest1}
	\zeta(w)L(w+s,\chi)\frac{\Gamma(w)x^{w+k}}{\Gamma(w+k+1)}\ll_q\exp{(C|w|\log|w|)},
\end{align}
for some constant $C$ and $|\textnormal{Im }w|\to \infty$. Since the function on the left-hand side of \eqref{midest1} is holomorphic for $|\textnormal{Im }w|>\eta'>0$, then, by using \eqref{rhsest1}, \eqref{lhsest1}, \eqref{midest1}, and Lemma \ref{phragmenlindelof}, we deduce that
\begin{align*}
	\zeta(w)L(w+s,\chi)\frac{\Gamma(w)x^{w+k}}{\Gamma(w+k+1)}=o(1),
\end{align*}
uniformly for $1-c\leq \textnormal{ Re }w\leq c$ and  $|\textnormal{Im }w|\to \infty$. Therefore,
\begin{align}\label{horintt111}
\int_{c\pm iT}^{1-c\pm iT}\zeta(w)&L(w+s,\chi)\frac{\Gamma(w)x^{w+k}}{\Gamma(w+k+1)}dw=o(1),
\end{align}
as $T\to \infty$. Using the evaluation $\z(0)=-\tf12$ and combining \eqref{perron1}, \eqref{cresapp1}, \eqref{r0zl1}, \eqref{r1zl1}, and \eqref{horintt111}, we deduce that
\begin{multline}\label{residue11}
\frac{1}{\Gamma(k+1)}\sideset{}{'}\sum_{n\leq x}\sigma_{-s}(\chi,n)(x-n)^{k}=\frac{x^{k+1}L(1+s,\chi)}{\Gamma(k+2)}-\frac{L(s,\chi)x^{k}}{2\Gamma(k+1)}\\
+\frac{1}{2\pi i}\int_{(1-c)}\zeta(w)L(w+s,\chi)\frac{\Gamma(w)x^{w+k}}{\Gamma(w+k+1)}dw,
\end{multline}
provided that $k\geq 0$ and $k>2c-\s-2$. Define
\begin{equation}\label{inti1111111}
		I(y):=\frac{1}{2\pi i}\int_{(1-c)}\frac{\eta(w,s)\Gamma(w)}{\Gamma(w+k+1)}y^{w} \,dw.
	\end{equation}
Using the functional equation \eqref{fezetal1} in the integrand on the right-hand side of \eqref{residue11} and inverting the order of summation and integration, we find that
\begin{align}\label{inti1111}
	&\frac{1}{2\pi i}\int_{(1-c)}\zeta(w)L(w+s,\chi)\frac{\Gamma(w)x^{w+k}}{\Gamma(w+k+1)}\,dw\notag\\
&=\frac{ix^k\pi^{s-1}}{\tau(\bar{\chi})q^{s-1}}\frac{1}{2\pi i}\int_{(1-c)}\frac{\eta(w,s)\Gamma(w)}{\Gamma(w+k+1)}\zeta(1-w)L(1-w-s,\bar{\chi})\left(\frac{\pi^2x}{q}\right)^{w} \,dw\notag\\\nonumber
	&=\frac{ix^k\pi^{s-1}}{\tau(\bar{\chi})q^{s-1}}\frac{1}{2\pi i}\int_{(1-c)}
\frac{\eta(w,s)\Gamma(w)}{\Gamma(w+k+1)}\left(\frac{\pi^2x}{q}\right)^{w}
\sum_{n=1}^{\infty}\frac{\sigma_{s}(\bar{\chi},n)}{n^{1-w}} \,dw\\\nonumber
	&=\frac{ix^k\pi^{s-1}}{\tau(\bar{\chi})q^{s-1}}\sum_{n=1}^{\infty}\frac{\sigma_{s}(\bar{\chi},n)}{n^{1+k}}\frac{1}{2\pi i}\int_{(1-c)}\frac{\eta(w,s)\Gamma(w)}{\Gamma(w+k+1)}\left(\frac{\pi^2 n x}{q}\right)^{w} \,dw\\
	&=\frac{ix^k\pi^{s-1}}{\tau(\bar{\chi})q^{s-1}}\sum_{n=1}^{\infty}\frac{\sigma_{s}(\bar{\chi},n)}{n}I\left(\frac{\pi^2 n x}{q}\right),
	\end{align}
provided that $k>2c-\s-1$. We compute the integral  $I(y)$ by using the residue calculus, shifting the line of integration to the right, and letting $c\to -\infty$.

Let $k$ be a positive integer and $\s\neq 0$. From \eqref{inti1111111}, we can write
\begin{equation*}
		I(y):=\frac{1}{2\pi i}\int_{(1-c)}F(w)\,dw,
	\end{equation*}
where
\begin{align*}
F(w):=\frac{\Gamma(w)\Gamma\left(\frac12(1-w)\right)\Gamma\left(\frac12(2-w-s)\right)y^{w}}{\Gamma(1+k+w)\Gamma\left(\frac12{w}\right)
\Gamma\left(\frac12(1+w+s)\right)}.
\end{align*}
Note that the poles of the function $F(w)$  on the right side of the line $1-c+it, -\infty<t<\infty$,
are at $w=2m+1$ and $w=2m+2-s$ for $m=0, 1, 2, \dots$.  Thus,
\begin{align*}
R_{2m+1}(F(w))&=(-1)^{m+1}\frac{2\Gamma(2m+1)\Gamma\left(-m-\frac12(s-1)\right)y^{2m+1}}{m!\Gamma(2+k+2m)\Gamma\left(m+\frac12\right)
\Gamma\left(1+m+\frac12(s)\right)}
\end{align*}
and
\begin{align*}
R_{2m+2-s}(F(w))&=(-1)^{m+1}\frac{2\Gamma(2m+2-s)\Gamma\left(-m+\frac12(s-1)\right)y^{2m+2-s}}{m!\Gamma(3+k+2m-s)
\Gamma\left(m+\frac12(2-s)\right)\Gamma\left(m+\frac{3}{2}\right)}.
\end{align*}
 With the aid of the duplication formula \eqref{dup} and the reflection formula \eqref{ref} for $\Gamma(s)$, we find that
\begin{align}\label{r2mF1}
R_{2m+1}(F(w))&=-\frac{2^{s-1}}{\cos(\pi s/2)}\frac{(2\sqrt{y})^{4m+2}}{(2m+k+1)!
\Gamma(2m+s+1)}
\end{align}
and
\begin{align}\label{r2msF1}
R_{2m+2-s}(F(w))&=\frac{2(2y)^{2-s}}{\cos(\pi s/2)}\frac{(2\sqrt{y})^{4m}}{(2m+1)!
\Gamma(2m+k+3-s)}.
\end{align}
	Now from \cite[pp.~77--78]{watsonbessel}, we recall that the modified Bessel function $I_{\nu}(z)$ is defined by
		\begin{align}\label{sumbesseli}
	I_{\nu}(z):=\sum_{m=0}^{\infty}\frac{(z/2)^{2m+\nu}}{m!\Gamma(m+1+\nu)},
	\end{align}
	and that $K_{\nu}(z)$ can be represented as
	\begin{align}\label{sumbesselk}
	K_{\nu}(z)=\frac{\pi}{2}\frac{I_{-\nu}(z)-I_{\nu}(z)}{\sin(\pi\nu)}.
	\end{align}
(We emphasize that the definition of $I_{\nu}(z)$ given in \eqref{sumbesseli} should not be confused with the definition of $I_{\nu}(z)$ given by Ramanujan in \eqref{defofI}.)
Therefore, from \eqref{sumbesselj}, \eqref{sumbesseli}, and \eqref{r2mF1}, for $k$ even,
\begin{align}\label{sumr2mF1}
\sum_{m=0}^{\infty}R_{2m+1}(F(w))
&=-\frac{2^{s-1-2k}y^{-k}}{\cos(\pi s/2)}\sum_{m=0}^{\infty}\frac{(2\sqrt{y})^{4m+2k+2}}{(2m+k+1)!
\Gamma(2m+1+s)}\nonumber\\
&=-\frac{2^{s-1-2k}y^{-k}}{\cos(\pi s/2)}\bigg\{\sum_{m=0}^{\infty}\frac{(2\sqrt{y})^{4m+2k+2}}{(2m+1)!
\Gamma(2m+1+s-k)}\nonumber\\
&\quad -\sum_{m=1}^{k/2}\frac{(2\sqrt{y})^{4m-2}}{(2m-1)!
\Gamma(2m-1+s-k)}\bigg\}\nonumber\\
&=-\frac{2^{-1-k}y^{(1-s-k)/2}}{\cos(\pi s/2)}(I_{-1+s-k}(4\sqrt{y})-J_{-1+s-k}(4\sqrt{y}))\nonumber\\
&\quad+\frac{2^{s-1-2k}y^{-k}}{\cos(\pi s/2)}\sum_{m=1}^{k/2}\frac{(2\sqrt{y})^{4m-2}}{(2m-1)!
\Gamma(2m-1+s-k)}\nonumber\\
&=-\frac{2^{-1-k}y^{(1-s-k)/2}}{\cos(\pi s/2)}(I_{-1+s-k}(4\sqrt{y})-J_{-1+s-k}(4\sqrt{y}))\nonumber\\
&\quad+\frac{2^{s+1}}{\cos(\pi s/2)}\sum_{m=1}^{k/2}\frac{2^{-4m}y^{1-2n}}{\Gamma(k-2m+2)
\Gamma(1-2m+s)}.
\end{align}
Similarly, for $k$ odd
\begin{align}\label{sumr2mF1o}
\sum_{m=0}^{\infty}R_{2m+1}(F(w))
&=-\frac{2^{-1-k}y^{(1-s-k)/2}}{\cos(\pi s/2)}(I_{-1+s-k}(4\sqrt{y})+J_{-1+s-k}(4\sqrt{y}))\nonumber\\
&\quad+\frac{2^{s+1}}{\cos(\pi s/2)}\sum_{m=1}^{(k+1)/2}\frac{2^{-4m}y^{1-2n}}{\Gamma(k-2m+2)
\Gamma(1-2m+s)}.
\end{align}
From \eqref{r2msF1}, \eqref{sumbesselj}, and \eqref{sumbesseli}, we find that
\begin{align}\label{sumr2msF1}
\sum_{m=0}^{\infty}R_{2m+1-s}(F(w))&=\frac{2^{-1-k}y^{(1-s-k)/2}}{\cos(\pi s/2)}(-J_{1-s+k}(4\sqrt{y})+I_{1-s+k}(4\sqrt{y})).
\end{align}
Invoking \eqref{sumbesselk} in the sum of \eqref{sumr2mF1}, \eqref{sumr2mF1o}, and \eqref{sumr2msF1}, we deduce that
\begin{align*}
&\sum_{m=0}^{\infty}\left(R_{2m+1}(F(w))+R_{2m+1-s}(F(w))\right)
=-\frac{\sin(\pi s/2)}{2^{k}y^{(-1+s+k)/2}}\notag\\&\quad\times\left(\frac{J_{1-s+k}(4\sqrt{y})+(-1)^{k+1}J_{-1+s-k}(4\sqrt{y})}{\sin{\pi s}}-(-1)^{k+1}\frac{2}{\pi}K_{1-s+k}(4\sqrt{y})\right)\nonumber\\
&\quad+\frac{2^{s+1}}{\cos(\pi s/2)}\sum_{m=1}^{\left\lfloor\frac{k+1}{2}\right\rfloor}\frac{2^{-4m}y^{1-2n}}{\Gamma(k-2m+2)
\Gamma(1-2m+s)}.
\end{align*}
Consider the positively oriented contour $\mathcal{R}_N$ formed by the points $\{1-c-iT,2
 N+\tf32-iT,2 N+\tf32+iT,1-c+iT\}$, where $T>0$ and $N$ is a positive
 integer. By the residue theorem,
 \begin{align}\label{resth1}
 \frac{1}{2\pi i}\int_{\mathcal{R}_N}F(w)\, dw
=\sum_{0\leq k\leq N}R_{2k+1}(F(w))+
 \sum_{0\leq k\leq N}R_{2k+1-s}(F(w)).
 \end{align}
 Recall Stirling's formula in the form \cite[p.~73, equation (5)]{davenport}
 \begin{equation*}
 \Gamma(s)=\sqrt{2\pi}e^{-s}s^{s-1/2}e^{f(s)},
 \end{equation*}
 for $-\pi<\operatorname{arg}s<\pi$ and $f(s)=O\left(1/|s|\right)$, as $|s|\to\infty$. Therefore, for fixed $T>0$ and $\sigma\to \infty$,
 \begin{equation}\label{stirling'sformula}
 \Gamma(s)=O\left(e^{-\sigma+(\sigma-1/2)\log\sigma}\right).
 \end{equation}
 Hence, for the integral over the right side of the rectangular contour $\mathcal{R}_N$,
 \begin{equation}\label{verint1}
\int_{2N+3/2-iT}^{2N+3/2+iT}F(w)\,  dw\ll_{T,s} y^{2N+3/2}e^{4N-(4N+2+k+\s)\log N }=o(1),
 \end{equation}
 as $N\to \infty$. Using Stirling's formula \eqref{strivert} to estimate the integrals over  the horizontal sides of $\mathcal{R}_N$, we find that
\begin{align}\label{horint1}
\int_{1-c\pm iT}^{\infty\pm iT}F(w)\, dw\ll_{s}\int_{1-c}^{\infty}y^{\sigma}T^{-2\b-\s-k}\, d\sigma
\ll_{s,y}\frac{y^{1-c}}{ T^{2c-\s-k-2}\log T}=o(1),
\end{align}
provided that $k>2c-\s-2$. Using \eqref{resth1}, \eqref{verint1}, and \eqref{horint1} in \eqref{inti1111111}, we deduce that
\begin{align}\label{iyint11}
I(y)&=\frac{\sin(\pi s/2)}{2^{k}y^{(-1+s+k)/2}}\left(\frac{J_{1-s+k}(4\sqrt{y})+(-1)^{k+1}J_{-1+s-k}(4\sqrt{y})}{\sin{\pi s}}\right.\\&\quad\left.-(-1)^{k+1}\frac{2}{\pi}K_{1-s+k}(4\sqrt{y})\right)
-\frac{2^{s+1}}{\cos(\pi s/2)}\sum_{m=1}^{\left\lfloor\frac{k+1}{2}\right\rfloor}\frac{2^{-4m}y^{1-2n}}{\Gamma(k-2m+2)
\Gamma(1-2m+s)}.\notag
\end{align}
Using the functional equation \eqref{evenfe1}, the reflection formula \eqref{ref}, and the duplication formula \eqref{dup}, for $y=\pi^2nx/q$, we find that
\begin{align}\label{extrak}
&\sum_{n=1}^{\infty}\frac{\s_s(\chi,n)}{n}\bigg\{\frac{2^{s+1}}{\cos(\pi s/2)}\sum_{m=1}^{\left\lfloor\frac{k+1}{2}\right\rfloor}\frac{2^{-4m}y^{1-2n}}{\Gamma(k-2m+2)
\Gamma(1-2m+s)} \bigg\}\nonumber\\
&=2i\tau{(\bar{\chi})}\frac{\pi^{1-s}}{q^{1-s}}
\sum_{n=1}^{\left\lfloor\frac{k+1}{2}\right\rfloor}(-1)^{n-1}\frac{x^{-2n+1}}{\Gamma(k-2n+2)}\frac{\z(2n)}{(2\pi)^{2n}}L(1-2n+s,\chi).
\end{align}
With the aid of \eqref{yj}, we see that
\begin{multline}\label{gmrel}
\sin(\pi s/2)\left(\frac{J_{1-s+k}(4\sqrt{y})+(-1)^{k+1}J_{-1+s-k}(4\sqrt{y})}{\sin{\pi s}}-(-1)^{k+1}\frac{2}{\pi}K_{1-s+k}(4\sqrt{y})\right)\\
=\tilde{G}_{1+k-s}(4\sqrt{y}).
\end{multline}
Combining \eqref{residue11}, \eqref{inti1111}, \eqref{iyint11}, and \eqref{extrak}, we see that
\begin{align}\label{residue111}
\frac{1}{\Gamma(k+1)}\sideset{}{'}\sum_{n\leq x}\sigma_{-s}(\chi,n)(x-n)^{k}
&=\frac{x^{k+1}L(1+s,\chi)}{\Gamma(k+2)}-\frac{L(s,\chi)x^{k}}{2\Gamma(k+1)}\\
&\quad+2\sum_{n=1}^{\left\lfloor\frac{k+1}{2}\right\rfloor}(-1)^{n-1}
\frac{x^{k-2n+1}}{\Gamma(k-2n+2)}\frac{\z(2n)}{(2\pi)^{2n}}L(1-2n+s,\chi)\nonumber\\
&\quad+\frac{i}{\tau(\bar{\chi})(2\pi)^k}
\sum_{n=1}^{\infty}\sigma_{-s}(\bar{\chi},n)\left(\frac{xq}{n}\right)^{\frac{1-s+k}{2}}\tilde{G}_{1-s+k}
\left(4\pi\sqrt{\frac{nx}{q}}\right),\notag
\end{align}
provided that $k\geq 0$, $\s\neq 0$, and $k>2c-\s-1$.

For $x>0$ fixed, by the asymptotic expansions for Bessel functions \eqref{asymbess}, \eqref{asymbess1}, and \eqref{asymbess2}, there exists a sufficiently large integer $N_0$ such that
\begin{align*}
\tilde{G}_{1+k-s}(4\pi\sqrt{\frac{nx}{q}})\ll_{q}\frac{1}{(nx)^{1/4}},
\end{align*}
for all $n>N_0$. Hence, for $x>0$,
\begin{align*}
\sum_{n>N_0}\left(\frac{qx}{n}\right)^{(1+k-s)/2}\s_s(n)\tilde{G}_{1+k-s}\left(4\pi\sqrt{\frac{nx}{q}}\right)&\ll_q x^{(2k-2\s-1)/4}\sum_{n>N_0}\frac{\s_{\s}(n)}{n^{(2k-2\s+3)/4}}\\&\ll_q x^{(2k-2\s-1)/4},
\end{align*}
provided that $k>|\s|+\tfrac{1}{2}$.
Therefore, for $k>|\s|+\tfrac{1}{2}$ and $x>0$, the
series
\begin{align*}
\sum_{n=1}^{\infty}\left(\frac{qx}{n}\right)^{(1+k-s)/2}\s_s(n)\tilde{G}_{1+k-s}\left(4\pi \sqrt{\frac{nx}{q}}\right)
\end{align*}
is absolutely and uniformly convergent for $0<x_1\leq x\leq x_2<\infty$. Thus, by differentiating a suitable number of times with the aid of \eqref{dergt}, we find that \eqref{residue111} may be then upheld for $k>|\s|+\tfrac{1}{2}$. Since $|\s|<\tf12$, the series on the left-hand side of \eqref{residue111} is continuous for $k>|\s|+\tfrac{1}{2}$. Conversely, we can see that the series on the left-hand side of \eqref{residue111} is continuous when $k>0$, which implies that $|\s|<\tf12$. Thus, the identity \eqref{residue111} is valid for $k>|\s|+\tfrac{1}{2}$ and $\s\neq 0$. Since the series on the right-hand side of \eqref{residue111} is absolutely and uniformly  convergent for $0<x_1\leq x\leq x_2<\infty$, we can take the limit as $s\to 0$ on both sides of \eqref{residue111} for $|\s|<\tf12$ and $k>|\s|+\tfrac{1}{2}$. Hence,  the identity \eqref{residue111} is valid for $k>|\s|+\tfrac{1}{2}$ with $|\s|<\tf12$.

Suppose that  the identity
\begin{align}\label{testiden}
&\frac{1}{\Gamma(k+1)}\sideset{}{'}\sum_{n\leq x}\sigma_{-s}(\chi,n)(x-n)^{k}
=\frac{x^{k+1}L(1+s,\chi)}{\Gamma(k+2)}-\frac{L(s,\chi)x^{k}}{2\Gamma(k+1)}\notag\\
&\quad+2\sum_{n=1}^{\left\lfloor\frac{k+1}{2}\right\rfloor}(-1)^{n-1}
\frac{x^{k-2n+1}}{\Gamma(k-2n+2)}\frac{\z(2n)}{(2\pi)^{2n}}L(1-2n+s,\chi)\nonumber\\
&\quad+\frac{i}{\tau(\bar{\chi})(2\pi)^k}
\sum_{n=1}^{\infty}\sigma_{-s}(\bar{\chi},n)\left(\frac{xq}{n}\right)^{\frac{1-s+k}{2}}\tilde{G}_{1-s+k}
\left(4\pi\sqrt{\frac{nx}{q}}\right),
\end{align}
is valid for some $k>0$. Let $\b>\max\{1,1-\s\}$. Then
\begin{align*}
\sum_{n=1}^{\infty}\frac{|\s_s(n)|}{n^{\b}}<\infty\\
\end{align*}
and
\begin{align*}
\operatorname{sup}_{0\leq h\leq 1}\left\lvert\sum_{m^2<n\leq (m+h)^2}\frac{\s_s(n)}{n^{\b-1/2}}\right\rvert=o(1),
\end{align*}
as $m\to \infty$.
Put $x = y^2$ in the identity \eqref{testiden}, where $y$
lies in an interval $J$ of length less than $1$. By Lemma \ref{ftnuequi}, $2y$ times the infinite series on the right-hand side of \eqref{testiden}, with $ x=y^2 $, is uniformly equi-convergent on $J$ with the differentiated series of the Fourier series of
a function with period $1$ which equals $A(y)\tilde{F}_{2-s+k}(y)$ on $I$, provided that
$k>|\s|-\tfrac{1}{2}$. But then, $k+1>|\s|+\tf12$. Hence, from \eqref{residue111},
\begin{align*}
&\frac{i}{\tau(\bar{\chi})(2\pi)^{k+1}}A(y)\tilde{F}_{2-s+k}(y)\nonumber\\&=A(y)\bigg\{\sideset{}{'}\sum_{n\leq y^2}\frac{\sigma_{-s}(\chi,n)(y^2-n)^{k+1}}{\Gamma(k+2)}-\frac{y^{2(k+2)}L(1+s,\chi)}{\Gamma(k+3)}
+\frac{L(s,\chi)y^{2(k+1)}}{2\Gamma(k+2)}\nonumber\\
&\hspace{2cm}-2\sum_{n=1}^{\left\lfloor\frac{k+2}{2}\right\rfloor}(-1)^{n-1}
\frac{y^{2(k-2n+2)}}{\Gamma(k-2n+3)}\frac{\z(2n)}{(2\pi)^{2n}}L(1-2n+s,\chi)\bigg\}\nonumber\\
&=A(y)\bigg\{\int_{0}^{y^2}\sideset{}{'}\sum_{n\leq t}\frac{\sigma_{-s}(\chi,n)(t-n)^{k}}{\Gamma(k+1)}\, dt-\frac{y^{2(k+2)}L(1+s,\chi)}{\Gamma(k+3)}+\frac{L(s,\chi)y^{2(k+1)}}{2\Gamma(k+2)}\nonumber\\
&\hspace{2cm}-2\sum_{n=1}^{\left\lfloor\frac{k+2}{2}\right\rfloor}(-1)^{n-1}
\frac{y^{2(k-2n+2)}}{\Gamma(k-2n+3)}\frac{\z(2n)}{(2\pi)^{2n}}L(1-2n+s,\chi)\bigg\}\nonumber\\
&=A(y)\bigg\{\int_{0}^{y}\sideset{}{'}\sum_{n\leq t^2}\frac{\sigma_{-s}(\chi,n)(t^2-n)^{k}2t}{\Gamma(k+1)}\, dt-\frac{y^{2(k+2)}L(1+s,\chi)}{\Gamma(k+3)}+\frac{L(s,\chi)y^{2(k+1)}}{2\Gamma(k+2)}\nonumber\\
&\hspace{2cm}-2\sum_{n=1}^{\left\lfloor\frac{k+2}{2}\right\rfloor}(-1)^{n-1}
\frac{y^{2(k-2n+2)}}{\Gamma(k-2n+3)}\frac{\z(2n)}{(2\pi)^{2n}}L(1-2n+s,\chi)\bigg\}.
\end{align*}
Note that $A(y)=1$ on $J$. Therefore, from Lemma \ref{localization} and the properties of the Fourier series of the function
\begin{align*}
\frac{2y}{\Gamma(k+1)}\sideset{}{'}\sum_{n\leq y^2}\sigma_{-s}(\chi,n)(y^2-n)^{k}
\end{align*}
in $I$, we see that the identity \eqref{residue111} holds for $k>|\s|-\tf12$, which completes the proof of Theorem \ref{bdrz006}.
\end{proof}

From  \eqref{M} and \eqref{gmrel}, we find that $\sin(\pi s/2)M_{1-s}(z)=\tilde{G}_{1-s}(z)$. The case $k=0$ of Theorem \ref{bdrz006} gives the following corollary.

\begin{corollary}\label{bdrz06}
If $\chi$ is a non-principal odd primitive character modulo $q$, $x>0$,
and $|\s|<1/2$, then
\begin{align}
\sideset{}{'}\sum_{n\leq x}\sigma_{-s}(\chi,n)&=xL(1+s,\chi)-\df12 L(s,\chi)\nonumber\\
&\quad+\frac{i\sin{(\pi s)/2}}{\tau(\bar{\chi})}\sum_{n=1}^{\infty}\sigma_{-s}(\bar{\chi},n)\left(\frac{qx}{n}\right)^{\frac{1-s}{2}}M_{1-s}
\left(4\pi\sqrt{\frac{nx}{q}}\right),\notag
\end{align}
where $M_{1-s}(z)$ is defined in \eqref{M}.
\end{corollary}

Next, we show that Theorem \ref{bdrz06} implies Theorem \ref{bdrz05}. We then
finish this section and hence finish the proof of Theorem \ref{bdrz01} by
proving that Theorem \ref{bdrz01} implies Theorem \ref{bdrz06}.

\medskip

\begin{proof}[Proof that Theorem \ref{bdrz06} implies Theorem \ref{bdrz05}.]
Recall that $L_s(a,q,x)$ and $M_{\nu}(z)$ are defined in \eqref{L} and \eqref{M}, respectively.  Thus,
\allowdisplaybreaks{\begin{align*}
	&L_s(a,q,x)=-\frac{x}{2}\sin\left(\frac{\pi s}{2}\right)\notag\\
&\quad\notag\times\sum_{m=1}^{\infty}\sum_{n=0}^{\infty}\left\{\frac{M_{1-s}\left(4\pi \sqrt{mx\left(n+\frac{a}{q}\right)}\right)}{(mx)^{(1+s)/2}(n+a/q)^{(1-s)/2}}-\frac{M_{1-s}\left(4\pi \sqrt{mx\left(n+1-\frac{a}{q}\right)}\right)}{(mx)^{(1+s)/2}(n+1-a/q)^{(1-s)/2}}\right\}\notag\\\nonumber
	&=-\frac{x}{2}\sin\left(\frac{\pi s}{2}\right)\nonumber\\
&\quad\times\sum_{m=1}^{\infty}\left\{\sum_{\substack{n=1\\n\equiv a\operatorname{mod}q}}^{\infty}\frac{M_{1-s}\left(4\pi \sqrt{\frac{mnx}{q}}\right)}{(mx)^{(1+s)/2}(n/q)^{(1-s)/2}}-\sum_{\substack{n=1\\n\equiv -a\operatorname{mod}q}}^{\infty}\frac{M_{1-s}\left(4\pi \sqrt{\frac{mnx}{q}}\right)}{(mx)^{(1+s)/2}(n/q)^{(1-s)/2}}\right\}\notag\\\nonumber
	&=-\frac{(qx)^{(1-s)/2}}{2\phi(q)}\sin\left(\frac{\pi s}{2}\right)\sum_{m=1}^{\infty}\sum_{n=1}^{\infty}\frac{M_{1-s}\left(4\pi \sqrt{\frac{mnx}{q}}\right)}{m^{(1+s)/2}n^{(1-s)/2}}
\sum_{\chi\operatorname{mod}q}\bar{\chi}(n)(\chi(a)-\chi(-a))\\\nonumber
	&=-\frac{(qx)^{(1-s)/2}}{\phi(q)}\sin\left(\frac{\pi s}{2}\right)\sum_{\substack{\chi\operatorname{mod}q\\\chi\operatorname{odd}}}\chi(a)\sum_{m=1}^{\infty}
\sum_{n=1}^{\infty}\bar{\chi}(n)n^s\frac{M_{1-s}\left(4\pi \sqrt{\frac{mnx}{q}}\right)}{(mn)^{(1+s)/2}}\\\nonumber
	&=-\frac{(qx)^{(1-s)/2}}{\phi(q)}\sin\left(\frac{\pi s}{2}\right)\sum_{\substack{\chi\operatorname{mod}q\\\chi\operatorname{odd}}}\chi(a)\sum_{n=1}^{\infty}\sum_{d\mid n}\bar{\chi}(d)d^s\frac{M_{1-s}\left(4\pi \sqrt{\frac{nx}{q}}\right)}{n^{(1+s)/2}}\\
	&=-\frac{(qx)^{(1-s)/2}}{\phi(q)}\sin\left(\frac{\pi s}{2}\right)\sum_{\substack{\chi\operatorname{mod}q\\\chi\operatorname{odd}}}\chi(a)
\sum_{n=1}^{\infty}\sigma_s(\bar{\chi},n)\frac{M_{1-s}\left(4\pi \sqrt{\frac{nx}{q}}\right)}{n^{(1+s)/2}}.
\end{align*}}
Now, from Lemma \ref{oddcharsin} and Theorem \ref{bdrz05},
\begin{align*}
 L_s(a,q,x)+\sum_{n=1}^{\infty}F\left(\frac{x}{n}\right)\frac{\sin\left(2\pi n a/q\right)}{n^{s}}&=-\frac{ix}{\phi(q)}\sum_{\substack{\chi  \neq \chi_0\operatorname{mod} q\\ \chi\operatorname {even}}}\chi(a)\tau(\bar{\chi})L(1+s,\chi)\\
 &\quad+\frac{i}{2\phi(q)}\sum_{\substack{\chi  \neq \chi_0\operatorname{mod} q\\ \chi\operatorname {even}}}\chi(a)\tau(\bar{\chi})L(s,\chi).
 \end{align*}
 Using the functional equation \eqref{evenfe1} of $L(s,\chi)$ for odd primitive characters, we find that
 \begin{align}\label{lasqsinf}
  L_s(a,q,x)+\sum_{n=1}^{\infty}&F\left(\frac{x}{n}\right)\frac{\sin\left(2\pi
      n
      a/q\right)}{n^{s}}=\frac{x\pi^{s+1/2}\Gamma\left(\frac12(1-s)\right)}{\Gamma\left(\frac12(2+s)\right)} \notag
  \frac{q^{-s}}{\phi(q)}\sum_{\substack{\chi \operatorname{mod} q\\ \chi\operatorname {odd}}}\chi(a)L(-s,\bar{\chi})\\\nonumber
  &\quad-\frac{\pi^{s-1/2}\Gamma\left(\frac12(2-s)\right)}{2\Gamma\left(\frac12(1+s)\right)}\frac{q^{1-s}}{\phi(q)}
  \sum_{\substack{\chi \operatorname{mod} q\\ \chi\operatorname {odd}}}\chi(a)L(1-s,\bar{\chi})\\\nonumber
&=\frac{x\pi^{s+1/2}\Gamma\left(\frac12(1-s)\right)}{2\Gamma\left(\frac12(2+s)\right)}\frac{q^{-s}}{\phi(q)}
\sum_{\substack{\chi \operatorname{mod} q}}(\chi(a)-\chi(q-a))L(-s,\bar{\chi})\\
  &\quad-\frac{\pi^{s-1/2}\Gamma\left(\frac12(2-s)\right)}{4\Gamma\left(\frac12(1+s)\right)}\frac{q^{1-s}}{\phi(q)}
  \sum_{\substack{\chi \operatorname{mod} q}}(\chi(a)-\chi(q-a)L(1-s,\bar{\chi}).
 \end{align}
From \cite[p.~249, Chapter 12]{apostol},
\begin{equation}\label{hurlf}
	L(s,\chi)=q^{-s}\sum_{h=1}^{q}\chi(h)\zeta(s,h/q).
\end{equation}
Multiplying both sides of \eqref{hurlf} by $\bar{\chi}(a)$ and summing over all characters $\chi$ modulo $q$, we deduce that
  \begin{equation}\label{hurwitz}
   	\zeta(s,a/q)=\frac{q^s}{\phi(q)}\sum_{\chi \operatorname{mod} q}\bar{\chi}(a)L(s,\chi),
   \end{equation}
   where $\zeta(s,a)$ denotes the Hurwitz zeta function. Using the duplication formula \eqref{dup} and the reflection formula \eqref{ref} for $\Gamma(s)$, we find that
   \begin{equation}\label{dupref}
   	\frac{\Gamma\left(\frac{1}{2}s\right)}{\Gamma\left(\frac12(1-s)\right)}=\frac{\cos(\tf12\pi s)\Gamma(s)}{2^{s-1}\sqrt{\pi}}.
   \end{equation}
 Utilizing \eqref{hurwitz} and \eqref{dupref} in \eqref{lasqsinf}, we see that
 \begin{align*}
 L_s(a,q,x)+\sum_{n=1}^{\infty}F\left(\frac{x}{n}\right)&\frac{\sin\left(2\pi n a/q\right)}{n^s}
=-x\frac{\sin(\pi s/2)\Gamma(-s)}{(2\pi)^{-s}}\left(\zeta\left(-s,\frac{a}{q}\right)-\zeta\left(-s,1-\frac{a}{q}\right)\right)\\\nonumber
&\quad-\frac{\cos(\pi s/2)\Gamma(1-s)}{2(2\pi)^{1-s}}\left(\zeta\left(1-s,\frac{a}{q}\right)-\zeta\left(1-s,1-\frac{a}{q}\right)\right),
  \end{align*}
  which completes the proof.
\end{proof}

The proof that  Theorem \ref{bdrz01} implies Theorem \ref{bdrz06} is similar
to the proof that Theorem \ref{bdrz02} implies Theorem \ref{bdrz04}, which we give in the next section.
\section{Proof of Theorem \ref{bdrz02}}\label{sect12}
Arguing as in the previous section, for $0<a<q$ and $q$ prime, we can show that Theorem \ref{bdrz02} is equivalent to the following theorem.
\begin{theorem}\label{bdrz03}
Let $q$ be a prime and $0<a<q$. Let
\begin{align}\label{gsaqx}
	&G_s(a,q,x)=\frac{x}{2}\cos\left(\frac{\pi s}{2}\right)\\
&\times\sum_{m=1}^{\infty}\sum_{n=0}^{\infty}\left\{\frac{H_{1-s}\left(4\pi \sqrt{mx\left(n+\frac{a}{q}\right)}\right)}{(mx)^{(1+s)/2}(n+a/q)^{(1-s)/2}}+\frac{H_{1-s}\left(4\pi \sqrt{mx\left(n+1-\frac{a}{q}\right)}\right)}{(mx)^{(1+s)/2}(n+1-a/q)^{(1-s)/2}}\right\},\notag
\end{align}
where $H_{\nu}(z)$ is defined in \eqref{H} and
where we assume that the product of the summation indices $mn$ tends to infinity.  Then, for $ |\s|<\frac{1}{2}$,
\begin{align*}
G_s(a,q,x)+\sum_{n=1}^{\infty}F\left(\frac{x}{n}\right)&\frac{\cos\left(2\pi n a/q\right)}{n^s}
=x\frac{\cos(\tf12\pi s)\Gamma(-s)}{(2\pi)^{-s}}\left(\zeta\left(-s,\frac{a}{q}\right)+\zeta\left(-s,1-\frac{a}{q}\right)\right)\\
&\quad-\frac{\sin(\tf12\pi s)\Gamma(1-s)}{2(2\pi)^{1-s}}\left(\zeta\left(1-s,\frac{a}{q}\right)+\zeta\left(1-s,1-\frac{a}{q}\right)\right).
\end{align*}
\end{theorem}

 We show that Theorem \ref{bdrz03} is equivalent to Theorem \ref{bdrz04}, which is a special case of the following theorem.

\begin{theorem}\label{bdrz004}
If $\chi$ is a non-principal even primitive character modulo $q$, $x>0$,
$|\s|<1/2$, and $k$ is a non-negative integer, then
\begin{align}
&\frac{1}{\Gamma(k+1)}\sideset{}{'}\sum_{n\leq x}\sigma_{-s}(\chi,n)(x-n)^k\nonumber\\
&=\frac{x^{k+1}L(1+s,\chi)}{\Gamma(k+2)}-\frac{x^kL(s,\chi)}{2\Gamma(k+1)}
+2\sum_{n=1}^{\left\lfloor\frac{k+1}{2}\right\rfloor}\frac{(-1)^{n-1}x^{k-2n+1}}{\Gamma(k-2n+2)}\frac{\z(2n)}{(2\pi)^{2n}}
L(1-2n+s,\chi)\nonumber\\
&\quad+\frac{1}{\tau(\bar{\chi})(2\pi)^k}\sum_{n=1}^{\infty}\sigma_{-s}(\bar{\chi},n)
\left(\frac{qx}{n}\right)^{(1-s+k)/2}{G}_{1-s+k}
\left(4\pi\sqrt{\frac{nx}{q}}\right),\notag
\end{align}
where ${G}_{\lambda-s}(z)$ is defined in \eqref{capg}. The series on the right-hand side
converges uniformly on any interval for $x > 0$ where the left-hand side is continuous.
The convergence is bounded on any interval $0 < x_1\leq x\leq x_2<\infty$ when $k=0$.
\end{theorem}

\begin{proof}
From \eqref{twistdivsum} and Lemma \ref{titchma}, for fixed $x>0$, we see that
\begin{align*}
\frac{1}{\Gamma(1+k)}\sideset{}{'}\sum_{n\leq x}\sigma_{-s}(\chi,n)(x-n)^k=\frac{1}{2\pi i}\int_{(c)}\zeta(w)L(w+s,\chi)\frac{\Gamma(w)x^{w+k}}{\Gamma(w+k+1)}dw,
\end{align*}
where $\max\{1,1-\s, \s\}<c$ and $k\geq 0$. Proceeding as we did in the proof of Theorem \ref{bdrz006}, we find that
\begin{align}\label{residue1}
\frac{1}{\Gamma(k+1)}\sideset{}{'}\sum_{n\leq x}\sigma_{-s}(\chi,n)(x-n)^{k}&=\frac{x^{k+1}L(1+s,\chi)}{\Gamma(k+2)}-\frac{L(s,\chi)x^{k}}{2\Gamma(k+1)}\notag\\
&\quad+\frac{1}{2\pi i}\int_{(1-c)}\zeta(w)L(w+s,\chi)\frac{\Gamma(w)x^{w+k}}{\Gamma(w+k+1)}dw,
\end{align}
provided that $k\geq 0$ and $k>2c-\s-2$. The functional equation for $L(2s,\chi)$ for an even primitive Dirichlet character $\chi$  is given by \cite[p.~69]{davenport}
\begin{equation}\label{evenfe}
\left(\df{\pi}{q}\right)^{-s}\Gamma(s)L(2s,\chi)=\df{\tau(\chi)}{\sqrt{q}}\left(\df{\pi}{q}\right)^{-(\tf12-s)}
\Gamma\left(\df12-s\right)L(1-2s,\overline{\chi}),
\end{equation}
where $\tau(\chi)$ is the Gauss sum defined in \eqref{tau}. Combining the functional equation \eqref{fe} of $\zeta(2w)$  and the functional equation \eqref{evenfe} of $L(2w+s,\chi)$ for even primitive $\chi$, we deduce the functional equation
\begin{equation}\label{fezetal}
	\zeta(w)L(w+s,\chi)=\frac{\pi^{2w+s-1}}{\tau(\bar{\chi})q^{w+s-1}}\eta(w,s)\zeta(1-w)L(1-w-s,\bar{\chi}),
\end{equation}
where
\begin{equation*}
	\eta(w,s)=\frac{\Gamma\left(\frac12(1-w)\right)\Gamma\left(\frac12(1-w-s)\right)}{\Gamma\left(\frac12{w}\right)
\Gamma\left(\frac12(w+s)\right)}.
\end{equation*}
Define
\begin{equation}\label{inti111111}
		I(y):=\frac{1}{2\pi i}\int_{(1-c)}\frac{\eta(w,s)\Gamma(w)}{\Gamma(w+k+1)}y^{w} \,dw.
	\end{equation}
Using the functional equation \eqref{fezetal} in the integrand on the right-hand side of \eqref{residue1} and inverting the order of summation and integration, we find that
\begin{align}\label{inti11111}
	&\frac{1}{2\pi i}\int_{(1-c)}\zeta(w)L(w+s,\chi)\frac{\Gamma(w)x^{w+k}}{\Gamma(w+k+1)}\,dw\notag\\
&=\frac{x^k\pi^{s-1}}{\tau(\bar{\chi})q^{s-1}}\frac{1}{2\pi i}\int_{(1-c)}\frac{\eta(w,s)\Gamma(w)}{\Gamma(w+k+1)}\zeta(1-w)L(1-w-s,\bar{\chi})\left(\frac{\pi^2x}{q}\right)^{w} \,dw\notag\\\nonumber
	&=\frac{x^k\pi^{s-1}}{\tau(\bar{\chi})q^{s-1}}\frac{1}{2\pi i}\int_{(1-c)}
\frac{\eta(w,s)\Gamma(w)}{\Gamma(w+k+1)}\left(\frac{\pi^2x}{q}\right)^{w}
\sum_{n=1}^{\infty}\frac{\sigma_{s}(\bar{\chi},n)}{n^{1-w}} \,dw\\\nonumber
	&=\frac{x^k\pi^{s-1}}{\tau(\bar{\chi})q^{s-1}}\sum_{n=1}^{\infty}\frac{\sigma_{s}(\bar{\chi},n)}{n^{1+k}}\frac{1}{2\pi i}\int_{(1-c)}\frac{\eta(w,s)\Gamma(w)}{\Gamma(w+k+1)}\left(\frac{\pi^2 n x}{q}\right)^{w} \,dw\\
	&=\frac{x^k\pi^{s-1}}{\tau(\bar{\chi})q^{s-1}}\sum_{n=1}^{\infty}\frac{\sigma_{s}(\bar{\chi},n)}{n}I\left(\frac{\pi^2 n x}{q}\right),
	\end{align}
provided that $k>2c-\s-1$. We compute the integral  $I(y)$ by using the residue calculus, shifting the line of integration to the right, and letting $c\to -\infty$.

Let $k$ be a positive integer and $\s\neq 0$. By \eqref{inti111111}, we can write
\begin{equation*}
		I(y):=\frac{1}{2\pi i}\int_{(1-c)}F(w)\,dw,
	\end{equation*}
where
\begin{align*}
F(w):=\frac{\Gamma(w)\Gamma\left(\frac12(1-w)\right)\Gamma\left(\frac12(1-w-s)\right)y^{w}}{\Gamma(1+k+w)\Gamma\left(\frac12{w}\right)
\Gamma\left(\frac12(w+s)\right)}.
\end{align*}
Note that the poles of the function $F(w)$  on the right side of the line $1-c+it, -\infty<t<\infty$,
are at $w=2m+1$ and $w=2m+1-s$, $m=0, 1, 2, \dots$.  Calculating the residues, we find that
\begin{align*}
R_{2m+1}(F(w))&=(-1)^{m+1}\frac{2\Gamma(2m+1)\Gamma\left(-m-\frac12s\right)y^{2m+1}}{m!\Gamma(2+k+2m)\Gamma\left(m+\frac12\right)
\Gamma\left(m+\frac12(s+1)\right)}
\end{align*}
and
\begin{align*}
R_{2m+1-s}(F(w))&=(-1)^{m+1}\frac{2\Gamma(2m+1-s)\Gamma\left(-m+\frac12s\right)y^{2m+1-s}}{m!\Gamma(2+k+2m-s)
\Gamma\left(m+\frac12(1-s)\right)
\Gamma\left(m+\frac12\right)}.
\end{align*}
 With the aid of the duplication formula \eqref{dup} and the reflection formula \eqref{ref}, we find that
\begin{align}\label{r2mF}
R_{2m+1}(F(w))&=\frac{2^{s-1}}{\sin(\pi s/2)}\frac{(2\sqrt{y})^{4m+2}}{(2m+k+1)!
\Gamma(2m+1+s)}
\end{align}
and
\begin{align}\label{r2msF}
R_{2m+1-s}(F(w))&=-\frac{(2y)^{1-s}}{\sin(\pi s/2)}\frac{(2\sqrt{y})^{4m}}{(2m)!
\Gamma(2m+k+2-s)}.
\end{align}
Consequently, from \eqref{sumbesselj}, \eqref{sumbesseli}, and \eqref{r2mF}, for $k$ even,
\begin{align}\label{sumr2mF}
\sum_{m=0}^{\infty}R_{2m+1}(F(w))&=\frac{2^{s-1-2k}y^{-k}}{\sin(\pi s/2)}\sum_{m=0}^{\infty}\frac{(2\sqrt{y})^{4m+2k+2}}{(2m+k+1)!
\Gamma(2m+1+s)}\nonumber\\
&=\frac{2^{s-1-2k}y^{-k}}{\sin(\pi s/2)}\bigg\{\sum_{m=0}^{\infty}\frac{(2\sqrt{y})^{4m+2}}{(2m+1)!
\Gamma(2m+1+s-k)}\nonumber\\
&\quad -\sum_{m=1}^{k/2}\frac{(2\sqrt{y})^{4m-2}}{(2m-1)!
\Gamma(2m-1+s-k)}\bigg\}\nonumber\\
&=\frac{2^{-1-k}y^{(1-s-k)/2}}{\sin(\pi s/2)}(I_{-1+s-k}(4\sqrt{y})-J_{-1+s-k}(4\sqrt{y}))\nonumber\\
&\quad-\frac{2^{s+1}}{\sin(\pi s/2)}\sum_{m=1}^{k/2}\frac{2^{-4m}y^{1-2n}}{\Gamma(k-2m+2)
\Gamma(1-2m+s)}.
\end{align}
For each odd integer $k$,
\begin{align}\label{sumr2mFo}
\sum_{m=0}^{\infty}R_{2m+1}(F(w))&=\frac{2^{-1-k}y^{(1-s-k)/2}}{\sin(\pi s/2)}(I_{-1+s-k}(4\sqrt{y})+J_{-1+s-k}(4\sqrt{y}))\nonumber\\
&\quad-\frac{2^{s+1}}{\sin(\pi s/2)}\sum_{m=1}^{(k+1)/2}\frac{2^{-4m}y^{1-2n}}{\Gamma(k-2m+2)
\Gamma(1-2m+s)}.
\end{align}
Similarly, from \eqref{r2msF}, we find that
\begin{align}\label{sumr2msF}
\sum_{m=0}^{\infty}R_{2m+1-s}(F(w))&=-\frac{2^{-1-k}y^{(1-s-k)/2}}{\sin(\pi s/2)}(J_{1-s+k}(4\sqrt{y})+I_{1-s+k}(4\sqrt{y})).
\end{align}
Utilizing \eqref{sumbesselk} in the sum of \eqref{sumr2mF}, \eqref{sumr2mFo}, and \eqref{sumr2msF}, we deduce that
\begin{align}\label{allres}
&\sum_{m=0}^{\infty}\left(R_{2m+1}(F(w))+R_{2m+1-s}(F(w))\right)
=-\frac{\cos(\pi s/2)}{2^{k}y^{(-1+s+k)/2}}\notag\\&\quad\times\left(\frac{J_{1-s+k}(4\sqrt{y})-(-1)^{k+1}J_{-1+s-k}(4\sqrt{y})}{\sin{\pi s}}-(-1)^{k+1}\frac{2}{\pi}K_{1-s+k}(4\sqrt{y})\right)\notag\\
&\quad-\frac{2^{s+1}}{\sin(\pi s/2)}\sum_{m=1}^{\left\lfloor\frac{k+1}{2}\right\rfloor}\frac{2^{-4m}y^{1-2n}}{\Gamma(k-2m+2)
\Gamma(1-2m+s)}.
 \end{align}
Using \eqref{yj}, we can show that
\begin{multline}\label{grel}
\cos(\pi s/2)\left(\frac{J_{1-s+k}(4\sqrt{y})-(-1)^{k+1}J_{-1+s-k}(4\sqrt{y})}{\sin{\pi s}}-(-1)^{k+1}\frac{2}{\pi}K_{1-s+k}(4\sqrt{y})\right)\\
={G}_{1+k-s}(4\sqrt{y}).
\end{multline}
Consider the positively oriented contour $\mathcal{R}_N$ formed by the points $\{1-c-iT,2
 N+\tf32-iT,2 N+\tf32+iT,1-c+iT\}$, where $T>0$ and $N$ is a positive
 integer. By the residue theorem,
 \begin{align*}
 \frac{1}{2\pi i}\int_{\mathcal{R}_N}F(w)\, dw
=\sum_{0\leq k\leq N}R_{2k+1}(F(w))+
 \sum_{0\leq k\leq N}R_{2k+1-s}(F(w)).
 \end{align*}
 By \eqref{stirling'sformula}, for the integral over the right side of the rectangular contour $\mathcal{R}_N$,
 \begin{equation*}
\int_{2N+3/2-iT}^{2N+3/2+iT}F(w)\,  dw\ll_{T,s} y^{2N+3/2}e^{4N-(4N+2+k+\s)\log N}=o(1),
 \end{equation*}
 as $N\to \infty$. Using Stirling's formula \eqref{strivert} to estimate the integrals over  the horizontal sides of $\mathcal{R}_N$, we find that
\begin{equation*}
\int_{1-c\pm iT}^{\infty\pm iT}F(w)\, dw\ll_{s}\int_{1-c}^{\infty}y^{\sigma}T^{-2\b-\s-k}\, d\sigma
\ll_{s,y}\frac{y^{1-c}}{ T^{2c-\s-k-2}\log T}=o(1),
\end{equation*}
provided that $k>2c-\s-2$. Combining \eqref{residue1}, \eqref{inti11111}, \eqref{allres}, and \eqref{grel}, we conclude that
\begin{align}\label{klargeint}
&\frac{1}{\Gamma(k+1)}\sideset{}{'}\sum_{n\leq x}\sigma_{-s}(\chi,n)(x-n)^k\nonumber\\
&=\frac{x^{k+1}L(1+s,\chi)}{\Gamma(k+2)}-\frac{x^kL(s,\chi)}{2\Gamma(k+1)}
+2\sum_{n=1}^{\left\lfloor\frac{k+1}{2}\right\rfloor}\frac{(-1)^{n-1}x^{k-2n+1}}{\Gamma(k-2n+2)}
\frac{\z(2n)}{(2\pi)^{2n}}L(1-2n+s,\chi)\nonumber\\
&\qquad+\frac{1}{\tau(\bar{\chi})(2\pi)^k}
\sum_{n=1}^{\infty}\sigma_{-s}(\bar{\chi},n)\left(\frac{qx}{n}\right)^{(1-s+k)/2}{G}_{1-s+k}
\left(4\pi\sqrt{\frac{nx}{q}}\right),
\end{align}
provided that $k\geq 0$, $\s\neq 0$, and $k>2c-\s-1$.  By the asymptotic expansions for Bessel functions \eqref{asymbess}, \eqref{asymbess1}, and \eqref{asymbess2}, Lemma \ref{fnuequi},  \eqref{gtransform}, and an argument like that in the proof in Theorem \ref{bdrz006}, we deduce the identity \eqref{klargeint} for $k>|\s|-\tf12$, with $|\s|<\tf12$.  Thus, we complete the proof of  Theorem \ref{bdrz02}.
\end{proof}

From the definition \eqref{H} and \eqref{gmrel}, we find that $\cos(\pi s/2)M_{1-s}(z)=G_{1-s}(z)$. The case $k=0$ of Theorem \ref{bdrz004} provides the following corollary.

\begin{corollary}\label{bdrz04}
If $\chi$ is a non-principal even primitive character modulo $q$, $x>0$,
and $|\s|<1/2$, then
\begin{align}
\sideset{}{'}\sum_{n\leq x}\sigma_{-s}(\chi,n)&=xL(1+s,\chi)-\tf12 L(s,\chi)\nonumber\\
&\quad+\frac{\cos{(\pi s)/2}}{\tau(\bar{\chi})}\sum_{n=1}^{\infty}\sigma_{-s}(\bar{\chi},n)\left(\frac{qx}{n}\right)^{\frac{1-s}{2}}H_{1-s}
\left(4\pi\sqrt{\frac{nx}{q}}\right),\notag
\end{align}
where $H_{1-s}(z)$ is defined in \eqref{H}.
\end{corollary}

 Next we show that Theorem \ref{bdrz04} implies Theorem \ref{bdrz03}.

\begin{proof}
First we write \eqref{gsaqx} as a sum over Dirichlet characters. To that end, for any prime $q$ and $0<a<q$,
{\allowdisplaybreaks\begin{align}\label{gsaqchar}
	&G_s(a,q,x)=\frac{x}{2}\cos\left(\frac{\pi s}{2}\right)\notag\\
&\quad\times\sum_{m=1}^{\infty}\sum_{n=0}^{\infty}\left\{\frac{H_{1-s}\left(4\pi \sqrt{mx\left(n+\frac{a}{q}\right)}\right)}{(mx)^{(1+s)/2}(n+a/q)^{(1-s)/2}}+\frac{H_{1-s}\left(4\pi \sqrt{mx\left(n+1-\frac{a}{q}\right)}\right)}{(mx)^{(1+s)/2}(n+1-a/q)^{(1-s)/2}}\right\}\notag\\\nonumber
	&=\frac{x}{2}\cos\left(\frac{\pi s}{2}\right)\sum_{m=1}^{\infty}\sum_{\substack{n=1\\n\equiv\pm a\operatorname{mod}q}}^{\infty}\frac{H_{1-s}\left(4\pi \sqrt{\frac{mnx}{q}}\right)}{(mx)^{(1+s)/2}(n/q)^{(1-s)/2}}\\\nonumber
	&=\frac{(qx)^{(1-s)/2}}{2\phi(q)}\cos\left(\frac{\pi s}{2}\right)\sum_{m=1}^{\infty}\sum_{n=1}^{\infty}\frac{H_{1-s}\left(4\pi \sqrt{\frac{mnx}{q}}\right)}{m^{(1+s)/2}n^{(1-s)/2}}\sum_{\chi\operatorname{mod}q}\bar{\chi}(n)(\chi(a)+\chi(-a))
\\\nonumber
	&=\frac{(qx)^{(1-s)/2}}{\phi(q)}\cos\left(\frac{\pi s}{2}\right)\sum_{\substack{\chi\operatorname{mod}q\\\chi\operatorname{even}}}\chi(a)
\sum_{m=1}^{\infty}\sum_{n=1}^{\infty}\bar{\chi}(n)n^s\frac{H_{1-s}\left(4\pi \sqrt{\frac{mnx}{q}}\right)}{(mn)^{(1+s)/2}}\\\nonumber
	&=\frac{(qx)^{(1-s)/2}}{\phi(q)}\cos\left(\frac{\pi s}{2}\right)\sum_{\substack{\chi\operatorname{mod}q\\\chi\operatorname{even}}}\chi(a)\sum_{n=1}^{\infty}\sum_{d\mid n}\bar{\chi}(d)d^s\frac{H_{1-s}\left(4\pi \sqrt{\frac{nx}{q}}\right)}{n^{(1+s)/2}}\\
	&=\frac{(qx)^{(1-s)/2}}{\phi(q)}\cos\left(\frac{\pi s}{2}\right)\sum_{\substack{\chi\operatorname{mod}q\\\chi\operatorname{even}}}\chi(a)
\sum_{n=1}^{\infty}\sigma_s(\bar{\chi},n)\frac{H_{1-s}\left(4\pi \sqrt{\frac{nx}{q}}\right)}{n^{(1+s)/2}},
\end{align}}%
where in the penultimate step we recall our assumption that the double series converges in the sense that the product of the indices $mn$ tends to infinity.  For the principal character $\chi_0$,
{\allowdisplaybreaks\begin{align}\label{prichar}
\sum_{n=1}^{\infty}\sigma_s(\chi_0,n)&\frac{H_{1-s}\left(4\pi \sqrt{\frac{nx}{q}}\right)}{n^{(1+s)/2}}=\sum_{m=1}^{\infty}\sum_{n=1}^{\infty}\chi_0(n)n^s\frac{H_{1-s}\left(4\pi \sqrt{\frac{mnx}{q}}\right)}{(mn)^{(1+s)/2}}\\\nonumber
&=\sum_{m=1}^{\infty}\sum_{\substack{n=1\\q\nmid n}}^{\infty}n^s\frac{H_{1-s}\left(4\pi \sqrt{\frac{mnx}{q}}\right)}{(mn)^{(1+s)/2}}\\\nonumber
&=\sum_{m=1}^{\infty}\sum_{n=1}^{\infty}n^s\frac{H_{1-s}\left(4\pi \sqrt{\frac{mnx}{q}}\right)}{(mn)^{(1+s)/2}}-q^{(s-1)/2}\sum_{m=1}^{\infty}
\sum_{n=1}^{\infty}n^s\frac{H_{1-s}\left(4\pi \sqrt{mnx}\right)}{(mn)^{(1+s)/2}}\\\nonumber
&=\sum_{n=1}^{\infty}\sigma_s(n)\frac{H_{1-s}\left(4\pi \sqrt{\frac{nx}{q}}\right)}{n^{(1+s)/2}}-q^{(s-1)/2}\sum_{n=1}^{\infty}\sigma_s(n)\frac{H_{1-s}\left(4\pi \sqrt{nx}\right)}{n^{(1+s)/2}}.
\end{align}}%
Combining \eqref{gsaqchar} and \eqref{prichar} and applying Lemma \ref{ldivsum}, we find that
{\allowdisplaybreaks\begin{align}\label{gasqevenchar1}
\notag &G_s(a,q,x)=\frac{(qx)^{(1-s)/2}}{\phi(q)}\cos\left(\frac{\pi s}{2}\right)\sum_{\substack{\chi\neq\chi_0\operatorname{mod}q\\\chi\operatorname{even}}}\chi(a)
\sum_{n=1}^{\infty}\sigma_s(\bar{\chi},n)\frac{H_{1-s}\left(4\pi \sqrt{\frac{nx}{q}}\right)}{n^{(1+s)/2}}\\\nonumber
&\quad+\frac{1}{\phi(q)}\left(\sideset{}{'}\sum_{n\leq x}\sigma_{-s}(n)-xZ(s,x)+\frac{1}{2}\zeta(s)\right)\notag\\&\quad-\frac{q^{1-s}}{\phi(q)}\left(\sideset{}{'}\sum_{n\leq x/q}\sigma_{-s}(n)-\frac{x}{q}Z(s,x/q)+\frac{1}{2}\zeta(s)
\right)\notag\\\nonumber
&=\frac{(qx)^{\frac{1-s}{2}}}{\phi(q)}\cos\left(\frac{\pi s}{2}\right)\sum_{\substack{\chi\neq\chi_0\operatorname{mod}q\\\chi\operatorname{even}}}\chi(a)
\sum_{n=1}^{\infty}\sigma_s(\bar{\chi},n)\frac{H_{1-s}\left(4\pi \sqrt{\frac{nx}{q}}\right)}{n^{(1+s)/2}}\\\nonumber
&\quad+\frac{1}{\phi(q)}\sideset{}{'}\sum_{n\leq x}\sigma_{-s}(n)-\frac{q^{1-s}}{\phi(q)}\sideset{}{'}\sum_{n\leq x/q}\sigma_{-s}(n)\notag\\&\quad+\frac{x}{\phi(q)q^s}\zeta(1+s)\left(1-\frac{1}{q^{-s}}\right)
-\frac{\zeta(s)}{2\phi(q)q^{s-1}}\left(1-\frac{1}{q^{1-s}}\right).
\end{align}}%
For each prime $q$, by Lemma \ref{evencahrcos},
 \begin{align}\label{evencahrcos1}
 \notag\sum_{n=1}^{\infty}F\left(\frac{x}{n}\right)
\frac{\cos\left(2\pi n a/q\right)}{n^{s}}&=q^{-s}\sideset{}{'}\sum_{1\leq n\leq x/q}\sigma_{-s}(n)+\frac{1}{\phi(q)}\sum_{\substack{\chi \operatorname{mod} q\\ \chi \operatorname {even}}}\chi(a)\tau(\bar{\chi})\sideset{}{'}\sum_{1\leq n\leq x}\sigma_{-s}(\chi,n)\\\nonumber
 &=q^{-s}\sideset{}{'}\sum_{1\leq n\leq x/q}\sigma_{-s}(n)-\frac{1}{\phi(q)}\sideset{}{'}\sum_{1\leq n\leq x}\sigma_{-s}(\chi_0,n)\\
 &\quad+\frac{1}{\phi(q)}\sum_{\substack{\chi  \neq \chi_0\operatorname{mod} q\\ \chi\operatorname {even}}}\chi(a)\tau(\bar{\chi})\sideset{}{'}\sum_{1\leq n\leq x}\sigma_{-s}(\chi,n).
 \end{align}
 Now,
 \begin{align}\label{divsumprchar1}
 \notag\sideset{}{'}\sum_{1\leq n\leq x}\sigma_{-s}(\chi_0,n)=\sideset{}{'}\sum_{1\leq n\leq x}\sum_{\substack{d\mid n\\q\nmid d}}d^{-s}&=\sideset{}{'}\sum_{1\leq n\leq x}\sum_{d\mid n}d^{-s}-q^{-s}\sideset{}{'}\sum_{1\leq n\leq x/q}\sum_{d\mid n}d^{-s}\\
 &=\sideset{}{'}\sum_{1\leq n\leq x}\sigma_{-s}(n)-q^{-s}\sideset{}{'}\sum_{1\leq n\leq x/q}\sigma_{-s}(n).
 \end{align}
 Substituting \eqref{divsumprchar1} into \eqref{evencahrcos1}, we find that
 \begin{align}\label{evencahrcos2}
 \sum_{n=1}^{\infty}F\left(\frac{x}{n}\right)\frac{\cos\left(2\pi n a/q\right)}{n^{s}}
 &=\frac{q^{1-s}}{\phi(q)}\sideset{}{'}\sum_{1\leq n\leq x/q}\sigma_{-s}(n)-\frac{1}{\phi(q)}\sideset{}{'}\sum_{1\leq n\leq x}\sigma_{-s}(n)\\\nonumber
 &\quad+\frac{1}{\phi(q)}\sum_{\substack{\chi  \neq \chi_0\operatorname{mod} q\\ \chi\operatorname {even}}}\chi(a)\tau(\bar{\chi})\sideset{}{'}\sum_{1\leq n\leq x}\sigma_{-s}(\chi,n).
 \end{align}
 Adding \eqref{gasqevenchar1} and \eqref{evencahrcos2} and using Theorem \ref{bdrz04}, we find that
 \begin{align}\label{gasqfcos}
 G_s(a,q,x)&+\sum_{n=1}^{\infty}F\left(\frac{x}{n}\right)\frac{\cos\left(2\pi na/q\right)}{n^{s}}=\frac{x}{\phi(q)q^s}\zeta(1+s)\left(1-\frac{1}{q^{-s}}\right)\notag\\
 &-\frac{\zeta(s)}{2\phi(q)q^{s-1}}\left(1-\frac{1}{q^{1-s}}\right)
 +\frac{x}{\phi(q)}\sum_{\substack{\chi  \neq \chi_0\operatorname{mod} q\\ \chi\operatorname {even}}}\chi(a)\tau(\bar{\chi})L(1+s,\chi)\notag\\
 &-\frac{1}{2\phi(q)}\sum_{\substack{\chi  \neq \chi_0\operatorname{mod} q\\ \chi\operatorname {even}}}\chi(a)\tau(\bar{\chi})L(s,\chi).
 \end{align}
 Recall that if $\chi_0$ is the principal character modulo the prime $q$, then
 \begin{equation}\label{princharl}
  	L(s,\chi_0)=\zeta(s)\left(1-\frac{1}{q^s}\right).
  \end{equation}
  Using the functional equations of $\zeta(s)$  and $L(s,\chi)$ for even primitive Dirichlet characters,  \eqref{fe} and \eqref{evenfe}, respectively, and \eqref{princharl}, we find from \eqref{gasqfcos} that
  \begin{align}\label{gasqfcos1}
  G_s&(a,q,x)+\sum_{n=1}^{\infty}F\left(\frac{x}{n}\right)\frac{\cos\left(2\pi n a/q\right)}{n^{s}}=\frac{x\pi^{s+1/2}\Gamma\left(-\frac{1}{2}s\right)}{\Gamma\left(\frac12(1+s)\right)}
  \frac{q^{-s}}{\phi(q)}\sum_{\substack{\chi \operatorname{mod} q\\ \chi\operatorname {even}}}\chi(a)L(-s,\bar{\chi})\notag\\\nonumber
  &\quad-\frac{\pi^{s-1/2}\Gamma\left(\frac12(1-s)\right)}{2\Gamma\left(\frac{1}{2}s\right)}\frac{q^{1-s}}{\phi(q)}
  \sum_{\substack{\chi \operatorname{mod} q\\ \chi\operatorname {even}}}\chi(a)L(1-s,\bar{\chi})\\\nonumber
&=\frac{x\pi^{s+1/2}\Gamma\left(-\frac{1}{2}s\right)}{2\Gamma\left(\frac12(1+s)\right)}\frac{q^{-s}}{\phi(q)}
\sum_{\substack{\chi \operatorname{mod} q}}(\chi(a)+\chi(q-a))L(-s,\bar{\chi})\\
  &\quad-\frac{\pi^{s-1/2}\Gamma\left(\frac12(1-s)\right)}{4\Gamma\left(\frac{1}{2}s\right)}\frac{q^{1-s}}{\phi(q)}
  \sum_{\substack{\chi \operatorname{mod} q}}(\chi(a)+\chi(q-a))L(1-s,\bar{\chi}).
  \end{align}
We complete the proof of Theorem \ref{bdrz03} by using \eqref{hurwitz} and \eqref{dupref} in \eqref{gasqfcos1}.
\end{proof}

Next we prove that Theorem \ref{bdrz02} implies Theorem \ref{bdrz04}.

\begin{proof}
Let $\chi$ be an even primitive character modulo $q$. Set $\theta=h/q$, where $1\leq h <q$. The Gauss sum $\tau(n,\chi)$ is defined by
\begin{equation*}
\tau(n,\chi)=\sum_{m=1}^{q}\chi(m)e^{2\pi i m n/q}.
\end{equation*}
Note that $\tau(1,\chi)=\tau(\chi)$, which is defined in \eqref{tau}. For any  character $\chi$ \cite[p.~165, Theorem 8.9]{apostol}
\begin{equation*}
	\tau(n,\chi)=\bar{\chi}(n)\tau(\chi).
\end{equation*}
Multiplying both sides of \eqref{ebdrz02} by $\bar{\chi}(h)/\tau(\bar{\chi})$
and summing over $h$, $1\leq h<q$, we find that the left-hand side yields
\begin{align}\label{lehasi}
\notag\frac{1}{\tau(\bar{\chi})}\sum_{h=1}^{q-1}\bar{\chi}(h)\sum_{n=1}^{\infty}&F\left(\frac{x}{n}\right)\frac{\cos\left(2\pi n h/q\right)}{n^{s}}=\frac{1}{\tau(\bar{\chi})}\sum_{n=1}^{\infty}\frac{F\left(\frac{x}{n}\right)}{n^{s}}
\sum_{h=1}^{q-1}\bar{\chi}(h)\cos\left(\frac{2\pi n h}{q}\right)\\\nonumber
&=\frac{1}{2\tau(\bar{\chi})}\sum_{n=1}^{\infty}\frac{F\left(\frac{x}{n}\right)}{n^{s}}
\sum_{h=1}^{q-1}\bar{\chi}(h)\left(e^{2\pi in h/q}+e^{-2\pi in h/q}\right)\\\nonumber
&=\frac{1}{2\tau(\bar{\chi})}\sum_{n=1}^{\infty}\frac{F\left(\frac{x}{n}\right)}{n^{s}}\tau(\bar{\chi})(\chi(n)+\chi(-n))
\\
&=\sideset{}{'}\sum_{n\leq x}\sigma_{-s}(\chi,n).
\end{align}
On the other hand, summing over $h$, $1\leq h\leq q$, on the right-hand side of \eqref{ebdrz02} gives
{\allowdisplaybreaks\begin{align}\label{rihasi}
\notag{}&\frac{x}{2\tau(\bar{\chi})}\cos\left(\frac{\pi s}{2}\right)\sum_{h=1}^{q-1}\sum_{m=1}^{\infty}\sum_{\substack{n=1\\n\equiv\pm h\operatorname{mod}q}}^{\infty}\bar{\chi}(h)\frac{H_{1-s}\left(4\pi \sqrt{\frac{mnx}{q}}\right)}{(mx)^{(1+s)/2}(n/q)^{(1-s)/2}}\\\nonumber
&=\frac{x}{2\tau(\bar{\chi})}\cos\left(\frac{\pi s}{2}\right)\sum_{m=1}^{\infty}\sum_{\substack{n=1}}^{\infty}\frac{H_{1-s}\left(4\pi \sqrt{\frac{mnx}{q}}\right)}{(mx)^{(1+s)/2}(n/q)^{(1-s)/2}}\sum_{\substack{h=1\\h\equiv\pm n\operatorname{mod}q}}^{q-1}\bar{\chi}(h)\\\nonumber
&=\frac{x}{\tau(\bar{\chi})}\cos\left(\frac{\pi s}{2}\right)\sum_{m=1}^{\infty}\sum_{\substack{n=1}}^{\infty}\bar{\chi}(n)\frac{H_{1-s}\left(4\pi \sqrt{\frac{mnx}{q}}\right)}{(mx)^{(1+s)/2}(n/q)^{(1-s)/2}}\\
&=\frac{(qx)^{(1-s)/2}\cos(\tf12\pi s)}{\tau(\bar{\chi})}\sum_{n=1}^{\infty}\frac{\sigma_{s}(\bar{\chi},n)}{n^{(1+s)/2}}H_{1-s}
\left(4\pi\sqrt{\frac{nx}{q}}\right).
\end{align}}%
Combining \eqref{lehasi}, \eqref{rihasi}, and \eqref{hurlf} with the functional equation \eqref{evenfe} of $L(s,\chi)$ for even primitive $\chi$, we obtain the equality in \eqref{ebdrz02}, which completes the proof of Theorem \ref{bdrz04}.
\end{proof}

\section{Koshliakov Transforms and Modular-type Transformations}
\label{ktmt}

In Section \ref{sect5}, we studied a generalization of the Vorono\"{\dotlessi} summation formula, namely \eqref{varlauform1}, the right-hand side of which consists of summing infinitely many integrals involving the kernel
\begin{equation}\label{kernel}
\left(\frac{2}{\pi}K_{s}(4\pi\sqrt{nt})-Y_{s}(4\pi\sqrt{nt})\right)\cos\left(\frac{\pi s}{2}\right)-J_{s}(4\pi\sqrt{n t})\sin\left(\frac{\pi s}{2}\right).
\end{equation}
Koshliakov \cite{kosh1938} remarkably found that for $-\tfrac{1}{2}<\nu<\tfrac{1}{2}$, the modified Bessel function $K_{\nu}(x)$ is self-reciprocal with respect to this kernel, i.e.,
\begin{equation}
\int_{0}^{\infty} K_{\nu}(t) \left( \cos(\nu \pi) \tilde{M}_{2 \nu}(2 \sqrt{xt}) -
\sin(\nu \pi) J_{2 \nu}(2 \sqrt{xt}) \right)\, dt = K_{\nu}(x).\label{koshlyakov-1}
\end{equation}
He also showed that for the same values of $\nu$, $xK_{\nu}(x)$ is self-reciprocal with respect to  the companion kernel $\sin(\nu \pi) J_{2 \nu}(2 \sqrt{xt}) -
\cos(\nu \pi) L_{2 \nu}(2 \sqrt{xt})$, i.e.,
\begin{equation}
\int_{0}^{\infty} tK_{\nu}(t) \left( \sin(\nu \pi) J_{2 \nu}(2 \sqrt{xt}) -
\cos(\nu \pi) L_{2 \nu}(2 \sqrt{xt}) \right)\, dt = xK_{\nu}(x).\label{koshlyakov-2}
\end{equation}
Here
\begin{equation*}
\tilde{M}_{\nu}(x) := \frac{2}{\pi} K_{\nu}(x) - Y_{\nu}(x) \quad\text{ and } \quad
L_{\nu}(x) := - \frac{2}{\pi} K_{\nu}(x) - Y_{\nu}(x),
\end{equation*}
It is easy to see that these
identities actually hold for complex $\nu$ with $ - \tfrac{1}{2}<$ Re $\nu<\tfrac{1}{2}$. It
must be mentioned here that the special case $z=0$ of \eqref{koshlyakov-1} was obtained by Dixon and
Ferrar \cite[p.~164, equation (4.1)]{dixfer3}.

Motivated by these results of Koshliakov, we begin with some definitions.

\begin{definition} Let $f(t,\nu)$ be a function analytic in the real variable $t$ and in the complex variable $\nu$. Then, we define the \textit{first Koshliakov transform} of a function $f(t,\nu)$ to be
\begin{equation*}
\int_{0}^{\infty} f(t,\nu) \left( \cos(\nu \pi) \tilde{M}_{2 \nu}(2 \sqrt{xt}) -
\sin(\nu \pi) J_{2 \nu}(2 \sqrt{xt}) \right)\, dt,
\end{equation*}
and the \textit{second Koshliakov transform} of a function $f(t,\nu)$ to be
\begin{equation*}
\int_{0}^{\infty} f(t,\nu) \left( \sin(\nu \pi) J_{2 \nu}(2 \sqrt{xt}) -
\cos(\nu \pi) L_{2 \nu}(2 \sqrt{xt}) \right)\, dt,
\end{equation*}
provided, of course, that the integrals converge.
\end{definition}
\textbf{Remark.} We note here that the first Koshliakov transform is the integral transform that arises naturally when one considers a function corresponding to the functional equation of an even Maass form in conjunction with a summation formula of Ferrar; see for example, the work of J.~Lewis and D.~Zagier \cite[p.~216--217]{lewzag}\footnote{The kernel $F_{s}(\xi)$ in \cite[p.~217]{lewzag} is equal to $2\pi$ times the kernel in \eqref{kernel} with $s$ replaced by $2s-1$. This is immediately seen by an application of \eqref{yj}.}.

The following two theorems, obtained in \cite[Theorems 5.3, 5.5]{dixitmoll}, give the necessary conditions for functions to equal their first or second Koshliakov transforms. The corollaries resulting from them \cite[Corollaries 5.4 and 5.6]{dixitmoll} give associated modular transformations. These results are stated below.
\begin{theorem}
\label{production-1}
Assume $\pm\frac{1}{2}\sigma < c = \textup{Re } z < 1 \pm
\frac{1}{2}\sigma$. Define $f(x, s)$ by
\begin{equation*}
f(x,s) = \frac{1}{2 \pi i} \int_{c - i \infty}^{c + i \infty}
x^{-z} F(z,s) \zeta(1 - z -s/2) \zeta(1-z + s/2) \, dz,
\end{equation*}
where $F(z,s)$ is a function satisfying $F(z,s) = F(1-z,s)$ and is such that the integral above converges.
Then $f$ is self-reciprocal (as a function of $x$) with respect to the kernel
\begin{equation*}
2 \pi \left( \cos \left( \tfrac{1}{2} \pi s \right) \tilde{M}_{s}( 4 \pi \sqrt{xy}) -
\sin \left( \tfrac{1}{2} \pi s \right) J_{s}( 4 \pi \sqrt{xy}) \right),
\end{equation*}
that is,
\begin{equation*}
f(y,s) = 2 \pi \int_{0}^{\infty} f(x,s)
\left[ \cos \left( \tfrac{1}{2} \pi s \right) \tilde{M}_{s}( 4 \pi \sqrt{xy}) -
\sin \left( \tfrac{1}{2} \pi s \right) J_{s}( 4 \pi \sqrt{xy}) \right]\, dx.
\end{equation*}
\end{theorem}





\begin{corollary}\label{cor5.4}
Let $f(x,s)$ be as in the previous theorem. Then, if $\alpha, \, \beta > 0$
and $\alpha \beta = 1$, and if $-1< \sigma <1$,
\begin{equation}\label{cor5.4result}
\sqrt{\alpha} \int_{0}^{\infty} K_{s/2}(2 \pi \alpha x) f(x,s)\, dx =
\sqrt{\beta} \int_{0}^{\infty} K_{s/2}(2 \pi \beta x) f(x,s)\, dx.
\end{equation}
\end{corollary}

\begin{theorem}
\label{production-2}
Assume $\pm\frac{1}{2}\sigma < c = \textup{Re } z < 1 \pm
\frac{1}{2}\sigma$. Define $f(x, s)$ by
\begin{equation}\label{fxzprod2}
f(x,s) = \frac{1}{2 \pi i} \int_{c - i \infty}^{c + i \infty}
\frac{F(z,s)}{(2 \pi)^{2z}}
\Gamma \left( z - \frac{1}{2}s \right)
\Gamma \left( z + \frac{1}{2}s \right)
\zeta\left(z -\frac{s}{2}\right) \zeta\left(z + \frac{s}{2}\right) x^{-z} \, dz,
\end{equation}
where $F(z,s)$ is a function satisfying $F(z,s) = F(1-z,s)$ and is such that the integral above converges.
Then $f$ is self-reciprocal (as a function of $x$) with respect to the kernel
\begin{equation*}
2 \pi \left( \sin \left( \tfrac{1}{2} \pi s \right) J_{s}( 4 \pi \sqrt{xy}) -
\cos \left( \tfrac{1}{2} \pi s \right) L_{s}( 4 \pi \sqrt{xy}) \right),
\end{equation*}
that is,
\begin{equation*}
f(y,s) = 2 \pi \int_{0}^{\infty} f(x,s)
\left[ \sin \left( \tfrac{1}{2} \pi s \right) J_{s}( 4 \pi \sqrt{xy}) -
\cos \left( \tfrac{1}{2} \pi s \right) L_{s}( 4 \pi \sqrt{xy}) \right]\, dx.
\end{equation*}
\end{theorem}

\begin{corollary}\label{cor5.6}
Let $f(x,s)$ be as in the previous theorem. Then, if $\alpha, \, \beta > 0$
and $\alpha \beta = 1$, and $-1< \sigma<1$,
\begin{equation}\label{cor5.6result}
\alpha^{3/2} \int_{0}^{\infty} x K_{s/2}(2 \pi \alpha x) f(x,s)\, dx =
\beta^{3/2} \int_{0}^{\infty} x K_{s/2}(2 \pi \beta x) f(x,s)\, dx.
\end{equation}
\end{corollary}

The identity in \eqref{koshlyakov-1} can be proved using Theorem \ref{production-1} by taking
\begin{equation*}
F(z,s)=\frac{\pi^{-z}}{4}\frac{\G\left(\frac{1}{2}z-\frac{1}{4}s\right)
\G\left(\frac{1}{2}z+\frac{1}{4}s\right)}{\zeta\left(1-z-\frac{1}{2}s\right)\zeta\left(1-z+\frac{1}{2}s\right)}
\end{equation*}
and then using the fact \cite[p.~115, formula 11.1]{ob} that for $c=$ Re
$z>\pm$ Re $\nu$ and $a>0$,
\begin{equation}\label{melkz}
\frac{1}{2\pi i}\int_{(c)}2^{z-2}a^{-z}\G\left(\frac{z}{2}-\frac{\nu}{2}\right)\G\left(\frac{z}{2}+\frac{\nu}{2}\right)x^{-z}\, dz=K_{\nu}(ax).
\end{equation}
Note that \eqref{ref}, \eqref{dup}, and \eqref{fe} imply $F(z,s)=F(1-z,s)$. In
the last step, replace $s$ by $2\nu$, $x$ by $x/(2\pi)$, and $y$ by $y/(2\pi)$
to obtain \eqref{koshlyakov-1}. Similarly, \eqref{koshlyakov-2} can be proved
using Theorem \ref{production-2} by taking
\begin{equation*}
F(z,s)=\frac{\pi^{z+1}}{\G\left(\frac{1}{2}z-\frac{1}{4}s\right)\G\left(\frac{1}{2}z+\frac{1}{4}s\right)
\zeta\left(z-\frac{1}{2}s\right)\zeta\left(z+\frac{1}{2}s\right)},
\end{equation*}
and then using \eqref{melkz} with $z$ replaced by $z+1$. The equality $F(z,s)=F(1-z,s)$ can be proved using \eqref{ref2}, \eqref{dup}, and \eqref{fe}. As before, in the final step, we replace $s$ by $2\nu$, $x$ by $x/(2\pi)$, and $y$ by $y/(2\pi)$ to obtain \eqref{koshlyakov-2}.

If we let $f(x,s)=K_{s/2}(2\pi x)$ in \eqref{cor5.4result}, we obtain
\begin{equation}\label{pfaff1}
\sqrt{\alpha} \int_{0}^{\infty} K_{s/2}(2 \pi \alpha x) K_{s/2}(2 \pi x)\, dx =
\sqrt{\beta} \int_{0}^{\infty} K_{s/2}(2 \pi \beta x) K_{s/2}(2 \pi x)\, dx,
\end{equation}
which is really a special case of Pfaff's transformation \cite[p.~110]{temme}
\begin{equation}\label{pfaff}
{}_2F_{1}(a,b;c;w)=(1-w)^{-a}{}_2F_{1}\left(a,c-b;c;\frac{w}{w-1}\right),
\end{equation}
as can be checked using the evaluation \cite[p.~384, Formula \textbf{2.16.33.1}]{prud}
\begin{align*}
&\int_{0}^{\infty}x^{a-1}K_{\mu}(bx)K_{\nu}(cx)\, dx\nonumber\\
&=2^{a-3}c^{-a-\mu}\frac{b^{\mu}}{\G(a)}\G\left(\frac{a+\mu+\nu}{2}\right)\G\left(\frac{a+\mu-\nu}{2}\right)
\G\left(\frac{a-\mu+\nu}{2}\right)\G\left(\frac{a-\mu-\nu}{2}\right)\nonumber\\
&\quad\times{}_2F_{1}\left(\frac{a+\mu+\nu}{2},\frac{a-\mu+\nu}{2};a;1-\frac{b^2}{c^2}\right),
\end{align*}
valid for Re $(b+c)>0$ and Re $a>|$Re $\mu|+|$Re $\nu|$. Similarly, letting $f(x,s)=xK_{s/2}(2\pi x)$ in \eqref{cor5.6result} yields
\begin{equation*}
\a^{3/2}\int_{0}^{\infty} x^2K_{s/2}(2 \pi \alpha x) K_{s/2}(2 \pi x)\, dx =
\b^{3/2}\int_{0}^{\infty} x^2K_{s/2}(2 \pi \beta x) K_{s/2}(2 \pi x)\, dx,
\end{equation*}
which is again a special case of Pfaff's transformation \eqref{pfaff}.


\subsection{A New Modular Transformation}\label{anmt}
In \cite[Theorems 4.5, 4.9]{dixitmoll}, transformations of the type given in
Corollary \ref{cor5.4} resulting from the choices $F(z,
s)=\G\left(z+\frac{1}{2}s\right)\G\left(1-z+\frac{1}{2}s\right)$ and $F(z,
s)=\G\left(\frac{1}{2}z+\frac{1}{4}s\right)\G\left(\frac{1}{2}
-\frac{1}{2}z+\frac{1}{4}s\right)$
were obtained. In the following theorem, we give a new example of a function
$f(x,s)$, equal to its first Koshliakov transform, constructed by choosing
\begin{equation*}
F(z, s)=\frac{1}{\sin(\pi z)-\sin\left(\frac{1}{2}\pi s\right)}
\end{equation*}
which, with the help of Corollary \ref{cor5.4}, gives a new modular
transformation. An integral involving a product of Riemann $\Xi$-functions
(defined in \eqref{Xit}) at two different arguments linked with this
transformation is also obtained.

\begin{theorem}\label{newthm}
Let $-1< \sigma<1$. Let
\begin{align}\label{fxzsp}
&f(x,s):=\frac{x^{s/2}\zeta(1+s)}{2\sin\left(\frac{1}{2}\pi
    s\right)}+\frac{x^{2-s/2}}{\pi\cos\left(\frac{1}{2}\pi
    s\right)}\sum_{n=1}^{\infty}\sigma_{-s}(n)\left(\frac{n^{s-1}-x^{s-1}}{n^2-x^2}\right)\\
&\quad-2^{-s}\pi^{-1-s}x^{-s/2}\left\{\G(s)\zeta(s)\left(\pi\tan\left(\frac{\pi s}{2}\right)-2(\gamma+\log x)\right)-\frac{(2\pi)^{s}\zeta'(1-s)}{\cos\left(\frac{1}{2}\pi s\right)}\right\}.\notag
\end{align}
Then, for $\a,\b>0$ and $ \a\b=1$,
\begin{align}\label{ane}
&\sqrt{\alpha} \int_{0}^{\infty} K_{s/2}(2 \pi \alpha x) f(x,s)\, dx= \sqrt{\beta} \int_{0}^{\infty} K_{s/2}(2 \pi \beta x) f(x,s)\, dx\nonumber\\
&=\frac{1}{4\pi^3}\int_{0}^{\infty}\G\left(\frac{s-1+it}{4}\right)\G\left(\frac{s-1-it}{4}\right)
\G\left(\frac{-s+1+it}{4}\right)\G\left(\frac{-s+1-it}{4}\right)\nonumber\\
&\quad\quad\quad\quad\quad\times\Xi\left(\frac{t-is}{2}\right)\Xi\left(\frac{t+is}{2}\right)
\frac{\cos\left(\tfrac{1}{2}t\log\a\right)}{t^2+(s+1)^2}\, dt.
\end{align}
\end{theorem}

\textbf{Remark:} When $x$ is an integer $m$, we adhere to the interpretation
$$\lim_{x\to m}\frac{m^{s-1}-x^{s-1}}{m^2-x^2}=\frac{(s-1)}{2}m^{s-3}.$$

\begin{proof}
Let
\begin{equation}\label{Fzsdef}
F(z,s)=\frac{1}{\sin(\pi z)-\sin\left(\frac{1}{2}\pi s\right)}.
\end{equation}
We first show that for $\pm  \frac{1}{2}\sigma < c = \textup{Re } z < 1 \pm
\frac{1}{2}\sigma$,
\begin{equation}\label{tbp}
\mathfrak{I}(x,s):=\frac{1}{2 \pi i} \int_{(c)}
F(z,s) \zeta(1 - z -s/2) \zeta(1-z + s/2)x^{-z} \, dz=f(x,s),
\end{equation}
where $f(x,s)$ is defined in \eqref{fxzsp}.

To prove \eqref{tbp}, first assume that Re $z>1\pm  \frac{1}{2}\sigma$, let $z=1-w$, and consider, for $\l=$ Re $w<\pm$ $\frac{\sigma}{2}$, the integral
\begin{equation*}
\mathfrak{H}(x,s):=\frac{1}{2 \pi i} \int_{(\l)}\frac{\zeta\left(w-\frac{s}{2}\right)\zeta\left(w+\frac{s}{2}\right)}{\sin(\pi w)-\sin\left(\frac{1}{2}\pi s\right)}x^{w-1} \, dw.
\end{equation*}
In order to use the formula
\begin{align}\label{prodze}
\zeta\left(w-\frac{s}{2}\right)\zeta\left(w+\frac{s}{2}\right)=\sum_{n=1}^{\infty}\frac{\sigma_{-s}(n)}{n^{w-\frac{s}{2}}},
\end{align}
which is valid for $\textup{ Re } w>1\pm \frac{1}{2}\sigma$, we need to shift the line
of integration to $\l'= \textup{ Re } w>1\pm \textup{Re }
\frac{1}{2}\sigma$. Considering a positively oriented rectangular contour with vertices
$[\lambda-iT,\lambda'-iT,\lambda'+iT,\lambda+iT,\lambda-iT]$ for $T>0$,
shifting the line of integration, considering the contributions of the simple poles
 at $\frac{1}{2}s$ and $1+\frac{1}{2}s$ and of the double pole
 at $1-\frac{1}{2}s$, and noting that the integrals along the
horizontal segments tend to zero as $T\to\infty$, by Cauchy's residue
theorem, we find that
\begin{align}\label{frakhxz}
\mathfrak{H}(x,s)&=\frac{1}{2 \pi
  i}\biggl(\frac{1}{x}\sum_{n=1}^{\infty}\sigma_{-s}(n)n^{s/2}\int_{(\l')}\frac{(n/x)^{-w}}{\sin(\pi
    w)-\sin\left(\frac{\pi s}{2}\right)}\, dw\notag\\
&\quad-2\pi i\left(R_{s/2}+R_{1+s/2}+R_{1-s/2}\right)\biggr),
\end{align}
where the interchange of the order of summation and integration can be easily
justified. The residues $R_{s/2}$, $R_{1+s/2}$, and $R_{1-s/2}$ are computed as
{\allowdisplaybreaks\begin{align}
R_{s/2}&=\lim_{w\to\frac{1}{2}s}\frac{(w-\frac{1}{2}s)}{\sin(\pi
  w)-\sin\left(\frac{1 }{2}\pi s
  \right)}\zeta\left(w-\frac{s}{2}\right)
\zeta\left(w+\frac{s}{2}\right)x^{w-1}\nonumber\\
&=-\frac{\zeta(s)x^{s/2-1}}{2\pi\cos\left(\frac{1}{2}\pi s\right)},\label{star}\\
R_{1+s/2}&=\lim_{w\to1+\frac{1}{2}s}\frac{(w-1-\frac{1}{2}s)}{\sin(\pi w)-\sin\left(\frac{1}{2}\pi s\right)}\zeta\left(w-\frac{s}{2}\right)\zeta\left(w+\frac{s}{2}\right)x^{w-1}\nonumber\\
&=-\frac{\zeta(1+s)x^{s/2}}{2\sin\left(\frac{1}{2}\pi s\right)},\label{starstar}\\
R_{1-s/2}&=\lim_{w\to1-\frac{1}{2}s}\frac{d}{dw}\left\{\frac{(w-1+\frac{1}{2}s)^2}{\sin(\pi w)-\sin\left(\frac{1}{2}\pi s\right)}\zeta\left(w-\frac{s}{2}\right)\zeta\left(w+\frac{s}{2}\right)x^{w-1}\right\}\notag\\
=2^{-s}&\pi^{-1-s}x^{-s/2}\left\{\G(s)\zeta(s)\left(\pi\tan\left(\frac{\pi s}{2}\right)-2(\gamma+\log x)\right)-\frac{(2\pi)^{s}\zeta'(1-s)}{\cos\left(\frac{1}{2}\pi s\right)}\right\}.\label{residues}
\end{align}}%
Next,
\begin{multline}\label{inxz}
\int_{(\l')}\frac{(n/x)^{-w}}{\sin(\pi w)-\sin\left(\frac{1}{2}\pi s\right)}\,
dw\\=\frac{1}{2\pi^2}\int_{(\l')}\frac{\pi}{\sin\left(\frac{1}{2}\pi
\left(w-\frac{1}{2}s\right)\right)}
\frac{\pi}{\cos\left(\frac{1}{2}\pi
\left(w+\frac{1}{2}s\right)\right)}\left(\frac{n}{x}\right)^{-w}\, dw.
\end{multline}
For $0<d$= Re $z<2$ \cite[p.~345, formula \bf{(12)}]{erdelyi},
\begin{equation*}
\frac{1}{2\pi i}\int_{(d)}\frac{\pi}{\sin\left(\frac{1}{2}\pi
    z\right)}x^{-z}\, dz=\frac{2}{1+x^2}.
\end{equation*}
Replace  $z$ by $w-\tf12s$ to obtain, for $\tf12\sigma<d'= \text{Re }
 w<2+\tf12\sigma$,
\begin{equation}\label{j10a}
\frac{1}{2\pi i}\int_{(d')}\frac{\pi}{\sin\left(\frac{1}{2}\pi\left(w-\frac{1}{2}s\right)\right)}x^{-w}\, dw=\frac{2x^{-s/2}}{1+x^2}.
\end{equation}
Also, replace $w$ by $w+1+s$ in \eqref{j10a}, so that for $-1-\tf12\sigma<d''=
\text{Re } w<1-\tf12\sigma$,
\begin{equation}\label{j10b}
\frac{1}{2\pi i}\int_{(d'')}\frac{\pi}{\cos\left(\frac{1}{2}\pi\left(w+\frac{1}{2}s\right)\right)}x^{-w}\, dw=\frac{2x^{1+s/2}}{1+x^2}.
\end{equation}
Employing \eqref{j10a} and \eqref{j10b} in \eqref{inxz} and using
\eqref{melconv},
we deduce that, for $\tf12\sigma<c'= \text{Re } w<1-\tf12\sigma$,
\begin{align}\label{remin}
\int_{(c')}\frac{(n/x)^{-w}}{\sin(\pi w)-\sin\left(\frac{1}{2}\pi s\right)}\,
dw&=\frac{4i}{\pi}\int_{0}^{\infty}\frac{t^{-s/2}}{(1+t^2)}
\frac{\left(n/(xt)\right)^{1+s/2}}{\left(1+n^2/(x^2t^2)\right)}\frac{dt}{t}\nonumber\\
&=\frac{4i}{\pi}\left(\frac{n}{x}\right)^{1+s/2}\int_{0}^{\infty}\frac{t^{-s}}{(1+t^2)\left(n^2/x^2+t^2\right)}\, dt.
\end{align}
From \cite[p.~330, formula \textbf{3.264.2}]{grn}, for $0<$ Re $\mu<4$, $|\arg b|<\pi$, and $|\arg h|<\pi$,
\begin{equation*}
\int_{0}^{\infty}\frac{t^{\mu-1}}{(b+t^2)(h+t^2)}\, dt=\frac{\pi}{2\sin\left(\frac{1}{2}\mu\pi\right)}\frac{h^{\mu/2-1}-b^{\mu/2-1}}{b-h}.
\end{equation*}
Letting $\mu=-s+1$, $b=1$, and $h=n^2/x^2$ above, employing the
resulting identity in \eqref{remin}, and simplifying, we see that, for
$\tf12\sigma<c'= \text{Re } w<1-\tf12\sigma$,
\begin{equation*}
\int_{(c')}\frac{(n/x)^{-w}}{\sin(\pi w)-\sin\left(\frac{1}{2}\pi s\right)}\,
dw=\frac{2i}{\cos\left(\frac{1}{2}\pi s\right)}
\left(\frac{\left(n/x\right)^{-s/2}
-\left(n/x\right)^{1+s/2}}{1-n^2/x^2}\right).
\end{equation*}
Employing the residue theorem again, we find that for $\l'= \text{Re }w>1\pm
\tf12\sigma$,
\begin{align}\label{frakhxz1}
\int_{(\l')}&\frac{(n/x)^{-w}}{\sin(\pi w)-\sin\left(\frac{1}{2}\pi
    s\right)}\, dw\notag\\&=\int_{(c')}\frac{(n/x)^{-w}}{\sin(\pi
  w)-\sin\left(\frac{1}{2}\pi s\right)}\, dw+2\pi
i\lim_{w\to1-s/2}\frac{\left(w-1+s/2\right)(n/x)^{-w}}{\sin(\pi
  w)-\sin\left(\frac{1}{2}\pi s\right)}\nonumber\\
&=\frac{2i}{\cos\left(\frac{1}{2}\pi s\right)}\left(\frac{\left(n/x\right)^{-s/2}
-\left(n/x\right)^{1+s/2}}{1-n^2/x^2}\right)-\frac{2i}{\cos\left(\frac{1}{2}\pi
s\right)}\left(n/x\right)^{s/2-1}\nonumber\\
&=\frac{2in^{-s/2}x^{3-s/2}}{\cos\left(\frac{1}{2}\pi s\right)}\left(\frac{n^{s-1}-x^{s-1}}{n^2-x^2}\right).
\end{align}
Now substitute \eqref{star}, \eqref{starstar}, \eqref{residues}, and \eqref{frakhxz1} in \eqref{frakhxz} to
see that, for $c''= \textup{Re }z>1\pm\frac{1}{2}\sigma$,
\begin{align*}
&\mathfrak{H}(x,z)=\frac{1}{2 \pi i} \int_{(c'')}\frac{\zeta(1 - z -s/2)
  \zeta(1-z + s/2)}{\sin(\pi z)-\sin\left(\frac{1}{2}\pi s\right)}x^{-z} \, dz\\
&=\frac{x^{s/2}\zeta(1+s)}{2\sin\left(\frac{1}{2}\pi
    s\right)}-\frac{\zeta(s)x^{s/2-1}}{2\pi\cos\left(\frac{1}{2}\pi s\right)}+\frac{x^{2-s/2}}{\pi\cos\left(\frac{1}{2}\pi s\right)}\sum_{n=1}^{\infty}\sigma_{-s}(n)\left(\frac{n^{s-1}-x^{s-1}}{n^2-x^2}\right)\nonumber\\
&\quad-2^{-s}\pi^{-1-s}x^{-s/2}\left\{\G(s)\zeta(s)\left(\pi\tan\left(\frac{\pi
        s}{2}\right)-2(\gamma+\log
    x)\right)-\frac{(2\pi)^{s}\zeta'(1-s)}{\cos\left(\frac{1}{2}\pi s\right)}\right\}.\notag
\end{align*}
Using the residue theorem again, we see that for $\pm \frac{1}{2}\sigma < c = \textup{Re } z < 1 \pm\frac{1}{2}\sigma$,
{\allowdisplaybreaks\begin{align*}
&\mathfrak{I}(x,s)=\mathfrak{H}(x,s)-\lim_{z\to
  1-s/2}\frac{\left(z-1+s/2\right)}{\sin(\pi
  z)-\sin\left(\frac{1}{2}\pi s\right)}\zeta(1 - z -s/2) \zeta(1-z + s/2)x^{-z}\nonumber\\
&=\frac{x^{s/2}\zeta(1+s)}{2\sin\left(\frac{1}{2}\pi s\right)}+\frac{x^{2-s/2}}{\pi\cos\left(\frac{1}{2}\pi s\right)}\sum_{n=1}^{\infty}\sigma_{-s}(n)\left(\frac{n^{s-1}-x^{s-1}}{n^2-x^2}\right)\nonumber\\
&\quad-2^{-s}\pi^{-1-s}x^{-s/2}\left\{\G(s)\zeta(s)\left(\pi\tan\left(\frac{\pi
        s}{2}\right)-2(\gamma+\log
    x)\right)-\frac{(2\pi)^{s}\zeta'(1-s)}{\cos\left(\frac{1}{2}\pi s\right)}\right\}\nonumber\\
&=f(x,s),
\end{align*}}%
where $f(x,s)$ is defined in \eqref{fxzsp}. The proof of the first equality
in \eqref{ane} now follows from \eqref{tbp}, Theorem \ref{production-1}, and Corollary
\ref{cor5.4}.

In order to establish the equality between the extreme sides of \eqref{ane},
note that
by \eqref{melkz}, for $c= \text{Re }z>\pm\tf12\sigma$,
\begin{equation}\label{bess}
\frac{1}{2\pi i}\int_{(c)}2^{z-2}(2\pi\a)^{-z}\G\left(\frac{z}{2}-\frac{s}{4}\right)\G\left(\frac{z}{2}+\frac{s}{4}\right)x^{-z}\, dz=K_{s/2}(2\pi\a x).
\end{equation}
From \eqref{bess} and \eqref{pf}, for $\pm\tf12\sigma< c = \textup{Re } z < 1
\pm\tf12\sigma$,
\begin{align*}
&\sqrt{\a}\int_{0}^{\infty}K_{s/2}(2\pi\a x)f(x,s)\, dx \notag
=\frac{\sqrt{\a}}{2\pi i}\int_{(c)}2^{-z-1}(2\pi\a)^{z-1}\\
&\quad\times\frac{\G\left(\frac{1}{2}-\frac{1}{2}z-\frac{1}{4}s\right)
\G\left(\frac{1}{2}-\frac{1}{2}z+\frac{1}{4}s\right)}{\sin(\pi
z)-\sin\left(\frac{1}{2}\pi s\right)}\zeta(1 - z -\tf12s) \zeta(1-z + \tf12s)\, dz.
\end{align*}
Now use the functional equation \eqref{fesym} for $\zeta(1 - z -s/2)$ and for
$\zeta(1 - z
+s/2)$, \eqref{feg}, \eqref{ref}, and \eqref{xis} to simplify the integrand,
thereby obtaining
\begin{align}\label{jfinbef}
&\sqrt{\a}\int_{0}^{\infty}K_{s/2}(2\pi\a x)f(x,s)\, dx\nonumber\\
&=\frac{1}{64\pi^3i\sqrt{\a}}\int_{(c)}\frac{\G\left(\frac{1}{2}z+\frac{1}{4}s
-\frac{1}{2}\right)\G\left(-\frac{1}{2}z+\frac{1}{4}s\right)\G\left(\frac{1}{2}-\frac{1}{2}z
-\frac{1}{4}s\right)\G\left(\frac{1}{2}z-\frac{1}{4}s\right)}{\left(\frac{1}{2}z
-\frac{1}{4}s-\frac{1}{2}\right)\left(-\frac{1}{2}z-\frac{1}{4}s\right)}\notag\\
&\quad\times\xi\left(z-\frac{s}{2}\right)
\xi\left(z+\frac{s}{2}\right)\a^{z}\, dz.
\end{align}
From \cite[equation (2.8)]{transf}, if $f(s,t)=\phi(s,it)\phi(s,-it)$, where
$\phi$ is analytic as a function of a real variable $t$ and of a complex
variable $s$, then
\begin{align}\label{iznsimp}
&\int_{0}^{\infty}f\left(s,\frac{t}{2}\right)\Xi\left(\frac{t+is}{2}\right)\Xi\left(\frac{t-is}{2}\right)
\cos\left(\tfrac{1}{2}t\log\a\right)\, dt\nonumber\\
&=\frac{1}{i\sqrt{\a}}\int_{\frac{1}{2}-i\infty}^{\frac{1}{2}+i\infty}\phi\left(s,z-\frac{1}{2}\right)
\phi\left(s,\frac{1}{2}-z\right)\xi\left(z-\frac{s}{2}\right)\xi\left(z+\frac{s}{2}\right)\a^{z}\, dz.
\end{align}
It is easy to see that with
\begin{equation*}
\phi(z,s)=\frac{\G\left(\frac{1}{2}z+\frac{1}{4}s-\frac{1}{4}\right)\G\left(\frac{1}{2}z-\frac{1}{4}s+\frac{1}{4}\right)}
{\left(\frac{1}{2}z-\frac{1}{4}s-\frac{1}{4}\right)},
\end{equation*}
and the fact that shifting the line of integration from Re $z=c$ to Re
$z=\frac{1}{2}$ and using the residue theorem leaves the integral in
\eqref{jfinbef} unchanged, this integral can be written in the form given on
the right-hand side of \eqref{iznsimp}, whence \eqref{iznsimp} proves the
equality between the extreme sides of \eqref{ane}.
\end{proof}

Theorem \ref{newthm} gives a new generalization of the following formula due
to Koshliakov \cite[equations (36), (40)]{kosh1937}, different from the one
given in \cite[Theorem 4.5]{dixitmoll}.
\begin{theorem}
Define
\begin{equation*}
\Lambda(x) = \frac{\pi^{2}}{6} + \gamma^{2} - 2 \gamma_{1} + 2 \gamma \log x
+ \frac{1}{2} \log^{2}x + \sum_{n=1}^{\infty} d(n)
\left( \frac{1}{x+n} - \frac{1}{n} \right),
\end{equation*}
where $\gamma_{1}$ is a Stieltjes constant defined in \eqref{stiel}. Then, for $\alpha, \beta > 0$ and $\alpha\beta=1$,
\begin{align*}
\sqrt{\alpha} \int_{0}^{\infty} K_{0}(2 \pi \alpha x) \Lambda(x) \, dx &= \sqrt{\beta} \int_{0}^{\infty} K_{0}(2 \pi \beta x) \Lambda(x) \, dx\nonumber\\
&=8 \int_{0}^{\infty} \frac{ \left( \Xi \left( \tfrac{1}{2}t \right) \right)^{2}}
{(1+t^{2})^{2}} \,
\frac{\cos \left( \tfrac{1}{2} t \log \alpha \right)}{\cosh \left( \tfrac{1}{2}
\pi t \right)} \, dt.
\end{align*}
\end{theorem}

Koshliakov's result can be obtained by letting $s\to 0$ in Theorem \ref{newthm}.

\subsection{Transformation Involving Modified Lommel Functions: A Series
  Analogue of Theorem \ref{newthm}}\label{asa}

In \cite[Theorem 6]{gui}, Guinand proved the following theorem.

\begin{theorem}\label{guisumpri}
If $f(x)$ and  $f'(x)$ are integrals, $f(x), xf'(x)$, and $x^2f''(x)$ belong to $L^{2}(0,\infty)$, and $0<|r|<1$, then
\begin{align*}
&\lim_{N\to\infty}\left\{\sum_{n=1}^{N}\left(1-\frac{n}{N}\right)^{2}\sigma_r(n)n^{-r/2}f(n)\right.\notag\\
&\hspace{.5in}\left.-\int_{0}^{N}
\left(1-\frac{x}{N}\right)^{2}f(x)\left(x^{-r/2}\zeta(1-r)+x^{r/2}\zeta(1+r)\right)\, dx\right\}\nonumber\\
&=\lim_{N\to\infty}\left\{\sum_{n=1}^{N}\left(1-\frac{n}{N}\right)^{2}\sigma_r(n)n^{-r/2}g(n)\right.\notag\\
&\hspace{.5in}\left.-\int_{0}^{N}
\left(1-\frac{x}{N}\right)^{2}g(x)\left(x^{-r/2}\zeta(1-r)+x^{r/2}\zeta(1+r)\right)\, dx\right\},
\end{align*}
where
\begin{align*}
\int_{0}^{x}g(y)y^{r/2}\, dy&=x^{(r+1)/2}\int_{0}^{\infty}y^{-\frac{1}{2}}f(y)\phi_{r+1}(4\pi\sqrt{xy})\, dy,\nonumber\\
\phi_{\nu}(z)&=\cos\left(\frac{1}{2}\pi\nu\right)J_{\nu}(z)-\sin\left(\frac{1}{2}\pi\nu\right)\left(Y_{\nu}(z)
+\frac{2}{\pi}K_{\nu}(z)\right),
\end{align*}
and $g(x)$ is chosen so that it is the integral of its derivative.
\end{theorem}
As discussed by Guinand \cite[equation (1)]{guinand}, this gives
\begin{align}\label{guisum1}
&\sum_{n=1}^{\infty}\sigma_{-s}(n)n^{s/2}G(n, s)-\zeta(1+s)\int_{0}^{\infty}t^{s/2}G(t, s)\, dt
-\zeta(1-s)\int_{0}^{\infty}t^{-s/2}G(t, s)\, dt\\
&=\sum_{n=1}^{\infty}\sigma_{-s}(n)n^{s/2}H(n, s)-\zeta(1+s)\int_{0}^{\infty}t^{s/2}H(t, s)\,
dt-\zeta(1-s)\int_{0}^{\infty}t^{-s/2}H(t, s)\, dt,\nonumber
\end{align}
where $G(x, s)$ satisfies the same conditions as those of $f$ in Theorem
\ref{guisumpri}, and where
\begin{align*}
H(x, s)&=\int_{0}^{\infty}G(t, s)\left(-2\pi\sin\left(\tfrac{1}{2}\pi
    s\right)J_{s}(4\pi\sqrt{xt})\right.\notag\\&\quad\left.-\cos\left(\tfrac{1}{2}\pi s\right)\left(2\pi
    Y_{s}(4\pi\sqrt{xt})-4K_{s}(4\pi\sqrt{xt})\right)\right)\, dt,
\end{align*}
that is, $H(x, s)$ is essentially the first Koshliakov transform of $G(x, s)$.

Even though \eqref{cor5.4result} and \eqref{guisum1}  are both modular
transformations and both involve the first Koshliakov transform of a
function, they are very different in nature as can be seen from the fact that
if $f(x,s)=K_{s/2}(2\pi x)$ in \eqref{cor5.4result}, we obtain Pfaff's
transformation, as discussed in \eqref{pfaff1} and \eqref{pfaff}, whereas,
letting $G(x, s)=K_{s/2}(2\pi\a x)$ in \eqref{guisum1} yields the
Ramanujan--Guinand formula, as discussed in \cite[Section
7]{dixitmoll}.

The function $f(x,s)$ of Theorem \ref{newthm} is equal to its first
Koshliakov transform, as can be seen from \eqref{Fzsdef}, \eqref{tbp}, and
Theorem \ref{production-1}. Thus there are two series transformations
associated with this $f(x,s)$ for a fixed $s$ such that $-1<\sigma<1$ -- one
resulting from letting $G(x, s)=H(x, s)=f(x,s)$ in \eqref{guisum1}, and the other
being the series analogue of Theorem \ref{newthm} obtainable by interchanging
the order of summation and integration in both expressions in the first equality in \eqref{ane}. The former
seems more formidable than the latter. Thus we attempt the latter, which gives
a beautiful transformation involving the modified Lommel functions.

There are several functions in the literature called Lommel functions.
However, the ones that are important for us are those defined by \cite[p.~346, equation (10)]{watsonbessel}
\begin{equation}\label{lommeldef2}
s_{\mu,\nu}(w):=\frac{w^{\mu+1}}{(\mu-\nu+1)(\mu+\nu+1)}{}_1F_{2}
\left(1;\frac{\mu-\nu+3}{2},\frac{\mu+\nu+3}{2};-\frac{1}{4}w^2\right)
\end{equation}
and \cite[p.~347, equation (2)]{watsonbessel}
\begin{align}\label{lommeldef1}
S_{\mu,\nu}(w)&=s_{\mu,\nu}(w)+\frac{2^{\mu-1}\G\left(\frac{\mu-\nu+1}{2}\right)
\G\left(\frac{\mu+\nu+1}{2}\right)}{\sin(\nu\pi)}\\
&\quad\quad\quad\quad\quad\times\left\{\cos\left(\frac{1}{2}(\mu-\nu)\pi\right)
J_{-\nu}(w)-\cos\left(\frac{1}{2}(\mu+\nu)\pi\right)J_{\nu}(w)\right\}\notag
\end{align}
for $\nu\notin\mathbb{Z}$, and
\begin{align}\label{lommeldefint}
S_{\mu,\nu}(w)&=s_{\mu,\nu}(w)+2^{\mu-1}\G\left(\frac{\mu-\nu+1}{2}\right)\G\left(\frac{\mu+\nu+1}{2}\right)\\
&\quad\quad\quad\quad\quad\times\left\{\sin\left(\frac{1}{2}(\mu-\nu)\pi\right)
J_{\nu}(w)-\cos\left(\frac{1}{2}(\mu-\nu)\pi\right)Y_{\nu}(w)\right\}\notag
\end{align}
for $\nu\in\mathbb{Z}$.
These functions are the solutions of an inhomogeneous form of the Bessel differential equation \cite[p.~345]{watsonbessel}, namely,
\begin{equation*}
w^2\frac{d^{2}y}{dw^{2}}+w\frac{dy}{dw}+(w^2-\nu^2)y=w^{\mu+1}.
\end{equation*}
Even though $s_{\mu,\nu}(w)$ is undefined when $\mu\pm\nu$ is an odd negative
integer, $S_{\mu,\nu}(w)$ has a limit at those values
\cite[p.~347]{watsonbessel}. These are the exceptional cases of the Lommel
function $S_{\mu,\nu}(w)$. For more details on the exceptional cases, the
reader is referred to \cite[pp.~348--349, Section 10.73]{watsonbessel} and to
a more recent article \cite{glasser}.

The modified Lommel functions or the Lommel functions of imaginary argument
\cite{zs} are defined by
\begin{equation}\label{modlf}
T_{\mu,\nu}(y):=-i^{1-\mu}S_{\mu,\nu}(iy),
\end{equation}
where $y$ is real. For further information on modified Lommel functions, the reader is
referred to \cite{rollinger} and \cite{shafer}.

Lommel functions, as well as modified Lommel functions, are very useful in  physics and mathematical physics. For example, see \cite{assjsv, goldstein, sitzer, thomaschan, torrey}.

As a series analogue of Theorem \ref{newthm}, we will now obtain the
following modular transformation consisting of infinite series of modified
Lommel functions, of which one is an exceptional case of Lommel functions.
\begin{theorem}\label{newthmser}
For $\a>0$ and $-1<\sigma<1$, let
{\allowdisplaybreaks\begin{align}
&\mathfrak{L}(s,\a):=\frac{\pi^{s/2}\G\left(-\frac{1}{2}s\right)\zeta(-s)}{8\sin\left(\frac{1}{2}\pi   s\right)}\a^{-(s+1)/2}
-2^{-2-s}\pi^{-(s+3)/2}\a^{-1+s/2}\G\left(\frac{1-s}{2}\right)\notag\\
&\quad\times\left\{\G(s)\zeta(s)\left(\pi\tan\left(\frac{\pi s}{2}\right)-\gamma+2\log(2\pi\alpha)-\psi\left(\frac{1-s}{2}\right)\right)
-\frac{(2\pi)^s\zeta'(1-s)}{\cos\left(\frac{1}{2}\pi s\right)}\right\}\nonumber\\
&\quad+\frac{\sqrt{\a}}{\pi\cos\left(\frac{1}{2}\pi s\right)}\sum_{n=1}^{\infty}\sigma_{-s}(n)n^{s/2}\bigg\{2^{s/2}\G\left(1+\frac{s}{2}\right)
T_{-1-s/2,-s/2}(2\pi n\a)\nonumber\\
&\quad\quad\quad\quad\quad\quad\quad\quad\quad-\sqrt{\pi}2^{-s/2}\G\left(\frac{3-s}{2}\right)
T_{-2+s/2,s/2}(2\pi n\a)\bigg\}.\notag
\end{align}}
Then, for $\a\b=1$,
\begin{align}\label{anee}
&\mathfrak{L}(s,\a)=\mathfrak{L}(s,\b)\nonumber\\
&=\frac{1}{4\pi^3}\int_{0}^{\infty}\G\left(\frac{s-1+it}{4}\right)\G\left(\frac{s-1-it}{4}\right)
\G\left(\frac{-s+1+it}{4}\right)\G\left(\frac{-s+1-it}{4}\right)\nonumber\\
&\quad\quad\quad\quad\quad\times\Xi\left(\frac{t-is}{2}\right)\Xi\left(\frac{t+is}{2}\right)
\frac{\cos\left(\tfrac{1}{2}t\log\a\right)}{t^2+(s+1)^2}\, dt.
\end{align}
\end{theorem}

\textbf{Remark:} We emphasize that there are very few results in the
literature involving infinite series of Lommel functions. Papers by R.~G.~Cooke
\cite{cooke} and Lewis and Zagier \cite[p.~213--218]{lewzag} are two examples. The special case of the Lommel function $S_{\mu,\nu}(z)$ that Lewis and Zagier consider in \cite[p.~214, Equation (2.15)]{lewzag} is
\begin{equation*}
\mathcal{C}_{s}(z)=\sqrt{z}\Gamma(2s+1)S_{-2s-\frac{1}{2},\frac{1}{2}}(z).
\end{equation*}
To the best of our knowledge, none of these papers involve infinite
series of exceptional cases of Lommel functions.
\begin{proof}
Consider the extreme left side of \eqref{ane}.
Using \eqref{melkz} and \eqref{fesym}, we find that
\begin{equation}\label{fterm}
\int_{0}^{\infty}\frac{x^{s/2}\zeta(1+s)}{2\sin\left(\frac{1}{2}\pi s\right)}K_{s/2}(2\pi\a x)\,
dx=\frac{\pi^{s/2}\Gamma\left(-\frac{s}{2}\right)\zeta(-s)}
{8\a^{1+s/2}\sin\left(\frac{1}{2}\pi
  s\right)}.
\end{equation}
Also note that formula $\textbf{2.16.20.1}$ of \cite[p.~365]{prud} asserts
that, for
$| \text{Re}\hspace{0.5mm} w | >$ Re $\nu$ and real $m>0$,
\begin{align}\label{mkl}
\int_{0}^{\infty} x^{w-1} K_{\nu}(mx) \log x \, dx &=
\frac{2^{w-3}}{m^{w}}
\Gamma \left( \frac{w + \nu}{2} \right)
\Gamma \left( \frac{w - \nu}{2} \right)\notag \\
&\quad\times \left\{ \psi \left( \frac{w + \nu}{2} \right) +
\psi \left( \frac{w - \nu}{2} \right)
-2 \log \left( \frac{m}{2} \right) \right\}.
\end{align}
Now use \eqref{melkz} and \eqref{mkl} to find, upon simplification, that
\begin{align}\label{sterm}
&-2^{-s}\pi^{-1-s}\int_{0}^{\infty}x^{-s/2}K_{s/2}(2\pi\a x)\nonumber\\
&\quad\quad\quad\quad\quad\quad\quad\times\left\{\G(s)\zeta(s)\left(\pi\tan\left(\frac{\pi s}{2}\right)-2(\gamma+\log x)\right)-\frac{(2\pi)^{s}\zeta'(1-s)}{\cos\left(\frac{\pi s}{2}\right)}\right\}\, dx\nonumber\\
&=-2^{-2-s}\pi^{-(s+3)/2}\a^{-1+s/2}\G\left(\frac{1-s}{2}\right)\\
&\quad\times\left\{\G(s)\zeta(s)\left(\pi\tan\left(\frac{\pi s}{2}\right)-\gamma+2\log(2\pi\alpha)-\psi\left(\frac{1-s}{2}\right)\right)-\frac{(2\pi)^s\zeta'(1-s)}{\cos\left(\frac{\pi s}{2}\right)}\right\}.\notag
\end{align}
In \cite[p.~347, formula \textbf{2.16.3.18}]{prud}, we find that, for $y>0$, $\textup{Re } c>0$, and $\textup{Re } a>|\textup{Re }\nu|$, the integral evaluation
(in corrected form)
\begin{align}\label{prudint}
PV\int_{0}^{\infty}\frac{x^{a-1}}{x^2-y^2}K_{\nu}(cx)\, dx
&=\frac{\pi^2y^{a-2}}{4\sin(\nu\pi)}\left(\cot\left(\frac{\pi(a+\nu)}{2}\right)I_{\nu}(cy)-
\cot\left(\frac{\pi(a-\nu)}{2}\right)I_{-\nu}(cy)\right)\nonumber\\
&\quad+2^{a-4}c^{2-a}\G\left(\frac{a+\nu}{2}-1\right)\G\left(\frac{a-\nu}{2}-1\right)\nonumber\\
&\quad\times{}_1F_{2}\left(1;2-\frac{\nu+a}{2}, 2+\frac{\nu-a}{2};\frac{c^2y^2}{4}\right).
\end{align}
(In \cite{prud}, the principal value designation $PV$ is missing.) Next we show that
\begin{align}\label{insum}
&\int_{0}^{\infty}x^{2-s/2}K_{s/2}(2\pi\a x)\sum_{n=1}^{\infty}\sigma_{-s}(n)\left(\frac{n^{s-1}-x^{s-1}}{n^2-x^2}\right)\, dx\\
&=\sum_{n=1}^{\infty}\sigma_{-s}(n)\left\{n^{s-1}PV\int_{0}^{\infty}\frac{x^{2-s/2}K_{s/2}(2\pi\a x)}{n^2-x^2}\, dx\right.\notag\\&\quad\left.-PV\int_{0}^{\infty}\frac{x^{1+s/2}K_{s/2}(2\pi\a x)}{n^2-x^2}\, dx\right\}.\notag
\end{align}
 Let $w(t)\in C_0^{\infty}$ be a smooth function so that $0\leq w(t)\leq 1$
 for all $t\in \mathbb{R}$, $w(t)$ has compact support in $(-\tf13,\tf13)$,
 and  $w(t)=1$ in $(-\tf14,\tf14)$. Then the right-hand side of \eqref{insum}
 is equal  to
\begin{align}\label{twosum}
&\sum_{n=1}^{\infty}\sigma_{-s}(n)PV\int_{0}^{\infty}x^{2-s/2}K_{s/2}(2\pi\a x)\left(\frac{n^{s-1}-x^{s-1}}{n^2-x^2}\right)\, dx\\\nonumber
&=\sum_{n=1}^{\infty}\sigma_{-s}(n)\int_{0}^{\infty}x^{2-s/2}K_{s/2}(2\pi\a x)\left(\frac{n^{s-1}-x^{s-1}}{n^2-x^2}\right)(1-w(x-n))\, dx\\\nonumber
&\quad+\sum_{n=1}^{\infty}\sigma_{-s}(n)PV\int_{0}^{\infty}x^{2-s/2}K_{s/2}(2\pi\a x)\left(\frac{n^{s-1}-x^{s-1}}{n^2-x^2}\right)w(x-n)\, dx.
\end{align}
Since the series in the integrand on the left-hand side of \eqref{insum}
is absolutely convergent,  we can interchange the summation and integration
of the first expression on the right-hand side of \eqref{twosum}. Note that if $m$
is a positive integer and $m-\tf12\leq x\leq m+\tf12$,
\begin{align*}
\sum_{n=1}^{\infty}\sigma_{-s}(n)\left(\frac{n^{s-1}-x^{s-1}}{n^2-x^2}\right)w(x-n)
=\sigma_{-s}(m)\left(\frac{m^{s-1}-x^{s-1}}{m^2-x^2}\right)w(x-m).
\end{align*}
Therefore,
\begin{align*}
&PV\int_{0}^{\infty}x^{2-s/2}K_{s/2}(2\pi\a x)\sum_{n=1}^{\infty}\sigma_{-s}(n)\left(\frac{n^{s-1}-x^{s-1}}{n^2-x^2}\right)w(x-n)\, dx\\\nonumber
&=\sum_{m=1}^{\infty}PV\int_{m-1/2}^{m+1/2}x^{2-s/2}K_{s/2}(2\pi\a x)\sum_{n=1}^{\infty}\sigma_{-s}(n)\left(\frac{n^{s-1}-x^{s-1}}{n^2-x^2}\right)w(x-n)\, dx\\\nonumber
&=\sum_{m=1}^{\infty}\sigma_{-s}(m)PV\int_{m-1/2}^{m+1/2}x^{2-s/2}K_{s/2}(2\pi\a x)\left(\frac{m^{s-1}-x^{s-1}}{m^2-x^2}\right)w(x-m)\, dx\\\nonumber
&=\sum_{m=1}^{\infty}\sigma_{-s}(m)PV\int_{0}^{\infty}x^{2-s/2}K_{s/2}(2\pi\a x)\left(\frac{m^{s-1}-x^{s-1}}{m^2-x^2}\right)w(x-m)\, dx.
\end{align*}
This justifies the interchange of summation and integration in \eqref{insum}.
Using \eqref{prudint}, we have
\begin{align}\label{-2deri}
&n^{s-1}\cdot PV\int_{0}^{\infty}\frac{x^{2-s/2}K_{s/2}(2\pi\a x)}{n^2-x^2}\, dx\\
&=-\frac{\pi^2n^{s/2}}{4\sin\left(\frac{1}{2}\pi s\right)}\left\{\cot\left(\tfrac{3\pi}{2}\right)I_{s/2}(2\pi n\a)-\cot\left(\tfrac{\pi}{2}(3-s)\right)I_{-s/2}(2\pi n\a)\right\}\nonumber\\
&\quad-n^{s-1}2^{-1-s/2}(2\pi\a)^{-1+s/2}\Gamma\left(\frac{1}{2}\right)
\Gamma\left(\frac{1-s}{2}\right){}_1F_{2}\left(1;\frac{1}{2},\frac{1+s}{2};\pi^2n^2\a^2\right)\nonumber\\
&=\frac{\pi^2n^{s/2}}{4\cos\left(\frac{1}{2}\pi s\right)}I_{-s/2}(2\pi n\a)-\frac{\pi^{(s-1)/2}\a^{-1+s/2}n^{s-1}}{4}
\Gamma\left(\frac{1-s}{2}\right)
{}_1F_{2}\left(1;\frac{1}{2},\frac{1+s}{2};\pi^2n^2\a^2\right).\notag
\end{align}
Another application of \eqref{prudint} yields
\begin{align}\label{-1deri}
&PV\int_{0}^{\infty}\frac{x^{1+s/2}K_{s/2}(2\pi\a x)}{x^2-n^2}\, dx
=\frac{\pi^2n^{s/2}}{4\sin\left(\frac{1}{2}\pi s\right)}\cot\left(\dfrac{\pi}{2}(2+s)\right)I_{s/2}(2\pi n\a)\nonumber\\
&\quad+\lim_{\nu\to s/2}\bigg\{2^{s/2-2}(2\pi\a)^{-s/2}\G\left(\frac{s}{4}
+\frac{\nu}{2}\right)\G\left(\frac{s}{4}-\frac{\nu}{2}\right)\notag\\
&\quad\times{}_1F_{2}\left(1;1-\frac{s}{4}-\frac{\nu}{2},1-\frac{s}{4}
+\frac{\nu}{2};\pi^2n^2\a^2\right)\notag\\&\quad-\frac{\pi^2n^{s/2}}{4\sin(\nu\pi)}
\cot\left(\frac{\pi}{2}\left(2+\frac{s}{2}-\nu\right)\right)
I_{-\nu}(2\pi n\a)\bigg\}\nonumber\\
&=:\frac{\pi^2n^{s/2}\cot\left(\frac{\pi s}{2}\right)}{4\sin\left(\frac{\pi s}{2}\right)}I_{s/2}(2\pi n\a)+L,
\end{align}
where we have denoted the limit by $L$. Note that
\begin{align*}
L&=\frac{1}{4}\lim_{\nu\to s/2}\G\left(\frac{s}{4}-\frac{\nu}{2}\right)\bigg\{(\pi\a)^{-s/2}
\G\left(\frac{s}{4}+\frac{\nu}{2}\right){}_1F_{2}\left(1;1-\frac{s}{4}-\frac{\nu}{2},1-\frac{s}{4}
+\frac{\nu}{2};\pi^2n^2\a^2\right)\notag\\&\quad-\frac{\pi n^{s/2}}{\sin(\nu\pi)}\cos\left(\pi\left(\frac{s}{4}-\frac{\nu}{2}\right)\right)
\G\left(1-\frac{s}{4}+\frac{\nu}{2}\right)I_{-\nu}(2\pi n \a)\bigg\}\nonumber\\
&=\frac{1}{4}\lim_{\nu\to s/2}\frac{1}{\left(\frac{s}{4}-\frac{\nu}{2}\right)}
\bigg\{(\pi\a)^{-s/2}\G\left(\frac{s}{4}+\frac{\nu}{2}\right){}_1F_{2}
\left(1;1-\frac{s}{4}-\frac{\nu}{2},1-\frac{s}{4}+\frac{\nu}{2};\pi^2n^2\a^2\right)\nonumber\\
&\quad-\frac{\pi n^{s/2}}{\sin(\nu\pi)}\cos\left(\pi\left(\frac{s}{4}-\frac{\nu}{2}\right)\right)\G\left(1-\frac{s}{4}
+\frac{\nu}{2}\right)I_{-\nu}(2\pi n \a)\bigg\},
\end{align*}
where we multiplied the expression on the right side of the first equality above by $\left(\frac{s}{4}-\frac{\nu}{2}\right)$ and used the functional equation of the Gamma function.

The last expression inside the limit symbol is of the form $\frac{0}{0}$, since
\begin{align*}
&\lim_{\nu\to s/2}\bigg\{(\pi\a)^{-s/2}\G\left(\frac{s}{4}+\frac{\nu}{2}\right){}_1F_{2}
\left(1;1-\frac{s}{4}-\frac{\nu}{2},1-\frac{s}{4}+\frac{\nu}{2};\pi^2n^2\a^2\right)\nonumber\\
&\quad\quad\quad\quad\quad\quad\quad\quad-\frac{\pi n^{s/2}}{\sin(\nu\pi)}\cos\left(\pi\left(\frac{s}{4}-\frac{\nu}{2}\right)\right)\G\left(1-\frac{s}{4}+
\frac{\nu}{2}\right)I_{-\nu}(2\pi n \a)\bigg\}\nonumber\\
&=(\pi\a)^{-s/2}\G\left(\frac{s}{2}\right){}_0F_{1}\left(-;1-\frac{s}{2};\pi^2n^2\a^2\right)-\frac{\pi
  n^{s/2}}{\sin\left(\frac{1}{2}\pi s\right)}I_{-s/2}(2\pi n \a)\nonumber\\
&=0,
\end{align*}
as can be seen by applying the reflection formula \eqref{ref} and the definition \cite[p.~911, formula \textbf{8.406, nos.~1-2}]{grn}
\begin{equation}\label{defbesi}
I_{\nu}(w):=
\begin{cases}
e^{-\pi\nu i/2}J_{\nu}(e^{\pi i/2}w), & \text{if $-\pi<$ arg $w\leq\frac{1}{2}\pi$,}\\
e^{3\pi\nu i/2}J_{\nu}(e^{-3\pi i/2}w), & \text{if $\frac{1}{2}\pi<$ arg $w\leq \pi$},
\end{cases}
\end{equation}
where $J_{\nu}(w)$ is defined in \eqref{sumbesselj}. Thus L'Hopital's rule yields
{\allowdisplaybreaks\begin{align}\label{0deri}
L&=-\frac{1}{2}\lim_{\nu\to s/2}\frac{d}{d\nu}\bigg\{(\pi\a)^{-s/2}
\G\left(\frac{s}{4}+\frac{\nu}{2}\right){}_1F_{2}\left(1;1-\frac{s}{4}-\frac{\nu}{2},1-\frac{s}{4}
+\frac{\nu}{2};\pi^2n^2\a^2\right)\nonumber\\
&\quad\quad\quad\quad\quad\quad\quad\quad-\frac{\pi n^{s/2}}{\sin(\nu\pi)}\cos\left(\pi\left(\frac{s}{4}-\frac{\nu}{2}\right)\right)
\G\left(1-\frac{s}{4}+\frac{\nu}{2}\right)I_{-\nu}(2\pi n \a)\bigg\}\nonumber\\
&=-\frac{1}{2}\lim_{\nu\to s/2}\bigg[(\pi\a)^{-s/2}\bigg\{\G\left(\frac{s}{4}+\frac{\nu}{2}\right)
\frac{d}{d\nu}{}_1F_{2}\left(1;1-\frac{s}{4}-\frac{\nu}{2},1-\frac{s}{4}+\frac{\nu}{2};\pi^2n^2\a^2\right)\nonumber\\
&\quad\quad\quad\quad\quad\quad\quad\quad\quad+\frac{1}{2}\G'\left(\frac{s}{4}+\frac{\nu}{2}\right){}_1F_{2}
\left(1;1-\frac{s}{4}-\frac{\nu}{2},1-\frac{s}{4}+\frac{\nu}{2};\pi^2n^2\a^2\right)\bigg\}\nonumber\\
&\quad\quad\quad\quad\quad-\pi n^{s/2}\frac{d}{d\nu}\bigg\{\frac{\cos\left(\pi\left(\frac{s}{4}-\frac{\nu}{2}\right)\right)}{\sin(\nu\pi)}
\G\left(1-\frac{s}{4}+\frac{\nu}{2}\right)I_{-\nu}(2\pi n \a)\bigg\}\bigg].
\end{align}}%
Using the series definition of ${}_1F_{2}$, we see that
{\allowdisplaybreaks\begin{align}\label{1deri}
&\left.\frac{d}{d\nu}{}_1F_{2}\left(1;1-\frac{s}{4}-\frac{\nu}{2},1-\frac{s}{4}+\frac{\nu}{2};
\pi^2n^2\a^2\right)\right|_{\nu=s/2}\nonumber\\
&=\sum_{m=0}^{\infty}(\pi n\a)^{2m}\frac{d}{d\nu}\left.\left(\frac{1}{\left(1-\frac{s}{4}-\frac{\nu}{2}\right)_m
\left(1-\frac{s}{4}+\frac{\nu}{2}\right)_m}\right)\right|_{\nu=s/2}\nonumber\\
&=\frac{1}{2}\sum_{m=0}^{\infty}\df{\left(\begin{matrix}-\psi\left(1-\frac{s}{4}-\frac{\nu}{2}\right)
+\psi\left(1+m-\frac{s}{4}-\frac{\nu}{2}\right)\\+\psi\left(1-\frac{s}{4}+\frac{\nu}{2}\right)
-\psi\left(1+m-\frac{s}{4}+\frac{\nu}{2}\right)\end{matrix}\right)}{\left(1-\frac{s}{4}-\frac{\nu}{2}\right)_m
\left(1-\frac{s}{4}+\frac{\nu}{2}\right)_m}(\pi n\a)^{2m}\Bigg|_{\nu=s/2}\nonumber\\
&=\frac{1}{2}\G\left(1-\frac{s}{2}\right)\sum_{m=0}^{\infty}\frac{\left(-\psi\left(1-\frac{s}{2}\right)
+\psi\left(1+m-\frac{s}{2}\right)-\gamma-\psi\left(1+m\right)\right)}{m!\G\left(1+m-\frac{s}2\right)}(\pi n\a)^{2m}\nonumber\\
&=-\frac{(\pi n\a)^{s/2}}{2}\G\left(1-\frac{s}{2}\right)\left(\psi\left(1-\frac{s}{2}\right)+\gamma\right)I_{-s/2}(2\pi n\a)\nonumber\\
&\quad+\frac{1}{2}\G\left(1-\frac{s}{2}\right)\sum_{m=0}^{\infty}\frac{\left(\psi\left(1+m-\frac{s}{2}\right)
-\psi\left(1+m\right)\right)}{m!\G\left(1+m-\frac{s}2\right)}(\pi n\a)^{2m},
\end{align}}%
where in the last step we used the fact $\psi(1)=-\gamma$, and also \eqref{defbesi} and \eqref{sumbesselj}. Next, as $\nu\to s/2$, the last expression in \eqref{0deri} simplifies to
{\allowdisplaybreaks\begin{align}\label{2deri}
&\left.\frac{d}{d\nu}\bigg\{\frac{\cos\left(\pi\left(\frac{s}{4}-\frac{\nu}{2}\right)\right)}{\sin(\nu\pi)}
\G\left(1-\frac{s}{4}+\frac{\nu}{2}\right)I_{-\nu}(2\pi n \a)\bigg\}\right|_{\nu=s/2}\nonumber\\
&=\lim_{\nu\to
  s/2}\bigg[\frac{\G\left(1-\frac{s}{4}+\frac{\nu}{2}\right)}{2\sin(\nu\pi)}I_{-\nu}(2\pi
n\a)\bigg\{\pi\left(\sin\left(\frac{\pi}{4}(s-2\nu)\right)
\right.\left.-2\cos\left(\frac{\pi}{4}(s-2\nu)\right)\cot(\nu\pi)\right)\nonumber\\
&\quad+\cos\left(\frac{\pi}{4}(s-2\nu)\right)\psi\left(1-\frac{s}{4}+\frac{\nu}{2}\right)\bigg\}
+\frac{\cos\left(\frac{\pi}{4}(s-2\nu)\right)\G\left(1-\frac{s}{4}+\frac{\nu}{2}\right)}{\sin(\nu\pi)}\frac{d}{d\nu}I_{-\nu}(2\pi n\a)\bigg]\nonumber\\
&=-\frac{I_{-s/2}(2\pi n\a)}{2\sin\left(\frac{1}{2}\pi s\right)}\left(\gamma+2\pi\cot\left(\frac{\pi
      s}{2}\right)\right)+\frac{1}{\sin\left(\frac{1}{2}\pi s\right)}\left.\frac{d}{d\nu}I_{-\nu}(2\pi n\a)\right|_{\nu=s/2}.
\end{align}}%
Substituting \eqref{1deri} and \eqref{2deri} in \eqref{0deri}, and using \eqref{ref} in the second step below, we find that
{\allowdisplaybreaks\begin{align}\label{3deri}
L&=-\frac{1}{2}\bigg[(\pi\a)^{-s/2}\bigg\{\G\left(\frac{s}{2}\right)\bigg(-\frac{(\pi n\a)^{s/2}}{2}\G\left(1-\frac{s}{2}\right)\left(\psi\left(1-\frac{s}{2}\right)+\gamma\right)I_{-s/2}(2\pi n\a)\nonumber\\
&\quad+\frac{1}{2}\G\left(1-\frac{s}{2}\right)\sum_{m=0}^{\infty}
\frac{\left(\psi\left(1+m-\frac{s}{2}\right)-\psi\left(1+m\right)\right)}{m!\G\left(1+m-\frac{s}2\right)}(\pi n\a)^{2m}\bigg)\nonumber\\
&\quad+\frac{1}{2}\G'\left(\frac{s}{2}\right)
{}_0F_{1}\left(-;1-\frac{s}{2};\pi^2n^2\a^2\right)\bigg\}\nonumber\\
&\quad-\pi n^{s/2}\bigg\{-\frac{I_{-s/2}(2\pi
  n\a)}{2\sin\left(\frac{1}{2}\pi s\right)}\left(\gamma+2\pi\cot\left(\frac{\pi
      s}{2}\right)\right)+\frac{1}{\sin\left(\frac{1}{2}\pi s\right)}\left.\frac{d}{d\nu}I_{-\nu}(2\pi n\a)\right|_{\nu=s/2}\bigg\}\bigg]\nonumber\\
&=-\frac{\pi}{2\sin\left(\frac{1}{2}\pi s\right)}\bigg[(\pi\a)^{-s/2}\bigg\{-\frac{(\pi n\a)^{s/2}}{2}\left(\psi\left(1-\frac{s}{2}\right)+\gamma\right)I_{-s/2}(2\pi n\a)\nonumber\\
&\quad+\frac{1}{2}\sum_{m=0}^{\infty}\frac{\left(\psi\left(1+m-\frac{s}{2}\right)
-\psi\left(1+m\right)\right)}{m!\G\left(1+m-\frac{s}2\right)}(\pi n\a)^{2m}+\frac{(\pi n\a)^{s/2}}{2}\psi\left(\frac{s}{2}\right)I_{-s/2}(2\pi n\a)\bigg\}\nonumber\\
&\quad-n^{s/2}\bigg\{-\frac{1}{2}I_{-s/2}(2\pi n\a)\left(\gamma+2\pi\cot\left(\frac{\pi s}{2}\right)\right)+\left.\frac{d}{d\nu}I_{-\nu}(2\pi n\a)\right|_{\nu=s/2}\bigg\}\bigg]\nonumber\\
&=-\frac{\pi}{2\sin\left(\frac{1}{2}\pi s\right)}\bigg[\frac{n^{s/2}}{2}I_{-s/2}(2\pi n\a)\left(\psi\left(\frac{s}{2}\right)-\psi\left(1-\frac{s}{2}\right)+2\pi\cot\left(\frac{\pi s}{2}\right)\right)\nonumber\\
&\quad+\frac{(\pi\a)^{-s/2}}{2}\sum_{m=0}^{\infty}\frac{\left(\psi\left(1+m-\frac{s}{2}\right)
-\psi\left(1+m\right)\right)}{m!\G\left(1+m-\frac{s}2\right)}(\pi n\a)^{2m}\nonumber\\
&\quad-n^{s/2}\left.\frac{d}{d\nu}I_{-\nu}(2\pi n\a)\right|_{\nu=s/2}\bigg].
\end{align}}%
The reflection formula \eqref{ref} implies that
\begin{equation}\label{refvar}
\psi\left(\frac{s}{2}\right)-\psi\left(1-\frac{s}{2}\right)=-\pi\cot\left(\frac{\pi s}{2}\right).
\end{equation}
Also, for $\nu\neq k$ or $k+\tfrac{1}{2}$, where $k$ is an integer, we have \cite[p.~929, formula \textbf{8.486.4}]{grn}
\begin{equation}\label{deribess}
\frac{\partial I_{\nu}(w)}{\partial\nu}=I_{\nu}(w)\log\left(\frac{w}{2}\right)-\sum_{m=0}^{\infty}
\frac{\psi(\nu+m+1)}{m!\G\left(\nu+m+1\right)}\left(\frac{w}{2}\right)^{\nu+2m}.
\end{equation}
Hence, employing \eqref{refvar} and \eqref{deribess} in \eqref{3deri}, we obtain
{\allowdisplaybreaks\begin{align}\label{4deri}
L&=-\frac{\pi}{2\sin\left(\frac{1}{2}\pi s\right)}\bigg[\frac{\pi
  n^{s/2}\cot\left(\frac{\pi s}{2}\right)}{2}I_{-s/2}(2\pi n\a)\notag\\
&\quad+\frac{(\pi\a)^{-s/2}}{2}\sum_{m=0}^{\infty}\frac{\left(\psi\left(1+m-\frac{s}{2}\right)
-\psi\left(1+m\right)\right)}{m!\G\left(1+m-\frac{s}2\right)}(\pi n\a)^{2m}\nonumber\\
&\quad-n^{s/2}\bigg\{-I_{-s/2}(2\pi n\a)\log(\pi n\a)+\sum_{m=0}^{\infty}\frac{\psi\left(1+m-\frac{s}{2}\right)}{m!\G\left(1+m-\frac{s}{2}\right)}(\pi n\a)^{2m-s/2}\bigg\}\bigg]\nonumber\\
&=-\frac{\pi}{4\sin\left(\frac{1}{2}\pi s\right)}\bigg[n^{s/2}I_{-s/2}(2\pi n\a)\left(\pi\cot\left(\frac{\pi s}{2}\right)+2\log(\pi n\a)\right)\nonumber\\
&\quad-(\pi\a)^{-s/2}\sum_{m=0}^{\infty}\frac{\left(\psi\left(1+m-\frac{s}{2}\right)
+\psi\left(1+m\right)\right)}{m!\G\left(1+m-\frac{s}2\right)}(\pi n\a)^{2m}\bigg].
\end{align}}%
Now substitute \eqref{4deri} in \eqref{-1deri} and use the definition \cite[p.~78, equation (6)]{watsonbessel}
\begin{equation}\label{kdef}
K_{\nu}(w)=\frac{\pi}{2}\frac{I_{-\nu}(w)-I_{\nu}(w)}{\sin(\nu\pi)}
\end{equation}
to deduce that
\begin{align}\label{5deri}
&PV\int_{0}^{\infty}\frac{x^{1+s/2}K_{s/2}(2\pi\a x)}{x^2-n^2}\, dx\nonumber\\
&=-\frac{\pi n^{s/2}}{2}K_{s/2}(2\pi n\a)\cot\left(\frac{\pi s}{2}\right)-\frac{\pi n^{s/2}}{2\sin\left(\frac{\pi s}{2}\right)}I_{-s/2}(2\pi n\a)\log(\pi n\a)\nonumber\\
&\quad+\frac{\pi(\pi\a)^{-s/2}}{4\sin\left(\frac{1}{2}\pi s\right)}\sum_{m=0}^{\infty}\frac{\left(\psi\left(1+m-\frac{s}{2}\right)
+\psi\left(1+m\right)\right)}{m!\G\left(1+m-\frac{s}2\right)}(\pi n\a)^{2m}.
\end{align}%
Hence, from \eqref{-2deri} and \eqref{5deri},
{\allowdisplaybreaks\begin{align}\label{intrepr}
&PV\int_{0}^{\infty}x^{2-s/2}K_{s/2}(2\pi\a x)\left(\frac{n^{s-1}-x^{s-1}}{n^2-x^2}\right)\, dx\nonumber\\
&=\frac{\pi^2n^{s/2}}{4\cos\left(\frac{1}{2}\pi s\right)}I_{-s/2}(2\pi n\a)-\frac{\pi^{\frac{s-1}{2}}\a^{-1+s/2}n^{s-1}}{4}\Gamma\left(\frac{1-s}{2}\right)
{}_1F_{2}\left(1;\frac{1}{2},\frac{1+s}{2};\pi^2n^2\a^2\right)\nonumber\\
&\quad-\frac{\pi n^{s/2}}{2}K_{s/2}(2\pi n\a)\cot\left(\frac{\pi
    s}{2}\right)-\frac{\pi n^{s/2}}{2\sin\left(\frac{1}{2}\pi s\right)}I_{-s/2}(2\pi n\a)\log(\pi n\a)\nonumber\\
&\quad+\frac{\pi(\pi\a)^{-s/2}}{4\sin\left(\frac{1}{2}\pi s\right)}\sum_{m=0}^{\infty}\frac{\left(\psi\left(1+m-\frac{s}{2}\right)
+\psi\left(1+m\right)\right)}{m!\G\left(1+m-\frac{s}2\right)}(\pi n\a)^{2m}.
 \end{align}}%
From \cite[p.~349, equation (3)]{watsonbessel},
\begin{align}\label{lommelcon}
&S_{\nu-1,\nu}(w)=-2^{\nu-2}\pi\G(\nu)Y_{\nu}(w)\\
&\quad+\frac{w^{\nu}}{4}\G(\nu)\sum_{m=0}^{\infty}\frac{(-1)^m(w/2)^{2m}}{m!\G(\nu+m+1)}
\left\{2\log\left(\frac{1}{2}w\right)-\psi(\nu+m+1)-\psi(m+1)\right\},\notag
\end{align}
where $S_{\nu-1,\nu}(w)$ is an exceptional case of the Lommel function
$S_{\mu,\nu}(w)$ defined in \eqref{lommeldef1} and \eqref{lommeldefint}. Dixon and Ferrar \cite[p.~38, equation (3.11)]{dixfer1}
denote the  infinite series on the right-hand side of \eqref{lommelcon} by $Y_{/\nu}(w)$.

Let $w=2\pi in\a$ and $\nu=-\frac{1}{2}s$  in \eqref{lommelcon}, and then use
\eqref{defbesi} and the relation \cite[p.~233, equation (9.27)]{temme}
\begin{equation}\label{yimag}
Y_{\nu}(iw)=e^{\frac{1}{2}(\nu+1)\pi
  i}I_{\nu}(w)-\frac{2}{\pi}e^{-\frac{1}{2}\nu\pi i}K_{\nu}(w),
\end{equation}
which is valid for $-\pi<\arg w\leq\frac{1}{2}\pi$,
to  find after simplification that
\begin{multline}\label{repser}
\sum_{m=0}^{\infty}\frac{\left(\psi\left(1+m-\frac{s}{2}\right)
+\psi\left(1+m\right)\right)}{m!\G\left(1+m-\frac{s}2\right)}(\pi n\a)^{2m}
=2(\pi n\a)^{s/2}I_{-s/2}(2\pi n\a)\log(\pi n\a)\\+2e^{i\pi s/2}(\pi
n\a)^{s/2}K_{s/2}(2\pi n\a)-\frac{4(2\pi
  in\a)^{s/2}}{\G\left(-\frac{1}{2}s\right)}
S_{-1-s/2,-s/2}(2\pi in\a).
\end{multline}
Also, let $\mu=\frac{s}{2}-2, \nu=\frac{s}2$ and $w=2\pi in\a$ in
\eqref{lommeldef1} and \eqref{lommeldef2},
 then substitute the latter into the former, use \eqref{yimag}, and simplify to
 obtain
\begin{multline*}
{}_1F_{2}\left(1;\frac{1}{2},\frac{1+s}{2};\pi^2n^2\a^2\right)
=(2\pi in\a)^{1-s/2}(1-s)
S_{s/2-2,s/2}(2\pi in\a)\\
-\pi^{(3-s)/2}i^{1-s/2}\G\left(\frac{s+1}{2}\right)
\left(ie^{i\pi s/4}I_{s/2}(2\pi n\a)
-\frac{2}{\pi}e^{-i\pi s/4}K_{s/2}(2\pi n\a)\right).
\end{multline*}
Hence, using \eqref{ref2}, we see that
\begin{align}\label{rephy}
&-\frac{\pi^{(s-1)/2}\a^{-1+s/2}n^{s-1}}{4}
\Gamma\left(\frac{1-s}{2}\right){}_1F_{2}\left(1;\frac{1}{2},\frac{1+s}{2};\pi^2n^2\a^2\right)\notag\\
&=-\frac{\pi^2n^{s/2}}{4\cos\left(\frac{1}{2}\pi s\right)}I_{s/2}(2\pi
n\a)-\frac{i\pi n^{s/2}e^{-i\pi s/2}}{2\cos\left(\frac{1}{2}\pi s\right)}K_{s/2}(2\pi n\a)\notag\\
&\quad-i^{1-\frac{s}{2}}\sqrt{\pi}\left(\frac{n}{2}\right)^{s/2}
\G\left(\frac{3-s}{2}\right)
S_{s/2-2,s/2}(2\pi in\a).
\end{align}
Finally substituting \eqref{repser} and \eqref{rephy} in \eqref{intrepr},
using \eqref{ref} and \eqref{kdef}, and
 simplifying, we obtain
\begin{align}\label{tterm}
PV\int_{0}^{\infty}x^{2-s/2}K_{s/2}(2\pi\a x)&\left(\frac{n^{s-1}-x^{s-1}}{n^2-x^2}\right)\, dx
=(2in)^{s/2}\G\left(1+\frac{s}{2}\right)S_{-1-s/2,-s/2}(2\pi in\a)\notag\\
&\quad-i^{1-s/2}\sqrt{\pi}\left(\frac{n}{2}\right)^{s/2}
\G\left(\frac{3-s}{2}\right)S_{-2+s/2,s/2}(2\pi in\a),
\end{align}
that is, the  integral in \eqref{tterm} can be written simply as a linear
combination of
two Lommel functions of imaginary arguments,
 one of which belongs to the exceptional case. (Note that we can replace
  the second subscript $-s/2$ of the first Lommel
  function in \eqref{tterm} by $s/2$, since the Lommel function
  $S_{\mu,\nu}(w)$ is an even function of $\nu$
   (see \cite[p.~348]{watsonbessel}).
From \eqref{ane}, \eqref{fterm}, \eqref{sterm}, and \eqref{tterm2}, we arrive
at \eqref{anee}. This completes the proof of Theorem \ref{newthmser}.
\end{proof}

   Since the integral evaluation \eqref{tterm} is
   new and has not been recorded in the tables of integrals
	such as \cite{grn} and \cite{prud}, we record it below as a theorem
        and rewrite it in terms of the modified Lommel function.

	\begin{theorem}\label{newinteval}
	Let $-2< \sigma<3$ and $y, \a>0$. Let the modified Lommel
        function $T_{\mu,\nu}(y)$ be defined in \eqref{modlf}. Then
\begin{align}\label{tterm2}
&PV\int_{0}^{\infty}x^{2-s/2}K_{s/2}(2\pi\a
x)\left(\frac{y^{s-1}-x^{s-1}}{y^2-x^2}\right)\, dx\\
&=(2y)^{s/2}\G\left(1+\frac{s}{2}\right)T_{-1-s/2,-s/2}(2\pi \a y)-\sqrt{\pi}\left(\frac{y}{2}\right)^{s/2}\G\left(\frac{3-s}{2}\right)T_{-2+s/2,s/2}(2\pi\a y).\notag
\end{align}
\end{theorem}

\textbf{Remark:} The Ramanujan--Guinand formula \cite[Theorem 1.2, Theorem 1.4]{transf}, containing
infinite series involving $\sigma_{s}(n)$ and the modified Bessel function
$K_{\nu}(z)$, is similar to the series occuring in the Fourier expansion
of non-holomorphic Eisenstein series \cite{cohen}, \cite[p.~243]{lewzag}. In view of the fact that the modular relation involving series of Lommel functions
appearing in Lewis and Zagier's work \cite[p.~217]{lewzag} characterizes the Fourier coefficients of even Maass forms \cite[p.~216, Proposition 3]{lewzag}, we surmise that the modular transformation consisting of the series involving $\sigma_{s}(n)$ and the modified Lommel
function that we have obtained in Theorem \ref{newthmser} may also have some important implications in the theory of Maass forms.

\subsection{An Integral Representation for $f(x,s)$ Defined in \eqref{fxzsp} and an Equivalent Formulation of Theorem \ref{newthm}}\label{iref}

Some elegant integral transformations involving the function $\varphi(x, s)$
defined in \eqref{plus} or, in particular, the function $\varphi(x,
s)-x^{s/2-1}\zeta(s)/(2\pi)$, were established in \cite[Theorems 6.3,
6.4]{dixitmoll}. In this subsection, we derive an equivalent form of Theorem
\ref{newthm} which yields yet another integral transformation involving this
function. These integrals also involve the Lommel function $S_{\nu,\nu}(w)$.

We first derive an integral representation for $f(x,s)$ defined in
\eqref{fxzsp}.
\begin{theorem}\label{fxsintrep}
Let $f(x,s)$ and $\varphi(x, s)$ be defined in \eqref{fxzsp} and \eqref{plus},
respectively.
Then, for $-1<\sigma<1$,
\begin{align}\label{fxzspint}
f(x,s)=\frac{2}{\pi}x^{-s/2}\int_{0}^{\infty}\frac{t^{1+s/2}}{x^2+t^2}\left(\varphi(t,
  s)-\frac{\zeta(s)}{2\pi}t^{s/2-1}\right)\, dt.
\end{align}
\end{theorem}

\begin{proof}
Note that from \cite[equation (6.9)]{dixitmoll}, for $\pm\tf12\sigma<c=
\text{Re }z<1\pm\tf12\sigma$,
\begin{align}\label{omegamellin1}
\frac{1}{2\pi i}\int_{c-i\infty}^{c+i\infty}\frac{\zeta\left(1-z+\frac{1}{2}s\right)\zeta\left(1-z-\frac{1}{2}s\right)}
{2\cos\left(\frac{1}{2}\pi\left(z+\frac{1}{2}s\right)\right)}x^{-z}\, dz=\varphi(x, s)-\frac{\zeta(s)}{2\pi}x^{s/2-1}.
\end{align}
Let $s=1$ in \eqref{betamel}, replace $x$ by $x^2$ and $w$ by
$\frac{1}{2}z-\frac{1}{4}s$, and use \eqref{ref} to find that for
$\tf12\sigma<d'= \text{Re } z<2+\tf12\sigma$,
\begin{align}\label{sinmell}
\frac{1}{2\pi i}\int_{d'-i\infty}^{d'+i\infty}\frac{\pi}{\sin\left(\frac{1}{2}\pi\left(z-\frac{1}{2}s\right)\right)}x^{-z}\, dz=\frac{2x^{-s/2}}{1+x^2}.
\end{align}
Hence, for $\pm\tf12\sigma<c= \text{Re } z<1\pm\tf12\sigma$,
\begin{align}\label{genesisfxs}
f(x,s) &= \frac{1}{2 \pi i} \int_{c - i \infty}^{c + i \infty}
\frac{\zeta(1 - z -\tf12s) \zeta(1-z + \tf12s)}{\sin(\pi
  z)-\sin\left(\frac{1}{2}\pi s\right)}x^{-z} \, dz\nonumber\\
&=\frac{1}{\pi}\frac{1}{2 \pi i} \int_{c - i \infty}^{c + i \infty}
\frac{\zeta(1 - z -\tf12s) \zeta(1-z + \tf12s)}{2\cos\left(\frac{1}{2}\pi\left(z+\frac{1}{2}s\right)\right)}\frac{\pi}{\sin\left(\frac{1}{2}\pi
\left(z-\frac{1}{2}s\right)\right)}\, dz\nonumber\\
&=\frac{2}{\pi}x^{-s/2}\int_{0}^{\infty}\frac{t^{1+s/2}}{x^2+t^2}\left(\varphi(t, s)-\frac{\zeta(s)}{2\pi}t^{s/2-1}\right)\, dt,
\end{align}
where in the last step we used \eqref{melconv}.
\end{proof}

Now we are in a position to state and prove an equivalent formulation of
Theorem \ref{newthm}.
\begin{theorem}
Let $S_{\mu,\nu}(w)$ be the Lommel function defined in \eqref{lommeldef1}. For $\a,\b>0$, $\a\b=1$, and $-1<\sigma<1$,
{\allowdisplaybreaks\begin{align}\label{newthmlomint}
&\sqrt{\a}\int_{0}^{\infty}S_{s/2,s/2}(2\pi\a t)\left(\varphi(t, s)-\frac{\zeta(s)}{2\pi}t^{s/2-1}\right)\, dt\nonumber\\
&=\sqrt{\b}\int_{0}^{\infty}S_{s/2,s/2}(2\pi\b t)\left(\varphi(t, s)-\frac{\zeta(s)}{2\pi}t^{s/2-1}\right)\, dt\nonumber\\
&=\frac{2^{s/2-2}\pi^{-\frac{5}{2}}}{\G\left(\frac{1-s}{2}\right)}\int_{0}^{\infty}
\G\left(\frac{s-1+it}{4}\right)\G\left(\frac{s-1-it}{4}\right)
\G\left(\frac{-s+1+it}{4}\right)\G\left(\frac{-s+1-it}{4}\right)\nonumber\\
&\quad\quad\quad\quad\quad\times\Xi\left(\frac{t-is}{2}\right)\Xi\left(\frac{t+is}{2}\right)
\frac{\cos\left(\tfrac{1}{2}t\log\a\right)}{t^2+(s+1)^2}\, dt.
\end{align}}
\end{theorem}
\begin{proof}
Using \eqref{fxzspint}, we write the extreme left side in \eqref{ane} as
\begin{align}\label{intmd}
&\sqrt{\alpha} \int_{0}^{\infty} K_{s/2}(2 \pi \alpha x) f(x,s)\, dx\nonumber\\
&=\frac{2\sqrt{\alpha}}{\pi}\int_{0}^{\infty} x^{-s/2}K_{s/2}(2 \pi \alpha x)\int_{0}^{\infty}\frac{t^{1+s/2}}{x^2+t^2}
\left(\varphi(t, s)-\frac{\zeta(s)}{2\pi}t^{s/2-1}\right)\, dt\, dx\nonumber\\
&=\frac{2\sqrt{\alpha}}{\pi}\int_{0}^{\infty}t^{1+s/2}\left(\varphi(t, s)-\frac{\zeta(s)}{2\pi}t^{s/2-1}\right)\int_{0}^{\infty} \frac{x^{-s/2}K_{s/2}(2 \pi \alpha x)}{x^2+t^2}\, dx\, dt.
\end{align}
The interchange of the order of integration given above is delicate and hence
we justify it below.

By Fubini's theorem \cite[p.~30, Theorem 2.2]{temme}, it suffices to show
that each of the two double integrals are absolutely convergent. We begin with the first one. Thus,
\begin{align*}
&\int_{0}^{\infty}\left| x^{-s/2}K_{s/2}(2 \pi \alpha x)\right|\int_{0}^{\infty}\left|\frac{t^{1+s/2}}{x^2+t^2}\left(\varphi(t, s)-\frac{\zeta(s)}{2\pi}t^{s/2-1}\right)\right|\, dt\, dx\nonumber\\
&=\int_{0}^{1}x^{-\sigma/2}\left| K_{s/2}(2 \pi \alpha x)\right|\int_{0}^{\infty}\frac{t^{1+\sigma/2}}{x^2+t^2}\left|\varphi(t, s)-\frac{\zeta(s)}{2\pi}t^{s/2-1}\right|\, dt\, dx\nonumber\\
&\quad+\int_{1}^{\infty}x^{-\sigma/2}\left| K_{s/2}(2 \pi \alpha x)\right|\int_{0}^{\infty}\frac{t^{1+\sigma/2}}{x^2+t^2}\left|\varphi(t, s)-\frac{\zeta(s)}{2\pi}t^{s/2-1}\right|\, dt\, dx\nonumber\\
&=:I_{1}(s,\a)+I_{2}(s,\a).
\end{align*}
Consider $I_{2}(s,\a)$ first. Note that
\begin{align*}
|I_2(s,\alpha)|&=\int_{1}^{\infty}x^{-\sigma/2}\left| K_{s/2}(2 \pi \alpha x)\right|\int_{0}^{\infty}\frac{t^{1+\sigma/2}}{x^2+t^2}\left|\varphi(t, s)-\frac{\zeta(s)}{2\pi}t^{s/2-1}\right|\, dt\, dx\nonumber\\
&\leq \int_{1}^{\infty}x^{-\sigma/2}\left| K_{s/2}(2 \pi \alpha x)\right|\int_{0}^{\infty}\frac{t^{1+\sigma/2}}{1+t^2}\left|\varphi(t, s)-\frac{\zeta(s)}{2\pi}t^{s/2-1}\right|\, dt\, dx.
\end{align*}
Now use the definition of $\varphi(t,s)$ in \eqref{plus}, the asymptotic
expansion of $K_{\nu}(z)$ from \eqref{asymbess2}, and the fact that $\sigma<1$
to see that the inner integral converges as $t\to\infty$. To
analyze the behavior of this integral as $t\to 0$, we observe from
\eqref{one} and \eqref{Lambda} that
\begin{align}\label{exppts}
\varphi(t,s)-\frac{\zeta(s)}{2 \pi}t^{s/2-1}  &= - \frac{\Gamma(s) \zeta(s)t^{-s/2}}{(2 \pi)^{s}}
 - \frac{t^{s/2}}{2} \zeta(s+1)  +
\frac{t^{s/2+1}}{\pi} \sum_{n=1}^{\infty} \frac{\sigma_{-s}(n)}{n^{2}+t^{2}}\nonumber\\
&= -\frac{\Gamma(s) \zeta(s)t^{-s/2}}{(2 \pi)^{s}} -
\frac{t^{s/2}}{2} \zeta(s+1)\nonumber\\
&\quad  +
\frac{t^{s/2+1}}{\pi} \sum_{n=0}^{\infty} (-1)^n\zeta(2n+2)\zeta(2n+2+s)t^{2n},
\end{align}
where the last equality holds for $|t|<1$. In the last step, we expanded
$1/(1+t^2/n^2)$ in a geometric series, interchanged the order of summation, and
then used \eqref{sz}. Now $\sigma>-1$ implies the convergence of this
integral
near $0$.
Thus,
\begin{equation}\label{boo}
\int_{0}^{\infty}\frac{t^{1+\sigma/2}}{1+t^2}\left|\varphi(t, s)-\frac{\zeta(s)}{2\pi}t^{s/2-1}\right|\, dt=O_{s}(1).
\end{equation}
 Using \eqref{asymbess2} and \eqref{boo}, we conclude that
$I_{2}(s,\alpha)$ converges.

Now consider $I_{1}(s,\alpha)$. Split the inner integral as
\begin{align}\label{i1salpha}
I_{1}(s,\a)&=\int_{0}^{1}x^{-\sigma/2}\left| K_{s/2}(2 \pi \alpha x)\right|\bigg\{\int_{0}^{x}\frac{t^{1+\sigma/2}}{x^2+t^2}\left|\varphi(t, s)-\frac{\zeta(s)}{2\pi}t^{s/2-1}\right|\, dt\nonumber\\
&\quad\quad\quad\quad\quad\quad\quad\quad\quad\quad\quad\quad+\int_{x}^{\infty}\frac{t^{1+\sigma/2}}{x^2+t^2}\left|\varphi(t, s)-\frac{\zeta(s)}{2\pi}t^{s/2-1}\right|\, dt\bigg\}\, dx\nonumber\\
&=:I_{3}(s,\a)+I_{4}(s,\a).
\end{align}
Using \eqref{exppts}, we see that
\begin{align}\label{bd1}
\int_{0}^{x}\frac{t^{1+\sigma/2}}{x^2+t^2}\left|\varphi(t, s)-\frac{\zeta(s)}{2\pi}t^{s/2-1}\right|\, dt=O_{s}\left(\frac{1}{x^2}\int_{0}^{x}t^{1+\sigma/2}\left(t^{\sigma/2}+t^{-\sigma/2}\right)\, dt\right).
\end{align}
Since $0<t<x<1$, if $0\leq\sigma<1$, then $t^{\sigma/2}+t^{-\sigma/2}=O\left(t^{-\sigma/2}\right)$, and so
\begin{align}\label{bd2}
\int_{0}^{x}\frac{t^{1+\sigma/2}}{x^2+t^2}\left|\varphi(t, s)-\frac{\zeta(s)}{2\pi}t^{s/2-1}\right|\, dt=O_{s}\left(\frac{1}{x^2}\int_{0}^{x}t\, dt\right)=O_{s}(1),
\end{align}
whereas if $-1<\sigma<0$, then $t^{\sigma/2}+t^{-\sigma/2}=O\left(t^{\sigma/2}\right)$, and hence
\begin{align*}
\int_{0}^{x}\frac{t^{1+\sigma/2}}{x^2+t^2}\left|\varphi(t, s)-\frac{\zeta(s)}{2\pi}t^{s/2-1}\right|\, dt=O_{s}\left(\frac{1}{x^2}\int_{0}^{x}t^{1+\sigma}\, dt\right)=O_{s}(x^{\sigma}).
\end{align*}
Thus, from \eqref{bd1} and \eqref{bd2},
\begin{align}\label{firstintegral}
I_{3}(s,\a)&=\int_{0}^{1}x^{-\sigma/2}\left| K_{s/2}(2 \pi \alpha x)\right|\int_{0}^{x}\frac{t^{1+\sigma/2}}{x^2+t^2}\left|\varphi(t, s)-\frac{\zeta(s)}{2\pi}t^{s/2-1}\right|\, dt\, dx\nonumber\\
&=O_{s}\left(\int_{0}^{1}\left(x^{\sigma/2}+x^{-\sigma/2}\right)\left| K_{s/2}(2 \pi \alpha x)\right|\, dx\right).
\end{align}
Lastly, it remains to consider $I_{4}(s,\a)$. Since $\displaystyle\frac{1}{x^2+t^2}\leq\frac{1}{t^2}$,
\begin{align}\label{secondintegral}
I_{4}(s,\a)\leq\int_{0}^{1}x^{-\sigma/2}\left| K_{s/2}(2 \pi \alpha x)\right|\int_{x}^{\infty}t^{\sigma/2-1}\left|\varphi(t, s)-\frac{\zeta(s)}{2\pi}t^{s/2-1}\right|\, dt\, dx.
\end{align}
Next, split the latter integral into two integrals, one having the limits of its inner integral from $t=x$ to $t=1$, and the other from $t=1$ to $t=\infty$. Since $\sigma<1$, it is easy to see that
\begin{equation*}
\int_{1}^{\infty}t^{\sigma/2-1}\left|\varphi(t, s)-\frac{\zeta(s)}{2\pi}t^{s/2-1}\right|\, dt=O_{s}(1),
\end{equation*}
and thus
\begin{align}\label{secint1}
&\int_{0}^{1}x^{-\sigma/2}\left| K_{s/2}(2 \pi \alpha x)\right|\int_{1}^{\infty}t^{\sigma/2-1}\left|\varphi(t, s)-\frac{\zeta(s)}{2\pi}t^{s/2-1}\right|\, dt\, dx\nonumber\\
&=O_{s}\left(\int_{0}^{1}x^{-\sigma/2}\left| K_{s/2}(2 \pi \alpha x)\right|\, dx\right)\nonumber\\
&=O_{s}\left(\int_{0}^{1}\left(x^{\sigma/2}+x^{-\sigma/2}\right)\left| K_{s/2}(2 \pi \alpha x)\right|\, dx\right).
\end{align}
Now
\begin{align*}
&\int_{0}^{1}x^{-\sigma/2}\left| K_{s/2}(2 \pi \alpha x)\right|\int_{x}^{1}t^{\sigma/2-1}\left|\varphi(t, s)-\frac{\zeta(s)}{2\pi}t^{s/2-1}\right|\, dt\, dx\nonumber\\
&=\int_{0}^{1}x^{-\sigma/2}\left| K_{s/2}(2 \pi \alpha x)\right|\int_{x}^{1/2}t^{\sigma/2-1}\left|\varphi(t, s)-\frac{\zeta(s)}{2\pi}t^{s/2-1}\right|\, dt\, dx\nonumber\\
&\quad+\int_{0}^{1}x^{-\sigma/2}\left| K_{s/2}(2 \pi \alpha x)\right|\int_{1/2}^{1}t^{\sigma/2-1}\left|\varphi(t, s)-\frac{\zeta(s)}{2\pi}t^{s/2-1}\right|\, dt\, dx.
\end{align*}
It is easy to see that
\begin{equation}\label{halfto1}
\int_{1/2}^{1}t^{\sigma/2-1}\left|\varphi(t, s)-\frac{\zeta(s)}{2\pi}t^{s/2-1}\right|\, dt=O_{s}(1).
\end{equation}
Observe that another application of \eqref{exppts} yields
\begin{align}\label{xtohalf}
\int_{x}^{1/2}t^{\sigma/2-1}\left|\varphi(t, s)-\frac{\zeta(s)}{2\pi}t^{s/2-1}\right|\, dt&=O_{s}\left(\int_{x}^{1/2}t^{\sigma/2-1}\left(t^{\sigma/2}+t^{-\sigma/2}+t^{1+\sigma/2}\right)\, dt\right)\nonumber\\
&=O_{s}(x^{\sigma})+O_{s}\left(\left|\log x\right|\right)+O_{s}(x^{\sigma+1}).
\end{align}
However, since $0<x<1$ and $-1<\sigma<1$, $O_{s}(x^{\sigma+1})=O_{s}(x^{\sigma})$. Thus, combining this with \eqref{halfto1} and \eqref{xtohalf}, we arrive at
\begin{align*}
\int_{x}^{1}t^{\sigma/2-1}\left|\varphi(t, s)-\frac{\zeta(s)}{2\pi}t^{s/2-1}\right|\, dt=O_{s}(x^{\sigma})+O_{s}(1)+O_{s}\left(\left|\log x\right|\right).
\end{align*}
This in turn gives
\begin{align}\label{secint2}
&\int_{0}^{1}x^{-\sigma/2}\left| K_{s/2}(2 \pi \alpha x)\right|\int_{x}^{1}t^{\sigma/2-1}\left|\varphi(t, s)-\frac{\zeta(s)}{2\pi}t^{s/2-1}\right|\, dt\, dx\nonumber\\
&=O_{s}\left(\int_{0}^{1}\left(x^{\sigma/2}+x^{-\sigma/2}+x^{-\sigma/2}\left|\log x\right|\right)\left| K_{s/2}(2 \pi \alpha x)\right|\, dx\right).
\end{align}
From \eqref{i1salpha}, \eqref{firstintegral}, \eqref{secondintegral}, \eqref{secint1}, and \eqref{secint2}, we deduce that
\begin{align}\label{i1salphaexp}
I_{1}(s,\alpha)=O_{s}\left(\int_{0}^{1}\left(x^{\sigma/2}+x^{-\sigma/2}+x^{-\sigma/2}\left|\log x\right|\right)\left| K_{s/2}(2 \pi \alpha x)\right|\, dx\right).
\end{align}
 From \cite[p.~375, equations (9.6.9), (9.6.8)]{stab}, as $y\to 0$, $K_{\nu}(y)\sim 2^{\nu-1}\G(\nu)y^{-\nu}$, when $\textup{Re }\nu>0$, and $K_{0}(y)\sim -\log y$. Hence,
\begin{equation*}
K_{s/2}(2\pi\a x)=
\begin{cases}
O_{s,\a}\left(x^{-|\sigma|/2}\right), \quad&\text{if}\hspace{2mm} s\neq 0,\\
O_{\a}(\log x), &\text{if}\hspace{2mm} s=0.
\end{cases}
\end{equation*}
If $s\neq 0$, then
\begin{align*}
&\int_{0}^{1}\left(x^{\sigma/2}+x^{-\sigma/2}+x^{-\sigma/2}\left|\log x\right|\right)\left| K_{s/2}(2 \pi \alpha x)\right|\, dx\nonumber\\
&=O_{s,\alpha}\left(\int_{0}^{1}\left(1+x^{-\sigma}+x^{-\sigma}\left|\log x\right|\right)\, dx\right)
=O_{s,\alpha}(1),
\end{align*}
since $\sigma<1$. If $s=0$, then
\begin{align*}
&\int_{0}^{1}\left(x^{\sigma/2}+x^{-\sigma/2}+x^{-\sigma/2}\left|\log x\right|\right)\left| K_{s/2}(2 \pi \alpha x)\right|\, dx\nonumber\\
&=O_{\alpha}\left(\int_{0}^{1}\left(x^{\sigma/2}+x^{-\sigma/2}+x^{-\sigma/2}\left|\log x\right|\right)\left|\log x\right|\, dx\right)\nonumber\\
&=O_{s,\alpha}(1),
\end{align*}
as $\sigma>-1$.

Along with \eqref{i1salphaexp}, this finally implies that $I_{1}(s,\alpha)$ converges. Hence, the double integral
\begin{equation*}
\int_{0}^{\infty} x^{-s/2}K_{s/2}(2 \pi \alpha x)\int_{0}^{\infty}\frac{t^{1+s/2}}{x^2+t^2}\left(\varphi(t, s)-\frac{\zeta(s)}{2\pi}t^{s/2-1}\right)\, dt\, dx
\end{equation*}
converges absolutely. Similarly, it can be shown that
\begin{equation*}
\int_{0}^{\infty}t^{1+s/2}\left(\varphi(t, s)-\frac{\zeta(s)}{2\pi}t^{s/2-1}\right)\int_{0}^{\infty} \frac{x^{-s/2}K_{s/2}(2 \pi \alpha x)}{x^2+t^2}\, dx\, dt
\end{equation*}
converges absolutely. This allows us to apply Fubini's theorem and
interchange the order of integration in \eqref{intmd}.

Now from \cite[p.~346, equation \textbf{2.16.3.16}]{prud}, for Re $u>0$, Re
$y>0$, and $\pm$ Re $\nu>-\tf12$,
\begin{align}\label{anotev}
\int_{0}^{\infty}\frac{x^{\pm \nu}}{x^2+y^2}K_{\nu}(ux)\,
dx=\frac{\pi^2y^{\pm\nu-1}}{4\cos(\nu\pi)}\left(\mathbf{H}_{\mp\nu}(uy)-Y_{\mp\nu}(uy)\right),
\end{align}
where $H_{\nu}(w)$ is the first Struve function defined by \cite[p.~942,
formula \textbf{8.550.1}]{grn}
\begin{equation*}
\mathbf{H}_{\nu}(w):=\sum_{m=0}^{\infty}(-1)^m\frac{(w/2)^{2m+\nu+1}}
{\G\left(m+\frac{3}{2}\right)\G\left(\nu+m+\frac{3}{2}\right)}.
\end{equation*}
Also from \cite[p.~42, formula (84)]{htf},
\begin{equation}\label{htfref}
\mathbf{H}_{\nu}(w)=Y_{\nu}(w)+\frac{2^{1-\nu}\pi^{-1/2}}{\G\left(\nu+\frac{1}{2}\right)}S_{\nu,\nu}(w).
\end{equation}
Hence, from \eqref{intmd}, \eqref{anotev}, \eqref{htfref}, and \eqref{ref2},
\begin{align}\label{intmd1}
&\sqrt{\alpha} \int_{0}^{\infty} K_{s/2}(2 \pi \alpha x) f(x,s) \notag dx\\&=\dfrac{\G\left(\frac{1}{2}(1-s)\right)\sqrt{\a}}{2^{s/2}\sqrt{\pi}}\int_{0}^{\infty}
S_{s/2,s/2}(2\pi\a t)\left(\varphi(t, s)-\frac{\zeta(s)}{2\pi}t^{s/2-1}\right)\, dt.
\end{align}
Using \eqref{ane} and \eqref{intmd1}, we obtain \eqref{newthmlomint}.
\end{proof}

\textbf{Remark:} From \eqref{fxzsp}, \eqref{fxzspint}, and the evaluation \cite[p.~345, formula \bf{(14)}]{erdelyi}
\begin{equation*}
\int_{0}^{\infty}\frac{t^s}{x^2+t^2}\, dt=\frac{\pi
  x^{s-1}}{2\cos\left(\frac{1}{2}\pi s\right)},\quad -1<\sigma<1,
\end{equation*}
we find that
\begin{equation}\label{fxzspint1}
\frac{2}{\pi}x^{-s/2}\int_{0}^{\infty}\frac{t^{1+s/2}\varphi(t,
  s)}{x^2+t^2}\,
dt=f(x,s)+\frac{x^{s/2-1}\zeta(s)}{2\pi\cos\left(\frac{1}{2}\pi s\right)}.
\end{equation}
Koshliakov's formula \cite[equation (7)]{koshli}\footnote{Note that there is
  a misprint in Koshliakov's formulation of this identity, namely, there is a
  plus sign in front of the infinite sum, which actually should be a minus sign.}, namely,
\begin{equation*}
\int_{0}^{\infty}\frac{t\varphi(t,0)}{1+t^2}\, dt=\gamma^2-2\gamma_1-\frac{1}{4}+\frac{\pi^2}{6}-\sum_{n=1}^{\infty}d(n)\left(\frac{1}{n}-\frac{1}{n+1}\right)
\end{equation*}
is a special case of \eqref{fxzspint1} when $s=0$ and $x=1$.

\section{Some Results Associated with the Second Koshliakov Transform}\label{secondkoshliakov}

In Section \ref{anmt}, we studied a modular transformation associated with
the function $f(x,s)$ defined in \eqref{fxzsp}, which in turn, was found by
choosing the function $F(z,s)$ given in \eqref{Fzsdef}. Then in
\eqref{fxzspint} of Section \ref{iref}, we found an integral representation
for  $f(x,s)$. This, however, completely obscures the discovery of this
choice of $F(z,s)$, and in fact, it was discovered by first considering
\eqref{omegamellin1} and asking ourselves, what function needs to be
multiplied with
$1/\left(2\cos\left(\frac{\pi}{2}\left(z+\frac{s}{2}\right)\right)\right)$ in
order to form a function $F(z,s)$ that satisfies $F(z,s)=F(1-z,s)$ and hence
to be able to use Theorem \ref{production-1}. This motivated us to consider
the function defined by the integral in the second equality in
\eqref{genesisfxs}, which we then evaluated in two different ways, leading to
$f(x,s)$ and its aforementioned integral representation.

With the same intention of constructing an $F(z,s)$ satisfying $F(z,s)=F(1-z,s)$ towards applying Theorem \ref{production-2} and obtaining a modular transformation, we introduce a function defined by means of the integral
\begin{equation*}
\frac{1}{2\pi i}\int_{c'-i\infty}^{c'+i\infty}\frac{2\sin\left(\frac{1}{2}\pi\left(z-\frac{1}{2}s\right)\right)}{(2\pi)^{2z}}
\Gamma\left(z-\frac{s}{2}\right)\Gamma\left(z+\frac{s}{2}\right)\zeta\left(z-\frac{s}{2}\right)
\zeta\left(z+\frac{s}{2}\right)x^{-z}\, dz
\end{equation*}
for $\pm\tf12\sigma<c'= \textup{Re } z<1\pm \tf12\sigma$. We first evaluate this integral in the half-plane Re $z>1\pm\tfrac{1}{2}\sigma$ as well as in the vertical strip $\pm\tf12\sigma<c'= \textup{Re } z<1\pm \tf12\sigma$ . It is easy to see, using \eqref{zetab}, \eqref{sinest}, and \eqref{strivert}, that this integral converges in both of these regions.

\begin{theorem}
For $c'= \textup{Re } z>1\pm\tf12\sigma$,
\begin{multline}\label{fxzprod2sim0}
\frac{1}{2\pi i}\int_{c'-i\infty}^{c'+i\infty}\frac{2\sin\left(\frac{1}{2}\pi\left(z-\frac{1}{2}s\right)\right)}{(2\pi)^{2z}}
\Gamma\left(z-\frac{s}{2}\right)\Gamma\left(z+\frac{s}{2}\right)\zeta\left(z-\frac{s}{2}\right)\zeta
\left(z+\frac{s}{2}\right)x^{-z}\, dz\\=\eta(x,s),
\end{multline}
where
\begin{equation*}
\eta(x, s):=2i\sum_{n=1}^{\infty}\sigma_{-s}(n)n^{s/2}\left(e^{\pi is/4}K_{s}\left(4\pi e^{\pi i/4}\sqrt{nx}\right)-e^{-\pi is/4}K_{s}\left(4\pi e^{-\pi i/4}\sqrt{nx}\right)\right).
\end{equation*}
Moreover, if
\begin{equation}\label{kappa}
\kappa(x,s):=\eta(x,s)-\frac{1}{2\pi^2x}\left(\frac{\G(1+s)\zeta(1+s)}{(2\pi\sqrt{x})^{s}}
+\frac{\G(1-s)\zeta(1-s)\cos\left(\frac{\pi s}{2}\right)}{(2\pi\sqrt{x})^{-s}}\right),
\end{equation}
then for $\pm\tf12\sigma<c= \textup{Re } z<1\pm\tf12\sigma$,
\begin{multline}\label{fxzprod2sim}
\frac{1}{2\pi i}\int_{c-i\infty}^{c+i\infty}\frac{2\sin\left(\frac{1}{2}\pi\left(z-\frac{1}{2}s\right)\right)}{(2\pi)^{2z}}
\Gamma\left(z-\frac{s}{2}\right)\Gamma\left(z+\frac{s}{2}\right)\zeta\left(z-\frac{s}{2}\right)
\zeta\left(z+\frac{s}{2}\right)x^{-z}\, dz\\=\kappa(x,s).
\end{multline}
\end{theorem}

\textbf{Remark:} Note that $\eta(x,s)$ is merely $i$ times the series in
\eqref{minus}.

\begin{proof}
The proof of \eqref{fxzprod2sim0} is similar to that given by Koshliakov in
\cite[equation (11)]{koshlond} for the special case $s=0$ of the series
\eqref{plus}, and employs \eqref{melkz} along with \eqref{prodze}. Then
\eqref{fxzprod2sim} follows from \eqref{fxzprod2sim0} by shifting the line of
integration from $c'= \textup{Re } z>1\pm\tf12\sigma$ to $\pm\tf12\sigma<c=
\textup{Re } z<1\pm\tf12\sigma$,  considering the contribution of the poles
of the integrand at $1\pm\frac{1}{2}s$,  and employing the residue theorem.
\end{proof}

Given below is an analogue of Theorem \ref{fxsintrep}.
\begin{theorem}\label{thmkappa}
Let $\kappa(x,s)$ be defined in \eqref{kappa}, and let
\begin{align*}
\Phi(x,s)&:=2\sum_{n=1}^{\infty}\sigma_{-s}(n)n^{s/2}K_{s}(4\pi\sqrt{nx})-\frac{\Gamma(1+s)
\zeta(1+s)x^{-1-s/2}}{(2\pi)^{2+s}}\\&\quad-\frac{\Gamma(1-s)\zeta(1-s)x^{-1+s/2}}{(2\pi)^{2-s}}
\end{align*}
Then, for $-1<\sigma<1$,
\begin{equation*}
\Phi(x,s)=\frac{x^{-s/2}}{\pi}\int_{0}^{\infty}\frac{t^{1+s/2}\kappa(t,s)}{t^2+x^2}\, dt.
\end{equation*}
\end{theorem}

\begin{proof}
The proof simply uses the identity \cite[equation (5.26)]{dixitmoll}
\begin{equation*}
\Phi(x,s)=\frac{1}{2\pi i}\int_{c-i\infty}^{c+i\infty}\Gamma\left(z-\frac{s}{2}\right)\Gamma\left(z+\frac{s}{2}\right)
\zeta\left(z-\frac{s}{2}\right)\zeta\left(z+\frac{s}{2}\right)\frac{x^{-z}}{(2\pi)^{z}}\, dz,
\end{equation*}
valid for $\pm\tf12\sigma<c= \textup{Re } z<1\pm\tf12\sigma$, \eqref{sinmell}, which is valid for $\tf12\sigma<d'= \text{Re }z<2+ \tf12\sigma$, and \eqref{melconv}.
\end{proof}
As of now, we have been unable to use \eqref{fxzprod2sim} to construct an $f(x,s)$ in \eqref{fxzprod2} to produce a modular
transformation of the form \eqref{cor5.6result}.

\section{The Second Identity on Page 336}\label{sect15}

On page 336 of his lost notebook, Ramanujan claims the following:

\begin{quote} \emph{Let $\sigma_s(n)=\sum_{d|n}d^s$ and let $\zeta(s)$ denote the Riemann
  zeta function. If $\alpha$, and $\beta$ are positive numbers such that
  $\a\b=4\pi^2$, then}
\begin{align}\label{p336id2}
&\a^{(s+1)/2}\left\{\frac{1}{\a}\zeta(1-s)
+\frac{1}{2}\zeta(-s)\tan\left(\tfrac{1}{2}\pi s\right)+\sum_{n=1}^{\infty}\sigma_s(n)\sin
n\a\right\}\nonumber\\
&=\b^{(s+1)/2}\left\{\frac{1}{\b}\zeta(1-s)
+\frac{1}{2}\zeta(-s)\tan\left(\tfrac{1}{2}\pi
  s\right)+\sum_{n=1}^{\infty}\sigma_s(n)\sin n\b\right\}.
\end{align}%
\end{quote}

As remarked in \cite{bclz}, this formula is easily seen to be false because the series are divergent.

Fix an $s$. If a correct version of Ramanujan's identity
\eqref{p336id2} exists, we believe that it should be a special case of
\eqref{guisum1}, where  $G(x, s)=H(x, s)=f(x,s)$, and $f(x,s)$ is self-reciprocal with respect to
the kernel
\begin{equation*}
-2\pi\sin\left(\tfrac{1}{2}\pi
  s\right)J_{s}(4\pi\sqrt{xt})-\cos\left(\tfrac{1}{2}\pi s\right)\left(2\pi
  Y_{s}(4\pi\sqrt{xt})-4K_{s}(4\pi\sqrt{xt})\right);
\end{equation*}
in other words, $f$ is equal to its first Koshliakov transform.

The appearance of $\tan\left(\frac{1}{2}\pi s\right)$ in Theorem \ref{newthm}
of Section \ref{ktmt} is pleasing when compared with \eqref{p336id2}. Thus, a
series analogue of this theorem (as attempted in Section \ref{asa}) or a
series transformation through Guinand's formula \eqref{guisum1} with the
choice of $G(x, s)=H(x, s)=f(x,s)$, with $f(x,s)$ defined in \eqref{fxzsp}, may shed
some light on \eqref{p336id2}, provided, of course, a correct version of \eqref{p336id2} does exist.

\bigskip

\begin{center}
\textbf{Acknowledgements}
\end{center}

\medskip

 A significant part of this research was accomplished when the
second author visited the University of Illinois at Urbana-Champaign in
March, July and November of 2013, and October, 2014. He is grateful to the
University of
Illinois, and especially to Professor Bruce C.~Berndt, for hospitality and
support.


\begin{thebibliography}{00}

\bibitem{stab}
M.~Abramowitz and I.~A.~Stegun, Handbook of Mathematical Functions, with Formulas, Graphs, and Mathematical Tables, 9th edition, Dover publications, New York, 1970.

\bibitem{assjsv}
N.~T.~Adelman, Y.~Stavsky and E.~Segal, \emph{Axisymmetric vibrations of radially polarized piezoelectric ceramic cylinders}, Journal of Sound and Vibration~\textbf{38}, no.~2 (1975), 245--254.

\bibitem{aar}
G.~E.~Andrews, R.~Askey and R.~Roy, \emph{Special Functions}, Encyclopedia of
Mathematics and its Applications, 71, Cambridge University Press,  Cambridge,
1999.

\bibitem{ab4}
G.~E.~Andrews and B.~C.~Berndt, \emph{Ramanujan's Lost Notebook}, Part IV,
Springer, New York, 2013.

\bibitem{apostol}
T.~M.~Apostol, \emph{Introduction to Analytic Number Theory}, Undergraduate
Texts in Mathematics, New York-Heidelberg, Springer-Verlag, 1976.



\bibitem{bII}
B.~C.~Berndt, \emph{Identities involving the coefficients of a class of Dirichlet series, I, II}, Trans.~Amer.~Math. Soc.~\textbf{ 137} (1969), 345--359; ibid.~\textbf{137} (1969) 361--374.


\bibitem{III}
B.~C.~Berndt, \emph{Identities involving the coefficients of a class of Dirichlet series, III}, Trans.~Amer.~Math. Soc.~\textbf{146} (1969), 323--348.

\bibitem{V}
B.~C.~Berndt, \emph{Identities involving the coefficients of a class of Dirichlet series, V}, Trans.~Amer.~Math. Soc.~\textbf{160} (1971), 139--156.

\bibitem{bcbconf}
B.~C.~Berndt, \emph{The Vorono\"{\dotlessi} summation formula}, in \emph{The Theory of Arithmetic Functions} (Proc. Conf. Western Michigan Univ., Kalamazoo, Mich.~1971), Lecture Notes in Math., Vol.~251, Springer, Berlin, 1972, pp.~21--36.

\bibitem{bcbrocky}
B.~C.~Berndt, \emph{On the Hurwitz zeta-function}, Rocky Mountain J.~Math.~{\bf 2}
(1972), 151--157.


\bibitem{bcbhurwitz}
B.~C.~Berndt, \emph{The gamma function and the Hurwitz zeta-function},
Amer.~Math.~Monthly {\bf 92} (1985), 126--130.

\bibitem{II}
B.~C.~Berndt, \emph{Ramanujan's Notebooks}, Part II,
Springer-Verlag, New York, 1989.


\bibitem{bclz}
B.~C.~Berndt, O.-Y.~Chan, S.-G.~Lim and A.~Zaharescu, \emph{Questionable
  claims found in Ramanujan's lost notebook}, in
{\it Tapas in Experimental Mathematics}, T.~Amdeberhan and
V.~H.~Moll, eds., Contemp.~Math., Vol.~457, American Mathematical
Society, Providence, RI, 2008, pp.~69--98.


\bibitem{besselbbskaz}
B.~C.~Berndt, S.~Kim, and A.~Zaharescu, \emph{Weighted divisor sums
and Bessel function series}, II, Adv.~Math.~\textbf{229} (2012),
2055--2097.


\bibitem{besselcrelle}
B.~C.~Berndt, S.~Kim, and A.~Zaharescu, \emph{Weighted divisor sums and
  Bessel function series},
III, J.~Reine Angew.~Math.~\textbf{683} (2013), 67--96.

\bibitem{cdp1}
B.~C.~Berndt, S.~Kim, and A.~Zaharescu, \emph{The circle and divisor problems,
  and double series of Bessel functions}, Adv. Math.~\textbf{236} (2013),
24--59.

\bibitem{besselsurvey}
B.~C.~Berndt, S.~Kim, and A.~Zaharescu, \emph{The circle and divisor problems, and Ramanujan's contributions through Bessel function series}, in \emph{Legacy of Ramanujan:  Proceedings of an International Conference
  in Celebration of the 125th Anniversary of Ramanujan's Birth: University of
  Delhi, 17--22 December 2012}, B.~C.~Berndt and D.~Prasad, eds., Ramanujan
Mathematical Society, Mysore, 2013, pp.~111--127.





\bibitem{bessel1}
B.~C.~Berndt and A.~Zaharescu, \emph{Weighted divisor sums and
Bessel function series}, Math.~Ann. {\bf 335} (2006), 249--283.


\bibitem{betcon}
S.~Bettin and J.~B.~Conrey, \emph{Period functions and cotangent sums}, Algebra and Number Theory~\textbf{7} No. 1 (2013), 215--242.

\bibitem{kosh}
N.~N.~Bogolyubov, L.~D.~Faddeev, A.~Yu. Ishlinskii, V.~N. Koshlyakov, and Yu.~A. Mitropol'skii, \emph{Nikolai Sergeevich Koshlyakov} (\emph{on the centenary of his birth}), Russ.~Math. Surveys \textbf{45} (1990), 197--202.

\bibitem{chandrasekharannarasimhan}
K.~Chandrasekharan and R.~Narasimhan, \emph{Hecke's functional equation and arithmetical identities}, Ann. of Math.~\textbf{74} (1961), 1--23.


\bibitem{chowwal}
S.~Chowla and H.~Walum, \emph{On the divisor problem}, Norske Vid. Selsk. Forh. (Trondheim)~\textbf{36} (1963), 127--134. (Proc. Sympos. Pure Math.~\textbf{8} (1965), 138--143).


\bibitem{chowlacp}
S.~Chowla, \emph{The Collected Papers of Sarvadaman Chowla},
Vol.~III, Les Publications Centre de Recherches Math\'ematiques,
Montreal, 1999.

\bibitem{cohen}
H.~Cohen, \emph{Some formulas of Ramanujan involving Bessel functions}, Publications Math\'{e}ma-tiques de Besan\c{c}on. Alg\`{e}bre et Th\'{e}orie des Nombres, 2010, 59--68.

\bibitem{cooke}
R.~G.~Cooke, \emph{Null-series involving Lommel functions}, J. London Math. Soc.~\textbf{1}, No. 1 (1930), 58--61.

\bibitem{cop}
E.~T.~Copson, \emph{Theory of Functions of a Complex Variable}, Oxford University Press, Oxford, 1935.

\bibitem{cramer}
H.~Cram\'{e}r, \emph{Contributions to the analytic theory of numbers}, Proc. 5th Scand. Math. Congress, Helsingfors, 1922, pp.~266--272.

\bibitem{davenport}
H.~Davenport, \emph{Multiplicative Number Theory}, 3rd ed.,
Springer-Verlag, New York, 2000.

\bibitem{transf}
A.~Dixit, \emph{Transformation formulas associated with integrals involving the Riemann $\Xi$-function}, Monatsh. Math.~\textbf{164}, No. 2 (2011), 133--156.

\bibitem{dixitmoll}
A.~Dixit and V.~H.~Moll, \emph{Self-reciprocal functions, powers of the Riemann zeta function and modular-type transformations}, J. Number Thy.~\textbf{147} (2015), 211--249.


\bibitem{dixfer1}
A.~L.~Dixon and W.~L.~Ferrar, \emph{Lattice-point summation formulae}, Quart.~J.~Math.~\textbf{2} (1931), 31--54.

\bibitem{dixfer2}
A.~L.~Dixon and W.~L.~Ferrar, \emph{Some summations over the lattice points of a circle \textup{(I)}}, Quart.~J.~Math.~\textbf{1} (1934), 48--63.

\bibitem{dixfer3}
A.~L.~Dixon and W.~L.~Ferrar, \emph{Infinite integrals of Bessel functions}, Quart.~J.~Math.~\textbf{1} (1935), 161--174.


\bibitem{es2}
S.~Egger and F.~Steiner, \emph{An exact trace formula and zeta functions for an infinite quantum graph with a non-standard Weyl asymptotics}, J.~Phys.~A: Math.~Theor.~\textbf{44}, No.~44 (2011), 185--202.

\bibitem{es3}
S.~Egger and F.~Steiner, \emph{A new proof of the Vorono\"{\dotlessi} summation formula}, J.~Phys.~A:~Math. Theor.~\textbf{44}, No. 22 (2011), 225--302.


\bibitem{es1}
S.~Endres and F.~Steiner, \emph{A simple infinite quantum graph}, Ulmer Seminare Funktionalanalysis und Differentialgleichungen~\textbf{14} (2009), 187--200.

\bibitem{htf}
A.~Erd\'{e}lyi, W.~Magnus, F.~Oberhettinger and F.~Tricomi, \emph{Higher Transcendental Functions, Vol.~II}, Bateman Manuscript Project, McGraw-Hill, New York, 1953.


\bibitem{erdelyi}
A.~Erd\'{e}lyi, W.~Magnus, F.~Oberhettinger and F.~Tricomi, \emph{Tables of Integral Transforms, Vol.~I}, Bateman Manuscript Project, McGraw-Hill, New York, 1954.

\bibitem{glasser}
M.~L.~Glasser, \emph{Integral representations for the exceptional univariate
  Lommel functions}, J.~Phys. A: Math.~Theor.~\textbf{43} (2010) 155--207.


\bibitem{goldstein}
S.~Goldstein, \emph{On the Vortex Theory of Screw Propellers}, Proc. R. Soc. Lond. A~\textbf{123} (1929), 440--465.

\bibitem{grn}
I.~S.~Gradshteyn and I.~M.~Ryzhik, eds., \emph{Table of Integrals,
Series, and Products}, 7th ed., Academic Press, San Diego, 2007.


\bibitem{gui}
A.~P.~Guinand, \emph{Summation formulae and self-reciprocal functions \textup{(II)}}, Quart.~J.~Math.~\textbf{10} (1939), 104--118.

\bibitem{guinand}
A.~P.~Guinand, \emph{Some rapidly convergent series for the Riemann $\xi$-function}, Quart.J.~Math. (Oxford)~\textbf{6} (1955), 156--160.

\bibitem{hafner}
J.~L.~Hafner, \emph{New omega theorems for two classical lattice point problems}, Invent.~Math.~\textbf{63} (1981), 181--186.

\bibitem{hardiv}
G.~H.~Hardy, \emph{On Dirichlet's divisor problem}, Proc.~London Math.~Soc.~(2) \textbf{15} (1916), 1--25.

\bibitem{hardycpII}
G.~H.~Hardy, \emph{Collected Papers of G.~H.~Hardy}, Vol.~II, Clarendon Press, Oxford, 1967.

\bibitem{Hardy}
G.~H.~Hardy and E.~M.~Wright, \emph{An Introduction to the Theory of
Numbers}, 6th ed., Oxford University Press, Oxford, 2008.

\bibitem{hejhal}
D.~A.~Hejhal, \emph{A note on the Vorono\"{\dotlessi} summation formula},  Monatsh. Math.~\textbf{87} (1979), 1--14.

\bibitem{huxley2}
M.~N.~Huxley, \emph{Exponential sums and lattice points}. III, Proc.~London Math.~Soc.~(3) \textbf{87} (2003), 591--609.

\bibitem{ivic}
A.~Ivi\'{c}, \emph{The Vorono\"{\dotlessi} identity via the Laplace transform}, Ramanujan J.~\textbf{2} (1998), 39--45.

\bibitem{jutila}
M.~Jutila, \emph{Lectures on a method in the theory of exponential sums}, Tata Inst.~Fund.~Res.~Lect.~Math.~Vol. 80, Springer, 1987.

\bibitem{kanrao}
S.~Kanemitsu and R.~Sitaramachandra Rao, \emph{On a conjecture of S.~Chowla and of S.~Chowla and H.~Walum, II}, J. Number Theory \textbf{20} (1985), 103--120.

\bibitem{koshliakov}
N.~S.~Koshliakov, \emph{On Voronoi's sum-formula},
Mess.~Math.~\textbf{58} (1929), 30--32.

\bibitem{koshlond}
N.~S.~Koshliakov, \emph{On an extension of some formulae of Ramanujan}, Proc. London Math. Soc., II Ser.~\textbf{41} (1936), 26--32.

\bibitem{koshli}
N.~S.~Koshliakov, \emph{Note on some infinite integrals}, C. R. (Dokl.) Acad. Sci. URSS ~\textbf{2} (1936), 247--250.

\bibitem{kosh1937}
N.~S.~Koshliakov, \emph{On a transformation of definite integrals
and its application to the theory of Riemann's function
$\zeta(s)$}, Comp.~Rend.~(Doklady) Acad.~Sci.~URSS \textbf{15}
(1937), 3--8.

\bibitem{kosh1938}
N.~S.~Koshliakov, \emph{Note on certain integrals involving Bessel functions}, Bull. Acad. Sci. URSS Ser. Math.~\textbf{2} No. 4, 417--420; English text 421--425 (1938).


\bibitem{lau}
Y.-K.~Lau, \emph{On a generalized divisor problem I}, Nagoya Math. J.~\textbf{165} (2002), 71--78.

\bibitem{al}
A.~ Laurin\u{c}ikas, \emph{Once more on the function $\sigma_{A}(M)$}, Lithuanian Math. J.~\textbf{32}, No. 1 (1992), 81--93; translation from Liet. Mat. Rink.~\textbf{32}, No. 1 (1992), 105--121.

\bibitem{lewzag}
J.~Lewis and D.~Zagier, \emph{Period functions for Maass wave forms. \textup{I}}, Ann.~Math.~\textbf{153}, No. 1 (2001), 191--258.

\bibitem{littlewood}
J.~E.~Littlewood, \emph{Lectures on the Theory of Functions}, Oxford Univ. Press, New York, 1944.

\bibitem{meurmanast}
T.~Meurman, \emph{A simple proof of Vorono\"{\dotlessi}'s identity}, Ast\'{e}risque~\textbf{209} (1992), 265--274.

\bibitem{meurman}
T.~Meurman, \emph{The mean square of the error term in a generalization of Dirichlet's divisor problem}, Acta Arith.~\textbf{74}, No. 4 (1996), 351--364.

\bibitem{ob}
F.~Oberhettinger, \emph{Tables of Mellin Transforms}, Springer-Verlag,
Berlin, 1974.

\bibitem{os}
F.~Oberhettinger, and K.~Soni, \emph{On some relations which are equivalent to functional equations involving the Riemann zeta function}, Math.~Z.~\textbf{127} (1972), 17--34.

\bibitem{nist}
F.~W.~J.~Olver, D.~W.~Lozier, R.~F.~Boisvert, and C.~W.~Clark, eds., \emph{NIST Handbook of Mathematical Functions}, Cambridge University Press, Cambridge, 2010.

\bibitem{oppenheim}
A.~Oppenheim, \emph{Some identities in the theory of numbers}, Proc. London Math. Soc.~\textbf{2}, No. 1 (1927), 295--350.

\bibitem{kp}
R.~B.~Paris and D.~Kaminski, \emph{Asymptotics and Mellin-Barnes Integrals},  Encyclopedia of Mathematics and its Applications, 85, Cambridge University Press, Cambridge, 2001.

\bibitem{prud}
A.~P.~Prudnikov, I.~A.~Brychkov and O.~I.~Marichev, \emph{Integrals and Series: Special Functions}, Vol.~2, Gordon and Breach, New York, 1986.



\bibitem{secnb}
S.~Ramanujan, \emph{Notebooks} (2 volumes), Tata Institute of Fundamental
Research, Bombay, 1957; second ed, 2012.

\bibitem{lnb}
S.~Ramanujan, \emph{The Lost Notebook and Other Unpublished
Papers}, Narosa, New Delhi, 1988.

\bibitem{rollinger}
C.~N.~Rollinger, \emph{Lommel functions with imaginary argument}, Quart.~Appl.~Math.~\textbf{21} (1963), 343--349.

\bibitem{segal}
S.~L.~Segal, \emph{A note on the average order of number-theoretic error
  terms}, Duke Math. J.~\textbf{32} No.~2, (1965), 279--284.

\bibitem{shafer}
R.~E.~Shafer, \emph{Lommel functions of imaginary argument}, Technical report, Lawrence Radiation Laboratory, Livermore, (1963).

\bibitem{sitzer}
M.~R.~Sitzer, \emph{Stress distribution in rotating aeolotropic laminated
  heterogeneous disc under action of a time-dependent
loading}, Z. Angew. Math. Phys.~\text{36} (1985), 134--145.

\bibitem{sound}
K.~Soundararajan, \emph{Omega results for the divisor and circle problems},
Internat.~Math.~Res. Not. (2003), No. 36, 1987--1998.

\bibitem{temme}
N.~M.~Temme, \emph{Special Functions: An Introduction to the Classical Functions of Mathematical Physics}, Wiley-Interscience Publication, New York, 1996.

\bibitem{thomaschan}
B.~K.~Thomas and F.~T.~Chan, \emph{Glauber $e^{-} +$ He elastic scattering amplitude: A useful integral representation}, Phys. Rev. A~\textbf{8}, 252--262.

\bibitem{titch}
E.~C.~Titchmarsh, \emph{The Theory of the Riemann Zeta Function}, Clarendon
Press, Oxford, 1986.


\bibitem{torrey}
H.~C.~Torrey, \emph{Bloch equations with diffusion terms}, Phys. Rev.~\textbf{104} (1956), 563--565.

\bibitem{tsang}
K.--M.~Tsang, \emph{Recent progress on the Dirichlet divisor problem and the mean square of the Riemann zeta-function}, Sci.~China Math.~\textbf{53} (2010), 2561--2572.

\bibitem{voronoi}
G.~F.~Vorono\"{\dotlessi}, \emph{Sur une fonction transcendante et
ses applications \`a la sommation de quelques s\'eries},
Ann.~\'Ecole Norm.~Sup.~(3) \textbf{21} (1904), 207--267,
459--533.

\bibitem{watsonbessel}
G.~N.~Watson, \emph{Theory of Bessel Functions}, 2nd ed.,
University Press, Cambridge, 1966.


\bibitem{wig}
S.~Wigert, \emph{Sur une extension de la s\'{e}rie de Lambert}, Arkiv Mat.~Astron.~Fys.~\textbf{19} (1925), 13 pp.


\bibitem{wilton}
J.~R.~Wilton, \emph{Vorono\"{\dotlessi}'s summation formula}, Quart.~J.~Math.~\textbf{3} (1932), 26--32.

\bibitem{wiltonextended}
J.~R.~Wilton, \emph{An extended form of Dirichlet's divisor problem}, Proc. London Math. Soc. (1934), s2-36 (1): 391--426.

\bibitem{zs}
C.~H.~Ziener and H.~P.~Schlemmer, \emph{The inverse Laplace transform of the modified Lommel functions}, Integral Transforms Spec.~Func.~\textbf{24}, No. 2 (2013), 141--155.

\bibitem{zygmund}
A.~Zygmund, \emph{On trigonometric integrals}, Ann. of Math.~\textbf{48} (1947), 393--440.

\bibitem{zygmundtrig}
A.~Zygmund, \emph{Trigonometric series. Vol. I, II.
Third edition. With a foreword by Robert A. Fefferman}. Cambridge Mathematical Library. Cambridge University Press, Cambridge, 2002.
\end{thebibliography}
\end{document}